\documentclass[11pt]{article}
\usepackage{amsmath, amssymb, amsthm, esint, hyperref, cite, graphics, graphicx}
\usepackage{geometry}

\usepackage{xcolor}

 \newcommand{\bx}{\boldsymbol{x}}
 \newcommand{\bu}{\boldsymbol{u}}

\theoremstyle{definition}

\newtheorem{theorem}{Theorem}[section]

\newtheorem{definition}{Definition}[section]
\newtheorem{proposition}{Proposition}[section]
\newtheorem{remark}{Remark}[section]
\newtheorem{lemma}{Lemma}[section]
\begin{document}

\title{\sc{Well-posedness of solutions to stochastic fluid-structure interaction}}
\author{Jeffrey Kuan and Sun\v{c}ica \v{C}ani\'{c}\\
Department of Mathematics\\
University of California Berkeley}
\maketitle
\begin{abstract}
In this paper we introduce a constructive approach to study well-posedness of solutions to stochastic fluid-structure interaction with stochastic noise.
We focus on a benchmark problem
in stochastic fluid-structure interaction, and prove the existence of a unique weak solution in the  probabilistically strong sense. 
The benchmark problem consists of the 2D time-dependent Stokes equations describing the flow of an incompressible, viscous fluid 
interacting with  a linearly elastic membrane modeled by the 1D linear wave equation. The membrane is stochastically forced by the time-dependent 
white noise. The fluid and the structure are linearly coupled.
The constructive existence proof is based on a time-discretization via an operator splitting approach. This introduces
a sequence of approximate solutions, which are random variables. We show the existence of a subsequence of approximate solutions
which converges, almost surely, to a weak solution in the probabilistically strong sense.
The proof is based on uniform energy estimates in terms of the {\emph{expectation}} of the energy norms,
which are the backbone for a weak compactness argument giving 
rise to a weakly convergent subsequence of {\emph{probability measures}} associated with the approximate solutions.
Probabilistic techniques
based on the Skorohod representation theorem and the Gy\"ongy-Krylov lemma are then employed to obtain almost sure convergence
of a subsequence of the random approximate solutions to a weak solution in the probabilistically strong sense.
The result shows that the deterministic benchmark FSI model is robust to stochastic noise, even in the presence of rough white noise in time.
To the best of our knowledge, this is the first well-posedness result for fully coupled stochastic fluid-structure interaction.
\if 1 = 0
Once we have established the existence of a subsequence of probability measures that converges to some probability measure 
as the time step goes to zero,
we would like to show that on a further subsequence, there exist the {\emph{random variables}} 
 that will converge almost surely to a random
variable which is a weak solution, almost surely, in the probabilistically strong sense.
Showing this almost sure convergence with respect, however, has to be done in two parts. 
In the first part we get a hold of a subsequence of approximate solutions that converge almost surely {\bf{but}} on another probability space,
and then use this information in the second part to construct a convergent subsequence of approximate solutions that converge on the {\emph{original}} probability space, by invoking 
\fi
\end{abstract}
\section{Introduction}
In this paper, we introduce a constructive approach to study solutions of stochastic fluid-structure interaction (SFSI) with stochastic noise.
This manuscript is written as an introduction to the use of stochastic techniques to study SFSI, and is aimed at audiences
that have experience with deterministic FSI, but may be new to stochastic analysis. 
We focus on a benchmark problem 
 in which a stochastically forced linearly elastic membrane interacts with the flow of a 
viscous incompressible Newtonian fluid in two spatial dimensions. 
The membrane is modeled by the linear wave equation, while the fluid is modeled by the 2D time-dependent Stokes equations. 
The problem is forced by a ``rough'' stochastic forcing given by a time-dependent white noise $\dot{W}(t)$,
where $W$ is a given one-dimensional Brownian motion with respect to a complete probability space 
$(\Omega, \mathcal{F}, \mathbb{P})$ 
with complete filtration $\{\mathcal{F}_{t}\}_{t \ge 0}$.
The fluid and the membrane are coupled via a two-way coupling
describing continuity of fluid and structure velocities at the fluid-structure interface, and continuity of contact forces at
the interface. The coupling is calculated at the linearized, fixed interface, rendering this problem a linear stochastic fluid-structure 
interaction problem. The goal is to show that despite the rough white noise, the resulting problem is well-posed, showing that
the underlying deterministic fluid-structure interaction problem is robust to noise. Indeed, we prove the existence of a 
unique weak solution in the probabilistically strong sense (see Definition~\ref{strong} in Section~\ref{sec:weak}) to this 
stochastic fluid-structure interaction problem. 
This means that there exist unique random variables (stochastic processes),
 describing the fluid velocity $\boldsymbol{u}$, the structure velocity $v$, and the structure displacement $\eta$, 
 such that those stochastic processes are adapted to the filtration $\{\mathcal{F}_{t}\}_{t \ge 0}$,
 i.e., they only depend on the past history of the processes up to time $t$ and not on the future,
which satisfy the weak formulation of the original problem almost surely. This is the main result of this manuscript.

To prove the existence of a unique weak solution in the probabilistically strong sense,
we design a constructive existence proof.
The constructive existence proof is based on semi-discretizing the problem in time by dividing the time interval $(0,T)$ into $N$
subintervals of width $\Delta t=T/N$, and  using a time-splitting scheme,
introduced in \cite{BGR}, to construct approximate solutions. 
The goal is to show that the approximate solutions converge almost surely with respect to a certain topology, to the unique weak solution as  $\Delta t$ goes to zero. 
In contrast to the  deterministic case, see  the works of Muha and \v{C}ani\'c in \cite{MuhaCanic13, BorSunMultiLayered,BorSun3d}, where a time-discretization
via operator splitting approach was used to study existence of weak solutions, 
the splitting scheme for the current problem involving stochastic noise needs to be constructed more carefully to obtain stability, see Section~\ref{split_scheme}.
In particular, the problem has to be split so that
the stochastic part is considered separately from the deterministic part, and the fluid and structure problems are split
and solved in a particular order
so that the resulting stochastic integrals involving the stochastic noise increments
can be evaluated and estimated to prove stability. 
See Remark~\ref{stochastic_increment} in Section~\ref{SemidiscreteProblem}.
More precisely, along each time sub-interval $(t^n_{N} ,t^{n+1}_{N}), n=0,\dots,N-1$, the following three sub-problems are solved to obtain
approximate solutions consisting of the fluid and structure velocities, and the structure displacement, $(\boldsymbol{u},v,\eta)$.
First, in {\bf{Step 1}}, the structure displacement and structure velocity are updated using only the structure displacement and structure velocity from the previous time step. The resulting random variables are measurable with respect to the sigma algebra ${\cal{F}}_{t^n_{N}}$. 
Then, in {\bf{Step 2}}, which is the stochastic step, 
 the structure velocity is updated by adding to the structure velocity calculated in Step 1
 the stochastic noise increment from time step ${t^n_{N}}$ to time step ${t^{n+1}_{N}}$. Since the structure velocity obtained in Step 1
 is a random variable that is measurable with respect to the sigma algebra ${\cal{F}}_{t^{n}_{N}}$, 
 and the stochastic increment from ${t^n_{N}}$ to ${t^{n+1}_{N}}$ is independent of it,
 we will be able to obtain boundedness of the stochastic integral involving these two quantities
 by using their independence. This will lead to stability.
The resulting updated structure velocity is a random variable that is measurable with respect to the sigma algebra ${\cal{F}}_{t^{n+1}_{N}}$.
Finally, in {\bf{Step~3}}, the fluid and structure velocities are updated by using the information from the just calculated structure velocity in Step 2.
This gives rise
to random variables that are measurable with respect to the sigma algebra ${\cal{F}}_{t^{n+1}_{N}}$.
We would like to show that the sequence or a subsequence of random variables constructed this way converges in a certain topology to
a weak solution in the probabilistically strong sense of the coupled SFSI problem. 

Based on this splitting scheme, uniform energy estimates in terms of expectation can be derived.
In addition to estimating the expectation of the kinetic and elastic energy of the problem, it is important to get 
a uniform bound on the expectation of the numerical dissipation, to show that the numerical dissipation is bounded and that it
in fact, approaches zero as the time step $\Delta t$ goes to zero, which is crucial in the convergence proof. 
This is provided in Proposition~\ref{uniformenergy}. Furthermore, another interesting observation is that the energy estimates
will have an extra term on the right-hand side which accounts for the energy pumped into the problem by the stochastic noise.
This is in addition to the energy/work contributions by the initial and boundary data. 
These energy estimates define an energy function space for the unknown functions $(\boldsymbol{u},v,\eta)$.
A separable subspace of the energy space, specified in \eqref{phase} in Section~\ref{sec:measures_weak_conv}
is called a {\emph{phase space}}, and is denoted by ${\cal{X}}$.

Uniform estimates are the backbone for weak compactness, giving rise to convergent subsequences whose limits 
potentially satisfy the original problem in a certain sense. 
In the deterministic case, the uniform energy estimates typically imply existence of weakly- and weakly*-convergent subsequences
in the appropriate topologies, which is usually sufficient to pass to the limit in the linear problem and recover the weak solution. 
This is, however, not the case with the 
stochastic problem. The main reason is that the uniform energy estimates are given in expectation 
-- they hold on average over all realizations, not pathwise for each outcome $\omega \in \Omega$ in the underlying probability space. 
To deal with this issue, a {\emph{compactness argument}} needs to be invoked, even though the underling stochastic FSI problem is linear.
The compactness argument is used to obtain the existence of weakly convergent subsequences
of {\emph{probability measures}}, or {\emph{laws}}, describing the {\emph{distributions}} of the random approximate solutions.
Once weak convergence of probability measures is established, one can work on getting an 
 almost sure convergence of the random approximate solutions, which is necessary to recover the weak solution.
 
To establish
weak convergence of probability measures, one must show that the probability measures are {\bf{tight}}.
More precisely, one must show that for each $\epsilon>0$,
there exists a {\bf{compact set}} in the phase space ${\cal{X}}$ of 
displacements and fluid and structure velocities, such that the probability that our approximate solutions 
$(\boldsymbol{u}_N,v_N,\eta_N)$ live in that compact set is greater than $1-\epsilon$. 
See Definition~\ref{tight} for tightness of measures. 
The proof of tightness of the sequence of probability measures $\mu_N$ corresponding to the laws of the approximate solutions $(\boldsymbol{u}_N,v_N,\eta_N)$  will follow
from a {\emph{deterministic}} compactness argument alla Aubin-Lions. The compactness argument
will establish the existence of 
a compact subset of the phase space ${\cal{X}}$ that contains the approximate solutions $(\boldsymbol{u}_N,v_N,\eta_N)$ with probability greater than $1 - \epsilon$, thus verifying the tightness property.

Once we have established the existence of a subsequence of probability measures $\mu_N$ that converges weakly to some probability measure $\mu$ as $N \to \infty$, or equivalently,
as $\Delta t \to 0$, we would like to show that on a further subsequence, the {\emph{random variables}} $(\boldsymbol{u},v,\eta)_N$
will converge almost surely to a random
variable with the law $\mu$, with respect to the probability space $(\Omega, \mathcal{F}, \mathbb{P})$. 
Showing this almost sure convergence with respect to the probability space $(\Omega, \mathcal{F}, \mathbb{P})$, however, has to be done in two parts. 
In the first part, we get a hold of a subsequence of approximate solutions that converge almost surely {\bf{but}} on another probability space,
and then use this information in the second part to construct a convergent subsequence of approximate solutions that converge on the {\emph{original}} probability space. The following is a more detailed albeit succinct description of the two parts. 

{\bf{Part 1.}} We use the Skorohod representation theorem to deduce that there exists a sequence of random variables 
$(\tilde{\boldsymbol{u}},\tilde{v},\tilde{\eta})_N$,
defined on {\bf a probability space 
$(\tilde\Omega, \tilde{\mathcal{F}}, \tilde{\mathbb{P}})$}, which is {\emph{not necessarily the same as}} the original probability space 
$(\Omega, \mathcal{F}, \mathbb{P})$, such that the laws of $(\tilde{\boldsymbol{u}},\tilde{v},\tilde{\eta})_N$ are $\mu_N$, and $(\tilde{\boldsymbol{u}},\tilde{v},\tilde{\eta})_N$ converge almost surely to a random variable $(\tilde{\boldsymbol{u}},\tilde{v},\tilde{\eta})$ with the law $\mu$, 
 on the ``tilde'' probability space.
 On this ``tilde'' probability space we also show that
the almost sure limit $(\tilde{\boldsymbol{u}},\tilde{v},\tilde{\eta})$ satisfies the weak formulation
of the original problem  almost surely, {\bf{but}} with respect to the ``tilde'' probability space. This means 
that this limit is a {\bf{weak solution to the original problem in the probabilistically {\emph{weak}} sense}},
see Definition~\ref{weak}. 
This result will be useful in showing the existence of a unique weak solution in the probabilistically {\emph{strong}} sense
on the {\emph{original}} probability space $(\Omega, \mathcal{F}, \mathbb{P})$, discussed in the second part.

{\bf{Part 2.}} 
We would like to be able to prove that our sequence of approximate solutions $(\boldsymbol{u},v,\eta)_N$,
obtained using our time-discretization via operator splitting approach described above, converges almost surely to a random variable 
$(\boldsymbol{u},v,\eta)$ on the original probability space, and satisfies the weak formulation almost surely on the original probability space.
Namely, we would like to prove that the limit is   a {\bf{weak solution to the original problem in the probabilistically {\emph{strong}} sense}}.
If we could obtain that the sequence $(\boldsymbol{u},v,\eta)_N$ converges {\bf{in probability}} to a random variable on the original probability space
$(\Omega, \mathcal{F}, \mathbb{P})$, namely $(\boldsymbol{u},v,\eta)_N  \xrightarrow[ ]{p}  (\boldsymbol{u},v,\eta)$,
then the almost sure convergence along a subsequence will follow immediately. 
To obtain convergence {\emph{in probability}} of $(\boldsymbol{u},v,\eta)_N$,
we will invoke a standard  Gy\"{o}ngy-Krylov argument \cite{GK}.

More precisely, to prove that $X_N=(\boldsymbol{u},v,\eta)_N$ converge {\emph{in probability}} to some random variable $X^*=(\boldsymbol{u},v,\eta)$ 
on $(\Omega, \mathcal{F}, \mathbb{P})$, $X_N  \xrightarrow[ ]{p}  X^*$, based on the Gy\"{o}ngy-Krylov lemma  \cite{GK},
we need to show that for every two subsequences $X_l$ and $X_m$, there exists a subsequence $x_k=(X_{l_k},X_{m_k})$
such that the following two properties hold:
\begin{enumerate}
\item The joint laws  $\nu_{X_{l_k},X_{m_k}}$ of the subsequence $x_k$ converge to some probability measure $\nu$ as $k\to\infty$;
\item The limiting law is supported on the diagonal: $\nu(\{(X, Y) : X = Y\}) = 1.$
\end{enumerate}
The first property will follow from the tightness of measures $\mu_l$ and $\mu_m$, which are the laws associated with the random variables 
$X_l=(\boldsymbol{u},v,\eta)_l$ and $X_m=(\boldsymbol{u},v,\eta)_m$. The tightness of the measures $\mu_l$ and $\mu_m$
 implies tightness of the joint measures $\nu_{X_{l},X_{m}}$ as well.
To show that the second property holds, we will use the result of Part 1 above, combined with a {\emph{deterministic uniqueness}} argument. 
Namely, Part 1 gives us the existence of the almost surely convergent subsequences  $\tilde{X}_l=(\tilde{\boldsymbol{u}},\tilde{v},\tilde\eta)_l$ and $\tilde{X}_m=(\tilde{\boldsymbol{u}},\tilde{v},\tilde\eta)_m$
on the ``tilde'' probability space that have the same
laws $\mu_l$ and $\mu_m$ as $X_l=(\boldsymbol{u},v,\eta)_l$ and $X_m=(\boldsymbol{u},v,\eta)_m$. 
Those two ``tilde'' subsequences of random variables converge to the limits $\tilde{X}^1$ and $\tilde{X}^2$, respectively, each of which has the law $\mu$,
and a joint law of $(\tilde{X}^1,\tilde{X}^2)$ equal to  $\nu$ from Property 1 above. 
Recall, from Step 1, that both $\tilde{X}^1$ and $\tilde{X}^2$ are weak solutions in the probabilistically weak sense. 
To show that this joint law $\nu$ is supported on the diagonal, namely, to show Property 2 above, it is sufficient to show that $\tilde{X}^1$ is equal to $\tilde{X}^2$ almost surely,
namely it will be sufficient to show that $\tilde{\mathbb{P}}(\tilde{X}^{1} = \tilde{X}^{2}) =1$.
Indeed, proving the diagonal condition from the Gy\"ongy-Krylov lemma
is associated with proving pathwise {{uniqueness}} of weak solutions, which we present in Section~\ref{deterministic}. 

Once the properties from the Gy\"{o}ngy-Krylov lemma have been verified, we can conclude that there exists a subsequence of 
$(\boldsymbol{u},v,\eta)_N$, which we continue to denote by $N$,
 such that $(\boldsymbol{u},v,\eta)_N \xrightarrow[ ]{p} (\boldsymbol{u},v,\eta)$, which implies almost sure
convergence along a subsequence on the original probability space. This is presented in Section~\ref{GKlemma}.

Finally, the proof that the limiting function $(\boldsymbol{u},v,\eta)$ recovered above is a weak solution in the probabilistically strong sense
is presented in Section~\ref{final}.

To the best of our knowledge, this is the first well-posedness result in the context of stochastic fluid-structure interaction.
The result shows that our deterministic benchmark FSI model is robust to stochastic noise, even in the presence of rough white noise in time.
This proof combines stochastic PDE analysis tools with deterministic FSI approaches.
Additionally, the constructive proof lays out a framework for the development of a numerical scheme
for this class of SFSI problems. 
In the next section, we provide a brief review of the related literature.

\if 1 = 0

More precisely, to prove that $X_N=(\boldsymbol{u},v,\eta)_N$ converge {\emph{in probability}} to some random variable $X^*=(\boldsymbol{u},v,\eta)$ 
on $(\Omega, \mathcal{F}, \mathbb{P})$, based on the Gy\"{o}ngy-Krylov lemma 
we need to show that for every two subsequences $X_l$ and $X_m$, there exists a subsequence
such that the following two properties hold:
\begin{enumerate}
\item The joint laws $\nu_{X_{l_k},X_{m_k}}$ converge to some probability measure $\nu$ as $k\to\infty$;
\item The limiting law is supported on the diagonal: $\nu(\{(X, Y) : X = Y\}) = 1.$
\end{enumerate}
The first property will follow from the tightness of measures $\mu_l$ and $\mu_m$, which are the laws associated with the random variables 
$X_l=(\boldsymbol{u},v,\eta)_l$ and $X_m=(\boldsymbol{u},v,\eta)_m$. This implies tightness of the joint measures as well.
To show that the second property holds, we will use the result of Part 1, above. 
Namely, Part 1 gives us the existence of the almost surely convergent subsequences on the ``tilde'' probability space that have the same
laws $\mu_l$ and $\mu_m$ as $X_l=(\boldsymbol{u},v,\eta)_l$ and $X_m=(\boldsymbol{u},v,\eta)_m$. 
Those two ``tilde'' subsequences of random variables converge to the limits $\tilde{X}^1$ and $\tilde{X}^2$, each of which has the law $\mu$,
and a joint law of $(\tilde{X}^1,\tilde{X}^2)$ equal to  $\nu$ from property 1. If we could show that this limiting joint law is such that it is only supported
when the two random variables for which the joint law is defined (such as $\tilde{X}^1$ and $\tilde{X}^2$) ``are the same'', then the diagonal 
condition from the Gy\"ongy-Krylov lemma would be satisfied. 
Indeed, proving the diagonal condition from the Gy\"ongy-Krylov lemma
is associated with proving {\emph{deterministic uniqueness}} of weak solutions, which will imply that 
$\tilde{X}^{1} = \tilde{X}^{2}$ pathwise almost surely, or $\tilde{\mathbb{P}}(\tilde{X}^{1} = \tilde{X}^{2}) =1$.
This verifies the second property of the Gy\"ongy-Krylov lemma,
and completes the proof showing that on a subsequence, which we continue to denote by $N$,
$(\boldsymbol{u},v,\eta)_N \xrightarrow[ ]{P} (\boldsymbol{u},v,\eta)$. 
The almost sure convergence on $(\Omega, \mathcal{F}, \mathbb{P})$
then follows immediately. 

More precisely, the Gy\"{o}ngy-Krylov Lemma, see Lemma~\ref{GK} in Section~\ref{original_space}, states that a sequence of random variables $X_n$
defined on a probability space $(\Omega, \mathcal{F}, \mathbb{P})$ converges {\bf{in probability}} to some random variable $X^*$,
denoted by $X_n \xrightarrow[ ]{P}  X^*$, if and only if for every two subsequences $X_l$ and $X_m$
of $X_n$, there exists a subsequence $x_k=(X_{l_k},X_{m_k})$ converging {\emph{in law} to a random variable $x^*$, 
and the limiting law of $x^*$ is supported on the diagonal. 
Thus to prove that $X_N=(\boldsymbol{u},v,\eta)_N$ converge {\emph{in probability}} to some random variable $X^*=(\boldsymbol{u},v,\eta)$ 
on $(\Omega, \mathcal{F}, \mathbb{P})$, we need to show that for every two subsequences $X_l$ and $X_m$, there exists a subsequence
such that the following two properties hold:
\begin{enumerate}
\item The joint laws $\nu_{X_{l_k},X_{m_k}}$ converge to some probability measure $\nu$ as $k\to\infty$;
\item The limiting law is supported on the diagonal: $\nu(\{(X, Y) : X = Y\}) = 1.$
\end{enumerate}
The first property will follow from the tightness of measures $\mu_l$ and $\mu_m$, which are the laws associated with the random variables 
$X_l=(\boldsymbol{u},v,\eta)_l$ and $X_m=(\boldsymbol{u},v,\eta)_m$, which implies tightness of the joint measures as well.
To show that the second property holds, we will use the result of Part 1, above. 
Namely, Part 1 gives us the existence of the almost surely convergent subsequences on the ``tilde'' probability space that have the same
laws $\mu_l$ and $\mu_m$ as $X_l=(\boldsymbol{u},v,\eta)_l$ and $X_m=(\boldsymbol{u},v,\eta)_m$. 
Those two ``tilde'' subsequences of random variables converge to the limits $\tilde{X}^1$ and $\tilde{X}^2$, each of which has the law $\mu$,
and a joint law of $(\tilde{X}^1,\tilde{X}^2)$ equal to  $\nu$ from property 1. If we could show that this limiting joint law is such that it is only supported
when the two random variables for which the joint law is defined (such as $\tilde{X}^1$ and $\tilde{X}^2$) ``are the same'', then the diagonal 
condition from the Gy\"ongy-Krylov lemma would be satisfied. 
Indeed, proving the diagonal condition from the Gy\"ongy-Krylov lemma
is associated with proving {\emph{deterministic uniqueness}} of weak solutions, which will imply that 
$\tilde{X}^{1} = \tilde{X}^{2}$ pathwise almost surely, or $\tilde{\mathbb{P}}(\tilde{X}^{1} = \tilde{X}^{2}) =1$.
This verifies the second property of the Gy\"ongy-Krylov lemma,
and completes the proof showing that on a subsequence, which we continue to denote by $N$,
$(\boldsymbol{u},v,\eta)_N \xrightarrow[ ]{P} (\boldsymbol{u},v,\eta)$. 
The almost sure convergence on $(\Omega, \mathcal{F}, \mathbb{P})$
then follows immediately. 

\fi


\section{Literature review}


The mathematical analysis of deterministic fluid-structure interaction 
began around twenty years ago by focusing on  rigorous well-posedness  for linearly coupled fluid-structure interaction models. Linearly coupled FSI models are models where the fluid and structure coupling conditions are evaluated along a fixed fluid-structure interface, and the fluid equations are posed on a fixed fluid domain, even though the structure is assumed to be elastic and displaces from its reference configuration. The results concerning these linearly coupled models typically deal with establishing existence/uniqueness of weak or strong solutions. The existence and uniqueness of a weak solution to a linearly coupled model involving an interaction between the linear Stokes equations and the equations of linear elasticity was established in \cite{Gunzburger} using a Galerkin method. The Navier-Stokes equations for an incompressible, viscous fluid
linearly coupled to immersed elastic solids were considered in \cite{BarGruLasTuff2,BarGruLasTuff,KukavicaTuffahaZiane}.
In particular, the work in \cite{BarGruLasTuff2} deals with showing the existence of energy-level weak solutions, by a careful examination of the trace regularity of the hyperbolic structure dynamics in terms of the normal stress  at the fluid-structure interface. The results in \cite{BarGruLasTuff,KukavicaTuffahaZiane} deal with establishing sufficient regularity of initial data that provides existence of strong solutions of the corresponding linearly coupled systems.

The well-posedness analysis of deterministic FSI models was extended later to nonlinearly coupled models, where the fluid domain changes in time according to the structure displacement, and hence the problem is a moving boundary problem where the fluid domain is not known a priori. There is by now an extensive mathematical literature dealing with the well-posedness of such models, see e.g., \cite{BdV1,CDEM,ChengShkollerCoutand,ChenShkoller,CSS1,CSS2,CG,Grandmont16,FSIforBIO_Lukacova,IgnatovaKukavica,ignatova2014well,Kuk,LengererRuzicka,Lequeurre,MuhaCanic13,
BorSun3d,BorSunMultiLayered,BorSunNonLinearKoiter,BorSunSlip,Raymond} and the references therein. Of these references, we note that the approach outlined in \cite{MuhaCanic13, BorSun3d,BorSunMultiLayered,BorSunNonLinearKoiter,BorSunSlip,withRoland} is closely related to the approach used  in the current manuscript. 
In particular, the approach is based on using a splitting scheme, known as the Lie operator splitting scheme, that discretizes the nonlinearly coupled problem in time by a time step $\Delta t$, and separates the coupled problem into fluid and structure subproblems. Then, compactness arguments of Aubin-Lions type (see \cite{aubin1963theoreme,lions1969quelques,MuhaCanicCompactness}) are used to pass to the limit as $\Delta t \to 0$ in the approximate weak formulations satisfied by the approximate solutions, in order to obtain a \textit{constructive existence proof} for weak solutions to nonlinearly coupled fluid-structure interaction problems. This approach proved to be quite robust for {\emph{deterministic}} fluid-structure interaction problems, since it provided existence of weak solutions
 for several different 
scenarios involving thin, thick, and multi-layered structures coupled to the flow of an incompressible, viscous fluid via the no-slip or Navier slip boundary conditions,
see \cite{MuhaCanic13}, \cite{BorSun3d}, \cite{BorSunMultiLayered}, \cite{BorSunNonLinearKoiter}, \cite{BorSunSlip}.

In the present work, a version of this 
approach is extended to deal with {\emph{stochastic}} fluid-structure interaction problems,
by combining stochastic calculus with 
stochastic operator splitting approaches
introduced in \cite{BGR} and analyzed in \cite{GK2}.
%
%
More precisely, we design a time-discretized, operator splitting method in just the right way so that all the stochastic integrals are well-defined,
and the resulting time-discretized scheme is stable, allowing us to show, using stochastic calculus,  an almost sure convergence of approximate solutions to a weak solution
in the probabilistically strong sense of the coupled fluid-structure interaction problem.
To the best of our knowledge, this is the first well-posedness 
result on fully coupled \textit{stochastic fluid-structure interaction}.
Our result builds on recent developments in the area of \textit{stochastic partial differential equations} (SPDEs).

\textit{Stochastic partial differential equations}
are PDEs that feature some sort of random noise forcing, such as white noise forcing in either time, or both time and space, or spatially homogeneous Gaussian noise that is independent at every time but potentially correlated in space. They are motivated by the fact that many 
real-life systems modeled by PDEs exhibit some type of random noise, which can significantly impact the resulting dynamics of the system. 
The current manuscript considers a stochastic linearly coupled fluid-structure interaction model involving the interaction between a fluid modeled by the linear Stokes equations and an elastic membrane modeled by the wave equation. Although 
the coupled stochastic FSI model has not been previously considered in the stochastic PDE literature, there are many works that study either stochastic fluid dynamics or stochastic wave equations separately, as we summarize below. 


In terms of stochastic fluid equations, the consideration of stochastic Navier-Stokes equations is an active area of research, see e.g., 
\cite{BTNS,CGNS,FlandoliGatarek,KukavicaStochasticNS}. The study of stochastic Navier-Stokes equations was initiated in the work of \cite{BTNS}, which considered an abstract stochastic equation of Navier-Stokes type, with an additive random noise forcing in time, and a random initial condition. It was shown that there exists a solution that satisfies the problem almost surely in a distributional sense. In the works of \cite{CGNS, FlandoliGatarek}, this abstract equation of Navier-Stokes type is extended to more general settings where there is nonlinear dependence of the intensity of the random noise forcing on the actual solution itself. These two works consider different abstract conditions on this nonlinear dependence and prove existence of martingale, or probabilistically weak, solutions to the resulting stochastic equations. Both of these works use a Galerkin scheme to construct solutions and obtain existence by establishing uniform bounds on the sequence of random functions satisfying the finite-dimensional Galerkin problems. 
We note that passing to the limit in the Galerkin solutions in \cite{CGNS, FlandoliGatarek} was done 
by using standard probabilistic methods, such as establishing tightness of laws, showing weak convergence in law, and invoking the Skorohod representation theorem, which are standard techniques that we will employ for our current problem as well. 
While there are many works on stochastic fluid dynamics, we mention in particular a recent work \cite{LNT}, which establishes the existence of local martingale solutions, which are martingale solutions up to some stopping time, for a system of one layer shallow water equations for fluid velocity and water depth in two spatial dimensions, driven by random noise forcing described by cylindrical Wiener processes. We remark that \cite{LNT} employs similar probabilistic methods in passing to the limit in a sequence of random approximate solutions (obtained by a Galerkin method) that motivated many of the probabilistic arguments in this manuscript, though the methods used in this current manuscript for constructing approximate solutions are different, as they are based on time discretization
using an operator splitting approach, and not spatial discretization using a Galerkin method. One reason for the use of time-discretization via operator splitting,
versus a Galerkin approach, is a possible extension to the moving boundary case. In the Galerkin case, the basis functions for the moving boundary case will 
depend on the \textit{random} solution itself, which is difficult to deal with.

In terms of stochastic wave equations, there is extensive work on well-posedness and properties of solutions. It is classically well-known that the stochastic wave equation with spacetime white noise has a mild solution only in dimension one, but not in dimensions two and higher (see for example \cite{DMini}). This is due to the fact that the fundamental solution of the linear wave equation in dimension two is not square integrable in spacetime, and in higher dimensions, it is not even function-valued. Hence, work on the stochastic wave equation in dimensions two and higher, focuses on considering stochastic wave equations with a more general type of noise, such as spatially homogeneous Gaussian noise (see for example \cite{DaPrato}) which is independent in time but correlated in space. 
In particular,  the authors of \cite{Dalang, DF, KZ} consider conditions for this spatially homogeneous Gaussian noise, such that the resulting stochastic wave equation has a solution that is function-valued (rather than just a distribution) in dimensions two and higher. Existence results for such stochastic wave equations in higher dimensions are also considered in \cite{CD}, and the H\"{o}lder continuity and regularity properties of stochastic wave equations in higher dimensions are considered in  \cite{CD, DSS}. 

We conclude this literature review by mentioning a recent work \cite{SVWE} by the current authors,  where a \textit{stochastic viscous wave equation} was 
derived as a model for a stochastic linearly coupled fluid-structure interaction problem in a  geometry that allowed 
the entire fluid-structure system to be modeled by a single {stochastic viscous wave equation}, describing the random displacement of the structure from its reference configuration. This model describes 
the interaction between a two-dimensional infinite plate, modeled by the 2D wave equation, and a 3D fluid in the lower half space, modeled by the stationary Stokes equations, under the additional influence of spacetime white noise (random noise that is formally independent at every point in space and time).  
The work in \cite{SVWE} considers well-posedness for the stochastic viscous wave equation and establishes 
existence and uniqueness of a mild solution in spatial dimensions one and two, in addition to improved H\"{o}lder regularity properties. 
This result is interesting because the classical heat and wave equations driven by spacetime white noise in dimension two, do not possess
a mild solution. The main reason why the stochastic viscous wave equation studied in \cite{SVWE} admits a H\"{o}lder continuous mild solution in dimension two
(which is the physical dimension)
is the ``right'' scaling and the regularity properties of the fractional derivative operator (Dirichlet-to-Neumann operator), which
models the effects of viscous fluid regularization on the elastodynamics of a stochastically perturbed 2D membrane. 

While the results in \cite{SVWE} provide an insight into the behavior of solutions to stochastic FSI, they are restricted by the fact that 
the stochastic viscous wave equation is not a fully coupled model, it is defined in a special geometry on the entire $\mathbb R^2$, 
and it does not include the fluid inertia effects. 
This allowed the use of mathematical techniques that are not available in the fully coupled case of stochastic FSI.
The goal of the current manuscript is to develop techniques for studying \textit{fully coupled} stochastic fluid-structure interaction systems,
defined on physically relevant geometries, including  fluid inertia effects described by the time-dependent Stokes equations. 

%

\section{Description of the model}\label{model}

The model problem considered here is defined on a fixed fluid domain, 
which is a rectangle $\Omega_{f} = [0, L] \times [0, R]$. The boundary $\partial \Omega_{f}$ of the fluid domain 
consists of four parts: the moving boundary part denoted by $\Gamma$ (it is the reference configuration of the moving boundary), 
the bottom of the ``channel'' denoted by $\Gamma_b$, and the
inlet and outlet parts of the boundary $\Gamma_{in}$ and $\Gamma_{out}$ where the pressure data is prescribed.
The flow in the fluid domain $\Omega_f$
is driven by the inlet and outlet pressure data, and by the motion of the moving boundary. See Fig.~\ref{domain}.
We will use $\boldsymbol{x} = (z, r)$ to denote the coordinates of points in the fluid domain.

\begin{figure}[htp!]
\center
                    \includegraphics[width = 0.9 \textwidth]{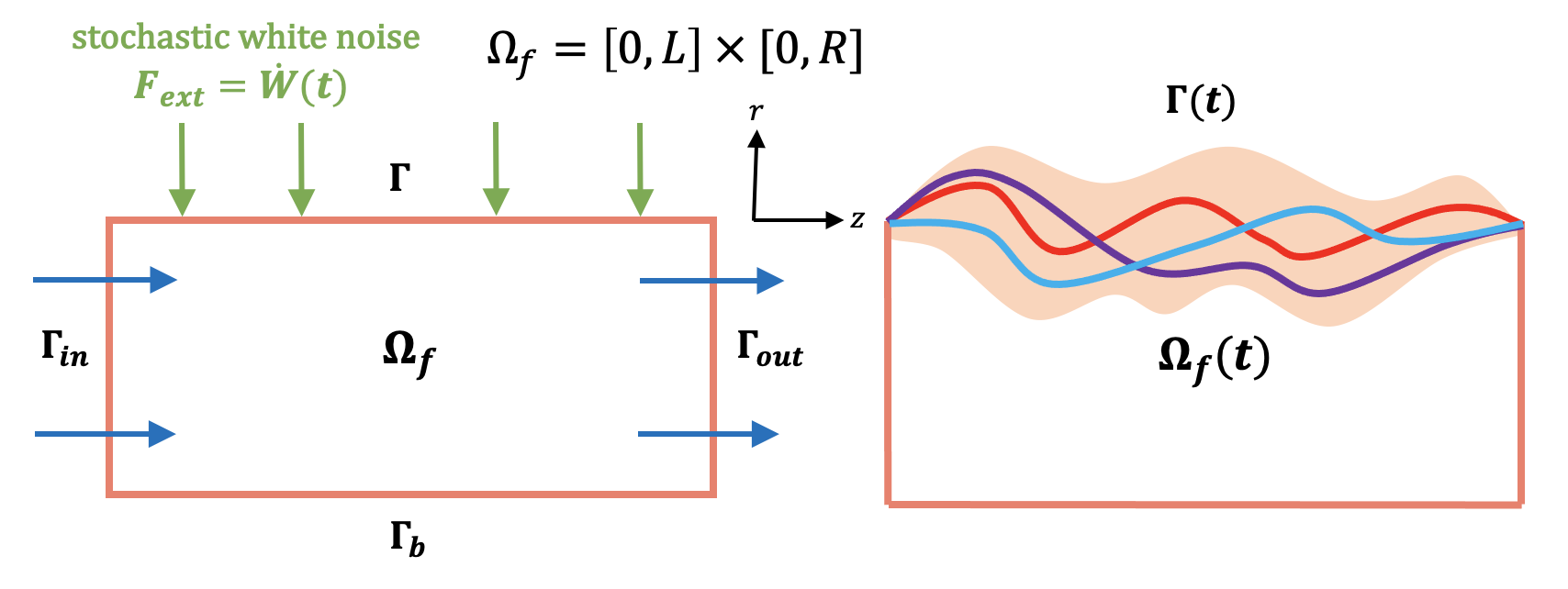}
  \caption{{\small\emph{Left: A sketch of the linearly coupled stochastic FSI problem, with $\Omega_{f}$ denoting the reference fluid domain, $\Gamma$ denoting the reference configuration of the structure, and $\dot{W}(t)$ denoting stochastic white noise forcing on the structure. Right: The different colors represent different possible outcomes for the random configuration $\Gamma(t)$ of the structure at some time $t$. The lightly shaded region represents a confidence interval of where the structure is likely to be.}}}
\label{domain}
\end{figure}

The {\bf{fluid flow}} in $\Omega_f$ will be modeled by the time-dependent Stokes equations for an incompressible, viscous fluid:
\begin{equation}\label{NS}
\left.
\begin{array}{rcl}
\partial_{t}\boldsymbol{u} &=& \nabla \cdot \boldsymbol{\sigma}, \\
\nabla \cdot {\boldsymbol u} &=& 0,
\end{array}
\right\} \quad {\rm in} \ \Omega_{f},
\end{equation}
where $\boldsymbol{u}(t, \boldsymbol{x}) = (u_{z}(t, \bx), u_{r}(t, \bx))$ is the fluid velocity, 
$
\boldsymbol{\sigma} = -p\boldsymbol{I} + 2\mu \boldsymbol{D}(\boldsymbol{u})
$
is the Cauchy stress tensor describing a Newtonian fluid, and $p$ is the fluid pressure.
This gives rise to the following system:
\begin{equation}\label{NS2}
\left.
\begin{array}{rcl}
\partial_{t}\boldsymbol{u} - \mu \Delta \boldsymbol{u} + \nabla p &=& 0, \\
\nabla \cdot {\boldsymbol u} &=& 0,
\end{array}
\right\} \quad {\rm in} \ \Omega_{f}.
\end{equation}

At the top boundary $\Gamma$ of the fluid domain, an elastic membrane interacts with the fluid flow. 
 We  assume that this elastic structure experiences displacement only in the vertical direction from its reference configuration $\Gamma$, 
 and we denote the magnitude of this displacement by $\eta(t, z)$. 
 The {\bf{elastodynamics of the structure}} will be modeled by the wave equation:
\begin{equation}\label{wave}
\eta_{tt} - \Delta \eta = f, \qquad \text{ on } \Gamma,
\end{equation}
where $f$ is an external forcing term. 

\if 1 = 0
\textbf{Fluid subproblem:} Next, we describe the fluid subproblem. We will assume that the fluid is an incompressible, viscous, Newtonian fluid residing in the fluid domain $\Omega_{f}$. With the assumption of linear coupling, we will assume as an approximation that the fluid domain is fixed over time to be $\Omega_{f}$. We will denote the fluid velocity by $\boldsymbol{u}(t, \boldsymbol{x}) = (u_{z}(t, x), u_{r}(t, x))$. We model the fluid by the linear Stokes equations, given by 
\begin{equation}\label{NS}
\left.
\begin{array}{rcl}
\partial_{t}\boldsymbol{u} &=& \nabla \cdot \boldsymbol{\sigma}, \\
\nabla \cdot {\boldsymbol u} &=& 0,
\end{array}
\right\} \quad {\rm in} \ \Omega_{f},
\end{equation}
where the first equation describes the balance of forces by Newton's second law and the second equation describes the incompressibility condition. The Cauchy stress tensor is defined by
\begin{equation*}
\boldsymbol{\sigma} = -p\boldsymbol{I} + 2\mu \boldsymbol{D}(\boldsymbol{u}),
\end{equation*}
so that the linear Stokes equations can also be expressed as
\begin{equation}\label{NS2}
\left.
\begin{array}{rcl}
\partial_{t}\boldsymbol{u} - \mu \Delta \boldsymbol{u} + \nabla p &=& 0, \\
\nabla \cdot {\boldsymbol u} &=& 0,
\end{array}
\right\} \quad {\rm in} \ \Omega_{f}.
\end{equation}
We will also make the following assumptions on the flow. We impose a symmetry condition at the bottom edge of $\Omega$, which represents the central axis of the full rectangular region, where $\Omega_{f}$ represents the upper half of this full rectangular region, where upon reflecting across the central axis, one can recover the dynamics on the full rectangular region. In particular,
\begin{equation}\label{velocityb}
u_{r} = \partial_{r}u_{z} = 0, \qquad \text{ on } \Gamma_{b}.
\end{equation}
On the boundaries representing the inlet and outlet, $\Gamma_{in}$ and $\Gamma_{out}$, we will assume that the incoming and outgoing flow is horizontal, so that
\begin{equation}\label{velocityinout}
u_{r} = 0, \qquad \text{ on } \Gamma_{in/out},
\end{equation}
and the flow is driven by changes in inlet and outlet pressure, which are prescribed at each time,
\begin{equation}\label{pressuredata}
p = P_{in/out}(t), \qquad \text{ on } \Gamma_{in/out},
\end{equation}
where $P_{in/out}(t) \in L^{2}(0, T)$ denotes the inlet and outlet pressure data. We will solve the linear problem on the time interval $[0, T]$, where $T > 0$ is given. 
\fi

The fluid and structure are coupled via two sets of coupling conditions, the kinematic and dynamic coupling conditions, which are evaluated along 
the {\emph{fixed interface}}. This is known as {\bf{linear coupling}}. 
The {\bf{kinematic coupling condition}} considered in this work describes the continuity of velocities at the fluid-structure interface
\begin{equation}\label{kinematic}
\boldsymbol{u} = \eta_{t} \boldsymbol{e}_{r}, \qquad \text{ on } \Gamma,
\end{equation}
also known as the {\emph{no-slip}} condition. 
The {\bf{dynamic coupling condition}} describes balance of forces at the interface. Namely, it states that the
elastodynamics of the thin elastic structure is driven by the jump in the force acting on the structure,
coming from the normal component of the normal fluid stress $\boldsymbol{\sigma} \boldsymbol{e_{r}} \cdot \boldsymbol{e_{r}}$ on one side,
and the external forcing $F_{ext}$ on the other:
\begin{equation*}
\eta_{tt} - \Delta \eta = -\boldsymbol{\sigma} \boldsymbol{e_{r}} \cdot \boldsymbol{e_{r}} + F_{ext}, \qquad \text{ on } \Gamma,
\end{equation*}
where $\boldsymbol{e_{r}}$ is the unit outer normal to the fixed fluid-structure interface $\Gamma$.

In this manuscript, we consider the external force $F_{ext}$ to be a stochastic force. 
In particular, as a start, we consider 
\begin{equation*}
F_{ext} = \dot{W}(t),
\end{equation*}
where $W$ is a one-dimensional Brownian motion in time. Note that the stochastic force is constant on the whole structure at each time. As a result, the stochastic noise is rough temporally but is constant spatially. We remark that although this is a simplified model, we use it to demonstrate the difficulties present in the stochastic case in the simplest possible setting. 

More precisely, we let $W$ denote a one-dimensional Brownian motion with respect to an underlying probability space with filtration, $(\Omega, \mathcal{F}, \{\mathcal{F}_{t}\}_{t \ge 0}, \mathbb{P})$, in which case $dW$ is formally the derivative of this Brownian motion. This is a purely formal notation that we will give precise meaning to later, as Brownian motion is almost surely nowhere differentiable. 

In addition, we will assume that the filtration $\{\mathcal{F}_{t}\}_{t \ge 0}$ is a \textbf{complete filtration}, 
which means that $\mathcal{F}_{t}$ contains all null sets of $(\Omega, \mathcal{F}, \mathbb{P})$ for every $t \ge 0$,
where a \textbf{null set} is defined to be any measurable set in $\mathcal{F}$ that has probability zero. This technical assumption will be useful to pass to the limit in our analysis of the stochastic problem above, as it allows us to bypass technicalities regarding null sets when considering almost sure limits of stochastic processes. \textit{In particular, the almost sure limit of $\mathcal{F}_{t}$ measurable random variables for any arbitrary $t \ge 0$ is still $\mathcal{F}_{t}$ measurable under the assumption of a complete filtration.} This is not a restrictive assumption, as one can complete a filtration by simply adding all null sets to $\mathcal{F}_{t}$ for all $t \ge 0$, and $W$ will still be a Brownian motion with respect to the completed filtration. See Section 1.4 in Revuz and Yor \cite{RY} for more information about complete filtrations.

{\bf{In summary}}, the coupled stochastic fluid-structure interaction problem studied in this manuscript, supplemented with initial and boundary data,
 is given by the following:
 {\emph{
 Find $(\bu,\eta)$ such that
}}

\begin{equation}\label{eq}
\left.
\begin{array}{rcl}
\eta_{tt} - \Delta \eta &=& -\boldsymbol{\sigma} \boldsymbol{e_{r}} \cdot \boldsymbol{e_{r}} + dW(t), \\
\boldsymbol{u} &=& \eta_{t} \boldsymbol{e_{r}}, 
\end{array}
\right\} \quad \text{ on } \Gamma,
\quad
\left.
\begin{array}{rcl}
\partial_{t}\boldsymbol{u} &=& \nabla \cdot \boldsymbol{\sigma}, \\
\nabla \cdot {\boldsymbol u} &=& 0,
\end{array}
\right\} \quad {\rm in} \ \Omega_{f},
\end{equation}
{\emph{with boundary data:}}
\begin{equation}\label{bdata}
\left.
\begin{array}{rcl}
u_{r} &=& 0, \\
p &=& P_{in/out}(t),
\end{array}
\right\} \qquad \text{ on } \Gamma_{in/out},
\qquad
u_{r} = \partial_{r}u_{z} = 0, \qquad \text{ on } \Gamma_{b},
\end{equation}
{\emph{and the following \textit{deterministic} initial data:}}
\begin{equation}\label{idata}
\boldsymbol{u}(0, z, r) = \boldsymbol{u}_{0}(z, r), \qquad \eta(0, z, R) = \eta_{0}(z), \qquad \partial_{t}\eta(0, z, R) = v_{0}(z),
\end{equation}
{\em{where $\boldsymbol{u}_{0} \in L^{2}(\Omega_{f})$, $\eta_{0} \in H_{0}^{1}(\Gamma)$, and $v_{0} \in L^{2}(\Gamma)$,}}
{\emph{and $W$ is a
given one-dimensional Brownian motion with respect to the complete probability space $(\Omega, \mathcal{F}, \mathbb{P})$ 
with complete filtration $\{\mathcal{F}_{t}\}_{t \ge 0}$.
}}

Thus, the problem is driven by deterministic inlet and outlet pressure data $P_{in/out}(t)$ prescribed on $\Gamma_{in/out}$, with the flow symmetry condition imposed at the bottom boundary $\Gamma_b$.
Notice that throughout this manuscript, we will be using $\Omega$ to denote the underlying probability space,
while $\Omega_{f}$ denotes the fluid domain.

\section{Definition of a weak solution and Main Result}\label{sec:weak}

To define the space of weak solutions to the above problem, we first introduce the function space for the fluid velocity:
\begin{equation}\label{VF}
\mathcal{V}_{F} = \{\boldsymbol{u} = (u_{z}, u_{r}) \in H^{1}(\Omega_{f})^{2}: \nabla \cdot \boldsymbol{u} = 0, \ u_{z} = 0 \text{ on } \Gamma, \ u_{r} = 0 \text{ on } \partial \Omega_{f} \text{\textbackslash} \Gamma\}.
\end{equation}
Since the structure subproblem is given by the wave equation with clamped ends, the natural space of functions for the structure is 
\begin{equation}\label{VS}
\mathcal{V}_{S} = H_{0}^{1}(\Gamma).
\end{equation}
Motivated by the energy inequality presented in Sec.~\ref{energy}, we introduce the 
following solution spaces in time for the fluid and structure subproblems:
\begin{equation}\label{WF}
\mathcal{W}_{F}(0, T) = L^{2}(\Omega; L^{\infty}(0, T; L^{2}(\Omega_{f}))) \cap L^{2}(\Omega; L^{2}(0, T; \mathcal{V}_{F})).
\end{equation}
\begin{equation}\label{WS}
\mathcal{W}_{S}(0, T) = L^{2}(\Omega; W^{1, \infty}(0, T; L^{2}(\Gamma))) \cap L^{2}(\Omega; L^{\infty}(0, T; \mathcal{V}_{S})).
\end{equation}
We emphasize that $\boldsymbol{u}$ and $\eta$ are random variables, and that the $L^2(\Omega)$ part of the solution spaces
reflects the fact that the energy estimate will hold in expectation.

Finally, we introduce the solution space for the {\emph{stochastic coupled FSI problem}:
\begin{equation}\label{W}
\mathcal{W}(0, T) = \{(\boldsymbol{u}, \eta) \in \mathcal{W}_{F}(0, T) \times \mathcal{W}_{S}(0, T) : \boldsymbol{u}|_{\Gamma} = \eta_{t} \boldsymbol{e_{r}} \text{ for almost every $t \in [0, T]$, a.s.}\}.
\end{equation}
Notice that in this solution space, the kinematic coupling condition is enforced strongly.

As in the deterministic case, we define weak solutions by integrating in space and time against an appropriate space of test functions,
which we define to be:
\begin{equation}\label{testspace}
\mathcal{Q}(0, T) = \{(\boldsymbol{q}, \psi) \in C_{c}^{1}([0, T); \mathcal{V}_{F} \times \mathcal{V}_{S}) : \boldsymbol{q}(t, z, R) = \psi(t, z) \boldsymbol{e_{r}}\}.
\end{equation}
These test functions are deterministic functions. Because the fluid domain does not change in time with the assumption of linear coupling, we can define 
\begin{equation}\label{Q}
\mathcal{Q} = \{(\boldsymbol{q}, \psi) \in \mathcal{V}_{F} \times \mathcal{V}_{S} : \boldsymbol{q}|_{\Gamma} = \psi \boldsymbol{e}_{r}\},
\end{equation}
and hence view the test functions as differentiable, compactly supported functions on $[0, T)$ that take values in the fixed function space $\mathcal{Q}$. 

To motivate the definition of a weak solution, we will proceed as in \cite{MuhaCanic13}. 
For the purposes of the derivation of the weak solution, we consider, for the moment, the case of a general deterministic external force $F_{ext}(t)$ in place of $\dot{W}(t)$, so that the first equation for the structure becomes 
\begin{equation*}
\eta_{tt} - \Delta \eta = -\boldsymbol{\sigma} \boldsymbol{e_{r}} \cdot \boldsymbol{e_{r}} + F_{ext}(t).
\end{equation*}
We will derive the standard deterministic partial differential equation definition of a weak solution, assuming that $F_{ext}(t)$ is a purely deterministic function in time, and then generalize this to the stochastic case.

We start by taking a test function $(\boldsymbol{q}, \psi) \in \mathcal{Q}(0, T)$, and multiplying the linear Stokes equation by $\boldsymbol{q}$ and integrating in space and time. We obtain
\begin{equation*}
\int_{0}^{T}\int_{\Omega_{f}} \partial_{t}\boldsymbol{u} \cdot \boldsymbol{q} d\boldsymbol{x} dt = \int_{0}^{T} \int_{\Omega_{f}} (\nabla \cdot \boldsymbol{\sigma}) \cdot \boldsymbol{q} d\boldsymbol{x} dt.
\end{equation*}
By integrating the first term by parts in time, we obtain:
\begin{align*}
\int_{0}^{T}\int_{\Omega_{f}} \partial_{t}\boldsymbol{u} \cdot \boldsymbol{q} d\boldsymbol{x} dt &= \int_{\Omega_{f}} \boldsymbol{u} \cdot \boldsymbol{q} d\boldsymbol{x}\Big\vert^{t = T}_{t = 0} - \int_{0}^{T} \int_{\Omega_{f}} \boldsymbol{u} \cdot \partial_{t}\boldsymbol{q} d\boldsymbol{x} dt 
= -\int_{\Omega_{f}} \boldsymbol{u_{0}} \cdot \boldsymbol{q}(0) d\boldsymbol{x} - \int_{0}^{T} \int_{\Omega_{f}} \boldsymbol{u} \cdot \partial_{t}\boldsymbol{q} d\boldsymbol{x} dt.
\end{align*}
By integrating the second term by parts in space and using the divergence free condition on $\boldsymbol{q}$, we obtain:
\begin{equation*}
\int_{\Omega_{f}} (\nabla \cdot \boldsymbol{\sigma}) \cdot \boldsymbol{q} d\boldsymbol{x} = \int_{\partial \Omega_{f}} (\boldsymbol{\sigma} \boldsymbol{n}) \cdot \boldsymbol{q} dS - 2\mu \int_{\Omega_{f}} \boldsymbol{D}(\boldsymbol{u}) : \boldsymbol{D}(\boldsymbol{q}) d\boldsymbol{x},
\end{equation*}
where $\boldsymbol{D}(\boldsymbol{u})$ and $\boldsymbol{D}(\boldsymbol{q})$ represent the symmetrized gradient. Using the definition of the Cauchy stress tensor, 
$
\boldsymbol{\sigma} = -p\boldsymbol{I} + 2\mu \boldsymbol{D}(\boldsymbol{u}),
$
and integrating in time, we obtain
\begin{align*}
&\int_{0}^{T} \int_{\Omega_{f}} (\nabla \cdot \boldsymbol{\sigma}) \cdot \boldsymbol{q} d\boldsymbol{x} dt = \int_{0}^{T} \int_{\Gamma_{in}} p q_{z} dr dt - \int_{0}^{T} \int_{\Gamma_{out}} pq_{z} dr dt - 2\mu \int_{0}^{T} \int_{\Omega_{f}} \boldsymbol{D}(\boldsymbol{u}) : \boldsymbol{D}(\boldsymbol{q}) d\boldsymbol{x} dt \\
&- \int_{0}^{T} \int_{\Gamma} \nabla \eta \cdot \nabla \psi dz dt + \int_{0}^{T} \int_{\Gamma} \partial_{t}\eta \partial_{t}\psi dz dt + \int_{\Gamma} v_{0} \psi(0) dz + \int_{0}^{T}\left(\int_{\Gamma} \psi dz\right) F_{ext}(t) dt.
\end{align*}

Putting this all together, we get that
\begin{small}
\begin{multline*}
- \int_{0}^{T} \int_{\Omega_{f}} \boldsymbol{u} \cdot \partial_{t}\boldsymbol{q} d\boldsymbol{x} dt + 2\mu \int_{0}^{T} \int_{\Omega_{f}} \boldsymbol{D}(\boldsymbol{u}) : \boldsymbol{D}(\boldsymbol{q}) d\boldsymbol{x} dt - \int_{0}^{T} \int_{\Gamma} \partial_{t}\eta\partial_{t}\psi dz dt + \int_{0}^{T} \int_{\Gamma} \nabla \eta \cdot \nabla \psi dz dt \\
= \int_{0}^{T} P_{in}(t) \left(\int_{\Gamma_{in}} q_{z} dr\right) dt - \int_{0}^{T} P_{out}(t) \left(\int_{\Gamma_{out}} q_{z} dr\right) dt 
+ \int_{\Omega_{f}} \boldsymbol{u_{0}} \cdot \boldsymbol{q}(0) d\boldsymbol{x} + \int_{\Gamma} v_{0}\psi(0) dz + \int_{0}^{T}\left(\int_{\Gamma} \psi dz\right) F_{ext}(t) dt,
\end{multline*}
\end{small}
where we used the fact that
$P_{in/out}(t) = p$ on $\Gamma_{in/out}$.

Now, we formally substitute 
$
F_{ext}(t) = \dot{W}(t),
$
into the definition of the deterministic weak solution, to get that the term containing $F_{ext}(t)$ can be interpreted  in the stochastic case as:
\begin{equation*}
\int_{0}^{T}\left(\int_{\Gamma} \psi dz\right) dW(t).
\end{equation*}
Since $W$ is a one dimensional Brownian motion and since $\int_{\Gamma} \psi dz$ is a deterministic function in time, we can interpret this term as a stochastic integral.

Before we give the definition of a weak solution to the stochastic FSI problem above, we recall the definition of a {\bf{stochastic basis}}. A \textbf{stochastic basis} $\mathcal{S}$ is an ordered quintuple (see \cite{LNT } for the notation)
\begin{equation*}
\mathcal{S} = (\Omega, \mathcal{F}, \{\mathcal{F}_{t}\}_{t \ge 0}, \mathbb{P}, W),
\end{equation*}
where $(\Omega, \mathcal{F}, \mathbb{P})$ is a probability space, $\{\mathcal{F}_{t}\}_{t \ge 0}$ is a complete filtration with respect to this probability space, and $W$ is a one-dimensional Brownian motion on the probability space with respect to the filtration $\{\mathcal{F}_{t}\}_{t \ge 0}$, meaning that:
(1) $W$ has continuous paths, almost surely,
(2) $W$ is adapted to the filtration $\{\mathcal{F}_{t}\}_{t \ge 0}$,
and (3) $W(t) - W(s)$ is independent of $\mathcal{F}_{s}$ for all $t \ge s$ and $W(t) - W(s) \sim N(0, t - s)$ for all $0 \le s \le t$,
where $N$ denotes the normal distribution.

We will define two notions of solution: (1) a weak solution in a probabilistically weak sense, and (2) a weak solution in a probabilistically strong sense. 
The second one is stronger than the first, but we will need the first to be able to prove the existence of a weak solution in a probabilistically strong sense.

\begin{definition}\label{weak}
An ordered triple $(\tilde{\mathcal{S}}, \tilde{\boldsymbol{u}}, \tilde{\eta})$ is a \textit{weak solution in a probabilistically \textbf{weak} sense} if there exists a stochastic basis
\begin{equation*}
\tilde{\mathcal{S}} = (\tilde{\Omega}, \tilde{\mathcal{F}}, \{\tilde{\mathcal{F}_{t}}\}_{t \ge 0}, \tilde{\mathbb{P}}, \tilde{W})
\end{equation*}
and $(\tilde{\boldsymbol{u}}, \tilde{\eta}) \in \mathcal{W}(0, T)$ with paths almost surely in $C(0, T; \mathcal{Q}')$, which satisfies:
\begin{itemize}
\item $(\tilde{\boldsymbol{u}}, \tilde{\eta})$ is adapted to the filtration $\{\tilde{\mathcal{F}_{t}}\}_{t \ge 0}$,
\item $\tilde{\eta}(0) = \eta_{0}$ almost surely, and
\item for all $(\boldsymbol{q}, \psi) \in \mathcal{Q}(0, T)$, 
\begin{small}
\begin{multline*}
- \int_{0}^{T} \int_{\Omega_{f}} \tilde{\boldsymbol{u}} \cdot \partial_{t}\boldsymbol{q} d\boldsymbol{x} dt + 2\mu \int_{0}^{T} \int_{\Omega_{f}} \boldsymbol{D}(\tilde{\boldsymbol{u}}) : \boldsymbol{D}(\boldsymbol{q}) d\boldsymbol{x} dt - \int_{0}^{T} \int_{\Gamma} \partial_{t}\tilde{\eta}\partial_{t}\psi dz dt + \int_{0}^{T} \int_{\Gamma} \nabla \tilde{\eta} \cdot \nabla \psi dz dt \\
= \int_{0}^{T} P_{in}(t) \left(\int_{\Gamma_{in}} q_{z} dr\right) dt - \int_{0}^{T} P_{out}(t) \left(\int_{\Gamma_{out}} q_{z} dr\right) dt 
+ \int_{\Omega_{f}} \boldsymbol{u_{0}} \cdot \boldsymbol{q}(0) d\boldsymbol{x} + \int_{\Gamma} v_{0}\psi(0) dz + \int_{0}^{T}\left(\int_{\Gamma} \psi dz\right) d\tilde{W},
\end{multline*}
\end{small}
almost surely.
\end{itemize}
\end{definition}

\begin{definition}\label{strong}
An ordered pair $(\boldsymbol{u}, \eta)$ is a \textit{weak solution in a probabilistically \textbf{strong} sense} if $(\boldsymbol{u}, \eta) \in \mathcal{W}(0, T)$ with paths almost surely in $C(0, T; \mathcal{Q}')$, satisfies:
\begin{itemize}
\item $(\boldsymbol{u}, \eta)$ is adapted to the filtration $\{\mathcal{F}_{t}\}_{t \ge 0}$
\item $\eta(0) = \eta_{0}$ almost surely, and 
\item for all $(\boldsymbol{q}, \psi) \in \mathcal{Q}(0, T)$, 
\begin{small}
\begin{multline*}
- \int_{0}^{T} \int_{\Omega_{f}} \boldsymbol{u} \cdot \partial_{t}\boldsymbol{q} d\boldsymbol{x} dt + 2\mu \int_{0}^{T} \int_{\Omega_{f}} \boldsymbol{D}(\boldsymbol{u}) : \boldsymbol{D}(\boldsymbol{q}) d\boldsymbol{x} dt - \int_{0}^{T} \int_{\Gamma} \partial_{t}\eta\partial_{t}\psi dz dt + \int_{0}^{T} \int_{\Gamma} \nabla \eta \cdot \nabla \psi dz dt \\
= \int_{0}^{T} P_{in}(t) \left(\int_{\Gamma_{in}} q_{z} dr\right) dt - \int_{0}^{T} P_{out}(t) \left(\int_{\Gamma_{out}} q_{z} dr\right) dt 
+ \int_{\Omega_{f}} \boldsymbol{u_{0}} \cdot \boldsymbol{q}(0) d\boldsymbol{x} + \int_{\Gamma} v_{0}\psi(0) dz + \int_{0}^{T}\left(\int_{\Gamma} \psi dz\right) dW.
\end{multline*}
\end{small}
almost surely.
\end{itemize}
\end{definition}





In a probabilistically strong solution as in the second definition above, we have a random solution satisfying the initial conditions on the originally given (arbitrary) probability space with a one dimensional Brownian motion with respect to a complete filtration. 
In a probabilistically weak solution, we have a weaker requirement that the random solution exists on a \textit{particular (not arbitrary) probability space}, where the initial 
conditions are satisfied ``in law''. We will show the existence of a weak solution in the probabilistically strong sense. However, to get to that
solution, we will first show existence of a convergent subsequence of probability measures corresponding to the laws of the approximate solutions, 
then construct a weak solution in the probabilistically weak sense using the Skorohod representation theorem, and then use the Gy\"{o}ngy-Krylov argument \cite{GK} to get to a weak solution in the probabilistically strong sense.

The main result of this work is stated in the following theorem.


\begin{theorem}[{\bf{Main Result}}] \label{Main theorem}
Let $\boldsymbol{u}_{0} \in L^{2}(\Omega_{f})$, $v_{0} \in L^{2}(\Gamma)$, and $\eta_{0} \in H_{0}^{1}(\Gamma)$. Let $P_{in/out} \in L^{2}_{loc}(0, \infty)$ and let $(\Omega, \mathcal{F}, \mathbb{P})$ be a probability space with a Brownian motion $W$ with respect to a given complete filtration $\{\mathcal{F}_{t}\}_{t \ge 0}$. Then, for any $T > 0$, there exists a unique weak solution in a probabilistically strong sense to the given stochastic fluid-structure interaction problem \eqref{eq}--\eqref{idata}. 
\end{theorem}

\section{A priori energy estimate}\label{energy}
We derive a formal energy estimate by assuming that the solution is pathwise regular enough to justify the integration by parts. 
We use $\| \cdot \|_{L^2(\Gamma)}$ and $(\cdot, \cdot)$ to denote the norm and inner product on $L^{2}(\Gamma)$,
and $\| \cdot \|_{L^2(\Omega_f)}$ and $\langle \cdot, \cdot \rangle$ to denote the norm and inner product on $L^{2}(\Omega_{f})$.

We define the total energy at time $T$ by
\begin{equation*}
E(T) := \frac{1}{2}\int_{\Gamma} |\nabla \eta|^{2} dz + \frac{1}{2} \int_{\Gamma} |v|^{2} dz + \frac{1}{2} \int_{\Omega_{f}} |\boldsymbol{u}|^{2} d\boldsymbol{x}=
\frac{1}{2}\left(\|\nabla \eta\|_{L^2(\Gamma)}^{2} + \|v\|_{L^2(\Gamma)}^{2} + \|\boldsymbol{u}\|_{L^2(\Omega_f)}^{2}\right),
\end{equation*}
and the total dissipation by time $T$ by
\begin{equation*}
D(T) = \int_{0}^{T} \int_{\Omega_{f}} |\boldsymbol{D}(\boldsymbol{u})|^{2} d\boldsymbol{x}.
\end{equation*}
To estimate the total energy and dissipation for the stochastic processes $\bu$, $v$ and $\eta$, we 
rewrite the stochastic fluid-structure interaction problem in the following stochastic differential formulation:
\begin{align*}
d\eta &= v dt, \\
dv &= (\Delta \eta - \boldsymbol{\sigma} \boldsymbol{e_{r}} \cdot \boldsymbol{e_{r}}) dt + dW, \\
d\boldsymbol{u} &= (\nabla \cdot \boldsymbol{\sigma}) dt.
\end{align*}
Notice that the first equation implies 
$
d(\nabla \eta) = (\nabla v) dt.
$
To obtain an energy estimate, we first apply It\"{o}'s formula to express the differentials of the $L^2$-norms of the stochastic processes 
that define the total energy  of the problem:
\begin{equation*}
d(\|\nabla \eta\|_{L^2(\Gamma)}^{2}) = 2 (\nabla \eta, \nabla v) dt,
\end{equation*}
\begin{equation*}
d(\|v\|_{L^2(\Gamma)}^{2}) = [2(\Delta \eta, v) - 2(\boldsymbol{\sigma} \boldsymbol{e_{r}} \cdot \boldsymbol{e_{r}}, v) + L] dt + 2(1, v) dW,
\end{equation*}
\begin{equation*}
d(\|\boldsymbol{u}\|_{L^2(\Omega_f)}^{2}) = 2\langle \nabla \cdot \boldsymbol{\sigma}, \boldsymbol{u}\rangle dt.
\end{equation*}
By adding these equations together, we obtain
that the differential of the total energy satisfies:
\begin{equation*}
d(\|\nabla \eta\|_{L^2(\Gamma)}^{2} + \|v\|_{L^2(\Gamma)}^{2} + \|\boldsymbol{u}\|_{L^2(\Omega_f)}^{2}) = [2\langle \nabla \cdot \boldsymbol{\sigma}, \boldsymbol{u} \rangle - 2(\boldsymbol{\sigma} \boldsymbol{e_{r}} \cdot \boldsymbol{e_{r}}, v) + L] dt + 2(1, v) dW,
\end{equation*}
where we have used that $(\Delta \eta, v) = -(\nabla \eta, \nabla v)$
under the assumption that $\eta$ and $v$ are smooth and vanish at the endpoints of $\Gamma$.
Recalling the kinematic coupling condition $\boldsymbol{u}|_{\Gamma} = v$, we obtain that
\begin{equation*}
\int_{\Omega_{f}} (\nabla \cdot \boldsymbol{\sigma}) \cdot \boldsymbol{u} d\boldsymbol{x} = \int_{\Gamma_{in}} pu_{z} dr - \int_{\Gamma_{out}} pu_{z} dr + \int_{\Gamma} (\boldsymbol{\sigma} \boldsymbol{e_{r}} \cdot \boldsymbol{e_{r}}) v dz - 2\mu \int_{\Omega_{f}} |\boldsymbol{D}(\boldsymbol{u})|^{2} d\boldsymbol{x},
\end{equation*}
which implies
\begin{multline*}
d\left(\frac{1}{2}\int_{\Gamma} |\nabla \eta|^{2} dz + \frac{1}{2} \int_{\Gamma} |v|^{2} dz + \frac{1}{2} \int_{\Omega_{f}} |\boldsymbol{u}|^{2} d\boldsymbol{x}\right) \\
= \left(\frac{L}{2} -2\mu \int_{\Omega_{f}} |\boldsymbol{D}(\boldsymbol{u})|^{2} d\boldsymbol{x} + \int_{\Gamma_{in}} pu_{z} dr - \int_{\Gamma_{out}} pu_{z} dr\right) dt + \left(\int_{\Gamma} v dz\right) dW.
\end{multline*}
Therefore, after integration, for all $T \ge 0$, we have
\begin{equation}\label{energyequality}
E(T) + 2\mu \int_{0}^{T} \int_{\Omega_{f}} |\boldsymbol{D}(\boldsymbol{u})|^{2} d\boldsymbol{x}dt = E_{0} + \frac{LT}{2} + \int_{0}^{T} \int_{\Gamma_{in}} P_{in}(t) u_{z}dr dt - \int_{0}^{T} \int_{\Gamma_{out}} P_{out}(t) u_{z} dr dt + \int_{0}^{T} \left(\int_{\Gamma} v dz\right) dW.
\end{equation}
We estimate the terms on the right hand side of \eqref{energyequality} as follows. 
For the pressure term we use H\"{o}lder's inequality, the trace inequality, Poincar\'{e}'s inequality, and Korn's inequality \cite{Korn} to get
\begin{align}\label{pressureenergy}
&\left|\int_{0}^{T} \left(\int_{\Gamma_{in}} u_{z} dr\right) P_{in}(t) dt\right| \le C\left|\int_{0}^{T} \left(\int_{\Gamma_{in}} |u_{z}|^{2} dr\right)^{1/2} P_{in}(t) dt\right| \le C \left|\int_{0}^{T} ||\nabla \boldsymbol{u}||_{L^{2}(\Omega_{f})} P_{in}(t) dt\right| \nonumber \\
&\le C \left|\int_{0}^{T} ||\boldsymbol{D} (\boldsymbol{u})||_{L^{2}(\Omega_{f})} P_{in}(t) dt\right| \le C(D(T))^{1/2} ||P_{in}(t)||_{L^{2}(0, T)} 
\le \epsilon D(T) + C(\epsilon) ||P_{in}(t)||_{L^{2}(0, T)}^{2}.
\end{align}
We note that the constant $C(\epsilon)$ depends only on $\epsilon$ and the parameters of the problem. The same computation holds for the outlet pressure.

For the stochastic integral, we will bound the {{expectation}} 
$
\mathbb{E}\left(\max_{0 \le \tau \le T} \left|\int_{0}^{\tau} \left(\int_{\Gamma} \partial_{t}\eta dz\right) dW\right|\right)
$ 
since the final energy estimate will be given in terms of expectation of
the total energy and dissipation at time $T$. 
To  bound 
this quantity, we use the Burkholder-Davis-Gundy (BDG) inequality under the assumption that the process $\partial_{t}\eta$ is a predictable stochastic process with respect to the given filtration $\{\mathcal{F}_{t}\}_{t \ge 0}$:
\begin{align}\label{BDG}
&\mathbb{E}\left(\max_{0 \le s \le T} \left|\int_{0}^{s} \left(\int_{\Gamma} \partial_{t}\eta dz\right) dW\right|\right) 
\le \mathbb{E} \left(\left|\int_{0}^{T} \left(\int_{\Gamma} \partial_{t}\eta dz\right)^{2} dt\right|^{1/2}\right) 
\le C\mathbb{E} \left(\left|\int_{0}^{T} ||\partial_{t}\eta||^{2}_{L^{2}(\Gamma)} dt\right|^{1/2}\right) \nonumber \\
&\le C\left(\mathbb{E}\left|\int_{0}^{T} ||\partial_{t}\eta||^{2}_{L^{2}(\Gamma)} dt\right|\right)^{1/2} 
\le CT^{1/2} \cdot \left[\mathbb{E}\left(\max_{0 \le t \le T} ||\partial_{t}\eta(t, \cdot)||^{2}_{L^{2}(\Gamma)}\right)\right]^{1/2} \nonumber \\
&\le C(\epsilon) T + \epsilon\mathbb{E}\left(\max_{0 \le t \le T} ||\partial_{t}\eta(t, \cdot)||^{2}_{L^{2}(\Gamma)}\right) 
\le C(\epsilon)T + \epsilon\mathbb{E}\left(\max_{0 \le t \le T} E(t)\right).
\end{align}
Now, we first use \eqref{pressureenergy} in \eqref{energyequality} to obtain
\begin{equation*}
E(T) + 2\mu D(T) \le E(0) + \frac{LT}{2} + 2\epsilon D(\tau) + C(\epsilon)\left(||P_{in}(t)||^{2}_{L^{2}(0, T)} + ||P_{out}(t)||^{2}_{L^{2}(0, T)}\right) + \int_{0}^{T} \left(\int_{\Gamma} vdz\right) dW,
\end{equation*}
and then choose $\epsilon < \frac{\mu}{2}$ and $\epsilon < \frac{1}{2}$ to get
\begin{equation*}
E(T) + \mu D(T) \le E(0) + \frac{LT}{2} + C(\epsilon)\left(||P_{in}(t)||^{2}_{L^{2}(0, T)} + ||P_{out}(t)||^{2}_{L^{2}(0, T)}\right) + \int_{0}^{T} \left(\int_{\Gamma} vdz\right) dW.
\end{equation*}

Taking a maximum over times $t \in [0, T]$, taking an expectation, and then using the estimate in \eqref{BDG}, we obtain
the following {\emph{a priori}} energy estimate for the coupled problem \eqref{eq}--\eqref{idata}:
\begin{equation*}
\mathbb{E}\left(\max_{0 \le t \le T} E(t) + \mu \int_{0}^{t} \int_{\Omega_{f}} |\boldsymbol{D}(\boldsymbol{u})|^{2} d\boldsymbol{x} ds \right) \le C\left(T + E(0) + ||P_{in}(t)||^{2}_{L^{2}(0, T)} + ||P_{out}(t)||^{2}_{L^{2}(0, T)}\right),
\end{equation*}
where $C$ is independent of $T$, depending only on the parameters of the problem. 

\begin{remark}
The right hand side of the energy estimate shows the four sources of energy input into the system: $E(0)$ represents the initial kinetic and potential energy, the two final terms represent the energy input from the inlet and outlet pressure, and $CT$ represents the energy input from the stochastic forcing on the structure.
\end{remark}

\section{The splitting scheme}\label{split_scheme}
To prove the existence of a weak solution to the given stochastic FSI problem we adapt a Lie operator splitting scheme 
that was first designed in the context of nonlinear fluid-structure interaction by Muha and \v{C}ani\'{c} in \cite{MuhaCanic13}. See also \cite{withRoland}.
The stochastic component provides a new difficulty as one must appropriately handle the stochastic noise in the splitting. Since the stochastic noise is added to the structure, one might first attempt a splitting with two subproblems: (1) the structure subproblem with the stochastic noise added, discretized using a Brownian increment, and (2)  the fluid subproblem. However, such a splitting {\bf{does not give rise to a stable scheme}}. To overcome this problem, we use a stochastic splitting introduced in \cite{BGR}, which has been used in stochastic differential equations to split stochastic effects from all other deterministic effects. We design a three part splitting scheme that involves a structure subproblem, a stochastic subproblem, and a fluid subproblem, {\bf{which gives rise to a stable and convergent scheme}}, as we show below. 

Given a fixed time $T > 0$, for each positive integer $N$, let $\Delta t = \frac{T}{N}$ denote the associated time step,
and let $t^{n}_{N} = n\Delta t$ denote the discrete times for $n = 0, 1, ...., N - 1, N$.
At each time step, we update the following vector using a three step method described below:
\begin{equation*}
\boldsymbol{X}^{n + \frac{i}{3}}_{N} 
= \left( \boldsymbol{u}^{n + \frac{i}{3}}_{N}, v^{n + \frac{i}{3}}_{N}, \eta^{n + \frac{i}{3}}_{N} \right)^T,\quad 
n = 0, 1, ...., N - 1, \quad i = 1, 2, 3,
\end{equation*}
where $i = 1$ corresponds to the result after updating the structure subproblem, $i = 2$ corresponds to the stochastic subproblem, and $i = 3$ corresponds to the  fluid subproblem, with the initial data
$\boldsymbol{X}^{0}_{N} = \left( \boldsymbol{u}_{0},  v_{0}, \eta_{0} \right)^T$ for each $N$. 

\if 1 = 0
Throughout the scheme, we will find it convenient to use the following nonstandard convention. The Brownian motion $\{W_{t}\}_{t \ge 0}$ is defined for $t \ge 0$, where by definition of a Brownian motion, $W(0) = 0$. We will extend each continuous sample path of Brownian motion to all $t \in \mathbb{R}$, including negative times, by setting
\begin{equation*}
W_{t} = 0, \qquad \text{ for } t \le 0.
\end{equation*}
Note that every sample path of $W_{t}$ is continuous on $\mathbb{R}$, and furthermore, for any arbitrary $T > 0$, the sample paths of $W_{t}$ are almost surely uniformly H\"{o}lder continuous on $(-\infty, T]$. We will furthermore extend the filtration $(\mathcal{F}_{t})_{t \ge 0}$ to the whole real line, by setting $\mathcal{F}_{t} = \mathcal{F}_{0}$ for $t < 0$. 
\fi

\subsection{The structure subproblem}

In this subproblem, we keep the fluid velocity fixed, so that 
\begin{equation*}
\boldsymbol{u}^{n + \frac{1}{3}}_{N} = \boldsymbol{u}^{n}_{N},
\end{equation*}
and update the structure displacement and the structure velocity by requiring that
$(\eta^{n + \frac{1}{3}}_{N}, v^{n + \frac{1}{3}}_{N})$ satisfy the following first order system in weak variational form:
\begin{equation*}
\int_{\Gamma} \frac{\eta^{n + \frac{1}{3}}_{N} - \eta^{n}_{N}}{\Delta t} \phi dz = \int_{\Gamma} v^{n + \frac{1}{3}}_{N} \phi dz, \qquad \text{ for all } \phi \in L^{2}(\Gamma),
\end{equation*}
\begin{equation}\label{structuresubp}
\int_{\Gamma} \frac{v^{n + \frac{1}{3}}_{N} - v^{n}_{N}}{\Delta t} \psi dz + \int_{\Gamma} \nabla \eta^{n + \frac{1}{3}}_{N} \cdot \nabla \psi dz = 0, \qquad \text{ for all } \psi \in H_{0}^{1}(\Gamma),
\end{equation}
where this system is solved pathwise for each $\omega \in \Omega$ separately. We note that $(\eta^{n + \frac{1}{3}}_{N}, v^{n + \frac{1}{3}}_{N})$ is a random variable taking values in $H_{0}^{1}(\Gamma) \times H_{0}^{1}(\Gamma)$. To verify this, we must check that it is a measurable function of the probability space.

\begin{proposition}
Suppose that $\eta^{n}_{N}$ and $v^{n}_{N}$ are $\mathcal{F}_{t^{n}_{N}}$ measurable random variables taking values in $H_{0}^{1}(\Gamma)$ and $L^{2}(\Gamma)$ respectively. Then, the structure problem \eqref{structuresubp} has a unique solution $(\eta^{n + \frac{1}{3}}_{N}, v^{n + \frac{1}{3}}_{N})$, which is a random variable taking values in $H_{0}^{1}(\Gamma) \times H_{0}^{1}(\Gamma)$ that is measurable with respect to $\mathcal{F}_{t^{n}_{N}}$. 
\end{proposition}

\begin{proof}
Let $F^{n}_{N}: (\eta^{n}_{N}, v^{n}_{N}) \to (\eta^{n + \frac{1}{3}}_{N}, v^{n + \frac{1}{3}}_{N})$ be the deterministic linear map that sends deterministic data $(\eta^{n}_{N}, v^{n}_{N}) \in H_{0}^{1}(\Gamma) \times L^{2}(\Gamma)$ to the unique solution $(\eta^{n + \frac{1}{3}}_{N}, v^{n + \frac{1}{3}}_{N}) \in H_{0}^{1}(\Gamma) \times H_{0}^{1}(\Gamma)$ satisfying the weak formulation \eqref{structuresubp} as a deterministic problem. We must show that this deterministic linear map $F^{n}_{N}: (\eta^{n}_{N}, v^{n}_{N}) \to (\eta^{n + \frac{1}{3}}_{N}, v^{n + \frac{1}{3}}_{N})$ is a continuous (or equivalently, bounded) linear map from $H_{0}^{1}(\Gamma) \times L^{2}(\Gamma)$ to $H_{0}^{1}(\Gamma) \times H_{0}^{1}(\Gamma)$. 
To do this, we must show that for given deterministic functions $\eta^{n}_{N}$ and $v^{n}_{N}$ in $H_{0}^{1}(\Gamma)$ and $L^{2}(\Gamma)$, 
there is a unique solution $(\eta^{n + \frac{1}{3}}_{N}, v^{n + \frac{1}{3}}_{N})$ to the above problem in $H_{0}^{1}(\Gamma) \times H_{0}^{1}(\Gamma)$,
 and that the solution map is a bounded linear map.

The existence of a unique weak solution follows from the Lax-Milgram lemma. Namely, by plugging the first equation
in \eqref{structuresubp} into the second equation, we see that $\eta^{n + \frac{1}{3}}_{N}$ must satisfy the following weak formulation:
\begin{equation}\label{structuresubpmod}
\int_{\Gamma} \eta^{n + \frac{1}{3}}_{N} \psi dz + (\Delta t)^{2} \int_{\Gamma} \nabla \eta^{n + \frac{1}{3}}_{N} \cdot \nabla \psi dz = (\Delta t) \int_{\Gamma} v^{n}_{N} \psi dz + \int_{\Gamma} \eta^{n}_{N} \psi dz, \qquad \text{ for all } \psi \in H_{0}^{1}(\Gamma).
\end{equation}
The bilinear form $B: H_{0}^{1}(\Gamma) \times H_{0}^{1}(\Gamma) \to \mathbb{R}$ (defined by the left-hand side)
\begin{equation*}
B(\eta, \psi) := \int_{\Gamma} \eta \psi dz + (\Delta t)^{2} \int_{\Gamma} \nabla \eta \cdot \nabla \psi dz,
\end{equation*} 
is clearly coercive and continuous, and furthermore, for any fixed but arbitrary $\eta^{n}_{N} \in H_{0}^{1}(\Gamma)$ and $v^{n}_{N} \in L^{2}(\Gamma)$, the map 
\begin{equation*}
\psi \to (\Delta t) \int_{\Gamma} v^{n}_{N} \psi dz + \int_{\Gamma} \eta^{n}_{N} \psi dz
\end{equation*}
is a linear functional on $H_{0}^{1}(\Gamma)$. So the existence of a unique $\eta^{n + \frac{1}{3}}_{N} \in H_{0}^{1}(\Gamma)$ satisfying the weak formulation above in \eqref{structuresubpmod} is given by the Lax-Milgram lemma applied to $H_{0}^{1}(\Gamma)$.
 One then recovers
\begin{equation*}
v^{n + \frac{1}{3}}_{N} = \frac{\eta^{n + \frac{1}{3}}_{N} - \eta^{n}_{N}}{\Delta t} \in H_{0}^{1}(\Gamma).
\end{equation*}

To show that the linear map $F^{n}: (\eta^{n}_{N}, v^{n}_{N}) \to (\eta^{n + \frac{1}{3}}_{N}, v^{n + \frac{1}{3}}_{N})$ is a bounded linear map from $H_{0}^{1}(\Gamma) \times L^{2}(\Gamma)$ to $H_{0}^{1}(\Gamma) \times H_{0}^{1}(\Gamma)$, we note that by substituting $\psi = \eta^{n + \frac{1}{3}}_{N}$ in \eqref{structuresubpmod}, we obtain
\begin{align*}
||\eta^{n + \frac{1}{3}}_{N}||_{H_{0}^{1}(\Gamma)}^{2} &\le C_{N} \left((\Delta t) \int_{\Gamma} v^{n}_{N} \cdot \eta^{n + \frac{1}{3}}_{N} dz + \int_{\Gamma} \eta^{n}_{N} \cdot \eta^{n + \frac{1}{3}}_{N} dz\right) \\
&\le C_{N} \left(||\eta^{n}_{N}||_{H_{0}^{1}(\Gamma)} + ||v^{n}_{N}||_{L^{2}(\Gamma)}\right) ||\eta^{n + \frac{1}{3}}_{N}||_{H_{0}^{1}(\Gamma)}.
\end{align*}
Hence,
$
||\eta^{n + \frac{1}{3}}_{N}||_{H_{0}^{1}(\Gamma)} \le C_{N} \left(||\eta^{n}_{N}||_{H_{0}^{1}(\Gamma)} + ||v^{n}_{N}||_{L^{2}(\Gamma)}\right)
$
for a constant $C_{N}$ depending only on $N$. 

Then, by the fact that $v^{n + \frac{1}{3}}_{N} = \frac{\eta^{n + \frac{1}{3}}_{N} - \eta^{n}_{N}}{\Delta t}$, we also have
$
||v^{n + \frac{1}{3}}_{N}||_{H_{0}^{1}(\Gamma)} \le C_{N} \left(||\eta^{n}_{N}||_{H_{0}^{1}(\Gamma)} + ||v^{n}_{N}||_{L^{2}(\Gamma)}\right).
$

Thus, $F^{n}_{N}: (\eta^{n}_{N}, v^{n}_{N}) \to (\eta^{n + \frac{1}{3}}_{N}, v^{n + \frac{1}{3}}_{N})$ is a bounded linear map from $H_{0}^{1}(\Gamma) \times L^{2}(\Gamma)$ to $H_{0}^{1}(\Gamma) \times H_{0}^{1}(\Gamma)$, and so the result of the structure subproblem, which consists of the random functions 
$
(\eta^{n + \frac{1}{3}}_{N}, v^{n + \frac{1}{3}}_{N}) = F^{n}_{N} \circ (\eta^{n}_{N}, v^{n}_{N})
$,
is a pair of $\mathcal{F}_{t^{n}_{N}}$ measurable random variables, taking values in $H_{0}^{1}(\Gamma) \times H_{0}^{1}(\Gamma)$. 
\end{proof}
To show that the approximate solutions defined by the subproblems converge to the weak solution of the continuous problem as $\Delta t \to 0$, 
we will need uniform bounds on the approximating sequences, which will follow from the uniform bounds on 
the discrete energy of the problem. 
For this purpose, we define the discrete energy at time $t_{n}$ by
\begin{equation}\label{discreteenergy}
E^{n + \frac{i}{3}}_{N} = \frac{1}{2}\left(\int_{\Omega_{f}} |\boldsymbol{u}^{n + \frac{i}{3}}_{N}|^{2} d\boldsymbol{x} + ||v^{n + \frac{i}{3}}_{N}||_{L^{2}(\Gamma)}^{2} + ||\nabla \eta^{n + \frac{i}{3}}_{N}||_{L^{2}(\Gamma)}^{2}\right),
\end{equation}
and we define the fluid dissipation at time $t_{n}$ by 
\begin{equation}\label{discretedissipation}
D^{n}_{N} = (\Delta t) \mu \int_{\Omega_{f}} |\boldsymbol{D}(\boldsymbol{u}^{n}_{N})|^{2}d\boldsymbol{x}.
\end{equation}
We emphasize that these are random variables.

\begin{proposition}
The following discrete energy equality is satisfied pathwise:
\begin{equation*}
E^{n + \frac{1}{3}}_{N} + \frac{1}{2} \left( ||v^{n + \frac{1}{3}}_{N} - v^{n}_{N}||^{2}_{L^{2}(\Gamma)}\right) + \frac{1}{2} \left(||\nabla \eta^{n + \frac{1}{3}}_{N} - \nabla \eta^{n}_{N}||^{2}_{L^{2}(\Gamma)}\right) = E^{n}_{N}.
\end{equation*}
\end{proposition}
\begin{proof}
Because $v^{n + \frac{1}{3}}_{N}  \in H_{0}^{1}(\Gamma)$, we can substitute $\psi = v^{n + \frac{1}{3}}_{N}$ in the weak formulation to obtain that pathwise,
\begin{equation*}
\int_{\Gamma} (v^{n + \frac{1}{3}}_{N} - v^{n}_{N}) \cdot v^{n + \frac{1}{3}}_{N} dz + (\Delta t) \int_{\Gamma} \nabla \eta^{n + \frac{1}{3}}_{N} \cdot \nabla v^{n + \frac{1}{3}}_{N} dz = 0. 
\end{equation*}
By using the identity
$
(a - b) \cdot a = \frac{1}{2}(|a|^{2} + |a - b|^{2} - |b|^{2}),
$
along with the fact that
$
v^{n + \frac{1}{3}}_{N} = \frac{\eta^{n + \frac{1}{3}}_{N} - \eta^{n}_{N}}{\Delta t},
$
we obtain that the following identity holds pathwise:
\begin{small}
\begin{align*}
\frac{1}{2} ||v^{n + \frac{1}{3}}_{N}||^{2}_{L^{2}(\Gamma)} &+ \frac{1}{2} ||\nabla \eta^{n + \frac{1}{3}}_{N}||^{2}_{L^{2}(\Gamma)} + \frac{1}{2} ||v^{n + \frac{1}{3}}_{N} - v^{n}_{N}||^{2}_{L^{2}(\Gamma)} + \frac{1}{2} ||\nabla \eta^{n + \frac{1}{3}}_{N} - \nabla \eta^{n}_{N}||^{2}_{L^{2}(\Gamma)} 
= \frac{1}{2} ||v^{n}_{N}||^{2}_{L^{2}(\Gamma)} + \frac{1}{2} ||\nabla \eta^{n}_{N}||^{2}_{L^{2}(\Gamma)}.
\end{align*}
\end{small}
The result follows once we note that $\boldsymbol{u}^{n + \frac{1}{3}}_{N} = \boldsymbol{u}^{n}_{N}$.
\end{proof}

\subsection{The stochastic subproblem}

In this subproblem, we incorporate only the effects of the stochastic forcing, which appears in only the structure equation. 
In this step, we keep the structure displacement and fluid velocity fixed
\begin{equation*}
\eta^{n + \frac{2}{3}}_{N} = \eta^{n + \frac{1}{3}}_{N}, \qquad \boldsymbol{u}^{n + \frac{2}{3}}_{N} = \boldsymbol{u}^{n + \frac{1}{3}}_{N},
\end{equation*}
and only update the structure velocity as
\begin{equation}\label{stochdef}
v^{n + \frac{2}{3}}_{N} = v^{n + \frac{1}{3}}_{N} + [W((n + 1) \Delta t) - W(n \Delta t)].
\end{equation}
In particular, we are splitting the stochastic part of the structure problem from the deterministic part. This is necessary to obtain a stable scheme. We state the following simple proposition.

\begin{proposition}
Suppose that $v^{n + \frac{1}{3}}_{N}$ is an $\mathcal{F}_{t^{n}_{N}}$ measurable random variable taking values in $H_{0}^{1}(\Gamma)$. Then, $v^{n + \frac{2}{3}}_{N}$ is an $\mathcal{F}_{t^{n + 1}_{N}}$ measurable random variable taking values in $H^{1}(\Gamma)$. 
\end{proposition}

Notice that the solution $v^{n + \frac{2}{3}}_{N}$ to the stochastic subproblem taking values in $H^{1}(\Gamma)$, satisfies pathwise the following integral equality, which will be useful later:
\begin{equation}\label{stochsubp}
\int_{\Gamma} \frac{v^{n + \frac{2}{3}}_{N} - v^{n + \frac{1}{3}}_{N}}{\Delta t} \psi dz = \int_{\Gamma} \frac{W((n + 1)\Delta t) - W(n\Delta t)}{\Delta t} \psi dz, \qquad \text{ for all } \psi \in H_{0}^{1}(\Gamma).
\end{equation}


\begin{proposition}
The following discrete energy identity holds pathwise:
\begin{equation*}
E^{n + \frac{2}{3}}_{N} = E^{n + \frac{1}{3}}_{N} + [W((n + 1) \Delta t) - W(n \Delta t)]\int_{\Gamma} v^{n + \frac{1}{3}}_{N} dz + \frac{L}{2}[W((n + 1) \Delta t) - W(n \Delta t)]^{2},
\end{equation*}
\end{proposition}

\begin{proof}
From
$
v^{n + \frac{2}{3}}_{N} = v^{n + \frac{1}{3}}_{N} + [W((n + 1) \Delta t) - W(n \Delta t)],
$
we get that
\begin{equation*}
\frac{1}{2}|v^{n + \frac{2}{3}}_{N}|^{2} = \frac{1}{2}|v^{n + \frac{1}{3}}_{N}|^{2} + v^{n + \frac{1}{3}}_{N} \cdot [W((n + 1) \Delta t) - W(n \Delta t)] + \frac{1}{2}[W((n + 1) \Delta t) - W(n \Delta t)]^{2}.
\end{equation*}
Therefore, after integrating over $\Gamma$, one gets the desired energy equality,
after recalling that $\eta$ and $\bu$ do not change in this subproblem.
\end{proof}

\subsection{The fluid subproblem}
In this subproblem, we keep the structure displacement fixed
\begin{equation*}
\eta^{n + 1}_{N} = \eta^{n + \frac{2}{3}}_{N},
\end{equation*}
and update the fluid and structure velocities.
To define the problem satisfied by the fluid and structure velocities,
 we introduce the following notation for the corresponding fixed time function spaces:
\begin{equation*}
\mathcal{V} = \{(\boldsymbol{u}, v) \in \mathcal{V}_{F} \times L^{2}(\Gamma) : \boldsymbol{u}|_{\Gamma} = v \boldsymbol{e_{r}}\},\quad
\mathcal{Q} = \{(\boldsymbol{q}, \psi) \in \mathcal{V}_{F} \times H_{0}^{1}(\Gamma) : \boldsymbol{q}|_{\Gamma} = \psi \boldsymbol{e}_{r}\},
\end{equation*}
where $\mathcal{V}_{F}$ is defined by \eqref{VF}. Note that the definition of $\mathcal{V}$ and $\mathcal{Q}$ does not depend on $N$ or $n$.

Then, the fluid subproblem is to find $(\boldsymbol{u}^{n + 1}_{N}, v^{n + 1}_{N})$ taking values in $\mathcal{V}$ pathwise, such that  
\begin{multline}\label{fluidsubp}
\int_{\Omega_{f}} \frac{\boldsymbol{u}^{n + 1}_{N} - \boldsymbol{u}^{n + \frac{2}{3}}_{N}}{\Delta t} \cdot \boldsymbol{q} d\boldsymbol{x} + 2\mu \int_{\Omega_{f}} \boldsymbol{D}(\boldsymbol{u}^{n + 1}_{N}) : \boldsymbol{D}(\boldsymbol{q}) d\boldsymbol{x} + \int_{\Gamma} \frac{v^{n + 1}_{N} - v^{n + \frac{2}{3}}_{N}}{\Delta t} \psi dz \\
= P^{n}_{N, in}\int_{0}^{R} (q_{z})|_{z = 0} dr - P^{n}_{N, out} \int_{0}^{R} (q_{z})|_{z = L} dr, \quad \forall (\boldsymbol{q}, \psi) \in \mathcal{Q},
\end{multline}
pathwise for each outcome $\omega \in \Omega$, where
$
P^{n}_{N, in/out} = \frac{1}{\Delta t} \int_{n\Delta t}^{(n + 1)\Delta t} P_{in/out}(t) dt.
$

\begin{proposition}\label{fluidsubprop}
Suppose that $\boldsymbol{u}^{n + \frac{2}{3}}_{N}$ and $v^{n + \frac{2}{3}}_{N}$ are $\mathcal{F}_{t^{n + 1}_{N}}$ measurable random variables taking values in $\mathcal{V}_{F}$ and $H^{1}(\Gamma)$ respectively. Then, the fluid subproblem \eqref{fluidsubp} has a unique solution $(\boldsymbol{u}^{n + 1}_{N}, v^{n + 1}_{N})$ that is an $\mathcal{F}_{t^{n + 1}_{N}}$ measurable random variable taking values in $\mathcal{V}$. 
\end{proposition}

\begin{proof}
We establish this result again using the Lax Milgram lemma. We let $T^{n}_{N}: \mathcal{V}_{F} \times H^{1}(\Gamma) \times \mathbb{R} \times \mathbb{R} \to \mathcal{V}$ denote the deterministic map that sends deterministic data $(\boldsymbol{u}^{n + \frac{2}{3}}_{N}, v^{n + \frac{2}{3}}_{N}, P^{n}_{N, in}, P^{n}_{N, out}) \in \mathcal{V}_{F} \times H^{1}(\Gamma) \times \mathbb{R} \times \mathbb{R}$ to the unique solution $(\boldsymbol{u}^{n + 1}_{N}, v^{n + 1}_{N}) \in \mathcal{V}$ satisfying the deterministic form of the weak formulation \eqref{fluidsubp}. We want to show that the deterministic linear map $T^{n}_{N}: (\boldsymbol{u}^{n + \frac{2}{3}}_{N}, v^{n + \frac{2}{3}}_{N}, P^{n}_{N, in}, P^{n}_{N, out}) \to (\boldsymbol{u}^{n + 1}_{N}, v^{n + 1}_{N})$ is a continuous map.
We start by showing that the bilinear form $B: \mathcal{V} \times \mathcal{V} \to \mathbb{R}$ given by
\begin{equation*}
B((\boldsymbol{u}, v), (\boldsymbol{q}, \psi)) = \int_{\Omega_{f}} \boldsymbol{u} \cdot \boldsymbol{q} d\boldsymbol{x} + 2\mu (\Delta t) \int_{\Omega_{f}} \boldsymbol{D}(\boldsymbol{u}) : \boldsymbol{D}(\boldsymbol{q}) d\boldsymbol{x} + \int_{\Gamma} v \psi dz,
\end{equation*}
is coercive and continuous. Coercivity follows from the Korn equality (see for example, Lemma 6 on pg. 377 in \cite{CDEM}), applied to  
\begin{equation*}
B((\boldsymbol{u}, v), (\boldsymbol{u}, v)) = \int_{\Omega_{f}} |\boldsymbol{u}|^{2} d\boldsymbol{x} + 2\mu (\Delta t) \int_{\Omega_{f}} |\boldsymbol{D}(\boldsymbol{u})|^{2} d\boldsymbol{x} + \int_{\Gamma} v^{2} dz,
\end{equation*}
to obtain
$
||\nabla \boldsymbol{u}||^{2}_{L^{2}(\Omega_{f})} = 2||\boldsymbol{D}(\boldsymbol{u})||^{2}_{L^{2}(\Omega_{f})}.
$
Continuity of the bilinear form $B$ follows from an application of the Cauchy-Schwarz inequality. 

Next, one can verify that the map sending
\begin{equation*}
(\boldsymbol{q}, \psi) \to \int_{\Omega_{f}} \boldsymbol{u}^{n + \frac{2}{3}}_{N} \cdot \boldsymbol{q} d\boldsymbol{x} + \int_{\Gamma} v^{n + \frac{2}{3}}_{N} \psi dz + (\Delta t) \left(P^{n}_{N, in}\int_{0}^{R} (q_{z})|_{z = 0} dr - P^{n}_{N, out}\int_{0}^{R} (q_{z})|_{z = L} dr\right),
\end{equation*}
is a continuous linear functional on $\mathcal{V}$. Thus, the existence of a unique $(\boldsymbol{u}^{n + 1}_{N}, v^{n + 1}_{N}) \in \mathcal{V}$ satisfying \eqref{fluidsubp} with the larger space of test functions $(\boldsymbol{q}, \psi) \in \mathcal{V}$ is guaranteed by the Lax-Milgram lemma. Note that $\mathcal{V}$ is a larger space than the space $\mathcal{Q}$ required for the test functions in the fluid subproblem \eqref{fluidsubp}. However, we still have the desired uniqueness of the solution in $\mathcal{V}$ if we restrict the test functions to $\mathcal{Q}$ as in \eqref{fluidsubp} because $\mathcal{Q}$ is dense in $\mathcal{V}$.

Then, using coercivity, the trace inequality for $\boldsymbol{u} \in H^{1}(\Omega_{f})$, and the fact that 
\begin{multline*}
B((\boldsymbol{u}^{n + 1}_{N}, v^{n + 1}_{N}), (\boldsymbol{u}^{n + 1}_{N}, v^{n + 1}_{N})) \\
= \int_{\Omega_{f}} \boldsymbol{u}^{n + \frac{2}{3}}_{N} \cdot \boldsymbol{u}^{n + 1}_{N} d\boldsymbol{x} + \int_{\Gamma} v^{n + \frac{2}{3}}_{N} \cdot v^{n + 1}_{N} dz + (\Delta t)\left(P^{n}_{N, in}\int_{0}^{R} (\boldsymbol{u}^{n + 1}_{N})_{z}|_{z = 0} dr - P^{n}_{N, out}\int_{0}^{R} (\boldsymbol{u}^{n + 1}_{N})_{z}|_{z = L} dr\right),
\end{multline*}
we obtain the continuity of the map $T^{n}$.

Thus, since $\boldsymbol{u}^{n + \frac{2}{3}}_{N}$ and $v^{n + \frac{2}{3}}_{N}$ are $\mathcal{F}_{t^{n + 1}_{N}}$ measurable by assumption, the random functions
$
(\boldsymbol{u}^{n + 1}_{N}, v^{n + 1}_{N}) = T^{n}_{N} \circ (\boldsymbol{u}^{n + \frac{2}{3}}_{N}, v^{n + \frac{2}{3}}_{N}),
$
which solve the fluid subproblem, are $\mathcal{F}_{t^{n + 1}_{N}}$ measurable random variables also. 
\end{proof}


\begin{proposition}
The following discrete energy identity holds pathwise:
\begin{multline*}
E^{n + 1}_{N} + 2\mu (\Delta t) \int_{\Omega_{f}} |\boldsymbol{D}(\boldsymbol{u}^{n + 1}_{N})|^{2} d\boldsymbol{x} + \frac{1}{2}\left(||\boldsymbol{u}^{n + 1}_{N} - \boldsymbol{u}^{n + \frac{2}{3}}_{N}||_{L^{2}(\Omega_{f})}^{2} \right) + \frac{1}{2}\left(||v^{n + 1}_{N} - v^{n + \frac{2}{3}}_{N}||_{L^{2}(\Gamma)}^{2} \right) \\
= E^{n + \frac{2}{3}}_{N} + (\Delta t) \left(P^{n}_{N, in}\int_{0}^{R} (\boldsymbol{u}^{n + 1}_{N})_{z}|_{z = 0} dr - P^{n}_{N, out} \int_{0}^{R} (\boldsymbol{u}^{n + 1}_{N})_{z}|_{z = L} dr\right).
\end{multline*}
\end{proposition}

\begin{proof}
We can substitute $(\boldsymbol{q}, \psi) = (\boldsymbol{u}^{n + 1}_{N}, v^{n + 1}_{N})$ into the weak formulation of the fluid subproblem since we showed in  Proposition \ref{fluidsubprop} that \eqref{fluidsubp} holds more generally for test functions in $\mathcal{V}$. We obtain
\begin{multline*}
\int_{\Omega_{f}} \frac{\boldsymbol{u}^{n + 1}_{N} - \boldsymbol{u}^{n + \frac{2}{3}}_{N}}{\Delta t} \cdot \boldsymbol{u}^{n + 1}_{N} d\boldsymbol{x} + 2\mu \int_{\Omega_{f}} |\boldsymbol{D}(\boldsymbol{u}^{n + 1}_{N})|^{2} d\boldsymbol{x} + \int_{\Gamma} \frac{v^{n + 1}_{N} - v^{n + \frac{2}{3}}_{N}}{\Delta t} \cdot v^{n + 1}_{N} dz \\
= P^{n}_{N, in}\int_{0}^{R} (\boldsymbol{u}^{n + 1}_{N})_{z}|_{z = 0} dr - P^{n}_{N, out} \int_{0}^{R} (\boldsymbol{u}^{n + 1}_{N})_{z}|_{z = L} dr.
\end{multline*}
The desired equality follows after multiplication by $\Delta t$, and by using the identity
$
(a - b) \cdot a = \frac{1}{2}(|a|^{2} + |a - b|^{2} - |b|^{2}).
$
\end{proof}
\if 1 = 0
By using the trace inequality for $\boldsymbol{u} \in H^{1}(\Omega)$, the Poincare inequality, and the Korn inequality, we get the estimate that the right hand side is 
\begin{multline*}
\left|(\Delta t) P^{n}_{in}\int_{0}^{R} (\boldsymbol{u}^{n + 1}_{N})_{z}|_{z = 0} dr - (\Delta t) P^{n}_{out} \int_{0}^{R} (\boldsymbol{u}^{n + 1}_{N})_{z}|_{z = L} dr\right| \le C(\Delta t) ||\nabla \boldsymbol{u}^{n + 1}_{N}||_{L^{2}(\Omega)} \cdot \left(|P^{n}_{in}| + |P^{n}_{out}|\right) \\
= C(\Delta t) ||\boldsymbol{D} (\boldsymbol{u}^{n + 1}_{N})||_{L^{2}(\Omega)} \cdot \left(|P^{n}_{in}| + |P^{n}_{out}|\right) \le C\epsilon ||D(\boldsymbol{u}^{n + 1}_{N})||_{L^{2}(\Omega)}^{2} + \frac{C}{\epsilon}(\Delta t)^{2} (|P^{n}_{in}|^{2} + |P^{n}_{out}|^{2}).
 \end{multline*}
 By choosing $\epsilon$ sufficiently small so that $C\epsilon < \mu \Delta t$, we get that 
 \begin{multline*}
E^{n + 1}_{N} - E^{n}_{N} + D^{n + 1}_{N} + \frac{1}{2} ||\boldsymbol{u}^{n + 1} - \boldsymbol{u}^{n + \frac{1}{2}}||_{L^{2}(\Omega)}^{2} + \frac{1}{2}||v^{n + 1} - v^{n + \frac{1}{2}}||_{L^{2}(\Gamma)}^{2} \le C_{\epsilon} (\Delta t)^{2} \left(|P^{n}_{in}|^{2} + |P^{n}_{out}|^{2}\right).
\end{multline*}
We have, by using the properties of Brownian motion, that 
\begin{align*}
(\Delta t)^{2} |P^{n}_{in/out}|^{2} &\le 2\left(\int_{n\Delta t}^{(n + 1)\Delta t} P_{in/out}(t) dt\right)^{2} + 2(W_{in/out}((n + 1)\Delta t) - W_{in/out}(n\Delta t))^{2} \\
&\le 2(\Delta t)||P_{in/out}||_{L^{2}(n\Delta t, (n + 1)\Delta t)}^{2} + 2(W_{in/out}((n + 1)\Delta t) - W_{in/out}(n\Delta t))^{2}.
\end{align*}
\fi

\subsection{The full, coupled semidiscrete problem}\label{SemidiscreteProblem}
By adding the weak formulations of the stochastic and fluid subproblems \eqref{stochsubp} and \eqref{fluidsubp}, and the second equation in the structure subproblem \eqref{structuresubp}, we have that the solution to the full semidiscrete problem is 
$(\boldsymbol{u}^{n + 1}_{N}, v^{n + 1}_{N})\in \mathcal{V}$, and $(v^{n + \frac{1}{3}}_{N}, \eta^{n + \frac{1}{3}}_{N}) \in H_{0}^{1}(\Gamma) \times H_{0}^{1}(\Gamma)$, satisfying the following equality pathwise:
\begin{equation}\label{semi1}
\begin{array}{l}
\displaystyle{\int_{\Omega_{f}} \frac{\boldsymbol{u}^{n + 1}_{N} - \boldsymbol{u}^{n}_{N}}{\Delta t} \cdot \boldsymbol{q} d\boldsymbol{x} + 2\mu \int_{\Omega_{f}} \boldsymbol{D}(\boldsymbol{u}^{n + 1}_{N}) : \boldsymbol{D}(\boldsymbol{q}) d\boldsymbol{x} + \int_{\Gamma} \frac{v^{n + 1}_{N} - v^{n}_{N}}{\Delta t} \psi dz + \int_{\Gamma} \nabla \eta^{n + 1}_{N} \cdot \nabla \psi dz }
\\
\displaystyle{= \int_{\Gamma} \frac{W((n + 1)\Delta t) - W(n\Delta t)}{\Delta t} \psi dz + P^{n}_{N, in}\int_{0}^{R} (q_{z})|_{z = 0} dr - P^{n}_{N, out} \int_{0}^{R} (q_{z})|_{z = L} dr, 
\  \forall (\boldsymbol{q}, \psi) \in \mathcal{Q},
}
\\
\hskip 1.3in \displaystyle{
\int_{\Gamma} \frac{\eta^{n + 1}_{N} - \eta^{n}_{N}}{\Delta t} \phi dz = \int_{\Gamma} v^{n + \frac{1}{3}}_{N} \phi dz, \quad  \forall \phi \in L^{2}(\Gamma),
}
\end{array}
\end{equation}
where
$
\displaystyle{P^{n}_{N, in/out} = \frac{1}{\Delta t} \int_{n\Delta t}^{(n + 1)\Delta t} P_{in/out}(t) dt.}
$
Note that $\eta^{n + 1}_{N} = \eta^{n + \frac{1}{3}}_{N}$ by the way we constructed the splitting scheme.

The following proposition provides uniform estimates on the expectation of the kinetic and elastic energy for the full, semidiscrete coupled problem (uniform in the number of time steps $N$, or equivalently, uniform in $\Delta t$), as well as uniform estimates on the expectation of
the numerical dissipation.
\if 1 = 0
We derived discrete energy estimates for each of the subproblems separately, and now we want to derive an energy estimate for the full semidiscrete problem. We accomplish this in the following proposition. These discrete energy estimates will be uniform in the number of time steps $N$. For convenience, we reproduce all energy identities for the structure, stochastic, and fluid subproblems below:
\begin{equation*}
E^{n + \frac{1}{3}}_{N} + \frac{1}{2} \left( ||v^{n + \frac{1}{3}}_{N} - v^{n}_{N}||^{2}_{L^{2}(\Gamma)}\right) + \frac{1}{2} \left(||\nabla \eta^{n + \frac{1}{3}}_{N} - \nabla \eta^{n}_{N}||^{2}_{L^{2}(\Gamma)}\right) = E^{n}_{N},
\end{equation*}
\begin{equation*}
E^{n + \frac{2}{3}}_{N} = E^{n + \frac{1}{3}}_{N} + [W((n + 1) \Delta t) - W(n \Delta t)]\int_{\Gamma} v^{n + \frac{1}{3}}_{N} dz + \frac{L}{2}[W((n + 1) \Delta t) - W(n \Delta t)]^{2},
\end{equation*}
\begin{multline*}
E^{n + 1}_{N} + 2\mu (\Delta t) \int_{\Omega_{f}} |\boldsymbol{D}(\boldsymbol{u}^{n + 1}_{N})|^{2} d\boldsymbol{x} + \frac{1}{2}\left(||\boldsymbol{u}^{n + 1}_{N} - \boldsymbol{u}^{n + \frac{2}{3}}_{N}||_{L^{2}(\Omega_{f})}^{2} \right) + \frac{1}{2}\left(||v^{n + 1}_{N} - v^{n + \frac{2}{3}}_{N}||_{L^{2}(\Gamma)}^{2} \right) \\
= E^{n + \frac{2}{3}}_{N} + (\Delta t) \left(P^{n}_{in}\int_{0}^{R} (\boldsymbol{u}^{n + 1}_{N})_{z}|_{z = 0} dr - P^{n}_{out} \int_{0}^{R} (\boldsymbol{u}^{n + 1}_{N})_{z}|_{z = L} dr\right).
\end{multline*}
\fi

\begin{proposition}\label{uniformenergy}
Let $N>0$  and let $\Delta t = \frac{T}{N}$. There exists a constant $C$ independent of $N$ and depending only on the
initial data, the parameters of the problem, and $||P_{in/out}||^{2}_{L^{2}(0, T)}$, such that the following uniform energy estimates hold:
\begin{enumerate}
\item \textit{\textbf{Uniform semidiscrete kinetic energy and elastic energy estimates:}} 
\begin{equation*}
\mathbb{E}\left(\max_{n = 0, 1, ..., N - 1} E^{n + \frac{1}{3}}_{N}\right) \le C, \ \ \mathbb{E}\left(\max_{n = 0, 1, ..., N - 1} E^{n + \frac{2}{3}}_{N}\right) \le C, \text{ and } \ \ \mathbb{E}\left(\max_{n = 0, 1, ..., N - 1} E^{n + 1}_{N}\right) \le C.
\end{equation*}
\item \textit{\textbf{Uniform semidiscrete viscous fluid dissipation estimate:}}
\begin{equation*}
\sum_{j = 1}^{N} \mathbb{E}(D^{j}_{N}) \le C.
\end{equation*}
\item \textit{\textbf{Uniform numerical dissipation estimates:}}
\begin{equation*}
\sum_{n = 0}^{N - 1} \left(\mathbb{E}\left( ||v^{n + \frac{1}{3}}_{N} - v^{n}_{N}||^{2}_{L^{2}(\Gamma)}\right) + \mathbb{E}\left(||\nabla \eta^{n + \frac{1}{3}}_{N} - \nabla \eta^{n}_{N}||^{2}_{L^{2}(\Gamma)}\right)\right) \le C. 
\end{equation*}
\begin{equation*}
\sum_{n = 0}^{N - 1} \mathbb{E}\left(||v^{n + \frac{2}{3}}_{N} - v^{n + \frac{1}{3}}_{N}||^{2}_{L^{2}(\Gamma)}\right) \le C.
\end{equation*}
\begin{equation*}
\sum_{n = 0}^{N - 1} \left(\mathbb{E}\left(||\boldsymbol{u}^{n + 1}_{N} - \boldsymbol{u}^{n + \frac{2}{3}}_{N}||_{L^{2}(\Omega_{f})}^{2}\right) + \mathbb{E}\left(||v^{n + 1}_{N} - v^{n + \frac{2}{3}}_{N}||_{L^{2}(\Gamma)}^{2}\right)\right) \le C. 
\end{equation*}
\end{enumerate}
\end{proposition}

\begin{proof}

First, recall the definitions of the discrete energy $E^{n + \frac{i}{3}}_{N}$ and the discrete fluid dissipation $D^{n}_{N}$ from \eqref{discreteenergy} and \eqref{discretedissipation}. We start with the second uniform numerical dissipation estimate. This estimate follows directly
from  the stochastic subproblem \eqref{stochdef} after integration
\begin{equation*}
\int_{\Gamma} |v^{n + \frac{2}{3}}_{N} - v^{n + \frac{1}{3}}_{N}|^{2} dz = L \cdot [W((n + 1)\Delta t) - W(n\Delta t)]^{2},
\end{equation*}
and summation of the expectations of both sides:
\begin{equation*}
\sum_{n = 0}^{N - 1} \mathbb{E}\left(||v^{n + \frac{2}{3}}_{N} - v^{n + \frac{1}{3}}_{N}||^{2}_{L^{2}(\Gamma)}\right) = \sum_{n = 0}^{N - 1} \mathbb{E}\left(L \cdot [W((n + 1)\Delta t) - W(n\Delta t)]^{2}\right) = LT.
\end{equation*}

We now verify the remaining uniform energy estimates. By summing the structure, stochastic, and fluid discrete energy identities, we obtain
\begin{small}
\begin{multline}\label{remark}
E^{n + 1}_{N} + \sum_{k = 0}^{n} \left(2\mu(\Delta t) \int_{\Omega_{f}} |\boldsymbol{D}(\boldsymbol{u}_{N}^{k + 1})|^{2} d\boldsymbol{x} + \frac{1}{2}\left(||\boldsymbol{u}^{k + 1}_{N} - \boldsymbol{u}^{k + \frac{2}{3}}_{N}||_{L^{2}(\Omega_{f})}^{2} \right) + \frac{1}{2}\left(||v^{k + 1}_{N} - v^{k + \frac{2}{3}}_{N}||_{L^{2}(\Gamma)}^{2} \right)\right) 
\\
+ \sum_{k = 0}^{n} \left(\frac{1}{2} \left( ||v^{k + \frac{1}{3}}_{N} - v^{k}_{N}||^{2}_{L^{2}(\Gamma)}\right) + \frac{1}{2} \left(||\nabla \eta^{k + \frac{1}{3}}_{N} - \nabla \eta^{k}_{N}||^{2}_{L^{2}(\Gamma)}\right)\right) 
\\
= E_{0} + (\Delta t) \sum_{k = 0}^{n} \left(P^{k}_{N, in}\int_{0}^{R} (\boldsymbol{u}^{k + 1}_{N})_{z}|_{z = 0} dr - P^{k}_{N, out} \int_{0}^{R} (\boldsymbol{u}^{k + 1}_{N})_{z}|_{z = L} dr \right) 
\\
+ \sum_{k = 0}^{n} \left([W((k + 1) \Delta t) - W(k \Delta t)]\int_{\Gamma} v^{k + \frac{1}{3}}_{N} dz + \frac{L}{2}[W((k + 1) \Delta t) - W(k \Delta t)]^{2}\right),
\end{multline}
\begin{multline*}
E^{n + \frac{2}{3}}_{N} + \sum_{k = 0}^{n - 1} \left(2\mu(\Delta t) \int_{\Omega_{f}} |\boldsymbol{D}(\boldsymbol{u}_{N}^{k + 1})|^{2} d\boldsymbol{x} + \frac{1}{2}\left(||\boldsymbol{u}^{k + 1}_{N} - \boldsymbol{u}^{k + \frac{2}{3}}_{N}||_{L^{2}(\Omega_{f})}^{2} \right) + \frac{1}{2}\left(||v^{k + 1}_{N} - v^{k + \frac{2}{3}}_{N}||_{L^{2}(\Gamma)}^{2} \right)\right) \\
+ \sum_{k = 0}^{n} \left(\frac{1}{2} \left( ||v^{k + \frac{1}{3}}_{N} - v^{k}_{N}||^{2}_{L^{2}(\Gamma)}\right) + \frac{1}{2} \left(||\nabla \eta^{k + \frac{1}{3}}_{N} - \nabla \eta^{k}_{N}||^{2}_{L^{2}(\Gamma)}\right)\right) \\
= E_{0} + (\Delta t) \sum_{k = 0}^{n - 1} \left(P^{k}_{N, in}\int_{0}^{R} (\boldsymbol{u}^{k + 1}_{N})_{z}|_{z = 0} dr - P^{k}_{N, out} \int_{0}^{R} (\boldsymbol{u}^{k + 1}_{N})_{z}|_{z = L} dr\right) \\
+ \sum_{k = 0}^{n} \left([W((k + 1) \Delta t) - W(k \Delta t)]\int_{\Gamma} v^{k + \frac{1}{3}}_{N} dz + \frac{L}{2}[W((k + 1) \Delta t) - W(k \Delta t)]^{2}\right),
\end{multline*}
and 
\begin{multline*}
E^{n + \frac{1}{3}}_{N} + \sum_{k = 0}^{n - 1} \left(2\mu(\Delta t) \int_{\Omega_{f}} |\boldsymbol{D}(\boldsymbol{u}_{N}^{k + 1})|^{2} d\boldsymbol{x} + \frac{1}{2}\left(||\boldsymbol{u}^{k + 1}_{N} - \boldsymbol{u}^{k + \frac{2}{3}}_{N}||_{L^{2}(\Omega_{f})}^{2} \right) + \frac{1}{2}\left(||v^{k + 1}_{N} - v^{k + \frac{2}{3}}_{N}||_{L^{2}(\Gamma)}^{2} \right)\right) \\
+ \sum_{k = 0}^{n} \left(\frac{1}{2} \left( ||v^{k + \frac{1}{3}}_{N} - v^{k}_{N}||^{2}_{L^{2}(\Gamma)}\right) + \frac{1}{2} \left(||\nabla \eta^{k + \frac{1}{3}}_{N} - \nabla \eta^{k}_{N}||^{2}_{L^{2}(\Gamma)}\right)\right) \\
= E_{0} + (\Delta t) \sum_{k = 0}^{n - 1} \left(P^{k}_{N, in}\int_{0}^{R} (\boldsymbol{u}^{k + 1}_{N})_{z}|_{z = 0} dr - P^{k}_{N, out} \int_{0}^{R} (\boldsymbol{u}^{k + 1}_{N})_{z}|_{z = L} dr\right) \\
+ \sum_{k = 0}^{n - 1} \left([W((k + 1) \Delta t) - W(k \Delta t)]\int_{\Gamma} v^{k + \frac{1}{3}}_{N} dz + \frac{L}{2}[W((k + 1) \Delta t) - W(k \Delta t)]^{2}\right),
\end{multline*}
\end{small}
for $n = 0, 1, ..., N - 1$. Therefore,
\begin{small}
\begin{multline*}
\mathbb{E}\left(\max_{i = 1, 2, 3} \left[\max_{n = 0, 1, ..., N - 1} E^{n + \frac{i}{3}}_{N}\right]\right) 
+ \sum_{k = 0}^{N - 1} \Bigg[\mathbb{E}\left(2\mu(\Delta t) \int_{\Omega_{f}} |\boldsymbol{D}(\boldsymbol{u}_{N}^{k + 1})|^{2} d\boldsymbol{x}\right) 
+ \frac{1}{2} \mathbb{E} \left(||\boldsymbol{u}^{k + 1}_{N} - \boldsymbol{u}^{k + \frac{2}{3}}_{N}||_{L^{2}(\Omega_{f})}^{2} \right) \\
+ \frac{1}{2} \mathbb{E} \left(||v^{k + 1}_{N} - v^{k + \frac{2}{3}}_{N}||_{L^{2}(\Gamma)}^{2} \right) 
+ \left(\frac{1}{2} \mathbb{E} \left( ||v^{k + \frac{1}{3}}_{N} - v^{k}_{N}||^{2}_{L^{2}(\Gamma)}\right) + \frac{1}{2} \mathbb{E} \left(||\nabla \eta^{k + \frac{1}{3}}_{N} - \nabla \eta^{k}_{N}||^{2}_{L^{2}(\Gamma)}\right)\right)\Bigg] \\
\le 2E_{0} + 2\mathbb{E}\left[\max_{n = 0, 1, ..., N - 1} (\Delta t) \sum_{k = 0}^{n} \left(P^{k}_{N, in}\int_{0}^{R} (\boldsymbol{u}^{k + 1}_{N})_{z}|_{z = 0} dr - P^{k}_{N, out} \int_{0}^{R} (\boldsymbol{u}^{k + 1}_{N})_{z}|_{z = L} dr\right)\right] \\
+ 2\mathbb{E}\left[\max_{n = 0, 1, ..., N - 1} \sum_{k = 0}^{n} \left([W((k + 1) \Delta t) - W(k \Delta t)]\int_{\Gamma} v^{k + \frac{1}{3}}_{N} dz + \frac{L}{2}[W((k + 1) \Delta t) - W(k \Delta t)]^{2}dr\right)\right].
\end{multline*}
\end{small}
What is left is to bound the quantities
\begin{equation*}
I_{1} := \mathbb{E}\left[\max_{n = 0, 1, ..., N - 1} (\Delta t) \sum_{k = 0}^{n} \left(P^{k}_{N, in}\int_{0}^{R} (\boldsymbol{u}^{k + 1}_{N})_{z}|_{z = 0} dr - P^{k}_{N, out} \int_{0}^{R} (\boldsymbol{u}^{k + 1}_{N})_{z}|_{z = L} dr\right)\right],
\end{equation*}
and
\begin{small}
\begin{multline*}
\mathbb{E}\left[\max_{n = 0, 1, ..., N - 1} \sum_{k = 0}^{n} \left([W((k + 1) \Delta t) - W(k \Delta t)]\int_{\Gamma} v^{k + \frac{1}{3}}_{N} dz + \frac{L}{2}[W((k + 1) \Delta t) - W(k \Delta t)]^{2}dr\right)\right] \\
\le \mathbb{E}\left[\max_{n = 0, 1, ..., N - 1} \sum_{k = 0}^{n} \left([W((k + 1) \Delta t) - W(k \Delta t)]\int_{\Gamma} v^{k + \frac{1}{3}}_{N} dz\right)\right] 
+ \frac{L}{2}\mathbb{E}\left(\sum_{k = 0}^{N - 1} [W((k + 1) \Delta t) - W(k \Delta t)]^{2}\right) := I_{2} + I_{3}.
\end{multline*}
\end{small}

\noindent \textbf{\textit{Bound for $I_{1}$:}} The same argument will work for $P^{k}_{N, in}$ and $P^{k}_{N, out}$ so without loss of generality, we perform the bounds below for $P^{k}_{N, in}$. We recall that
$\displaystyle{
P^{k}_{N, in} = \frac{1}{\Delta t} \int_{k\Delta t}^{(k + 1)\Delta t} P_{in}(t) dt,
}$
where $P^{k}_{N, in}$ is deterministic. Therefore, we have the following bound, for the term in $I_{1}$ that involves $P^{k}_{N, in}$:
\begin{small}
\begin{multline*}
\mathbb{E}\left[\max_{n = 0, 1, ..., N - 1} (\Delta t) \sum_{k = 0}^{n} \left(P^{k}_{N, in}\int_{0}^{R} (\boldsymbol{u}^{k + 1}_{N})_{z}|_{z = 0} dr\right)\right] 
\le \mathbb{E}\left(\sum_{k = 0}^{N - 1} (\Delta t) |P^{k}_{N, in}| \left|\int_{0}^{R} (\boldsymbol{u}^{k + 1}_{N})_{z}|_{z = 0} dr\right|\right) \\
\le \sum_{k = 0}^{N - 1} \mathbb{E} \left[(\Delta t) \frac{1}{4\epsilon} |P^{k}_{N, in}|^{2} + \epsilon(\Delta t)\left(\int_{0}^{R} (\boldsymbol{u}^{k + 1}_{N})_{z}|_{z = 0} dr\right)^{2}\right] \\
\le \sum_{k = 0}^{N - 1} \mathbb{E} \left[\frac{1}{4\epsilon} \cdot \frac{1}{\Delta t} \left(\int_{k\Delta t}^{(k + 1)\Delta t} P_{in}(t) dt\right)^{2} + C\epsilon(\Delta t) \int_{0}^{R} (\boldsymbol{u}^{k + 1}_{N})_{z}^{2}|_{z = 0} dr\right] \\
\le \sum_{k = 0}^{N - 1} \mathbb{E} \left[\frac{1}{4\epsilon} ||P_{in}||_{L^{2}(k\Delta t, (k + 1)\Delta t)}^{2} + C\epsilon(\Delta t) \int_{\Omega_{f}} |\boldsymbol{D}(\boldsymbol{u}^{k + 1}_{N})|^{2} d\boldsymbol{x} \right] 
= \frac{1}{4\epsilon} ||P_{in}||^{2}_{L^{2}(0, T)} + \sum_{k = 0}^{N - 1} \mathbb{E}\left(C\epsilon(\Delta t) \int_{\Omega_{f}} |\boldsymbol{D}(\boldsymbol{u}^{k + 1}_{N})|^{2} d\boldsymbol{x}\right),
\end{multline*}
where we used Korn's inequality in the last line.
\end{small}
Therefore,
\begin{equation*}
I_{1} \le \frac{1}{4\epsilon}||P_{in}||^{2}_{L^{2}(0, T)} + \frac{1}{4\epsilon}||P_{out}||^{2}_{L^{2}(0, T)} + \sum_{k = 0}^{N - 1} \mathbb{E}\left(2C\epsilon(\Delta t) \int_{\Omega_{f}} |\boldsymbol{D}(\boldsymbol{u}^{k + 1}_{N})|^{2} d\boldsymbol{x}\right).
\end{equation*}
Note that the constant $C$ is independent of $\Delta t$ and $N$. It is the geometric constant arising from the application of the Poincar\'{e} inequality on the fluid domain $\Omega_{f}$.

\vspace{0.1in}

\noindent \textbf{\textit{Bound for $I_{2}$:}} Next, we examine $I_{2}$ and start with an estimate involving the absolute values:
\begin{align*}
I_{2} 
\le \mathbb{E}\left(\max_{n = 0, 1, ..., N - 1} \left|\sum_{k = 0}^{n} \left(\int_{0}^{L} v^{k + \frac{1}{3}}_{N} dz\right) \cdot [W((k + 1)\Delta t) - W(k\Delta t)]\right|\right).
\end{align*}
Next, we consider the expression under the absolute value sign, and consider it as the following {\emph{stochastic integral}}:
\begin{equation*}
 \sum_{k = 0}^{n} \left(\int_{0}^{L} v^{k + \frac{1}{3}}_{N} dz\right) \cdot [W((k + 1)\Delta t) - W(k\Delta t)]
=  \int_{0}^{(n + 1)\Delta t} f(t) dW(t),
\end{equation*}
where $f(t)$ is the random function on $[0,T]$ defined by:
\begin{equation}\label{remark2}
f(t) = \sum_{k = 0}^{N - 1} \left(\int_{0}^{L} v^{k + \frac{1}{3}}_{N} dz\right) \cdot 1_{(k\Delta t, (k + 1)\Delta t]}(t).
\end{equation}
Because $v^{k + \frac{1}{3}}_{N}$ is $\mathcal{F}_{t^{k}_{N}}$ measurable, this integrand is predictable. 
{\emph{This is a direct consequence of how we split the stochastic part of the
problem from the structure subproblem.}} 
Without such a splitting, we would not be able to make the same conclusion.  
Hence, since the stochastic integral is a continuous process in time, we have
\begin{equation*}
I_{2} \le \mathbb{E}\left(\max_{0 \le s \le T} \left|\int_{0}^{s} f(t) dW\right|\right).
\end{equation*}
Using the BDG inequality, we obtain that
\begin{small}
\begin{multline*}
I_{2} \le \mathbb{E}\left[\left(\int_{0}^{T} |f(t)|^{2} dt\right)^{1/2}\right] = \mathbb{E}\left[(\Delta t)^{1/2} \left(\sum_{k = 0}^{N - 1} \left(\int_{0}^{L} v^{k + \frac{1}{3}}_{N} dz\right)^{2}\right)^{1/2} \right] 
\le \epsilon(\Delta t)\mathbb{E} \sum_{k = 0}^{N - 1} \left(\int_{0}^{L} v^{k + \frac{1}{3}}_{N} dz\right)^{2} + \frac{1}{4\epsilon} \\
\le \epsilon L (\Delta t)\mathbb{E} \sum_{k = 0}^{N - 1} ||v^{k + \frac{1}{3}}_{N}||_{L^{2}(\Gamma)}^{2} + \frac{1}{4\epsilon} 
\le \epsilon L N(\Delta t) \mathbb{E} \left(\max_{k = 0, 1, ..., N - 1} ||v^{k + \frac{1}{3}}_{N}||^{2}_{L^{2}(\Gamma)}\right) + \frac{1}{4\epsilon} \le 2\epsilon L N(\Delta t) \mathbb{E} \left(\max_{n = 0, 1, ..., N - 1} E^{n + \frac{1}{3}}_{N}\right) + \frac{1}{4\epsilon}.
\end{multline*}
\end{small}

\vspace{0.1in}

\noindent \textbf{\textit{Bound for $I_{3}$:}} Finally, by using the properties of Brownian motion, we immediately deduce that
\begin{equation*}
I_{3} := \frac{L}{2}\mathbb{E}\left(\sum_{k = 0}^{N - 1} [W((k + 1) \Delta t) - W(k \Delta t)]^{2}\right) = \frac{LT}{2}.
\end{equation*}

\vspace{0.2in}

\noindent \textbf{\textit{Conclusion:}} From the above estimates, we conclude that
\begin{small}
\begin{multline*}
\mathbb{E}\left(\max_{i = 1, 2, 3} \left[\max_{n = 0, 1, ..., N - 1} E^{n + \frac{i}{3}}_{N}\right]\right) 
+ \sum_{k = 0}^{N - 1} \Bigg[\mathbb{E}\left(2\mu(\Delta t) \int_{\Omega_{f}} |\boldsymbol{D}(\boldsymbol{u}_{N}^{k + 1})|^{2} d\boldsymbol{x}\right) 
+ \frac{1}{2} \mathbb{E} \left(||\boldsymbol{u}^{k + 1}_{N} - \boldsymbol{u}^{k + \frac{2}{3}}_{N}||_{L^{2}(\Omega_{f})}^{2} \right) \\
+ \frac{1}{2} \mathbb{E} \left(||v^{k + 1}_{N} - v^{k + \frac{2}{3}}_{N}||_{L^{2}(\Gamma)}^{2} \right) 
+ \left(\frac{1}{2} \mathbb{E} \left( ||v^{k + \frac{1}{3}}_{N} - v^{k}_{N}||^{2}_{L^{2}(\Gamma)}\right) + \frac{1}{2} \mathbb{E} \left(||\nabla \eta^{k + \frac{1}{3}}_{N} - \nabla \eta^{k}_{N}||^{2}_{L^{2}(\Gamma)}\right)\right)\Bigg] \\
\le 2E_{0} + \frac{1}{2\epsilon}||P_{in}||^{2}_{L^{2}(0, T)} + \frac{1}{2\epsilon}||P_{out}||^{2}_{L^{2}(0, T)} + \sum_{k = 0}^{N - 1} \mathbb{E}\left(4C\epsilon(\Delta t) \int_{\Omega_{f}} |\boldsymbol{D}(\boldsymbol{u}^{k + 1}_{N})|^{2} d\boldsymbol{x}\right) 
+ 4\epsilon L T \cdot \mathbb{E} \left(\max_{n = 0, 1, ..., N - 1} E^{n + \frac{1}{3}}_{N}\right) + \frac{1}{2\epsilon} + LT.
\end{multline*}
\end{small}
We note that the constant $C$ depends only on the fluid domain $\Omega_{f}$ and not on $\Delta t$ or $N$. The result follows once we fix $\epsilon > 0$, independent of $\Delta t$, such that $4C\epsilon < \mu$ and $4\epsilon LT < \frac{1}{2}$, and move the 
associated terms from the right hand side to the left hand side. We emphasize that this gives a uniform energy estimate because the choice of $\epsilon$ is independent of $\Delta t$ and hence $N$.

\if 1 = 0
We use the previous discrete energy estimates for the subproblems, which hold for $n = 0, 1, ..., N - 1$:
\begin{equation*}
(E^{n + \frac{1}{2}}_{N} - E^{n}_{N}) + \frac{1}{2} \mathbb{E}\left( ||v^{n + \frac{1}{2}}_{N} - v^{n}_{N}||^{2}_{L^{2}(\Gamma)}\right) + \frac{1}{2} \mathbb{E}\left(||\nabla \eta^{n + \frac{1}{2}}_{N} - \eta^{n}_{N}||^{2}_{L^{2}(\Gamma)}\right) = 0,
\end{equation*}
and
 \begin{multline*}
(E^{n + 1}_{N} - E^{n + \frac{1}{2}}_{N})  + D^{n + 1}_{N} + \frac{1}{2} \mathbb{E}\left(||\boldsymbol{u}^{n + 1} - \boldsymbol{u}^{n + \frac{1}{2}}||_{L^{2}(\Omega)}^{2}\right) + \frac{1}{2}\mathbb{E}\left(||v^{n + 1} - v^{n + \frac{1}{2}}||_{L^{2}(\Gamma)}^{2}\right) \\
\le (\Delta t)\left(1 + ||P_{in}||^{2}_{L^{2}(n\Delta t, (n + 1)\Delta t)} + ||P_{out}||^{2}_{L^{2}(n\Delta t, (n + 1)\Delta t)}\right).
\end{multline*}
Summing these over values of $n$ gives the desired result, once we note that $E^{0}_{N} := E_{0}$ is independent of $N$, as it depends only on the initial data, and
\begin{multline*}
\sum_{n = 1}^{N} (\Delta t)\left(1 + ||P_{in}||^{2}_{L^{2}(n\Delta t, (n + 1)\Delta t)} + ||P_{out}||^{2}_{L^{2}(n\Delta t, (n + 1)\Delta t)}\right) \\
= N(\Delta t) + (\Delta t) ||P_{in}||^{2}_{L^{2}(0, T)} + (\Delta t) ||P_{out}||^{2}_{L^{2}(0, T)} \le T(1 + ||P_{in}||^{2}_{L^{2}(0, T)} + ||P_{out}||^{2}_{L^{2}(0, T)}).
\end{multline*}
In fact, one can specifically take 
\begin{equation*}
C := 2E_{0} + 2T(1 + ||P_{in}||^{2}_{L^{2}(0, T)} + ||P_{out}||^{2}_{L^{2}(0, T)})
\end{equation*}
in all of the above inequalities.
\fi

\end{proof}

\begin{remark}\label{stochastic_increment}
We remark that it is in these energy estimates that one can see the importance of using our particular splitting strategy to obtain a stable scheme. 
Namely, this splitting strategy enabled us to estimate the terms involving the white noise as stochastic integrals, such as the second to last term in 
estimate \eqref{remark}. Because  $v^{k+\frac{1}{3}}_{N}$ is ${\cal{F}}_{t^k_{N}}$ measurable, 
the stochastic increment $[W((k + 1) \Delta t) - W(k \Delta t)]$ is independent of the integral of $v^{k+\frac{1}{3}}_{N}$, and hence,
we were able to rewrite this term as a stochastic integral, see \eqref{remark2}. 
\end{remark}

\section{Approximate solutions}

We use the solutions at fixed times of our semidiscrete scheme, $\boldsymbol{u}^{n + \frac{i}{3}}_{N}$, $\eta^{n + \frac{i}{3}}_{N}$, and $v^{n + \frac{i}{3}}_{N}$ for $i = 1, 2, 3$, to create approximate solutions for the given stochastic FSI problem in time on the time interval $[0, T]$, for each $N$,
which we will need to pass to the limit as $\Delta t \to 0$. 
The approximate solutions will be defined as piecewise functions in time.
However,
we must be careful in  this construction of approximate solutions 
to make sure that they are adapted to the given filtration $\{\mathcal{F}_{t}\}_{t \ge 0}$ associated to the given Brownian motion.

\subsection{Definition of approximate solutions}

We start with the fluid. We define the approximate random function $\boldsymbol{u}_{N}$ on $[0, T] \times \Omega_{f}$ to be the piecewise constant function
\begin{equation*}
\boldsymbol{u}_{N}(t, \cdot) = \boldsymbol{u}^{n - 1}_{N}, \qquad \text{ for } t \in ((n - 1)\Delta t, n \Delta t].
\end{equation*}
Note that because $\boldsymbol{u}^{n}_{N}$ is $\mathcal{F}_{t^{n}_{N}}$ measurable, the choice of $\boldsymbol{u}^{n - 1}_{N}$ instead of $\boldsymbol{u}^{n}_{N}$ above is used so that the resulting process $\boldsymbol{u}_{N}$ is adapted to the filtration $\{\mathcal{F}_{t}\}_{t \ge 0}$.

Next, we consider the functions associated with the structure. Note that $\eta^{n}_{N}$, $\eta^{n + \frac{1}{3}}_{N}$, and $v^{n + \frac{1}{3}}_{N}$ are $\mathcal{F}_{t^{n}_{N}}$ measurable while $v^{n + \frac{2}{3}}_{N}$ is $\mathcal{F}_{t^{n + 1}_{N}}$ measurable. It turns out that we will not need to keep track of $v^{n + \frac{2}{3}}_{N}$ when passing to the limit, since it does not appear in \eqref{semi1}. So it suffices to define
\begin{equation*}
\eta_{N}(t, \cdot) = \eta^{n - 1}_{N}, \qquad v_{N}(t, \cdot) = v^{n - 1}_{N}, \qquad v^{*}_{N}(t, \cdot) = v^{n - \frac{2}{3}}_{N}, \qquad \text{ for } t \in ((n - 1)\Delta t, n \Delta t],
\end{equation*}
and these are all adapted to the given filtration $\{\mathcal{F}_{t}\}_{t \ge 0}$. Note that $v_{N}$ defined on $[0, T] \times \Gamma$ is pathwise the trace of the fluid velocity $\boldsymbol{u}_{N}$ defined on $[0, T] \times \Omega_{f}$ for all $t \in [0, T]$, but this is not true for $v^{*}_{N}$, since $v^{*}_{N}$ is the structure velocity obtained after the structure subproblem in the semidiscrete scheme, which does not update the fluid velocity directly. 

We also introduce a piecewise linear interpolation $\overline{\eta}_{N}$ of $\eta_{N}$ to add additional regularity to the structure displacement, 
since we will want the structure displacement to be in 
$W^{1, \infty}(0, T; L^{2}(\Gamma))$ almost surely in the limit as $\Delta t \to 0$. Thus,  $\overline{\eta}_{N}$ is piecewise linear such that
\begin{equation}\label{eta_bar}
\overline{\eta}_{N}(n\Delta t) = \eta^{n}_{N}, \qquad \text{ for } n = 0, 1, ..., N.
\end{equation}
Note that $\overline{\eta}_{N}$ has Lipschitz continuous paths in time, and furthermore, 
\begin{equation}\label{v_star}
\partial_{t}\overline{\eta}_{N} = v^{*}_{N}.
\end{equation}
Because both $\eta^{n}_{N}$ and $\eta^{n + 1}_{N}$ are $\mathcal{F}_{t^{n}_{N}}$ adapted, $\overline{\eta}_{N}$ is adapted to the filtration $\{\mathcal{F}_{t}\}_{t \ge 0}$. We will also introduce a piecewise constant function $\eta_{N}^{\Delta t}$ for the structure displacement, given by
\begin{equation}\label{etatimeshift}
\eta^{\Delta t}_{N}(t, \cdot) = \eta^{n}_{N}, \qquad \text{ for } t \in ((n - 1)\Delta t, n\Delta t].
\end{equation}
Note that $\eta^{\Delta t}_{N}$ is adapted to the filtration $\{\mathcal{F}_{t}\}_{t \ge 0}$ and is a time-shifted version of $\eta_{N}$, which is emphasized in the notation by the superscript of $\Delta t$. This time-shifted structure displacement will be useful for passing to the limit in Section \ref{limit}.

We will also consider the corresponding piecewise linear interpolations for the fluid velocity and structure velocity, which satisfy
\begin{equation}\label{uv_bar}
\overline{\boldsymbol{u}}_{N}(n\Delta t) = \boldsymbol{u}^{n}_{N}, \qquad \overline{v}_{N}(n\Delta t) = v^{n}_{N}, \qquad \text{ for } n = 0, 1, ..., N.
\end{equation}
We will need to consider $\overline{\boldsymbol{u}}_{N}$ and $\overline{v}_{N}$ because we will express the discrete time derivatives $\frac{\boldsymbol{u}^{n + 1}_{N} - \boldsymbol{u}^{n}_{N}}{\Delta t}$ and $\frac{v^{n + 1}_{N} - v^{n}_{N}}{\Delta t}$ in the semidiscrete formulation \eqref{semi1} in terms of the time derivatives of $\overline{\boldsymbol{u}}_{N}$ and $\overline{v}_{N}$. We will also need to consider piecewise constant time-shifted functions $\boldsymbol{u}_{N}^{\Delta t}$ and $v_{N}^{\Delta t}$ for the fluid velocity and the structure velocity, defined by
\begin{equation}\label{uvtimeshift}
\boldsymbol{u}_{N}^{\Delta t}(t, \cdot) = \boldsymbol{u}^{n}_{N}, \qquad v_{N}^{\Delta t}(t, \cdot) = v^{n}_{N}, \qquad \text{ for } t \in ((n - 1)\Delta t, n\Delta t].
\end{equation}
We note that $\boldsymbol{u}_{N}^{\Delta t}$ and $v_{N}^{\Delta t}$ are time-shifted versions of $\boldsymbol{u}_{N}$ and $v_{N}$. We will need these time-shifted functions because the fluid dissipation estimate in Proposition \ref{uniformenergy} implies that $\boldsymbol{u}_{N}^{\Delta t}$, rather than $\boldsymbol{u}_{N}$, is uniformly bounded in $L^{2}(\Omega; L^{2}(0, T; H^{1}(\Omega_{f})))$. See Proposition \ref{uniformbound}.

We make the following important observation. Unlike $\overline{\eta}_{N}$, we note that $\overline{\boldsymbol{u}}_{N}$ and $\overline{v}_{N}$ are \textit{not} necessarily adapted to the filtration $\{\mathcal{F}_{t}\}_{t \ge 0}$, even though they can still be considered as random variables taking values in their appropriate path spaces. Similarly, $\boldsymbol{u}^{\Delta t}_{N}$ and $v^{\Delta t}_{N}$, unlike $\boldsymbol{u}_{N}$ and $v_{N}$, are \textit{not} necessarily adapted to the filtration $\{\mathcal{F}_{t}\}_{t \ge 0}$. However, this will not be an issue, because we will see later in Lemma \ref{almostsuresub} that $\overline{\boldsymbol{u}}_{N}$, $\boldsymbol{u}_{N}^{\Delta t}$, $\overline{v}_{N}$, and $v^{\Delta t}_{N}$ are almost surely ``close to" the random processes $\boldsymbol{u}_{N}$ and $\boldsymbol{v}_{N}$, which are adapted to the filtration $\{\mathcal{F}_{t}\}_{t \ge 0}$, as $N \to \infty$ along a subsequence.


We summarize some of the previously discussed measure theoretic properties of the stochastic approximate solutions in the following proposition, for future reference.

\begin{proposition}
Recall that $W$ is a one dimensional Brownian motion with respect to the probability space with complete filtration, $(\Omega, \mathcal{F}, \{\mathcal{F}_{t}\}_{t \ge 0}, \mathbb{P})$. For all $N \in \mathbb{N}$, $\boldsymbol{u}_{N}$, $v_{N}$, $v_{N}^{*}$, $\eta_{N}$, and $\overline{\eta}_{N}$ are adapted to the filtration $\{\mathcal{F}_{t}\}_{t \ge 0}$ with left continuous paths, with $\overline{\eta}_{N}$ having continuous paths. In addition, for some fixed $t > 0$ and for each $N$, define $n_{0} = \lfloor\frac{t}{\Delta t} \rfloor + 1$. Then, $W_{\tau} - W_{t}$ is independent of each of the random variables in the following collection of random variables for each $N$ and for each $\tau > t$:
\begin{equation*}
\{\boldsymbol{u}_{N}^{n - 1}, v^{n - 1}_{N}, v^{n - \frac{2}{3}}_{N} : 1 \le n \le n_{0}\}, \{\eta^{n}_{N} : 0 \le n \le n_{0}\}, \{\overline{\eta}_{N}(s) : s \in [0, n_{0}\Delta t]\}.
\end{equation*}
\end{proposition}
\subsection{Uniform boundedness of approximate solutions}

Using the previous discrete energy estimates, 
we establish uniform boundedness of the approximate solutions in the following proposition. We note that in contrast to the case of deterministic FSI, the uniform boundedness of these (random) approximate solutions is only in \textit{expectation}.

\begin{proposition}\label{uniformbound}
The following uniform boundedness results hold:
\begin{itemize}
\item $(\eta_{N})_{N \in \mathbb{N}}$ is uniformly bounded in $L^{2}(\Omega; L^{\infty}(0, T; H^{1}_{0}(\Gamma)))$.
\item $(v_{N})_{N \in \mathbb{N}}$ is uniformly bounded in $L^{2}(\Omega; L^{\infty}(0, T; L^{2}(\Gamma)))$.
\item $(v_{N}^{\Delta t})_{N \in \mathbb{N}}$ is uniformly bounded in $L^{2}(\Omega; L^{2}(0, T; H^{1/2}(\Gamma)))$. 
\item $(v^{*}_{N})_{N \in \mathbb{N}}$ is uniformly bounded in $L^{2}(\Omega; L^{\infty}(0, T; L^{2}(\Gamma)))$. 
\item $(\boldsymbol{u}_{N})_{N \in \mathbb{N}}$ is uniformly bounded in $L^{2}(\Omega; L^{\infty}(0, T; L^{2}(\Omega_{f})))$.
\item $(\boldsymbol{u}_{N}^{\Delta t})_{N \in \mathbb{N}}$ is uniformly bounded in $L^{2}(\Omega; L^{2}(0, T; H^{1}(\Omega_{f})))$. 
\end{itemize}
\end{proposition}

\begin{proof}
The only part of this result that does not follow directly from Proposition~\ref{uniformenergy} is to show that $(\boldsymbol{u}_{N}^{\Delta t})_{N \in \mathbb{N}}$ is uniformly bounded in $L^{2}(\Omega; L^{2}(0, T; H^{1}(\Omega_{f})))$. We compute
\begin{small}
\begin{multline*}
||\boldsymbol{u}_{N}^{\Delta t}||^{2}_{L^{2}(\Omega; L^{2}(0, T; H^{1}(\Omega_{f})))} = \mathbb{E} \left(\int_{0}^{T} ||\boldsymbol{u}_{N}^{\Delta t}||^{2}_{H^{1}(\Omega_{f})} dt\right) = (\Delta t) \mathbb{E} \left(\sum_{k = 1}^{N} ||\boldsymbol{u}^{k}_{N}||^{2}_{H^{1}(\Omega_{f})}\right) 
\le C(\Delta t) \mathbb{E}\left(\sum_{k = 1}^{N} ||\boldsymbol{D}(\boldsymbol{u}^{k}_{N})||^{2}_{L^{2}(\Omega_{f})}\right).
\end{multline*}
\end{small}
The result follows from the uniform boundedness of the sum of the dissipation terms (recall that the $(\Delta t)$ term is included in the definition of the energy dissipation \eqref{discretedissipation}). By taking the trace of the $r$ component of the fluid velocity $\boldsymbol{u}^{n}_{N}$, which is in $H^{1/2}(\Gamma)$, we get the corresponding boundedness of $(v_{N}^{\Delta t})_{N \in \mathbb{N}}$ in $L^{2}(\Omega; L^{2}(0, T; H^{1/2}(\Gamma)))$. 
\end{proof}

We also state the corresponding uniform boundedness property for the linear interpolations $(\overline{\eta}_{N})_{N \in \mathbb{N}}$. Note that in terms of distributional derivatives, $\partial_{t} \overline{\eta}_{N} = v_{N}^{*}$
holds pathwise for $\omega \in \Omega$.

Therefore, we have:
\begin{proposition}\label{uniformetabound}
The sequence of linear interpolations of the structure displacements $(\overline{\eta}_{N})_{N \in \mathbb{N}}$ is uniformly bounded in $L^{2}(\Omega; L^{\infty}(0, T; H^{1}_{0}(\Gamma))) \cap L^{2}(\Omega; W^{1, \infty}(0, T; L^{2}(\Gamma)))$. 
\end{proposition}

\begin{remark}
To be very precise, one must check that the stochastic approximate solutions are measurable, as random variables taking values in a given path space. 
The measurability of these stochastic processes is easy to see by using the measurability properties of the functions $\boldsymbol{u}^{n}_{N}$, $v^{n + \frac{i}{3}}_{N}$, and $\eta^{n}_{N}$. For example, $\eta_{N}$ is measurable as a map from the probability space $\Omega$ to $L^{\infty}(0, T; H_{0}^{1}(\Gamma))$ because $\eta_{N}$ can be considered as the composition of a measurable map $F_{1}$ with a continuous map $F_{2}$.
$F_1$ is the map from $\omega \in \Omega$ to the space of bounded sequences of length $N$ with values in $H_{0}^{1}(\Gamma)$, given by
$
F_{1}: \omega \to (\eta^{0}_{N}, \eta^{1}_{N}, ..., \eta^{N - 1}_{N}),
$
which is measurable by the measurability properties of each $\eta^{n}_{N}$. $F_2$ is 
the map from the space of bounded sequences of length $N$ with values in $H_{0}^{1}(\Gamma)$ to $L^{\infty}(0, T; H_{0}^{1}(\Gamma))$, given by
$
F_{2}: (\eta^{0}_{N}, \eta^{1}_{N}, ..., \eta^{N - 1}_{N}) \to \sum_{k = 0}^{N - 1} \eta^{k}_N \cdot 1_{(k\Delta t, (k + 1)\Delta t]}(t),
$
which is continuous.
\end{remark}

\section{Passage to the limit}
We would like to show that our approximate solution sequences converge in a certain sense, to a 
weak solution of the original problem. In the deterministic case, uniform boundedness in the energy norms
associated with the approximate structure displacements and fluid and structure velocities is typically sufficient to
obtain weakly or weakly* convergent
subsequences which can be shown to converge to a weak solution of the original continuous problem. 
No compactness result is needed in the deterministic linear case.
In the stochastic case, we only have uniform boundedness of our approximate solution sequences {\emph{in expectation}},
and so we cannot deduce existence of weakly or weakly* convergent subsequences pathwise. 
We need a {\emph{compactness type argument}} to be able to get to a convergent subsequence, even though we
are working with a linear FSI problem. A compactness argument will first imply the existence of a
convergent subsequence of the {\emph{probability measures}} which describe the {\emph{laws}} or equivalently, the {\emph{distributions}} of the approximate solutions. From here, we will eventually be able to get to {\emph{almost sure}} convergence of the stochastic approximate solutions themselves. 

We start by designing compactness arguments that will provide weak convergence of the probability measures describing the laws of our random approximate solutions.

\subsection{Weak convergence of measures}\label{sec:measures_weak_conv}
We first show that along subsequences, the probability measures, or the laws describing the distributions of our
 stochastic approximate solutions constructed earlier, converge to a probability measure, as the time step $\Delta t \to 0$, or $N \to \infty$.  
For this purpose,
we recall that we are given a probability space with complete filtration $(\Omega, \mathcal{F}, \{\mathcal{F}_{t}\}_{t \ge 0}, \mathbb{P})$,
with a one dimensional Brownian motion $W$ with respect to the given filtration. 
For each $N$, we define the probability measure (or the law) $\mu_{N}$:
\begin{equation}\label{muN}
\mu_{N} = \mu_{\eta_{N}} \times \mu_{\overline{\eta}_{N}} \times \mu_{\eta_{N}^{\Delta t}} \times \mu_{\boldsymbol{u}_{N}} \times \mu_{v_{N}} \times \mu_{\boldsymbol{u}_{N}} \times \mu_{v^{*}_{N}} \times \mu_{\overline{\boldsymbol{u}}_{N}} \times \mu_{\overline{v}_{N}} \times \mu_{\boldsymbol{u}^{\Delta t}_{N}} \times \mu_{v^{\Delta t}_{N}} \times \mu_{W},
\end{equation}
defined on the phase space $\mathcal{X}$:
\begin{equation}\label{phase}
\mathcal{X} = [L^{2}(0, T; L^{2}(\Gamma))]^{3} \times [L^{2}(0, T; L^{2}(\Omega_{f})) \times L^{2}(0, T; L^{2}(\Gamma))]^{4} \times C(0, T; \mathbb{R}).
\end{equation}

Here, $\mu_{\eta_{N}}$ denotes the law of $\eta_{N}$ on $L^{2}(0, T; L^{2}(\Gamma))$, $\mu_{\boldsymbol{u}_{N}}$ denotes the law of $\boldsymbol{u}_{N}$ on $L^{2}(0, T; L^{2}(\Omega_{f}))$, $\mu_{W}$ denotes the law of $W$ on $C(0, T; \mathbb{R})$, and so on. Thus, $\mu_{N}$ is the joint law of the random variables $\eta_{N}$, $\overline{\eta}_{N}$, $\eta^{\Delta t}_{N}$, $\boldsymbol{u}_{N}$, $v_{N}$, $\boldsymbol{u}_{N}$,$v^{*}_{N}$, $\overline{\boldsymbol{u}}_{N}$, $\overline{v}_{N}$, $\boldsymbol{u}^{\Delta t}_{N}$, $v^{\Delta t}_{N}$, and $W$.
As we shall see below, it is easier to work with the fluid velocity and the structure velocity in pairs, which is the reason why in \eqref{muN} above,
we consider $(\mu_{\boldsymbol{u}_{N}}, \mu_{v_{N}}), (\mu_{\boldsymbol{u}_{N}},\mu_{v^{*}_{N}})$, 
$(\mu_{\overline{\boldsymbol{u}}_{N}},\mu_{\overline{v}_{N}})$, and $(\mu_{\boldsymbol{u}^{\Delta t}_{N}}, \mu_{v^{\Delta t}_{N}})$.
The main result of this subsection is the following.

\begin{theorem}\label{weakconv}
Along a subsequence (which we will continue to denote by $N$), $\mu_{N}$ converges weakly as probability measures to a probability measure $\mu$ on $\mathcal{X}$. 
\end{theorem}

To show weak convergence of these probability measures along a subsequence, stated in Theorem~\ref{weakconv},
we must show that the probability measures are \textbf{tight}. 
\begin{definition}\label{tight}
The probability measures $\mu_N$ are {\bf{tight}} if  for every $\epsilon > 0$, 
there exists a compact set $A_{\epsilon}$, compact in $\mathcal{X}$, such that
\begin{equation*}
\mu_{N}(A_{\epsilon}) > 1 - \epsilon, \qquad \text{ for all } N.
\end{equation*}
\end{definition}
To get a hold of the compact subset $A_\epsilon$, we will need the following 
two deterministic compactness results for the 
structure displacements $\{\eta_N(\omega)\}$ and for the fluid and structure velocities $\{\boldsymbol{u}_{N}(\omega)\}$
and $\{ v_{N}(\omega)\}$. The two results are obtained in the following two lemmas.

The first lemma, which will be applied to the structure displacements $\{\eta_N(\omega)\}$, is a direct 
consequence of the classical Aubin-Lions compactness lemma \cite{aubin1963theoreme,lions1969quelques}:

\begin{lemma}\label{structurecompact}
The following  holds:
$
[W^{1, \infty}(0, T; L^{2}(\Gamma)) \cap L^{\infty}(0, T; H_{0}^{1}(\Gamma))] \subset \subset L^{\infty}(0, T; L^{2}(\Gamma)).
$
\end{lemma}


The Aubin-Lions compactness lemma actually gives a stronger compact embedding of 
$W^{1, \infty}(0, T; L^{2}(\Gamma)) \cap L^{\infty}(0, T; H_{0}^{1}(\Gamma))$ into $C(0, T; L^{2}(\Gamma))$, but since we 
want $\eta_{N}$ and $\overline{\eta}_{N}$ to take values in the same path space, we use 
$L^{\infty}(0, T; L^{2}(\Gamma))$ since $\eta_{N}$ is not continuous. 

To handle the compactness argument for the structure and fluid velocities, we consider the subsets $\mathcal{K}$ and $\mathcal{K}_{R}$ in $L^{2}(0, T; L^{2}(\Omega_{f})) \times L^{2}(0, T; L^{2}(\Gamma))$, defined as follows. 

\begin{definition}[{\bf{Definition of $\mathcal{K}$ and $\mathcal{K}_{R}$}}]\label{KR}
The sets $\mathcal{K}$ and $\mathcal{K}_{R}$ of paths (or realizations) are defined as follows:

$\bullet$   For the pathwise left continuous approximate functions  $\boldsymbol{u}_{N}(\omega), v_{N}(\omega)$ on $[0, T]$, 
we define:
\begin{equation*}
\mathcal{K} = \{(\boldsymbol{u}, v) \in L^{2}(0, T; L^{2}(\Omega_{f})) \times L^{2}(0, T; L^{2}(\Gamma)): 
\boldsymbol{u} = \boldsymbol{u}_{N}(\omega) \text{ and } v = v_{N}(\omega) \text{ for some $\omega \in \Omega$ and $N \in \mathbb{N}$}\}.
\end{equation*}
%
%
%

$\bullet$ For any arbitrary positive constant $R$, define $\mathcal{K}_{R}$ to be the subset of paths $(\boldsymbol{u}_{N}(\omega), v_{N}(\omega)) \in \mathcal{K}$ where $\omega$ and $N$ satisfy the following properties.
\begin{enumerate}
\item 
{\emph{Uniform\  boundedness:}} $||(\boldsymbol{u}_{N}^{\Delta t}, v_{N}^{\Delta t})||_{L^{2}(0, T; H^{1}(\Omega_{f})) \times L^{2}(0, T; H^{1/2}(\Gamma))} \le R$, $||\boldsymbol{u}_{N}||_{L^{\infty}(0, T; L^{2}(\Omega_{f}))} \le R$, \newline$||v_{N}||_{L^{\infty}(0, T; L^{2}(\Gamma))} \le R$,
$||\eta_{N}||_{L^{\infty}(0, T; H_{0}^{1}(\Gamma))} \le R$.
\item {\emph{Boundedness of numerical dissipation:}} 
$\sum_{n = 0}^{N - 1} ||\boldsymbol{u}^{n + 1}_{N} - \boldsymbol{u}^{n + \frac{2}{3}}_{N}||^{2}_{L^{2}(\Omega_{f})} \le R$, 
$\sum_{n = 0}^{N - 1} ||v^{n + \frac{1}{3}}_{N} - v^{n}_{N}||^{2}_{L^{2}(\Gamma)} \le R$, 
$\sum_{n = 0}^{N - 1} ||v^{n + \frac{2}{3}}_{N} - v^{n + \frac{1}{3}}_{N}||^{2}_{L^{2}(\Gamma)} \le R$, 
$\sum_{n = 0}^{N - 1} ||v^{n + 1}_{N} - v^{n + \frac{2}{3}}_{N}||^{2}_{L^{2}(\Gamma)} \le R$.
\item {\emph{Boundedness of fluid dissipation:}}
$(\Delta t) \sum_{n = 1}^{N} \int_{\Omega_{f}} |\boldsymbol{D}(\boldsymbol{u}^{n}_{N})|^{2} d\boldsymbol{x} \le R.$
\item {\emph{Boundedness of $1/4$-H\"{o}lder exponent of Brownian motion:}}
$\displaystyle{\sup_{s, t \in [0, T], s \ne t} \frac{|W(t) - W(s)|}{|t - s|^{1/4}} \le R}.$
\end{enumerate}
\end{definition}

\begin{remark}
In the fourth condition above, any positive H\"{o}lder exponent that is strictly less than $1/2$ would suffice, since Brownian motion is ``almost" $1/2$-H\"{o}lder continuous, but we have fixed $1/4$ for concreteness.
\end{remark}
The following lemma provides the desired compactness result for $(\boldsymbol{u}_{N}, v_{N})$.
\begin{lemma}\label{KRcompact}
For any arbitrary positive constant $R$, the set $\mathcal{K}_{R}$ is precompact in $L^{2}(0, T; L^{2}(\Omega_{f})) \times L^{2}(0, T; L^{2}(\Gamma))$.
\end{lemma}

\begin{proof}
We use the Simon's compactness theorem \cite{Simon,MuhaCanicCompactness}. 
According to Simon's theorem, it suffices to check two conditions.

\vspace{0.1in}

\noindent \textbf{\textit{First condition:}} We must first show that for any $0 < t_{1} < t_{2} < T$, the collection
$
\left\{\int_{t_{1}}^{t_{2}} f(t)dt : f \in \mathcal{K}_{R}\right\}
$
is relatively compact in $L^{2}(\Omega_{f}) \times L^{2}(\Gamma)$. Consider a sequence $\{f_{m}(t, \cdot)\}_{m = 1}^{\infty}$ in $\mathcal{K}_{R}$, where $f_{m}(t, \cdot) = (\boldsymbol{u}_{m}(t, \cdot), v_{m}(t, \cdot))$. We want to show that there is a subsequence $\left\{\int_{t_{1}}^{t_{2}} f_{m_{k}}(t) dt\right\}_{k = 1}^{\infty}$ that converges in $L^{2}(\Omega_{f}) \times L^{2}(\Gamma)$. 

For each $m$, there exists some $N_{m}$ and $\omega_{m} \in \Omega$ (both depending on $m$) such that 
\begin{align*}
\boldsymbol{u}_{m}(t) &= \boldsymbol{u}_{0} \cdot 1_{t \in [0, (\Delta t)_{m}]} + \sum_{n = 1}^{N_{m} - 1} \boldsymbol{u}^{n}_{N_{m}}(\omega_{m}) \cdot 1_{t \in (n(\Delta t)_{m}, (n + 1)(\Delta t)_{m}]}, \\
v_{m}(t) &= v_{0} \cdot 1_{t \in [0, (\Delta t)_{m}]} + \sum_{n = 1}^{N_{m} - 1} v^{n}_{N_{m}}(\omega_{m}) \cdot 1_{t \in (n(\Delta t)_{m}, (n + 1)(\Delta t)_{m}]},
\end{align*}
where $(\Delta t)_{m} = T/N_{m}$. Therefore, we have that 
\begin{equation*}
\int_{t_{1}}^{t_{2}} \boldsymbol{u}_{m}(t) dt = a_{m} \boldsymbol{u}_{0} + \int_{\max(t_{1}, (\Delta t)_{m})}^{\max(t_{2}, (\Delta t)_{m})} \boldsymbol{u}_{m}(t) dt, \qquad \int_{t_{1}}^{t_{2}} v_{m}(t) dt = a_{m} v_{0} + \int_{\max(t_{1}, (\Delta t)_{m})}^{\max(t_{2}, (\Delta t)_{m})} v_{m}(t) dt,
\end{equation*}
where $a_{m} = \max(0, (\Delta t)_{m} - t_{1})$. Because $a_{m} \in [0, T]$, we can find a subsequence $\{m_{k}\}_{k = 1}^{\infty}$ such that $a_{m_{k}} \to a$ as $k \to \infty$, for some $a \in [0, T]$. Because $\boldsymbol{u}_{0}$ and $v_{0}$ are the fixed initial data for the fluid velocity and the structure velocity, $a_{m_{k}} \boldsymbol{u}_{0}$ and $a_{m_{k}} v_{0}$ converge along this subsequence in $L^{2}(\Omega_{f})$ and $L^{2}(\Gamma)$.

It remains to show that the sequences in $k$ given by
\begin{equation}\label{shiftintegrals}
\int_{\max\left(t_{1}, (\Delta t)_{m_{k}}\right)}^{\max\left(t_{2}, (\Delta t)_{m_{k}}\right)} \boldsymbol{u}_{m_{k}}(t) dt \qquad \text{ and }\qquad \int_{\max\left(t_{1}, (\Delta t)_{m_{k}}\right)}^{\max\left(t_{2}, (\Delta t)_{m_{k}}\right)} v_{m_{k}}(t) dt
\end{equation}
converge in $L^{2}(\Omega_{f})$ and $L^{2}(\Gamma)$ respectively along a further subsequence.  Because of the compact embedding 
$
H^{1}(\Omega_{f}) \times H^{1/2}(\Gamma) \subset \subset L^{2}(\Omega_{f}) \times L^{2}(\Gamma),
$
it suffices to show that the two sequences in $k$ given in \eqref{shiftintegrals} are uniformly bounded in $H^{1}(\Omega_{f})$ and $H^{1/2}(\Gamma)$.
This can be easily verified by using the uniform boundedness property of functions in $\mathcal{K}_{R}$ in Definition \ref{KR}:
\begin{align*}
&\left|\left|\int_{\max\left(t_{1}, (\Delta t)_{m_{k}}\right)}^{\max\left(t_{2}, (\Delta t)_{m_{k}}\right)} \boldsymbol{u}_{m_{k}}(t) dt\right|\right|_{H^{1}(\Omega_{f})} + \left|\left|\int_{\max\left(t_{1}, (\Delta t)_{m_{k}}\right)}^{\max\left(t_{2}, (\Delta t)_{m_{k}}\right)} v_{m_{k}}(t) dt\right|\right|_{H^{1/2}(\Gamma)} \\
&\le \int_{(\Delta t)_{m_{k}}}^{T} ||\boldsymbol{u}_{m_{k}}(t)||_{H^{1}(\Omega_{f})} dt + \int_{(\Delta t)_{m_{k}}}^{T} ||v_{m_{k}}(t)||_{H^{1/2}(\Gamma)} dt \\
&\le T^{1/2} \left(\int_{(\Delta t)_{m_{k}}}^{T} ||\boldsymbol{u}_{m_{k}}(t)||^{2}_{H^{1}(\Omega_{f})} dt\right)^{1/2} + T^{1/2}\left(\int_{(\Delta t)_{m_{k}}}^{T} ||v_{m_{k}}(t)||^{2}_{H^{1/2}(\Gamma)} dt\right)^{1/2} \le 2T^{1/2}R.
\end{align*}
Thus, we can further refine the subsequence $\{m_{k}\}_{k = 1}^{\infty}$ to obtain that $\left\{\int_{t_{1}}^{t_{2}} \left(\boldsymbol{u}_{m_{k}}(t), v_{m_{k}}(t)\right) dt\right\}_{k = 1}^{\infty}$ converges in $L^{2}(\Omega_{f}) \times L^{2}(\Gamma)$, where we continue to denote the refined subsequence by $\{m_{k}\}_{k = 1}^{\infty}$. 

\vspace{0.1in}

\noindent \textbf{\textit{Second condition:}} We must show that $||\tau_{h}f - f||_{L^{2}(h, T; L^{2}(\Omega_{f}) \times L^{2}(\Gamma))} \to 0$ uniformly for all $f = (\boldsymbol{u}, v) \in \mathcal{K}_{R}$, as $h \to 0$. Here $\tau_{h}$ for $h > 0$ denotes the time shift map
$
(\tau_{h}f)(t, \cdot) = f(t - h, \cdot).
$
Consider an arbitrary $\epsilon > 0$. We want to find $h > 0$ sufficiently small such that 
\begin{equation*}
||\tau_{h}\boldsymbol{u} - \boldsymbol{u}||_{L^{2}(h, T; L^{2}(\Omega_{f}))} < \epsilon
\quad {\rm and} \quad 
||\tau_{h}v - v||_{L^{2}(h, T; L^{2}(\Gamma))} < \epsilon\quad \forall (\boldsymbol{u}, v) \in \mathcal{K}_{R}.
\end{equation*}

To verify this, we can write $h = l(\Delta t) + s$, for each $\Delta t = \frac{T}{N}$, where $0 \le s < \Delta t$, so that
\begin{equation*}
||\tau_{h}\boldsymbol{u} - \boldsymbol{u}||_{L^{2}(h, T; L^{2}(\Omega_{f}))} \le ||\tau_{s}\tau_{l\Delta t} \boldsymbol{u} - \tau_{l\Delta t} \boldsymbol{u}||_{L^{2}(h, T; L^{2}(\Omega_{f}))} + ||\tau_{l\Delta t} \boldsymbol{u} - \boldsymbol{u}||_{L^{2}(h, T; L^{2}(\Omega_{f}))},
\end{equation*}
\begin{equation*}
||\tau_{h}v - v||_{L^{2}(h, T; L^{2}(\Gamma))} \le ||\tau_{s} \tau_{l\Delta t} v - \tau_{l\Delta t} v||_{L^{2}(h, T; L^{2}(\Gamma))} + ||\tau_{l\Delta t} v - v||_{L^{2}(h, T; L^{2}(\Gamma))}.
\end{equation*}
The above estimates require one estimate for the small $s$ time shift, and one for the larger $l\Delta t$ time shift.
We will handle the first time shift estimate using the numerical dissipation estimate holding for  $\mathcal{K}_{R}$, specified in Definition \ref{KR},
 and we will handle the second time shift estimate using an Ehrling property.

{\emph{Estimate for time shift by $s$:}} Consider arbitrary $(\boldsymbol{u}_{N}, v_{N}) \in \mathcal{K}_{R}$. Recalling that $0 \le s < \Delta t$, we compute
\begin{small}
\begin{align*}
||\tau_{s}\tau_{l\Delta t}\boldsymbol{u}_{N} - \tau_{l\Delta t} \boldsymbol{u}_{N}||^{2}_{L^{2}(h, T; L^{2}(\Omega_{f}))} = s \sum_{n = 0}^{N - l - 2} ||\boldsymbol{u}^{n + 1}_{N} - \boldsymbol{u}^{n}_{N}||^{2}_{L^{2}(\Omega_{f})} 
\le s \sum_{n = 0}^{N - 1} ||\boldsymbol{u}^{n + 1}_{N} - \boldsymbol{u}^{n}_{N}||^{2}_{L^{2}(\Omega_{f})} 
 \le sR.
\end{align*}
\end{small} where we used that $\boldsymbol{u}^{n}_{N}=\boldsymbol{u}^{n + \frac{2}{3}}_{N}$ and the numerical dissipation estimate in the last inequality.
Similarly,
\begin{small}
\begin{align*}
||\tau_{s}&\tau_{l\Delta t}v_{N} - \tau_{l\Delta t} v_{N}||^{2}_{L^{2}(h, T; L^{2}(\Gamma))} = s \sum_{n = 0}^{N - l - 2} ||v^{n + 1}_{N} - v^{n}_{N}||^{2}_{L^{2}(\Gamma)} \le s \sum_{n = 0}^{N - 1} ||v^{n + 1}_{N} - v^{n}_{N}||^{2}_{L^{2}(\Gamma)} \\
&\le 3s \left(\sum_{n = 0}^{N - 1} ||v^{n + \frac{1}{3}}_{N} - v^{n}_{N}||^{2}_{L^{2}(\Gamma)} + \sum_{n = 0}^{N - 1} ||v^{n + \frac{2}{3}}_{N} - v^{n + \frac{1}{3}}_{N}||^{2}_{L^{2}(\Gamma)} + \sum_{n = 0}^{N - 1} ||v^{n + 1}_{N} - v^{n + \frac{2}{3}}_{N}||^{2}_{L^{2}(\Gamma)}\right) \le 9sR.
\end{align*}
\end{small}
Recalling that $h = s + l\Delta t$ so that $0 < s \le h$, we can make these quantities arbitrarily small by taking $h$ sufficiently small, since $R$ is a fixed arbitrary positive constant.

{\emph{Estimate for time shift by $l\Delta t$:}} Consider arbitrary $(\boldsymbol{u}_{N}, v_{N}) \in \mathcal{K}_{R}$. We want to estimate 
\begin{equation*}
||\tau_{l\Delta t} \boldsymbol{u}_{N} - \boldsymbol{u}_{N}||_{L^{2}(h, T; L^{2}(\Omega_{f}))} + ||\tau_{l\Delta t} v_{N} - v_{N}||_{L^{2}(h, T; L^{2}(\Gamma))}.
\end{equation*}
This is identically zero if $h < \Delta t$, so we assume for the following estimate that $h \ge \Delta t$. We use the chain of embeddings
$
H^{1}(\Omega_{f}) \times H^{1/2}(\Gamma) \subset \subset L^{2}(\Omega_{f}) \times L^{2}(\Gamma) \subset \mathcal{Q}',
$
where $\mathcal{Q}$ is the test space defined in \eqref{Q}. Applying the uniform Ehrling property, see e.g., \cite{RenardyPDE,MuhaCanicCompactness},
we obtain 
\begin{align*}
&||\tau_{l\Delta t} \boldsymbol{u}_{N} - \boldsymbol{u}_{N}||_{L^{2}(h, T; L^{2}(\Omega_{f}))} + ||\tau_{l\Delta t} v_{N} - v_{N}||_{L^{2}(h, T; L^{2}(\Gamma))} \\
&\le 2||\tau_{l\Delta t}(\boldsymbol{u}_{N}, v_{N}) - (\boldsymbol{u}_{N}, v_{N})||_{L^{2}(h, (l + 1)\Delta t; L^{2}(\Omega_{f}) \times L^{2}(\Gamma))} + 2||\tau_{l\Delta t}(\boldsymbol{u}_{N}, v_{N}) - (\boldsymbol{u}_{N}, v_{N})||_{L^{2}((l + 1)\Delta t, T; L^{2}(\Omega_{f}) \times L^{2}(\Gamma))} \\
&\le 2||\tau_{l\Delta t}(\boldsymbol{u}_{N}, v_{N}) - (\boldsymbol{u}_{N}, v_{N})||_{L^{2}(h, (l + 1)\Delta t; L^{2}(\Omega_{f}) \times L^{2}(\Gamma))} + \delta ||\tau_{l\Delta t}(\boldsymbol{u}_{N}, v_{N}) - (\boldsymbol{u}_{N}, v_{N})||_{L^{2}((l + 1)\Delta t, T; H^{1}(\Omega_{f}) \times H^{1/2}(\Gamma))} \\
&+ C(\delta) ||\tau_{l\Delta t}(\boldsymbol{u}_{N}, v_{N}) - (\boldsymbol{u}_{N}, v_{N})||_{L^{2}((l + 1)\Delta t, T; \mathcal{Q}')} 
 := I_{1} + I_{2} + I_{3}.
\end{align*}
To estimate $I_{1}$, we use the triangle inequality, the assumption that $h \ge \Delta t$, and the uniform boundedness property of $\mathcal{K}_{R}$ in Definition \ref{KR}:
\begin{equation*}
I_{1} \le 2||\tau_{l\Delta t}(\boldsymbol{u}_{N}, v_{N})||_{L^{2}(h, (l + 1)\Delta t; L^{2}(\Omega_{f}) \times L^{2}(\Gamma))} + 2||(\boldsymbol{u}_{N}, v_{N})||_{L^{2}(h, (l + 1)\Delta t; L^{2}(\Omega_{f}) \times L^{2}(\Gamma))} \le 8(\Delta t)^{1/2}R \le 8h^{1/2}R.
\end{equation*}
To estimate $I_{2}$, we use the triangle inequality
and the uniform boundedness property of $\mathcal{K}_{R}$ in Definition \ref{KR}:
\begin{align*}
I_{2} &\le \delta \left(||\tau_{l\Delta t}(\boldsymbol{u}_{N}, v_{N})||_{L^{2}((l + 1)\Delta t, T; H^{1}(\Omega_{f}) \times H^{1/2}(\Gamma))} + ||(\boldsymbol{u}_{N}, v_{N})||_{L^{2}((l + 1)\Delta t, T; H^{1}(\Omega_{f}) \times H^{1/2}(\Gamma))}\right) \\
&\le 2\delta ||(\boldsymbol{u}_{N}, v_{N})||_{L^{2}(\Delta t, T; H^{1}(\Omega_{f}) \times H^{1/2}(\Gamma))} \le 2\delta R.
\end{align*}
To estimate $I_{3}$, we multiply the first equation in  the weak formulation \eqref{semi1} by $\Delta t$ to obtain:
\begin{multline*}
\int_{\Omega_{f}} (\boldsymbol{u}^{n + l}_{N} - \boldsymbol{u}^{n}_{N}) \cdot \boldsymbol{q} d\boldsymbol{x} 
+ \int_{\Gamma} (v^{n + l}_{N} - v^{n}_{N}) \psi dz 
= \int_{\Gamma} [W((n + l)\Delta t) - W(n\Delta t)] \psi dz 
+ (\Delta t) \sum_{k = 1}^{l} \Bigg(P^{n + k - 1}_{N, in}\int_{0}^{R} (q_{z})|_{z = 0} dr \\
- P^{n + k - 1}_{N, out} \int_{0}^{R} (q_{z})|_{z = L} dr 
- 2\mu \int_{\Omega_{f}} \boldsymbol{D}(\boldsymbol{u}^{n + k}_{N}) : \boldsymbol{D}(\boldsymbol{q}) d\boldsymbol{x} - \int_{\Gamma} \nabla \eta^{n + k}_{N} \cdot \nabla \psi dz\Bigg),\qquad  \forall(\boldsymbol{q}, \psi) \in \mathcal{Q}.
\end{multline*}
We estimate the terms on the right hand side as follows. For $(\boldsymbol{q}, \psi) \in \mathcal{Q}$, where $\mathcal{Q}$ is defined in \eqref{Q},
with 
$
||(\boldsymbol{q}, \psi)||_{\mathcal{Q}} \le 1,
$
we have the following estimates. 
\begin{itemize}
\item Using Cauchy-Schwarz and the boundedness of the $1/4$-Holder exponent of Brownian motion  in the definition of $\mathcal{K}_{R}$, 
see Definition~\ref{KR}, we obtain 
\begin{small}
\begin{align*}
\left|\int_{\Gamma} [W((n + l)\Delta t) - W(n\Delta t)]\psi dz\right| \le \left(\int_{\Gamma} |W((n + l)\Delta t) - W(n\Delta t)|^{2} dz\right)^{1/2} 
\le \left(\int_{\Gamma} \left|R(l\Delta t)^{1/4}\right|^{2} dz\right)^{1/2} \le C(l\Delta t)^{1/4}.
\end{align*}
\end{small}
\item Next, we recall the definition of the discretized pressure $P^{n}_{N, in/out} = \frac{1}{\Delta t}\int_{n\Delta t}^{(n + 1)\Delta t} P_{in/out}(t) dt$,
and use the trace inequality on the integral involving $q_{z}$ to obtain
\begin{align*}
(\Delta t) &\left|\sum_{k = 1}^{l} P^{n + k - 1}_{N, in} \int_{0}^{R} (q_{z})|_{z = 0} dr\right| = (\Delta t) \left|\sum_{k = 1}^{l} P^{n + k - 1}_{N, in}\right| \cdot \left|\int_{0}^{R} (q_{z})|_{z = 0} dr\right| 
\le C(\Delta t) \left|\sum_{k = 1}^{l} P^{n + k - 1}_{N, in}\right| \\
&= C\left|\int_{n\Delta t}^{(n + l)\Delta t} P_{in}(t) dt\right| 
 \le C(l\Delta t)^{1/2} ||P_{in}||_{L^{2}(n\Delta t, (n + l)\Delta t)} 
 \le C(l\Delta t)^{1/2} ||P_{in}||_{L^{2}(0, T)} = C(l\Delta t)^{1/2}.
\end{align*}
The same estimate holds for the outlet pressure term.
\item Using Cauchy-Schwarz and the uniform fluid dissipation estimate in Definition~\ref{KR} of $\mathcal{K}_{R}$, we get
\begin{align*}
(\Delta t) &\left|\sum_{k = 1}^{l} 2\mu \int_{\Omega_{f}} \boldsymbol{D}(\boldsymbol{u}^{n + k}_{N}) : \boldsymbol{D}(\boldsymbol{q}) d\boldsymbol{x}\right| 
\le C(\Delta t) \sum_{k = 1}^{l} \left(\int_{\Omega_{f}} |\boldsymbol{D}(\boldsymbol{u}^{n + k}_{N})|^{2} d\boldsymbol{x}\right)^{1/2} \\
&\le Cl^{1/2}(\Delta t) \left(\sum_{k = 1}^{l} \int_{\Omega_{f}} |\boldsymbol{D}(\boldsymbol{u}^{n + k}_{N})|^{2} d\boldsymbol{x}\right)^{1/2} 
\le Cl^{1/2}(\Delta t) \left(\sum_{k = 1}^{N} \int_{\Omega_{f}} |\boldsymbol{D}(\boldsymbol{u}^{k}_{N})|^{2} d\boldsymbol{x}\right)^{1/2} \le C(l\Delta t)^{1/2}.
\end{align*}
\item Using Cauchy-Schwarz and the uniform boundedness of $\eta_{N}$ in Definition~\ref{KR} of $\mathcal{K}_{R}$, we get:
\begin{equation*}
(\Delta t) \left|\sum_{k = 1}^{l} \int_{\Gamma} \nabla \eta_{N}^{n + k} \cdot \nabla \psi dz\right| \le (\Delta t) \sum_{k = 1}^{l} \int_{\Gamma} |\nabla \eta_{N}^{n + k} \cdot \nabla \psi| dz \le C(\Delta t)\sum_{k = 1}^{l} ||\eta^{n + k}_{N}||_{H_{0}^{1}(\Gamma)} \le Cl(\Delta t).
\end{equation*}
\end{itemize}
Here, all constants $C$ are independent of $n$, $l$, and $\Delta t$ and hence $N$, but can depend on the fixed, arbitrary constant $R$, and on the given parameters of the problem.
Combining all of these estimates together, we obtain that
\begin{equation}\label{Qprimeest}
||(\boldsymbol{u}^{n + l}_{N}, v^{n + l}_{N}) - (\boldsymbol{u}^{n}_{N}, v^{n}_{N})||_{\mathcal{Q}'} \le C(l\Delta t)^{1/4},
\end{equation}
where we use the estimate $0 \le l(\Delta t) \le T$ to reduce all exponents on $(l\Delta t)$ to the smallest one, which is $1/4$. 
Hence,
\begin{align*}
&||\tau_{l\Delta t}(\boldsymbol{u}_{N}, v_{N}) - (\boldsymbol{u}_{N}, v_{N})||^{2}_{L^{2}((l + 1)\Delta t, T; \mathcal{Q}')} \\
&= (\Delta t) \sum_{n = 1}^{N - 1 - l} ||(\boldsymbol{u}^{n + l}_{N}, v^{n + l}_{N}) - (u^{n}_{N}, v^{n}_{N})||_{\mathcal{Q}'}^{2} 
\le C(\Delta t) \sum_{n = 0}^{N - 1} (l\Delta t)^{1/2} \le C(l\Delta t)^{1/2}.
\end{align*}
and so 
$
\displaystyle{I_{3} := C(\delta) ||\tau_{l\Delta t}(\boldsymbol{u}_{N}, v_{N}) - (\boldsymbol{u}_{N}, v_{N})||_{L^{2}((l + 1)\Delta t, T; \mathcal{Q}')} \le C(\delta)(l\Delta t)^{1/4}.}
$

Combining the estimates for $I_{1}$, $I_{2}$, and $I_{3}$, we obtain
\begin{equation*}
||\tau_{l\Delta t} \boldsymbol{u}_{N} - \boldsymbol{u}_{N}||_{L^{2}(h, T; L^{2}(\Omega_{f}))} + ||\tau_{l\Delta t} v_{N} - v_{N}||_{L^{2}(h, T; L^{2}(\Gamma))} \le 8h^{1/2} R + 2\delta R + C(\delta)(l\Delta t)^{1/4}.
\end{equation*}

\vspace{0.1in}

We can now conclude the verification of the second condition of Simon's compactness result. Namely, we have shown that
\begin{align*}
||\tau_{h}\boldsymbol{u} - \boldsymbol{u}||_{L^{2}(h, T; L^{2}(\Omega_{f}))} &\le (sR)^{1/2} + 8h^{1/2}R + 2\delta R + C(\delta)(l\Delta t)^{1/4}, \\ 
||\tau_{h}v - v||_{L^{2}(h, T; L^{2}(\Gamma))} &\le 3(sR)^{1/2} + 8h^{1/2}R +  2\delta R + C(\delta)(l\Delta t)^{1/4}.
\end{align*}
Now, since $h = s + l\Delta t$ and $s, l\Delta t \in [0, h]$, we get
\begin{align*}
||\tau_{h}\boldsymbol{u} - \boldsymbol{u}||_{L^{2}(h, T; L^{2}(\Omega_{f}))} &\le (hR)^{1/2} + 8h^{1/2}R + 2\delta R + C(\delta) h^{1/4},
\\ 
||\tau_{h}v - v||_{L^{2}(h, T; L^{2}(\Gamma))} &\le 3(hR)^{1/2} + 8h^{1/2}R + 2\delta R + C(\delta) h^{1/4}.
\end{align*}
Therefore, given $\epsilon > 0$, we can first choose $\delta>0$ so that $2\delta R < \frac{\epsilon}{2}$, which fixes a value for $C(\delta)$. Then, we can choose $h>0$ sufficiently small so that
\begin{equation*}
3(hR)^{1/2} + 8h^{1/2}R + C(\delta)h^{1/4} < \frac{\epsilon}{2}.
\end{equation*}
This establishes the desired equicontinuity estimate, and hence Lemma~\ref{KRcompact} follows from Simon's compactness theorem. 
\end{proof}

Finally, we note that we have obtained compactness results only for the velocity approximate function $v_{N}$ and not $v_{N}^{*}$. In addition, when passing to the limit, we will consider the linear interpolations and time-shifted versions of the fluid velocity and of the structure displacement and velocity. We recall that the linear interpolations are piecewise linear functions defined by \eqref{eta_bar}, \eqref{uv_bar}, and the time-shifted functions are piecewise constant functions defined by \eqref{etatimeshift}, \eqref{uvtimeshift}. Hence, we will need the following result. 

\begin{lemma}\label{almostsuresub}
For an appropriate subsequence, which we continue to denote by $N$, 
\begin{align*}
||v_{N} - v_{N}^{*}||_{L^{2}(0, T; L^{2}(\Gamma))} \to 0, \qquad \text{ as $N \to \infty$, almost surely,}\\
||v_{N} - \overline{v}_{N}||_{L^{2}(0, T; L^{2}(\Gamma))} \to 0, \qquad \text{ as $N \to \infty$, almost surely,}\\
||v_{N} - v^{\Delta t}_{N}||_{L^{2}(0, T; L^{2}(\Gamma))} \to 0, \qquad \text{ as $N \to \infty$, almost surely,}\\
||\boldsymbol{u}_{N} - \overline{\boldsymbol{u}}_{N}||_{L^{2}(0, T; L^{2}(\Omega_{f}))} \to 0, \qquad \text{ as $N \to \infty$, almost surely,}\\
||\boldsymbol{u}_{N} - \boldsymbol{u}^{\Delta t}_{N}||_{L^{2}(0, T; L^{2}(\Omega_{f}))} \to 0, \qquad \text{ as $N \to \infty$, almost surely,}\\
||\eta_{N} - \overline{\eta}_{N}||_{L^{2}(0, T; L^{2}(\Gamma))} \to 0, \qquad \text{ as $N \to \infty$, almost surely,} \\
||\eta_{N} - \eta^{\Delta t}_{N}||_{L^{2}(0, T; L^{2}(\Gamma))} \to 0, \qquad \text{ as $N \to \infty$, almost surely.} \\
\end{align*}
\end{lemma}

\begin{proof}
We start by showing the first convergence result. 
To do that, we introduce the events 
\begin{equation*}
E_{j, N} = \left\{||v_{N} - v_{N}^{*}||_{L^{2}(0, T; L^{2}(\Gamma))} \le \frac{1}{j}\right\},\quad j \ge 1,
\end{equation*}
and show that the probability that the complements of $E_{j, N}$ occur for infinitely many $j$, is zero.
Indeed, by multiplying by 
$\Delta t$ the uniform numerical dissipation estimate from Proposition \ref{uniformenergy} and keeping only the first term on the left hand side, we obtain
\begin{equation}\label{vstardiff}
\mathbb{E}\left(\Delta t \sum_{n = 0}^{N - 1} ||v^{n + \frac{1}{3}}_{N} - v^{n}_{N}||^{2}_{L^{2}(\Gamma)}\right) = \mathbb{E}\left(||v_{N} - v_{N}^{*}||^{2}_{L^{2}(0, T; L^{2}(\Gamma))}\right) \le C(\Delta t).
\end{equation}
By Chebychev's inequality, we get
$\mathbb{P}(E_{j, N}^{c}) \le C(\Delta t)j^{2} = CTN^{-1}j^{2}.$
Thus, for the events $E_{j, N = j^{4}}$, we have
$
\sum_{j = 1}^{\infty} \mathbb{P}(E_{j, N = j^{4}}^{c}) \le CT\sum_{j = 1}^{\infty} \frac{1}{j^{2}} < \infty.
$
Therefore, by the Borel-Cantelli lemma, 
\begin{equation*}
\mathbb{P}\left(E^{c}_{j, N = j^{4}} \text{ occurs for infinitely many $j$}\right) = 0.
\end{equation*}
This implies that for almost every $\omega \in \Omega$, there exists $j_{0}(\omega)$ such that 
$||v_{N_j} - v_{N_j}^{*}||_{L^{2}(0, T; L^{2}(\Gamma))} \le \frac{1}{j}$ for all $j \ge j_{0}(\omega)$,
where  $N_{j} := j^{4}$, which implies the desired result, where our subsequence $N_{j}$ will continue to be denoted by $N$ for simplicity of notation. 

To show the the remaining convergence results, we use Proposition \ref{uniformenergy} to conclude that there exists a uniform constant $C$ independent of $N$ such that
\begin{equation*}
\sum_{n = 0}^{N - 1} \mathbb{E}\left(||v_{N}^{n + 1} - v^{n}_{N}||^{2}_{L^{2}(\Gamma)}\right) \le C, \quad \sum_{n = 0}^{N - 1} \mathbb{E}\left(||\boldsymbol{u}^{n + 1}_{N} - \boldsymbol{u}^{n}_{N}||_{L^{2}(\Omega_{f})}^{2}\right) \le C,
\quad
\sum_{n = 0}^{N - 1} \mathbb{E}\left(||\nabla \eta^{n + 1}_{N} - \nabla \eta^{n}_{N}||_{L^{2}(\Gamma)}^{2}\right) \le C,
\end{equation*}
where we recall that $\boldsymbol{u}^{n + \frac{2}{3}}_{N} = \boldsymbol{u}^{n}_{N}$ and $\eta^{n + \frac{1}{3}}_{N} = \eta^{n + 1}_{N}$, and where we used the triangle inequality to obtain the first estimate. Then, the same argument as above gives the desired result, once we note that
\begin{align}\label{vlindiff}
&\mathbb{E}\left(||\overline{v}_{N} - v_{N}||^{2}_{L^{2}(0, T; L^{2}(\Gamma))}\right) \le (\Delta t) \sum_{n = 0}^{N - 1} \mathbb{E}\left(||v_{N}^{n + 1} - v^{n}_{N}||^{2}_{L^{2}(\Gamma)}\right) \le C(\Delta t) \to 0,
\quad& \text{ as } N \to \infty,
\\
\label{vshiftdiff}
&\mathbb{E}\left(||v^{\Delta t}_{N} - v_{N}||^{2}_{L^{2}(0, T; L^{2}(\Gamma))}\right) = (\Delta t) \sum_{n = 0}^{N - 1} \mathbb{E}\left(||v_{N}^{n + 1} - v^{n}_{N}||^{2}_{L^{2}(\Gamma)}\right) \le C(\Delta t) \to 0,
\quad& \text{ as } N \to \infty,
\\
\label{ulindiff}
&\mathbb{E}\left(||\overline{\boldsymbol{u}}_{N} - \boldsymbol{u}_{N}||^{2}_{L^{2}(0, T; L^{2}(\Omega_{f}))}\right) \le (\Delta t) \sum_{n = 0}^{N - 1} \mathbb{E}\left(||\boldsymbol{u}^{n + 1}_{N} - \boldsymbol{u}^{n}_{N}||_{L^{2}(\Omega_{f})}^{2}\right) \le C(\Delta t) \to 0, \quad &\text{ as } N \to \infty,
\\
\label{ushiftdiff}
&\mathbb{E}\left(||\boldsymbol{u}^{\Delta t}_{N} - \boldsymbol{u}_{N}||^{2}_{L^{2}(0, T; L^{2}(\Omega_{f}))}\right) = (\Delta t) \sum_{n = 0}^{N - 1} \mathbb{E}\left(||\boldsymbol{u}^{n + 1}_{N} - \boldsymbol{u}^{n}_{N}||_{L^{2}(\Omega_{f})}^{2}\right) \le C(\Delta t) \to 0, \quad &\text{ as } N \to \infty,
\\
\label{etalindiff}
&\mathbb{E}\left(||\overline{\eta}_{N} - \eta_{N}||^{2}_{L^{2}(0, T; L^{2}(\Gamma))}\right) \le (\Delta t) \sum_{n = 0}^{N - 1} \mathbb{E}\left(||\eta_{N}^{n + 1} - \eta^{n}_{N}||^{2}_{L^{2}(\Gamma)}\right) \le C'(\Delta t) \to 0, 
\quad &\text{ as } N \to \infty,
\\
\label{etashiftdiff}
&\mathbb{E}\left(||\eta^{\Delta t}_{N} - \eta_{N}||^{2}_{L^{2}(0, T; L^{2}(\Gamma))}\right) = (\Delta t) \sum_{n = 0}^{N - 1} \mathbb{E}\left(||\eta_{N}^{n + 1} - \eta^{n}_{N}||^{2}_{L^{2}(\Gamma)}\right) \le C'(\Delta t) \to 0, 
\quad &\text{ as } N \to \infty,
\end{align}
where we used Poincar\'{e}'s inequality to deduce \eqref{etalindiff} and \eqref{etashiftdiff}.
\end{proof}
\if 1 = 0
From the inequalities on pg.~13, we obtain the following pathwise inequality,
\begin{multline*}
\max_{i = 1, 2, 3} \left(\max_{n = 0, 1, ..., N - 1} E^{n + \frac{i}{3}}\right) + \sum_{k = 0}^{N - 1} \Bigg(2\mu(\Delta t) \int_{\Omega} |\boldsymbol{D}(\boldsymbol{u}_{N}^{k + 1})|^{2} d\boldsymbol{x} + \frac{1}{2} ||\boldsymbol{u}^{k + 1} - \boldsymbol{u}^{k + \frac{2}{3}}||_{L^{2}(\Omega)}^{2} \\
+ \frac{1}{2} ||v^{k + 1} - v^{k + \frac{2}{3}}||_{L^{2}(\Gamma)}^{2} + \frac{1}{2} ||v^{k + \frac{1}{3}}_{N} - v^{k}_{N}||^{2}_{L^{2}(\Gamma)} + \frac{1}{2} ||\nabla \eta^{k + \frac{1}{3}}_{N} - \nabla \eta^{k}_{N}||^{2}_{L^{2}(\Gamma)}\Bigg) \\
\le E_{0} + \left[\max_{n = 0, 1, ..., N - 1} (\Delta t) \sum_{k = 0}^{n} \left(P^{k}_{in}\int_{0}^{R} (\boldsymbol{u}^{k + 1}_{N})_{z}|_{z = 0} dr - P^{k}_{out} \int_{0}^{R} (\boldsymbol{u}^{k + 1}_{N})_{z}|_{z = L} dr\right)\right] \\
+ \left[\max_{n = 0, 1, ..., N - 1} \sum_{k = 0}^{n} \left([W((k + 1) \Delta t) - W(k \Delta t)]\int_{\Gamma} v^{k + \frac{1}{3}} dz + \frac{L}{2}[W((k + 1) \Delta t) - W(k \Delta t)]^{2}dr\right)\right].
\end{multline*}

We estimate the terms on the right hand side similarly to before, without expectation so that we keep the estimate pathwise. We begin with
\begin{multline*}
\max_{n = 0, 1, ..., N - 1} (\Delta t) \sum_{k = 0}^{n} \left(P^{k}_{in}\int_{0}^{R} (\boldsymbol{u}^{k + 1}_{N})_{z}|_{z = 0} dr\right) \le \sum_{k = 0}^{N - 1} \left[(\Delta t) \frac{1}{4\epsilon} |P^{k}_{in}|^{2} + \epsilon(\Delta t)\left(\int_{0}^{R} (\boldsymbol{u}^{k + 1}_{N})_{z}|_{z = 0} dr\right)^{2}\right] \\
\le \sum_{k = 0}^{N - 1} \left(\frac{1}{4\epsilon} \cdot \frac{1}{\Delta t} \left(\int_{k\Delta t}^{(k + 1)\Delta t} P_{in}(t) dt\right)^{2} + C\epsilon(\Delta t) \int_{0}^{R} (\boldsymbol{u}^{k + 1}_{N})_{z}^{2}|_{z = 0} dr\right) \\
\le \sum_{k = 0}^{N - 1} \left(\frac{1}{4\epsilon} ||P_{in}||_{L^{2}(k\Delta t, (k + 1)\Delta t)}^{2} + C\epsilon(\Delta t) \int_{\Omega} |\boldsymbol{D}(\boldsymbol{u}^{k + 1}_{N})|^{2} d\boldsymbol{x} \right).
\end{multline*}
We emphasize that this holds pathwise. For the next term, we use the fact that Brownian motion is almost surely $1/4$-H\"{o}lder continuous. Therefore, a.s.
\begin{equation*}
|W(t_{1}) - W(t_{2})| \le C(\omega)|t_{1} - t_{2}|^{1/4}, \qquad \text{ for $t_{1}, t_{2} \in [0, T]$},
\end{equation*}
for a constant $C(\omega)$ depending only on $\omega \in \Omega$, which is finite almost surely. In particular, $C(\omega)$ does not depend on $t_{1}, t_{2} \in [0, T]$. Thus,
\begin{multline*}
\max_{n = 0, 1, ..., N - 1} \sum_{k = 0}^{n} \left([W((k + 1) \Delta t) - W(k \Delta t)]\int_{\Gamma} v^{k + \frac{1}{3}}_{N} dz + \frac{L}{2}[W((k + 1) \Delta t) - W(k \Delta t)]^{2}dr\right) \\
\le L^{1/2} \sum_{k = 0}^{N - 1} |W((k + 1)\Delta t) - W(k\Delta t)| \cdot \left(\int_{\Gamma} \left|v^{k + \frac{1}{3}}_{N} \right|^{2} dz\right)^{1/2} + \frac{L}{2} \sum_{k = 0}^{N - 1} |W((k + 1)\Delta t) - W(k\Delta t)|^{2} \\
\le L^{1/2} \left(\max_{k = 0, 1, ..., N - 1} ||v^{k + \frac{1}{3}}_{N}||_{L^{2}(\Gamma)}\right) \sum_{k = 0}^{N - 1} |W((k + 1)\Delta t) - W(k\Delta t)| + \frac{L}{2} \sum_{k = 0}^{N - 1} |W((k + 1)\Delta t) - W(k\Delta t)|^{2} 
\end{multline*}
\fi

Notice that this result follows from the numerical dissipation estimates in Proposition \ref{uniformenergy},
which imply convergence to zero in expectation, of the numerical dissipation terms, shown in \eqref{vlindiff}, \eqref{vshiftdiff}, \eqref{ulindiff}, \eqref{ushiftdiff}, \eqref{etalindiff}, and \eqref{etashiftdiff},
from which we were able to deduce the almost sure convergence.

\begin{proof}[Proof of Theorem \ref{weakconv}]
To show weak convergence of probability measures along a subsequence, we must show that the probability measures $\mu_N$ are {tight},
see Definition~\ref{tight}. 
Here, we note that for reasons that will be clear later (see Step 2 below), we will take $N$ to be the subsequence provided by Lemma \ref{almostsuresub} and begin with this indexing convention of $N$. 


\vspace{0.1in}

\noindent \textbf{Step 1: Weak convergence of $\mu_{\overline{\eta}_{N}}$ and $\mu_{\boldsymbol{u}_{N}} \times \mu_{v_{N}}$ along a subsequence.} We show this by showing that $\mu_{\overline{\eta}_{N}}$ and $\mu_{\boldsymbol{u}_{N}} \times \mu_{v_{N}}$ are tight. 
To show the tightness of $\mu_{\overline{\eta}_{N}}$, we define the set
\begin{equation*}
A_{R} = \{\overline{\eta} \in W^{1, \infty}(0, T; L^{2}(\Gamma)) \cap L^{\infty}(0, T; H_{0}^{1}(\Gamma)) : ||\overline{\eta}||_{W^{1, \infty}(0, T; L^{2}(\Gamma))} \le R, ||\overline{\eta}||_{L^{\infty}(0, T; H_{0}^{1}(\Gamma))} \le R\}.
\end{equation*}
By Lemma \ref{structurecompact}, $\overline{A_{R}}$ is a compact set in $L^{2}(0, T; L^{2}(\Gamma))$ since $L^{\infty}(0, T; L^{2}(\Gamma))$ embeds continuously into $L^{2}(0, T; L^{2}(\Gamma))$, where the closure is taken in the topology of $L^{2}(0, T; L^{2}(\Gamma))$. So by Chebychev's inequality and the previous uniform boundedness results, we have that for an arbitrary $\epsilon > 0$,
\begin{equation*}
\mu_{\bar\eta_N}(\overline{A_{R}}) > 1 - \epsilon,
\end{equation*}
if $R$ is chosen sufficiently large. So there exists a subsequence, which we continue to denote by $N$, for which $\mu_{\overline{\eta}_{N}}$ converges weakly to some probability measure $\mu_{\eta}$ on $L^{2}(0, T; L^{2}(\Gamma))$. 

To show the tightness of $\mu_{\boldsymbol{u}_{N}} \times \mu_{v_{N}}$, recall the definition of the set $\mathcal{K}_{R}$, and note that by Lemma \ref{KRcompact}, $\overline{\mathcal{K}_{R}}$ is a compact set in $L^{2}(0, T; L^{2}(\Omega_{f})) \times L^{2}(0, T; L^{2}(\Gamma))$. Furthermore, using the uniform boundedness estimates from Proposition \ref{uniformbound}
combined with Chebychev's inequality, we have that for any $\epsilon > 0$, we can find $R$ sufficiently large such that
\begin{equation*}
(\mu_{\boldsymbol{u}_{N}} \times \mu_{v_{N}})(\overline{\mathcal{K}_{R}}) > 1 - \epsilon.
\end{equation*}
Hence, there exists a subsequence, which we continue to denote by $N$, 
for which the measures $\mu_{\boldsymbol{u}_{N}} \times \mu_{v_{N}}$ converge weakly to some limiting probability measure on $L^{2}(0, T; L^{2}(\Omega_{f})) \times L^{2}(0, T; L^{2}(\Gamma))$, which we denote by $\mu_{\boldsymbol{u}} \times \mu_{v}$.

\vspace{0.1in}

\noindent \textbf{Step 2: Weak convergence of $\mu_{\boldsymbol{u}_{N}} \times \mu_{v^{*}_{N}}$, $\mu_{\overline{\boldsymbol{u}}_{N}} \times \mu_{\overline{v}_{N}}$, $\mu_{\boldsymbol{u}^{\Delta t}_{N}} \times \mu_{v^{\Delta t}_{N}}$, $\mu_{\eta_{N}}$, and $\mu_{\eta^{\Delta t}_{N}}$ along the subsequence obtained from Step 1.} 
Since 
$
\mu_{\boldsymbol{u}_{N}} \times \mu_{v_{N}} \Longrightarrow \mu_{\boldsymbol{u}} \times \mu_{v},
$
by the definition of weak convergence, we have
\begin{equation*}
\mathbb{E}[f(\boldsymbol{u}_{N}, v_{N})] \to \int_{L^{2}(0, T; L^{2}(\Omega_{f})) \times L^{2}(0, T; L^{2}(\Gamma))} f d(\mu_{\boldsymbol{u}} \times \mu_{v}),
\end{equation*}
for all bounded, Lipschitz continuous functions $f: L^{2}(0, T; L^{2}(\Omega_{f})) \times L^{2}(0, T; L^{2}(\Gamma)) \to \mathbb{R}$. However, because 
$
||v_{N} - v_{N}^{*}||_{L^{2}(0, T; L^{2}(\Gamma))} \to 0 \ \text{ a.s.}
$
due to Lemma \ref{almostsuresub}, 
we have that by the Lipschitz continuity of $f$, 
\begin{equation*}
|f(\boldsymbol{u}_{N}, v_{N}) - f(\boldsymbol{u}_{N}, v_{N}^{*})| \le \text{Lip}(f) ||v_{N} - v_{N}^{*}||_{L^{2}(0, T; L^{2}(\Gamma))} \to 0, \ \text{ a.s. as $N \to \infty$.}
\end{equation*}
Hence, by the bounded convergence theorem,
$\mathbb{E}[f(\boldsymbol{u}_{N}, v_{N})] - \mathbb{E}[f(\boldsymbol{u}_{N}, v^{*}_{N})] \to 0,$ { as } $N \to \infty$,
and hence,
\begin{equation*}
\mathbb{E}[f(\boldsymbol{u}_{N}, v^{*}_{N})] \to \int_{L^{2}(0, T; L^{2}(\Omega_{f})) \times L^{2}(0, T; L^{2}(\Gamma))} f d(\mu_{\boldsymbol{u}} \times \mu_{v}),
\end{equation*}
for all bounded, Lipschitz continuous functions $f: L^{2}(0, T; L^{2}(\Omega_{f})) \times L^{2}(0, T; L^{2}(\Gamma)) \to \mathbb{R}$. Thus, along the subsequence generated from Step 1, we have that both $\mu_{\boldsymbol{u}_{N}} \times \mu_{v_{N}}$ and $\mu_{\boldsymbol{u}_{N}} \times \mu_{v_{N}^{*}}$ converge weakly to the same limiting probability measure $\mu_{\boldsymbol{u}} \times \mu_{v}$ on $L^{2}(0, T; L^{2}(\Omega_{f})) \times L^{2}(0, T; L^{2}(\Gamma))$. 

The same argument can be used to show that $\mu_{\overline{\boldsymbol{u}}_{N}} \times \mu_{\overline{v}_{N}}$ and $\mu_{\boldsymbol{u}^{\Delta t}_{N}} \times \mu_{v^{\Delta t}_{N}}$ also converge weakly to 
$\mu_{\boldsymbol{u}} \times \mu_{v}$. This follows from the result on
a.s. convergence of $||\overline{\boldsymbol{u}}_{N} - \boldsymbol{u}_{N}||_{L^{2}(0, T; L^{2}(\Omega_{f}))}$, 
$||\overline{v}_{N} - v_{N}||_{L^{2}(0, T; L^{2}(\Gamma))}$, $||\boldsymbol{u}^{\Delta t}_{N} - \boldsymbol{u}_{N}||_{L^{2}(0, T; L^{2}(\Omega_{f}))}$, and $||v^{\Delta t}_{N} - v_{N}||_{L^{2}(0, T; L^{2}(\Gamma))}$
in Lemma \ref{almostsuresub}.


Finally, we have from Step 1 that $\mu_{\overline{\eta}_{N}}$ converges weakly to some probability measure $\mu_{\eta}$, as probability measures on $L^{2}(0, T; L^{2}(\Gamma))$. 
Then, the weak convergence of $\mu_{\eta_{N}}$ and $\mu_{\eta^{\Delta t}_{N}}$ to this same weak limit $\mu_{\eta}$ follows from the same argument as above, and the result from Lemma \ref{almostsuresub} that 
$
||\overline{\eta}_{N} - \eta_{N}||_{L^{2}(0, T; L^{2}(\Gamma))} \to 0
$
and
$
||\eta^{\Delta t}_{N} - \eta_{N}||_{L^{2}(0, T; L^{2}(\Gamma))} \to 0, \text{ as } N \to \infty,
$
a.s.

\vspace{0.1in}

\noindent \textbf{Step 3: Tightness of full measures $\mu_{N}$ along the subsequence obtained from Step 1.} We now consider the full probability measures 
$\mu_{N}$ specified in \eqref{muN} on the phase space $\mathcal{X}$ specified in \eqref{phase}.
We want to show that these probability measures $\mu_{N}$ are tight along the subsequence $N$ constructed as a result of Step 1. 

Consider $\epsilon > 0$. We want to construct a compact set in the phase space $\mathcal{X}$ for which the probability measure $\mu_{N}$ has probability greater than $1 - \epsilon$ on this compact set, for all $N$. We will construct this compact set component-wise,
using $\pi_{1},\dots,\pi_{12}$ to denote the projections  onto the components 1 through 12 of  $\mu_N$.

By the weak convergence of the measures $\mu_{\overline{\eta}_{N}}$, $\mu_{\eta_{N}}$, and $\mu_{\eta^{\Delta t}_{N}}$, by Prohorov's theorem (see for example Proposition 6.1 in \cite{LNT}), there exist compact sets $B_{1}$, $B_{2}$, and $B_{3}$ in $L^{2}(0, T; L^{2}(\Gamma))$ such that 
\begin{equation*}
\pi_{1}(\mu_{N})(B_{1}) > 1 - \frac{\epsilon}{8}, \qquad \pi_{2}(\mu_{N})(B_{2}) > 1 - \frac{\epsilon}{8}, \qquad \pi_{3}(\mu_{N})(B_{3}) > 1 - \frac{\epsilon}{8}, \qquad \text{ for all } N.
\end{equation*}

Similarly, because $(\boldsymbol{u}_{N}, v_{N})$, $(\boldsymbol{u}_{N}, v_{N}^{*})$, $(\overline{\boldsymbol{u}}_{N}, \overline{v}_{N})$, and $(\boldsymbol{u}^{\Delta t}_{N}, v^{\Delta t}_{N})$ converge weakly along this subsequence $N$ by Step 1 and Step 2, there exist compact sets $B_{4, 5}$, $B_{6, 7}$, $B_{8, 9}$, and $B_{10, 11}$ in $L^{2}(0, T; L^{2}(\Omega_{f})) \times L^{2}(0, T; L^{2}(\Gamma))$ such that
\begin{align*}
&\pi_{4, 5}(\mu_{N})(B_{4, 5}) > 1 - \frac{\epsilon}{8}, \qquad \pi_{6, 7}(\mu_{N})(B_{6, 7}) > 1 - \frac{\epsilon}{8}, \\
&\pi_{8, 9}(\mu_{N})(B_{8, 9}) > 1 - \frac{\epsilon}{8}, \qquad \pi_{10, 11}(\mu_{N})(B_{10, 11}) > 1 - \frac{\epsilon}{8}, \qquad \text{ for all } N.
\end{align*}

Finally, the last component of $\mu_{N}$, which is $\mu_{W}$, is constant in $N$. Hence, the probability measures $\pi_{12}(\mu_{N})$ defined on $C(0, T; \mathbb{R})$ are trivially, weakly compact. 
Therefore, the collection $\pi_{12}(\mu_{N})$ for all $N$ is tight, and hence, there exists a compact set $B_{12} \subset C(0, T; \mathbb{R})$ such that 
\begin{equation*}
\pi_{12}(\mu_{N})(B_{12}) > 1 - \frac{\epsilon}{8}, \qquad \text{ for all } N.
\end{equation*}

Based on this construction, we have the set
$
M_{\epsilon} := B_{1} \times B_{2} \times B_{3} \times B_{4, 5} \times B_{6, 7} \times B_{8, 9} \times B_{10, 11} \times B_{12},
$
which is a compact subset of the phase space $\mathcal{X}$, satisfying
$
\mu_{N}(M_{\epsilon}) > 1 - \epsilon,  \text{ for all } N.
$
This establishes the desired tightness of the probability measures, and completes the proof of Proposition \ref{structurecompact}.
\end{proof}

\subsection{Continuity properties of the weak limit}

To be able to prove appropriate measure theoretic properties of the limiting solutions, we need to establish continuity properties of the limiting solution. This is because many measure theoretic properties are simpler for stochastic processes with continuous paths in time. This is simple to do for the structure displacements, since the approximate structure displacements $\overline{\eta}_{N}$ all have Lipschitz continuous paths. However, because the approximate fluid and structure velocities $\boldsymbol{u}_{N}$ and $v_{N}$ have paths that are not continuous, we want to establish that the limiting solutions for the fluid and structure velocities have continuous paths in time, with an appropriate notion of continuity.

 First, we introduce the following definition, which will be used throughout the remainder of the manuscript.

\begin{definition}
Let $B$ be an arbitrary Banach space and let $f, g: [0, T] \to B$. The function $g: [0, T] \to B$ is a \textbf{version} of $f$ if $f = g$ a.e. on $[0, T]$.
\end{definition}

The goal is to show that the limit function $(\boldsymbol{u}, v)$ is in $C(0, T; \mathcal{Q}')$ almost surely, or more precisely, that the limiting measure $\mu$ is supported on a measurable subset of phase space $\mathcal{X}$, with the projections onto the functions involving the fluid and structure velocities being in $C(0, T; \mathcal{Q}')$ almost surely. This continuity property will allow us to conclude later that the resulting limit process is well-behaved in a stochastic measure theoretic sense. 

To do this, we will use the idea of $p$-variation for functions in time taking values in $\mathcal{Q}'$. The notion of considering total variations of functions is a classical idea \cite{Orlicz}, \cite{Young}. We remark however that our definition below differs slightly from classical definitions of total $p$-variation.

\begin{definition}\label{variation}
For any real number $p \ge 1$ and any $\delta > 0$, we define the \textit{$p$-variation of length scale $\delta$} of a given function $(\boldsymbol{u}, v): [0, T] \to \mathcal{Q}'$ by
\begin{equation*}
V^{\delta}_{p}(\boldsymbol{u}, v) = \sup_{|P| \le \delta} \sum_{i = 1}^{M} ||(\boldsymbol{u}(t_{i}), v(t_{i})) - (\boldsymbol{u}(t_{i - 1}), v(t_{i - 1}))||^{p}_{\mathcal{Q}'},
\end{equation*}
where $P$ denotes a partition $0 \le t_{0} < t_{1} < ... < t_{M} \le T$ for some positive integer $M$, and the condition $|P| \le \delta$ means that $|t_{i} - t_{i - 1}| \le \delta$ for all $i = 1, 2, ..., M$. 
\end{definition}
We introduce this definition of the $p$-variation of length scale $\delta$ because we will invoke estimates on the time shifts, as in \eqref{Qprimeest}, in order to deduce continuity in $\mathcal{Q}'$. The strategy will be to show that almost surely, the limiting  fluid velocity and structure velocity,
denoted by the pair $(\boldsymbol{u}, v)$, has a variation that goes to zero as the length scale $\delta$ goes to zero, which would imply that 
the pair $(\boldsymbol{u}, v)$ cannot have any discontinuities and is hence continuous in $\mathcal{Q}'$. We hence want to define and examine the subset of functions whose $p$-variation of length scale $\delta$ is bounded above by a certain parameter $\epsilon$. We do this in the following lemma.

\begin{lemma}\label{Aclosed}
Let $A_{p, \delta, \epsilon}$ be the set of functions $(\boldsymbol{u}, v): [0, T] \to \mathcal{Q}'$ in 
$
X = L^{2}(0, T; L^{2}(\Omega_{f})) \times L^{2}(0, T; L^{2}(\Gamma))
$
such that the following properties hold:
\begin{enumerate}
\item $(\boldsymbol{u}, v)$ has a version that is left continuous on $[0, T]$ as a function of time, taking values in $\mathcal{Q}'$. 
\item This version of $(\boldsymbol{u}, v)$ is also right continuous at $t = 0$ as a function taking values in $\mathcal{Q}'$.
\item For this (necessarily unique) left continuous version, $V^{\delta}_{p}(\boldsymbol{u}, v) \le \epsilon$. 
\end{enumerate}
Then, for any $p \ge 1$, $\delta > 0$, and $\epsilon > 0$, $A_{p, \delta, \epsilon}$ is a closed set in $X$.
\end{lemma}

\begin{proof}
To show that $A_{p, \delta, \epsilon}$ is a closed set in $X$, we consider a sequence $\{(\boldsymbol{u}_{n}, v_{n})\}_{n = 1}^{\infty}$ in $A_{p, \delta, \epsilon}$ that converges to some element $(\boldsymbol{u}, v) \in X$ in the norm of $X$. We claim that $(\boldsymbol{u}, v) \in A_{p, \delta, \epsilon}$. 

We start by showing Property 3 above, namely $V^{\delta}_{p}(\boldsymbol{u}, v) \le \epsilon$. Because 
$(\boldsymbol{u}_{n}, v_{n}) \to (\boldsymbol{u}, v)$ in $L^{2}(0, T; L^{2}(\Omega_{f}) \times L^{2}(\Gamma))$,
we have that along a subsequence, which we will continue to denote by the same index, we have
$(\boldsymbol{u}_{n}(t), v_{n}(t)) \to (\boldsymbol{u}(t), v(t))$ in $L^{2}(\Omega_{f}) \times L^{2}(\Gamma)$, for a.e. $t \in [0, T]$.
Since $L^{2}(\Omega_{f}) \times L^{2}(\Gamma)$ embeds continuously into $\mathcal{Q}'$,
\begin{equation}\label{aeconvQprime}
(\boldsymbol{u}_{n}(t), v_{n}(t)) \to (\boldsymbol{u}(t), v(t)), \ {\rm in} \ \mathcal{Q}', \  \forall t \in S,
\end{equation}
where $S$ is some measurable subset of $[0, T]$,
 which consists of the points $t \in [0, T]$ for which the convergence above holds. Note that $S$ has almost every $t \in [0, T]$. 

Consider any partition $P$ with $|P| \le \delta$, \textit{consisting only of points in $S$.} Now, from the fact that $V^{\delta}_{p}(\boldsymbol{u}_{n}, v_{n}) \le \epsilon$ for all $n$, we have that
$
\sum_{i = 1}^{M} ||(\boldsymbol{u}_{n}(t_{i}), v_{n}(t_{i})) - (\boldsymbol{u}_{n}(t_{i - 1}), v_{n}(t_{i - 1}))||^{p}_{\mathcal{Q}'} \le \epsilon.
$
Because the partition is finite and because the partition consists of points in $S$ for which the convergence \eqref{aeconvQprime} holds, in the limit as $n \to \infty$:
\begin{equation}\label{varS}
\sum_{i = 1}^{M} ||(\boldsymbol{u}(t_{i}), v(t_{i})) - (\boldsymbol{u}(t_{i - 1}), v(t_{i - 1}))||^{p}_{\mathcal{Q}'} \le \epsilon.
\end{equation}
This verifies Property 3 for partitions $|P| \le \delta$ consisting of points in $S$. To show Properties 1 and 2, we use the above inequality \eqref{varS} and construct a version of $(\boldsymbol{u}, v)$ which will satisfy all the properties of the set 
$A_{p, \delta, \epsilon}$. Then, we will conclude the proof by verifying Property 3 for this new version and extending the verification of Property 3 to general partitions $|P| \le \delta$ consisting of any points in $[0, T]$.
We start with Property 1 above, namely that $(\boldsymbol{u}, v)$ must have a version that is left continuous. We do this in the following steps.

\vspace{0.1in}

\noindent \textbf{Step 1:} First, we show that at each point $t \in [0, T]$, the left and right limits of $(\boldsymbol{u}, v)$ along points in $S$ must exist. This will be useful, as $S$ is dense in $[0, T]$. In addition, the density of $S$ in $[0, T]$ means that for all $t \in [0, T]$, the notion of a left and right limit along points in $S$ makes sense.

To show this, consider any point $t_{0} \in [0, T]$. We emphasize that $t_{0}$ is not necessarily in the set $S$. We claim that the left and right limit at $t_{0}$ along points in $S$ must exist. In particular, for any sequence $\{t_{n}\}_{n = 1}^{\infty}$ with $t_{n} \in S$ that increases to $t_{0}$ or decreases to $t_{0}$, we claim that
$\lim_{n \to \infty} (\boldsymbol{u}(t_{n}), v(t_{n}))$ exists in $\mathcal{Q}'$.

We show this by contradiction. Suppose there exists a strictly increasing sequence $\{t_{n}\}_{n = 1}^{\infty}$ with $t_{n} \in S$ and $t_{n} \nearrow t_{0}$, such that
$\lim_{n \to \infty} (\boldsymbol{u}(t_{n}), v(t_{n}))$ does not exist in $\mathcal{Q}'$.
The same argument will hold in the case of a decreasing sequence. This implies that $\{(\boldsymbol{u}(t_{n}), v(t_{n}))\}_{n = 1}^{\infty}$ does not converge in $\mathcal{Q}'$, and hence is not a Cauchy sequence. Thus, there exists $\epsilon_{0} > 0$, such that given any $N$, there exists $n_{1}, n_{2} \ge N$ such that
\begin{equation*}
||(\boldsymbol{u}(t_{n_{1}}), v(t_{n_{1}})) - (\boldsymbol{u}(t_{n_{2}}), v(t_{n_{2}}))||_{\mathcal{Q}'} \ge \epsilon_{0}.
\end{equation*}
Note that we have called this constant $\epsilon_{0}$ to distinguish it from the $\epsilon$ in the definition of $A_{p, \delta, \epsilon}$. 
Now, choose $M$ sufficiently large such that 
\begin{equation*}
M(\epsilon_{0})^{p} > \epsilon.
\end{equation*}
Choose a partition $P$ consisting of points $s_{0}, ..., s_{2M - 1}$ in $S$ with the following properties.
\begin{enumerate}
\item For each $i = 0, 1, ..., 2M - 1$, we have that $t_{0} - \delta < s_{i} < t_{0}$.
\item The sequence $s_{0}, ..., s_{2M - 2}, s_{2M - 1}$ is strictly increasing.
\item For even $i = 0, 2, ..., 2M - 2$,
$
||(\boldsymbol{u}(s_{i}), v(s_{i})) - (\boldsymbol{u}(s_{i + 1}), v(s_{i + 1}))||_{\mathcal{Q}'} \ge \epsilon_{0}.
$
This can be accomplished by using the non-convergent sequence $\{(\boldsymbol{u}(t_{n}), v(t_{n}))\}_{n = 1}^{\infty}$ and the definition of $\epsilon_{0}$ to choose the $s_{i}$ from the sequence $\{t_{n}\}_{n = 1}^{\infty}$, as $t_{n} \in S$ for all $n$. 
\end{enumerate}
Since $M$ was chosen so that $M(\epsilon_{0})^{p} > \epsilon$, for this partition $P$, consisting of points in $S$ with $|P| \le \delta$, we have 
\begin{equation*}
\sum_{i = 1}^{2M - 1} ||(\boldsymbol{u}(s_{i}), v(s_{i})) - (\boldsymbol{u}(s_{i - 1}), v(s_{i - 1}))||^{p}_{\mathcal{Q}'} > \epsilon
\end{equation*}
which is a contradiction.

Furthermore, the left and right limits along points in $S$ are well-defined. Suppose for contradiction that
\begin{equation*}
\lim_{n \to \infty} (\boldsymbol{u}(s_{n}), v(s_{n})) = L_{1} \ne L_{2} = \lim_{n \to \infty} (\boldsymbol{u}(t_{n}), v(t_{n})),
\end{equation*}
for two increasing sequences $s_{n} \nearrow t_{0}$ and $t_{n} \nearrow t_{0}$, consisting of points in $S$. Then, we can construct a new sequence $\{r_{k}\}_{k = 1}^{\infty}$, where we set $r_{0} = s_{0}$. Then, we set $r_{1}$ to be any $t_{n}$ for which $t_{n} \in (s_{0}, t_{0})$. We continue, creating an interlaced sequence where all odd indices of $r_{k}$ come from the sequence of $s_{n}$ points, and all even indices of $r_{k}$ come from the sequence of $t_{n}$ points, where along the odd sequence, the indices of the corresponding $s_{n}$ points is strictly increasing, and similarly for the even sequence of the $t_{n}$ points. We can also perform this construction so that the points in $r_{k}$ are strictly increasing to $t_{0}$. However, one can see that $\lim_{n \to \infty} (\boldsymbol{u}(r_{n}), v(r_{n}))$ does not exist, which contradicts our earlier result. So the left and right limits along points in $S$ are well-defined. 

\vspace{0.1in}

\noindent \textbf{Step 2:} In Step 1, we have shown that
$\lim_{t \to t_{0}^{-}, t \in S} (\boldsymbol{u}(t), v(t))$ and $\lim_{t \to t_{0}^{+}, t \in S} (\boldsymbol{u}(t), v(t))$
both exist for all $t \in [0, T]$, where these are limits in $\mathcal{Q}'$. We show that there can only be countably many points $t_{0} \in (0, T)$ for which these limits, which take values in $\mathcal{Q}'$, do not agree. 

To do this, we argue by contradiction. Suppose that there are uncountably many points in $(0, T)$ for which these limits do not agree. Then, there exists $\rho > 0$ sufficiently small such that there are infinitely many points $t_{0} \in (0, T)$ for which
\begin{equation*}
||\lim_{t \to t_{0}^{-}, t \in S} (\boldsymbol{u}(t), v(t)) - \lim_{t \to t_{0}^{+}, t \in S} (\boldsymbol{u}(t), v(t))||_{\mathcal{Q}'} \ge \rho.
\end{equation*}
Let $M$ be sufficiently large such that $M\left(\frac{\rho}{2}\right)^{p} > \epsilon$ and select $t_{1}, ..., t_{M}$ points of discontinuity in $(0, T)$ with 
\begin{equation*}
||\lim_{t \to t_{n}^{-}, t \in S} (\boldsymbol{u}(t), v(t)) - \lim_{t \to t_{n}^{+}, t \in S} (\boldsymbol{u}(t), v(t))||_{\mathcal{Q}'} \ge \rho, \qquad \text{ for } n = 1, 2, ..., M.
\end{equation*}
We can order these points as $t_{1} < t_{2} < ... < t_{M}$, and select $2M$ points $\{s_{n, i}\}_{1 \le n \le M, i = 1, 2}$ in $S$, such that
\begin{enumerate}
\item $s_{1, 1} < s_{1, 2} < s_{2, 1} < s_{2, 2} < ... < s_{M, 1} < s_{M, 2}$.
\item For each $n = 1, 2, ..., M$, 
$
t_{n} - \frac{\delta}{2} < s_{n, 1} < t_{n} < s_{n, 2} < t_{n} + \frac{\delta}{2}.
$
\item For each $n = 1, 2, ..., M$,
$
||(\boldsymbol{u}(s_{n, 1}), v(s_{n, 1})) - \lim_{t \to t_{n}^{-}, t \in S} (\boldsymbol{u}(t), v(t))||_{\mathcal{Q}'} < \frac{\rho}{4},
$
and
$
||(\boldsymbol{u}(s_{n, 2}), v(s_{n, 2})) - \lim_{t \to t_{n}^{+}, t \in S} (\boldsymbol{u}(t), v(t))||_{\mathcal{Q}'} < \frac{\rho}{4}.
$
\end{enumerate}
Then, we can form a partition of points in $S$ that interlaces the sequence $s_{1, 1} < s_{1, 2} < s_{2, 1} < s_{2, 2} < ... < s_{M, 1} < s_{M, 2}$ with additional points so that the resulting partition $P$ has $|P| < \delta$, since $S$ is dense in $[0, T]$. We can do this in a way that keeps the points $s_{n, i}$ for $i = 1, 2$ consecutive in the partition for each $n = 1, 2, ..., M$. Since $M\left(\frac{\rho}{2}\right)^{p} > \epsilon$, we have that the variation for this resulting partition is greater than $\epsilon$, which is a contradiction.

The same argument as above implies that \textit{there are only countably many points $t_{0} \in S$ for which}
\begin{equation*}
\lim_{t \to t_{0}^{-}, t \in S} (\boldsymbol{u}(t), v(t)) \ne (\boldsymbol{u}(t_{0}), v(t_{0})).
\end{equation*}
Thus, we define $S^{*}$ to be the set of points $t_{0} \in S$ for which
\begin{equation*}
\lim_{t \to t_{0}^{-}, t \in S} (\boldsymbol{u}(t), v(t)) = (\boldsymbol{u}(t_{0}), v(t_{0})).
\end{equation*}
Since countable sets have measure zero, $S^{*}$ still has the property that $[0, T] - S^{*}$ is of measure zero. So in particular, $S^{*}$ is still dense in $[0, T]$. \textit{We emphasize that now, $(\boldsymbol{u}(t), v(t))$ has the useful property that it is left continuous on $S^{*}$.}

\vspace{0.1in}

\noindent \textbf{Step 3:} Because $S^{*} \subset S$ and is still a dense set in $[0, T]$, the result from Step 1 implies that:\\
\centerline{$\lim_{t \to t_{0}^{-}, t \in S^{*}} (\boldsymbol{u}(t), v(t))$ and $\lim_{t \to t_{0}^{+}, t \in S^{*}} (\boldsymbol{u}(t), v(t))$ exist
for all $t_{0} \in [0, T]$.}
However, these limits are only along points in $S^{*}$. By the density of $S^{*}$ in $[0, T]$ and the fact that $[0, T] - S^{*}$ has measure zero, we can redefine $(\boldsymbol{u}, v)$ up to a version, so that
\begin{equation}\label{version1and2}
(\boldsymbol{u}(t_{0}), v(t_{0})) \text{ is unchanged if } t_{0} \in S^{*},\ {\rm and} \ 
(\boldsymbol{u}(t_{0}), v(t_{0})) = \lim_{t \to t_{0}^{-}, t \in S^{*}} (\boldsymbol{u}(t), v(t)) \ \text{ if } t_{0} \in [0, T] - S^{*}. 
\end{equation}
For the remainder of this proof, $(\boldsymbol{u}, v)$ will denote this newly defined version in \eqref{version1and2}. We then claim that for this version,
\begin{equation}\label{step3}
\lim_{t \to t_{0}^{-}, t \in S^{*}} (\boldsymbol{u}(t), v(t)) = \lim_{t \to t_{0}^{-}} (\boldsymbol{u}(t), v(t)) \qquad \text{ and } \qquad \lim_{t \to t_{0}^{+}, t \in S^{*}} (\boldsymbol{u}(t), v(t)) = \lim_{t \to t_{0}^{+}} (\boldsymbol{u}(t), v(t)),
\end{equation}
for all $t \in [0, T]$. We will just prove the first statement, for the limit from the left, as the statement for the limit from the right is proved analogously. To see this, note that by the definition of the version and by the definition of $S^{*}$ in Step 2,
\begin{equation}\label{leftcont}
(\boldsymbol{u}(t_{0}), v(t_{0})) = \lim_{t \to t_{0}^{-}, t \in S^{*}} (\boldsymbol{u}(t), v(t)), \qquad \text{ for all } t_{0} \in [0, T].
\end{equation}
So given any strictly increasing sequence $t_{n} \nearrow t_{0}$ where $t_{n}$ is not necessarily in $S^{*}$, we want to show that
\begin{equation*}
\lim_{n \to \infty} (\boldsymbol{u}(t_{n}), v(t_{n})) = \lim_{t \to t_{0}^{-}, t \in S^{*}} (\boldsymbol{u}(t), v(t)).
\end{equation*}
To do this, we use the density of $S^{*}$ in $[0, T]$ along with \eqref{leftcont} to construct a sequence $s_{n}$ such that 
\begin{enumerate}
\item $s_{0} \le t_{0}$ and $t_{n - 1} < s_{n} \le t_{n}$ for all $n \ge 1$.
\item $|s_{n} - t_{n}| < 2^{-n}$ for all $n$. 
\item $||(\boldsymbol{u}(s_{n}), v(s_{n})) - (\boldsymbol{u}(t_{n}), v(t_{n}))||_{\mathcal{Q}'} < 2^{-n}$ for all $n$. 
\item $s_{n} \in S^{*}$ for all $n$.
\end{enumerate}
This is possible because $S^{*}$ is dense in $[0, T]$, and shows the desired result, as $s_{n}$ is a strictly increasing sequence converging to $t_{0}$ by Property 1 and 2, and by Property 3 and 4 we have
\begin{equation*}
\lim_{n \to \infty} (\boldsymbol{u}(t_{n}), v(t_{n})) = \lim_{n \to \infty} (\boldsymbol{u}(s_{n}), v(s_{n})) = \lim_{t \to t_{0}^{-}, t \in S^{*}} (\boldsymbol{u}(t), v(t)).
\end{equation*}
Note that this version of $(\boldsymbol{u}(t), v(t))$ on $[0, T]$  is left continuous by \eqref{step3} and \eqref{leftcont}, with only countably many points of discontinuity by Step 2.

\vspace{0.1in}

\noindent \textbf{Conclusion:} We have constructed a left continuous version of $(\boldsymbol{u}(t), v(t))$ on $[0, T]$ taking values in $\mathcal{Q}'$ in Step 3. At the left boundary, $t = 0$, we can set the version of $(\boldsymbol{u}, v)$ so that $(\boldsymbol{u}(0), v(0)) = \lim_{t \to 0^{+}} (\boldsymbol{u}(t), v(t))$, so that we have right continuity at $t = 0$. This is possible since this limit exists by Step 1 and \eqref{step3}. For the newly defined version of $(\boldsymbol{u}(t), v(t))$, we have that 
\begin{equation*}
\sum_{i = 1}^{N} ||(\boldsymbol{u}(x_{i}), v(x_{i})) - (\boldsymbol{u}(x_{i - 1}), v(x_{i - 1}))||^{p}_{\mathcal{Q}'} \le \epsilon,
\end{equation*}
for all partitions $P$ consisting of points in $S^{*}$ with $|P| \le \delta$, since we did not change the original $(\boldsymbol{u}(t), v(t))$ on points of $S^{*}$, which is a subset of $S$. We can now show that this $p$-variation inequality holds more generally for all partitions $P$ with points in $[0, T]$ with $|P| \le \delta$. To do this, we note that since $S^{*}$ is dense in $[0, T]$, we can approximate any partition $P$ of arbitrary points in $[0, T]$ with $|P| \le \delta$ by a sequence of partitions $\{P_{k}\}_{k \ge 1}$ of points in $S^{*}$ with $|P_{k}| \le \delta$ containing the same number of points as $P$. We can do this by approaching any partition points of $P$ in $(0, T]$ from the left by points in $S^{*}$, and approaching $t = 0$ from the right by points in $S^{*}$ if $t = 0$ is a partition point in $P$. We then obtain the desired result by taking the limit in $k$ as the partitions $P_{k}$ approach $P$. This process of taking the limit uses the fact that the version of $(\boldsymbol{u}, v)$ as defined in Step 3 is left continuous on $[0, T]$ and right continuous at $t = 0$. Therefore, we conclude that $V^{\delta}_{p}(\boldsymbol{u}, v) \le \epsilon$.
\end{proof}

The next lemma, along with the weak convergence of the laws $\mu_{N}$, will allow us to use the result above to prove almost sure continuity in $\mathcal{Q}'$ of the limiting fluid and structure velocity. In particular, this next lemma will show that if the length scale $\delta$ is chosen appropriately, then eventually, for large enough $N$ (or equivalently small enough $\Delta t$), the approximate solutions will have $p$-variation (for $p > 4$) with length scale $\delta$ bounded above uniformly with high probability. This is to be expected, due to the time shift estimate \eqref{Qprimeest}, which is \textit{independent of $N$}.

For the following results, we recall the definition of $\mu_{N}$ on the phase space $\mathcal{X}$ from \eqref{muN} and \eqref{phase}, and we denote by $\pi_{4, 5}\mu_{N}$ the projection onto the fourth and fifth components of $\mathcal{X}$, which gives the law of $(\boldsymbol{u}_{N}, v_{N})$ on $L^{2}(0, T; L^{2}(\Omega_{f})) \times L^{2}(0, T; L^{2}(\Gamma))$.

\begin{lemma}\label{Abound}
For any $p > 4$ and any $\epsilon > 0$, there exists $\delta_{0} > 0$ sufficiently small and $N_{0}$ sufficiently large such that for all $0 < \delta \le \delta_{0}$,
\begin{equation*}
\pi_{4, 5}\mu_{N}(A_{p, \delta, \epsilon}) > 1 - \epsilon, \qquad \text{ for all } N \ge N_{0}.
\end{equation*}
\end{lemma}

\begin{proof}
We start by first introducing a set $\mathcal{K}_{R, N}$. 
Let $\mathcal{K}_{R, N}$ be the collection of paths in $\mathcal{K}_{R}$ 
(introduced in Definition \ref{KR}),
corresponding to path realizations of the random variables $(\boldsymbol{u}_{N}, v_{N})$ for fixed $N$, satisfying the properties in the definition of $\mathcal{K}_{R}$. In particular, $\mathcal{K}_{R} = \bigcup_{N = 1}^{\infty} \mathcal{K}_{R, N}$. Notice that we can choose $R$ large enough so that 
\begin{equation*}
\pi_{4, 5}\mu_{N}(\overline{\mathcal{K}_{R, N}}) > 1- \epsilon, \qquad \text{for all $N$},
\end{equation*}
where the closure is taken in $L^{2}(0, T; L^{2}(\Omega_{f})) \times L^{2}(0, T; L^{2}(\Gamma))$. Recall from \eqref{Qprimeest} that
\begin{equation*}
||(\boldsymbol{u}^{n + l}_{N}, v^{n + l}_{N}) - (\boldsymbol{u}^{n}_{N}, v^{n}_{N})||_{\mathcal{Q}'} \le C_{R}(l\Delta t)^{1/4},
\end{equation*}
where $C_{R}$ is a constant depending only on $R$ for all $(\boldsymbol{u}_{N}, v_{N}) \in \mathcal{K}_{R}$. In particular, $C_{R}$ is independent of $n$, $l$, and $N$ and hence $\Delta t$. We will use this estimate on the increments in $\mathcal{Q}'$ to choose $\delta_{0} > 0$ and $N_{0}$ such that 
$\mathcal{K}_{R, N} \subset A_{p, \delta, \epsilon},  \text{ for all } N \ge N_{0} \text{ and } 0 < \delta \le \delta_{0}$,
and ultimately, 
$\overline{\mathcal{K}_{R, N}} \subset A_{p, \delta, \epsilon}, \text{ for all } N \ge N_{0} \text{ and } 0 < \delta \le \delta_{0},$
from which the result 
$
\pi_{4, 5}\mu_{N}(A_{p, \delta, \epsilon}) > 1 - \epsilon,  \text{ for all } N \ge N_{0}
$
and $0 < \delta \le \delta_{0}$ will follow. Indeed,
for any given partition $P$ with $|P| \le \delta$, the following estimate holds:
\begin{equation*}
\sum_{i = 1}^{M} ||(\boldsymbol{u}(x_{i}), v(x_{i})) - (\boldsymbol{u}(x_{i - 1}), v(x_{i - 1}))||^{p}_{\mathcal{Q}'}
\le C_{R}\sum_{k = 1}^{l} n_{k}(k\Delta t)^{p/4},
\end{equation*}
for any $(\boldsymbol{u}, v) \in \mathcal{K}_{R}$, where $n_{k}$ is the number of increments that have indices $k$ apart and $l$ is the maximum integer for which $l \Delta t < \delta + \Delta t$. This is true by the fact that the paths $(\boldsymbol{u}, v)$ in $\mathcal{K}_{R}$ are defined as piecewise constant functions taking values $(\boldsymbol{u}^{n}_{N}, v^{n}_{N})$, and by inequality \eqref{Qprimeest}. Because the partition $P$ has $|P| \le \delta$, we have that $l$ must satisfy
\begin{equation}\label{lcond}
l\Delta t < \delta + \Delta t = \delta + N^{-1}T\quad {\rm and} \quad 
\sum_{k = 1}^{l} n_{k}(k\Delta t) \le (N - 1)\Delta t = T - \Delta t. 
\end{equation}
Therefore, since $p > 4$, we have that for any partition $P$ with $|P| \le \delta$ and for any $(\boldsymbol{u}, v) \in \mathcal{K}_{R}$,
\begin{align*}
\sum_{i = 1}^{M} ||(\boldsymbol{u}(x_{i}), v(x_{i})) - (\boldsymbol{u}(x_{i - 1}), v(x_{i - 1}))||^{p}_{\mathcal{Q}'} &\le C_{R}\sum_{k = 1}^{l} n_{k}(k\Delta t)^{p/4} \le C_{R}\left(\sum_{k = 1}^{l} n_{k}(k\Delta t)\right) (l\Delta t)^{\frac{p}{4} - 1} \\
&\le C_{R} T (l\Delta t)^{\frac{p}{4} - 1} \le  C_{R} T (\delta + N^{-1}T)^{\frac{p}{4} - 1},
\end{align*}
where we used \eqref{lcond}. The proof is complete once we choose $\delta_{0} > 0$ sufficiently small and $N_{0}$ sufficiently large such that
\begin{equation*}
C_{R}T(\delta_{0} + N_{0}^{-1}T)^{\frac{p}{4} - 1} < \epsilon.
\end{equation*}
Therefore, for $(\boldsymbol{u}_{N}, v_{N})$ in $\mathcal{K}_{R}$ for any $N \ge N_{0}$ and $0 < \delta \le \delta_{0}$, we have
$
V^{\delta}_{p}(\boldsymbol{u}_{N}, v_{N}) \le \epsilon.
$
Thus,
\begin{equation*}
\mathcal{K}_{R, N} \subset A_{p, \delta, \epsilon}, \qquad \text{ for all } N \ge N_{0} \text{ and } 0 < \delta \le \delta_{0}.
\end{equation*}
Since $A_{p, \delta, \epsilon}$ is closed in $L^{2}(0, T; L^{2}(\Omega_{f})) \times L^{2}(0, T; L^{2}(\Gamma))$ by Lemma \ref{Aclosed}, we conclude that
\begin{equation*}
\overline{\mathcal{K}_{R, N}} \subset A_{p, \delta, \epsilon}, \qquad \text{ for all } N \ge N_{0} \text{ and } 0 < \delta \le \delta_{0},
\end{equation*}
where the closure is taken with respect to the norm of $L^{2}(0, T; L^{2}(\Omega_{f})) \times L^{2}(0, T; L^{2}(\Gamma))$. Since $\pi_{4, 5}\mu_{N}(\overline{\mathcal{K}_{R, N}}) > 1 - \epsilon$ for all positive integers $N$ by the initial choice of $R$, this implies the result.
\end{proof}

\begin{lemma}\label{continuitylemma}
For the weak limit $\mu$,
\begin{equation*}
\pi_{4, 5}\mu(X \cap C(0, T; \mathcal{Q}')) = 1,
\end{equation*}
where $X := L^{2}(0, T; L^{2}(\Omega_{f})) \times L^{2}(0, T; L^{2}(\Gamma))$. Furthermore, $\pi_{4, 5}\mu$ is supported on a Borel measurable subset of $X$ such that every function has a version in $C(0, T; \mathcal{Q}')$ that is equal to $(\boldsymbol{u}_{0}, v_{0})$ at $t = 0$. 
\end{lemma}

\begin{remark}
We remark that $X \cap C(0, T; \mathcal{Q}')$ is a Borel measurable subset of $X$, and hence the statement above makes sense. To see this, note that the inclusion map
\begin{equation*}
\iota: X \to L^{2}(0, T; \mathcal{Q}')
\end{equation*}
is continuous since $L^{2}(\Omega_{f}) \times L^{2}(\Gamma)$ embeds continuously into $\mathcal{Q}'$. It suffices to show that $C(0, T; \mathcal{Q}')$ is a Borel measurable subset of $L^{2}(0, T; \mathcal{Q}')$. Then, $X \cap C(0, T; \mathcal{Q}')$ is the preimage of $C(0, T; \mathcal{Q}')$ under $\iota$, and is hence measurable in $X$, which is the desired result. 

To show that $C(0, T; \mathcal{Q}')$ is a Borel measurable subset of $L^{2}(0, T; \mathcal{Q}')$, we note that a closed ball of arbitrary radius $R$ in $L^{\infty}(0, T; \mathcal{Q}')$ is Borel measurable in $L^{2}(0, T; \mathcal{Q}')$ since it is a closed set in $L^{2}(0, T; \mathcal{Q}')$. Since one can express an open ball as a countable union of closed balls, an open ball of arbitrary radius $R$ in $L^{\infty}(0, T; \mathcal{Q}')$ is also measurable in $L^{2}(0, T; \mathcal{Q}')$. Since $C(0, T; \mathcal{Q}')$ is closed in $L^{\infty}(0, T; \mathcal{Q}')$ in the topology of $L^{\infty}(0, T; \mathcal{Q}')$, and since closed and open balls of $L^{\infty}(0, T; \mathcal{Q}')$ are Borel measurable in $L^{2}(0, T; \mathcal{Q}')$, this implies that $C(0, T; \mathcal{Q}')$ is Borel measurable in $L^{2}(0, T; \mathcal{Q}')$.
\end{remark}

\begin{proof}
Fix $p > 4$ and set $\epsilon_{k} = 2^{-k}$. Then, by Lemma \ref{Abound}, there exists a decreasing sequence of positive real numbers $\{\delta_{k}\}_{k = 1}^{\infty}$ and an increasing sequence of positive integers $\{N_{k}\}_{k = 1}^{\infty}$, such that 
\begin{equation*}
\pi_{4, 5}\mu_{N}(A_{p, \delta_{k}, \epsilon_{k}}) > 1 - \epsilon_{k}, \qquad \text{ for all } N \ge N_{k}, \text{ and } k \in \mathbb{Z}^{+}.
\end{equation*}
Note that since $\mu_{N}$ converges weakly to $\mu$, we have that $\pi_{4, 5}\mu_{N}$ converges weakly to $\pi_{4, 5}\mu$. For each fixed positive integer $k$, since $A_{p, \delta_{k}, \epsilon_{k}}$ is a closed set in $X$, we have by Portmanteau's theorem for weak convergence of probability measures that
\begin{equation*}
\pi_{4, 5}\mu(A_{p, \delta_{k}, \epsilon_{k}}) \ge \limsup_{N \to \infty} \pi_{4, 5}\mu_{N}(A_{p, \delta_{k}, \epsilon_{k}}) \ge 1 - \epsilon_{k}.
\end{equation*}
By the Borel Cantelli lemma and by the choice of $\epsilon_{k} = \frac{1}{2^{k}}$ so that $\sum_{k = 1}^{\infty} \frac{1}{2^{k}} < \infty$, 
\begin{equation*}
\pi_{4, 5}\mu(A_{p, \delta_{k}, \epsilon_{k}}^{c} \text{occurs infinitely often in $k \in \mathbb{Z}^{+}$}) = 0.
\end{equation*}
So almost surely, $\pi_{4, 5}\mu$ takes values in the set $\{A_{p, \delta_{k}, \epsilon_{k}} \text{occurs for infinitely many $k$}\}$. However, one can show that
\begin{equation}\label{continuityinclusion}
\{A_{p, \delta_{k}, \epsilon_{k}} \text{ occurs for infinitely many $k$}\} \subset X \cap C(0, T; \mathcal{Q}'),
\end{equation}
which then implies the result. To see why this is true, suppose that 
\begin{equation*}
(\boldsymbol{u}, v) \in \{A_{p, \delta_{k}, \epsilon_{k}} \text{ occurs for infinitely many $k$}\}.
\end{equation*}
By the fact that $V^{\delta_{k}}_{p}(\boldsymbol{u}, v) \le \epsilon_{k}$, we must have that for every $t_{0} \in [0, T]$ (modified appropriately for the endpoint cases $t_{0} = 0$ and $t_{0} = T$),
\begin{equation*}
||(\boldsymbol{u}(t), v(t)) - (\boldsymbol{u}(t_{0}), v(t_{0}))||_{\mathcal{Q}'} \le \epsilon_{k}^{1/p}, \qquad \text{ for all } t \in (t_{0} - \delta_{k}, t_{0} + \delta_{k}) \cap [0, T],
\end{equation*}
for any $k$ such that $(\boldsymbol{u}, v) \in A_{p, \delta_{k}, \epsilon_{k}}$. But since $(\boldsymbol{u}, v) \in A_{p, \delta_{k}, \epsilon_{k}}$ for infinitely many $k$ and since $\epsilon_{k} = 2^{-k} \to 0$ as $k \to \infty$, this implies that 
\begin{equation*}
\lim_{t \to t_{0}} (\boldsymbol{u}(t), v(t)) \text{ exists and equals } (\boldsymbol{u}(t_{0}), v(t_{0})).
\end{equation*}
This shows the desired result in \eqref{continuityinclusion}. Therefore, we have shown the first part of the lemma, that $\pi_{4, 5}\mu(X \cap C(0, T; \mathcal{Q}')) = 1$. 

It remains to show that $\pi_{4, 5}\mu$ is supported more specifically on a Borel measurable subset of $X$ that consists entirely of functions that have a version that is in $C(0, T; \mathcal{Q}')$ with value $(\boldsymbol{u}_{0}, v_{0})$ at time $t = 0$. Define the set $B_{R}$ to be the set of functions $(\boldsymbol{u}, v) \in L^{2}(0, T; L^{2}(\Omega_{f}) \times L^{2}(\Gamma))$ such that
\begin{equation}\label{BRdef}
||(\boldsymbol{u}(\cdot), v(\cdot)) - (\boldsymbol{u}_{0}, v_{0})||_{L^{2}(0, h; \mathcal{Q}')} \le C_{R}h^{3/4}, \qquad \text{ for all } 0 < h \le T,
\end{equation}
where $C_{R}$ is the constant from the estimate \eqref{Qprimeest}.

One can check that for every $R > 0$, every element of $\mathcal{K}_{R}$ satisfies \eqref{BRdef} and hence is in $B_{R}$. This is because by using \eqref{Qprimeest}, which states that
\begin{equation*}
||(\boldsymbol{u}^{n + l}_{N}, v^{n + l}_{N}) - (\boldsymbol{u}^{n}_{N}, v^{n}_{N})||_{\mathcal{Q}'} \le C_{R}(l\Delta t)^{1/4}, \qquad \text{ for all } (\boldsymbol{u}_{N}, v_{N}) \in \mathcal{K}_{R},
\end{equation*}
we have that for all $0 < h \le T$ and for all $(\boldsymbol{u}, v) \in \mathcal{K}_{R}$,
\begin{equation*}
||(\boldsymbol{u}(\cdot), v(\cdot)) - (\boldsymbol{u}_{0}, v_{0})||^{2}_{L^{2}(0, h; \mathcal{Q}')} \le \int_{0}^{h} ||(\boldsymbol{u}(s), v(s)) - (\boldsymbol{u}_{0}, v_{0})||^{2}_{\mathcal{Q}'} ds \le C_{R}^{2} h \cdot \left(h^{1/4}\right)^{2} = C_{R}^{2} h^{3/2}.
\end{equation*}
Furthermore, one checks easily that $B_{R}$ is closed in $L^{2}(0, T; L^{2}(\Omega_{f}) \times L^{2}(\Gamma))$ since a sequence that converges in $L^{2}(0, T; L^{2}(\Omega_{f}) \times L^{2}(\Gamma))$ also converges in $L^{2}(0, T; \mathcal{Q}')$, in which case one can take the limit in \eqref{BRdef} to get the corresponding property for the limit function. Since $\mathcal{K}_{R} \subset B_{R}$ and $B_{R}$ is closed in $X$, we obtain that
\begin{equation*}
\overline{\mathcal{K}_{R}} \subset B_{R} \subset X, \qquad \text{ for all } R > 0.
\end{equation*}

Consider any $\epsilon > 0$. Choose $R$ sufficiently large so that
\begin{equation*}
\pi_{4, 5}\mu_{N}(\overline{K_{R}}) > 1 - \epsilon, \qquad \text{ for all } N.
\end{equation*}
Then, by Portmanteau's theorem,
\begin{equation*}
\pi_{4, 5}\mu(B_{R}) \ge \limsup_{N \to \infty} \pi_{4, 5}\mu_{N}(B_{R}) \ge \limsup_{n \to \infty} \pi_{4, 5}\mu_{N}(\overline{K_{R}}) \ge 1 - \epsilon.
\end{equation*}
So there exists an increasing sequence $\{R_{k}\}_{k = 1}^{\infty}$ such that 
$
\pi_{4, 5}\mu\left(\bigcup_{k = 1}^{\infty} B_{R_{k}}\right) = 1.
$
Thus,
\begin{equation*}
\pi_{4, 5}\mu\left[\left(\bigcup_{k = 1}^{\infty} B_{R_{k}}\right) \cap C(0, T; \mathcal{Q}')\right] = 1,
\quad 
{\rm where} 
\quad
\left(\bigcup_{k = 1}^{\infty} B_{R_{k}}\right) \cap C(0, T; \mathcal{Q}') = \left(\bigcup_{k = 1}^{\infty} B_{R_{k}}\right) \cap X \cap C(0, T; \mathcal{Q}')
\end{equation*}
is a Borel measurable subset of $X$. However, we note that any function in $\left(\bigcup_{k = 1}^{\infty} B_{R_{k}}\right) \cap C(0, T; \mathcal{Q}')$ must have the property that its (unique) continuous version taking values in $\mathcal{Q}'$ must be equal to $(\boldsymbol{u}_{0}, v_{0})$ at $t = 0$. To see this, if instead, $(\boldsymbol{u}(0), v(0)) \ne (\boldsymbol{u}_{0}, v_{0})$, let
\begin{equation*}
d = ||(\boldsymbol{u}(0), v(0)) - (\boldsymbol{u}_{0}, v_{0})||_{\mathcal{Q}'} > 0.
\end{equation*}
Then, one can show that there exists $h_{0}$ such that for all $0 < h \le h_{0}$, 
\begin{equation*}
||(\boldsymbol{u}(\cdot), v(\cdot)) - (\boldsymbol{u}_{0}, v_{0})||_{L^{2}(0, h; \mathcal{Q}')} \ge \frac{d}{2}h^{1/2}.
\end{equation*}
Therefore, this function cannot satisfy an estimate of the type
\begin{equation*}
||(\boldsymbol{u}(\cdot), v(\cdot)) - (\boldsymbol{u}_{0}, v_{0})||_{L^{2}(0, h; \mathcal{Q}')} \le Ch^{3/4}, \qquad \text{ for all } 0 < h \le h_{0},
\end{equation*}
for any $C$, and so this function cannot be in any $B_{R}$. This completes the proof. 
\end{proof}

\if 1 = 0
\begin{lemma}
Suppose that $f_{n} \in L^{\infty}(0, T; B)$ are left continuous functions on $[0, T]$, such that
\begin{equation*}
f_{n} \to f \qquad \text{ in } L^{\infty}(0, T; B).
\end{equation*}
Then, there exists a left continuous version of $f$, which we will denote by $f^{*}$. Furthermore, 
\begin{equation*}
\sup_{t \in [0, T]} ||f^{*}(t) - f_{n}(t)||_{B} \to 0 \qquad \text{ as } n \to \infty,
\end{equation*}
so that $f_{n}$ converges uniformly to the left continuous function $f^{*}$. 
\end{lemma}

\begin{proof}
There exists a constant $M$ such that 
\begin{equation*}
||f - f_{n}||_{L^{\infty}(0, T; B)} \le M, \qquad \text{ for all } n.
\end{equation*}

Define a strictly increasing sequence of positive integers $\{N_{k}\}_{k = 1}^{\infty}$ as follows. Set $N_{1} = 1$ and define the sequence $\{N_{k}\}_{k = 1}^{\infty}$ inductively as follows. For each positive integer $k \ge 2$, let $N_{k}$ be the smallest positive integer such that $N_{k} > N_{k - 1}$ and 
\begin{equation*}
||f - f_{n}||_{L^{\infty}(0, T; B)} \le \frac{M}{k}, \qquad \text{ for all } n \ge N_{k}.
\end{equation*}

Next, we will define a sequence of sets $\{A_{n}\}_{n = 1}^{\infty}$ defined as follows. For a given $n$, define $k(n)$ to be the smallest positive integer such that $N_{k(n)} \le n < N_{k(n) + 1}$. Then, define
\begin{equation*}
A_{n} = \left\{t \in [0, T] : ||f(t) - f_{n}(t)||_{B} > \frac{M}{k(n)}\right\},
\end{equation*}
and define
\begin{equation*}
A = \bigcup_{n = 1}^{\infty} A_{n}.
\end{equation*}
Since the sets $A_{n}$ are sets of measure zero, $A$ is also a set of measure zero. We observe that the complement of $A$ is dense in $[0, T]$, since $A$ is a set of measure zero. Note that $f_{n} \to f$ uniformly on the set $[0, T] - A$. Since the uniform limit of left continuous functions is left continuous, we have that $f$ is a left continuous function on $[0, T] - A$. 

We have that $f$ is left continuous on $[0, T] - A$, but it is not necessarily left continuous on $[0, T]$. Therefore, we must construct a version $f^{*}$ of $f$ that is left continuous. We let
\begin{equation*}
f^{*}(t) = f(t), \qquad \text{ for all } t \in [0, T] - A.
\end{equation*}
We now have to define $f^{*}$ on the set $A$, which has measure zero. Consider $t \in A$. To define $f^{*}(t)$ for $t \in A$, consider a decreasing sequence $\{t_{i}\}_{i = 1}^{\infty}$ such that $t_{i} \to t$ and $t_{i} \in [0, T] - A$. We can do this since $[0, T] - A$ is dense in $[0, T]$. 

We claim that $\{f(t_{i})\}_{i = 1}^{\infty}$ is a Cauchy sequence in $B$. Assuming this claim for now, we then have that this sequence converges, so that we can define $f^{*}(t) = \lim_{i \to \infty} f(t_{i})$. We prove this claim now. Consider an arbitrary $\epsilon > 0$. We want to show that there exists $m$ sufficiently large such that 
\begin{equation*}
||f(t_{i}) - f(t_{j})||_{B} < \epsilon, \qquad \text{ for all } i, j \ge m.
\end{equation*}
To see this, choose $k_{0}$ such that 
\begin{equation*}
\frac{M}{k_{0}} < \frac{\epsilon}{4}.
\end{equation*}
Recall that $||f(s) - f_{N_{k_{0}}}(s)||_{B} < \frac{M}{k_{0}}$ for all $s \notin [0, T] - A$ and in particular, this holds for all $s = t_{i}$ since the sequence $\{t_{i}\}_{i = 1}^{\infty}$ was chosen so that $t_{i} \in [0, T] - A$. Because $f_{N_{k_{0}}}$ is left continuous on $[0, T]$, we can choose $\delta > 0$ such that 
\begin{equation*}
||f_{N_{k_{0}}}(s) - f_{N_{k_{0}}}(t)||_{B} < \frac{\epsilon}{4}, \qquad \text{ for all } t \le s < t + \delta. 
\end{equation*}
Because $t_{i} \searrow t$, there exists $m$ sufficiently large such that for all $i \ge m$, $t \le t_{i} < t + \delta$. Then, for all $i, j \ge m$,
\begin{multline*}
||f(t_{i}) - f(t_{j})||_{B} \\
\le ||f(t_{i}) - f_{N_{k_{0}}}(t_{i})||_{B} + ||f_{N_{k_{0}}}(t_{i}) - f_{N_{k_{0}}}(t)||_{B} + ||f_{N_{k_{0}}}(t_{j}) - f(t_{j})||_{B} + ||f(t_{j}) - f_{N_{k_{0}}}(t_{j})||_{B} \\
< \frac{\epsilon}{4} + \frac{\epsilon}{4} + \frac{\epsilon}{4} + \frac{\epsilon}{4} = \epsilon.
\end{multline*}

Next, we show that this definition of $f^{*}(t)$ for $t \in A$ is well-defined. Suppose that $\{s_{i}\}$ and $\{t_{i}\}$ are both sequences that decrease to $t \in A$. We have shown that $f(s_{i})$ and $f(t_{i})$ both converge as $i \to \infty$. Thus, we must show uniqueness of the limit, that
\begin{equation*}
\lim_{i \to \infty} f(s_{i}) = \lim_{i \to \infty} f(t_{i}).
\end{equation*}
We do this by showing that given any $\epsilon > 0$, there exists $m$ such that
\begin{equation*}
||f(s_{i}) - f(t_{i})||_{B} < \epsilon, \qquad \text{ for all } i \ge m.
\end{equation*}
To do this, choose $k_{0}$ such that
\begin{equation*}
\frac{M}{k_{0}} < \frac{\epsilon}{4},
\end{equation*}
so that 
\begin{equation*}
||f(s) - f_{N_{k_{0}}}(s)||_{B} < \frac{M}{k_{0}} < \frac{\epsilon}{4}, \qquad \text{ for all } s \in [0, T] - A.
\end{equation*}
In particular, this holds for $s = s_{i}$ and $s = t_{i}$ since $s_{i}$ and $t_{i}$ are chosen to be in $[0, T] - A$. Then, by the left continuity of $f_{N_{k_{0}}}$, there exists $\delta > 0$ such that 
\begin{equation*}
||f_{N_{k_{0}}}(s) - f_{N_{k_{0}}}(t)||_{B} < \frac{\epsilon}{4}, \qquad \text{ for all } t \le s < t + \delta.
\end{equation*} 
Choose $m$ sufficiently large such that
\begin{equation*}
t \le s_{i} < t + \delta \qquad \text{ and } t \le t_{i} < t + \delta, \qquad \text{ for all } i \ge m.
\end{equation*}
Then, for all $i \ge m$,
\begin{multline*}
||f(s_{i}) - f(t_{i})||_{B} \\
\le ||f(s_{i}) - f_{N_{k_{0}}}(s_{i})||_{B} + ||f_{N_{k_{0}}}(s_{i}) - f_{N_{k_{0}}}(t)||_{B} + ||f_{N_{k_{0}}}(t_{i}) - f_{N_{k_{0}}}(t)||_{B} + ||f(t_{i}) - f_{N_{k_{0}}}(t_{i})||_{B} \\
\le \frac{\epsilon}{4} + \frac{\epsilon}{4} + \frac{\epsilon}{4} + \frac{\epsilon}{4} = \epsilon.
\end{multline*}
Therefore, 
\begin{equation*}
\lim_{i \to \infty} f(s_{i}) = \lim_{i \to \infty} f(t_{i}),
\end{equation*}
and the definition of $f^{*}(t)$ for $t \in A$ is well-defined.

Now that we have defined $f^{*}$, we must check that $f^{*}$ is indeed left continuous. Let $\{t_{i}\}_{i = 1}^{\infty}$ be any sequence $t_{i} \to t$ that decreases to $t$. We want to show that
\begin{equation*}
f^{*}(t) = \lim_{i \to \infty} f(t_{i}).
\end{equation*}
We note that this is not immediate if any of the points in the sequence $t_{i}$ belong to $A$. However, since the complement of $A$ in $[0, T]$ is dense in $[0, T]$, we can choose a sequence $\{\tau_{i}\}_{i = 1}^{\infty}$ such that 
\begin{equation*}
t_{i} \le \tau_{i} < t_{i} + \frac{1}{i}, \qquad ||f(t_{i}) - f(\tau_{i})||_{B} < \frac{1}{i}, \qquad \tau_{i} \notin A.
\end{equation*}
(Note that if $t = T$, then we can set $f^{*}(T)$ to anything to get left continuity so we are only considering the case where $0 \le t < T$. Then, if any of the $t_{i} = T$, remove these from the list). By the previous step, $\lim_{i \to \infty} f(\tau_{i}) = f^{*}(t)$ since $\tau_{i} \to t$ also. To show that $\lim_{i \to \infty} f(t_{i}) = f^{*}(t)$ also, given $\epsilon > 0$, there exists $m$ such that for all $i \ge m$,
\begin{equation*}
||f^{*}(t) - f(\tau_{i})||_{B} < \frac{\epsilon}{2}, \qquad \text{ for all } i \ge m.
\end{equation*}
We can also take $m$ sufficiently large such that 
\begin{equation*}
\frac{1}{m} < \frac{\epsilon}{2}.
\end{equation*}
Then, for all $i \ge m$,
\begin{equation*}
||f^{*}(t) - f(t_{i})||_{B} \le ||f^{*}(t) - f(\tau_{i})||_{B} + ||f(\tau_{i}) - f(t_{i})||_{B} \le \frac{\epsilon}{2} + \frac{1}{i} \le \frac{\epsilon}{2} + \frac{1}{m} < \epsilon.
\end{equation*}
Hence, $f^{*}$ is a left continuous version of $f$. 

Finally, we must show that 
\begin{equation*}
\sup_{t \in [0, T]} ||f^{*}(t) - f_{n}(t)||_{B} \to 0 \qquad \text{ as } n \to \infty.
\end{equation*}
We have that
\begin{equation*}
\sup_{t \in [0, T], t \notin A} ||f^{*}(t) - f_{n}(t)||_{B} \le \frac{M}{k(n)},
\end{equation*}
by construction. We claim that for all $n$,
\begin{equation*}
\sup_{t \in [0, T]} ||f^{*}(t) - f_{n}(t)||_{B} \le \frac{M}{k(n)} + \frac{1}{n}.
\end{equation*}
This is due to the left continuity of all of the $f_{n}$ and of $f^{*}$. In particular, given any $t \in A$, choose $t_{0} > t$ such that $t_{0} \in [0, T] - A$ and
\begin{equation*}
||f^{*}(t_{0}) - f^{*}(t)||_{B} < \frac{1}{2n}, \qquad ||f_{n}(t_{0}) - f(t)||_{B} < \frac{1}{2n}.
\end{equation*}
Then, 
\begin{equation*}
||f^{*}(t) - f_{n}(t)|| \le ||f^{*}(t_{0}) - f^{*}(t)||_{B} + ||f^{*}(t_{0}) - f_{n}(t_{0})|| + ||f_{n}(t_{0}) - f(t)||_{B} \le \frac{M}{k(n)} + \frac{1}{n}.
\end{equation*}
Since $k(n) \to \infty$ as $n \to \infty$, this establishes the desired claim. 

\end{proof}

\fi

\subsection{Skorohod representation theorem}

We now use the classical Skorohod representation theorem to translate weak convergence of probability measures to almost sure convergence of random variables, which will allow us to pass to the limit in the semidiscrete weak formulation. 
However, this will be at the expense of working on a different probability space. 
Namely, the Skorohod representation theorem provides the {\emph{existence}} of {\emph{a probability space}}, on which we 
will have almost sure convergence of new random variables with the same laws as the original approximate solutions, to a weak solution with the law $\mu$ from Theorem~\ref{weakconv}. 
This probability space is not necessarily the same as the original probability space
on which our problem is posed. Nevertheless, we can get back to the original probability space by using 
another result, known as the Gy\"{o}ngy-Krylov lemma, see Section~\ref{GKlemma}, to show that along a subsequence, the original approximate solutions on the original probability space converge almost surely to a limit with the same law $\mu$ from Theorem ~\ref{weakconv}. 

More precisely, showing convergence of our approximate solutions almost surely to a weak solution on the original probability space, 
consists of two steps. First, we use the Skorohod representation theorem to show that {\emph{there exists}} a probability space,
which we denote by ``tilde'', on which 
a sequence of random variables that are equal to our approximate solutions {\emph{in law}} converges almost surely in ${\cal{X}}$ as $N\to \infty$,
to a weak solution on the ``tilde'' probability space,
where the law of this weak solution is equal to $\mu$, obtained  in Theorem~\ref{weakconv}. 
Thus, in  this step, we prove the existence of a weak solution in a probabilistically {\emph{weak}} sense, see Definition~\ref{weak}.
Then, in step two, we show using  the Gy\"{o}ngy-Krylov lemma, that we can bring 
that weak solution back to the original probability space, implying that we will have constructed a  weak solution
in a probabilistically strong sense, see Definition~\ref{strong}, of the original continuous problem. This will complete the existence proof, which is the main result of this manuscript.

To achieve these goals, we first obtain almost sure convergence along a subsequence of approximate solutions on a ``tilde'' probability space 
using Skorohod's theorem. A statement of the Skorohod representation theorem, which holds for probability measures on complete separable metric spaces, can be found in Proposition 6.2 in \cite{LNT}.


Before we state the result, we introduce the notation ``$=_d$'' to denote 
random variables that are ``equal in distribution'' i.e., the random variables have the same laws as random variables taking values on the same given phase space ${\cal{X}}$. Namely, we will say that a random variable $X$ is equal in distribution (or equal in law) to
the random variable $\tilde{X}$, and denote 
$$
X =_d \tilde{X} \quad  {\rm if} \quad  \mu_{X} = \mu_{\tilde{X}},
$$
where $\mu_{X}$ for example is the probability measure on ${\cal{X}}$ describing the law of the random variable $X$ on ${\cal{X}}$.

Recall again the definition of the laws corresponding to the approximate solutions \eqref{muN}, and 
the definition of the corresponding phase space \eqref{phase}.
\begin{lemma}\label{properties}\label{SkorohodLemma}
Let $\mu$ denote the probability measure obtained as a weak limit of the measures $\mu_N$ from Theorem~\ref{weakconv}.
Then, there exists a probability space $(\tilde{\Omega}, \tilde{\mathcal{F}}, \tilde{\mathbb{P}})$ and $\mathcal{X}$-valued random variables 
on $(\tilde{\Omega}, \tilde{\mathcal{F}}, \tilde{\mathbb{P}})$:
\begin{equation*}
(\tilde{\eta}, \tilde{\overline{\eta}}, \tilde{\eta}^{\Delta t}, \tilde{\boldsymbol{u}}, \tilde{v}, \tilde{\boldsymbol{u}}^{*}, \tilde{v}^{*}, \tilde{\overline{\boldsymbol{u}}}, \tilde{\overline{v}}, \tilde{\boldsymbol{u}}^{\Delta t}, \tilde{v}^{\Delta t}, \tilde{W}),
\ {\rm and} \
(\tilde{\eta}_{N}, \tilde{\overline{\eta}}_{N}, \tilde{\eta}^{\Delta t}_{N}, \tilde{\boldsymbol{u}}_{N}, \tilde{v}_{N}, \tilde{\boldsymbol{u}}_{N}^{*}, \tilde{v}_{N}^{*}, \tilde{\overline{\boldsymbol{u}}}_{N}, \tilde{\overline{v}}_{N}, \tilde{\boldsymbol{u}}^{\Delta t}_{N}, \tilde{v}^{\Delta t}_{N}, \tilde{W}_{N}), \text{ for each $N$},
\end{equation*}
such that
\begin{equation*}
(\tilde{\eta}_{N}, \tilde{\overline{\eta}}_{N}, \tilde{\eta}^{\Delta t}_{N}, \tilde{\boldsymbol{u}}_{N}, \tilde{v}_{N}, \tilde{\boldsymbol{u}}_{N}^{*}, \tilde{v}_{N}^{*}, \tilde{\overline{\boldsymbol{u}}}_{N}, \tilde{\overline{v}}_{N}, \tilde{\boldsymbol{u}}^{\Delta t}_{N}, \tilde{v}^{\Delta t}_{N}, \tilde{W}_{N}) =_{d} (\eta_{N}, \overline{\eta}_{N}, \eta^{\Delta t}_{N}, \boldsymbol{u}_{N}, v_{N}, \boldsymbol{u}_{N}, v_{N}^{*}, \overline{\boldsymbol{u}}_{N}, \overline{v}_{N}, \boldsymbol{u}^{\Delta t}_{N}, v^{\Delta t}_{N}, W),
\end{equation*}
for all $N$, and
\begin{equation}\label{Skorohod}
(\tilde{\eta}_{N}, \tilde{\overline{\eta}}_{N}, \tilde{\eta}^{\Delta t}_{N}, \tilde{\boldsymbol{u}}_{N}, \tilde{v}_{N}, \tilde{\boldsymbol{u}}_{N}^{*}, \tilde{v}_{N}^{*}, \tilde{\overline{\boldsymbol{u}}}_{N}, \tilde{\overline{v}}_{N}, \tilde{\boldsymbol{u}}^{\Delta t}_{N}, \tilde{v}^{\Delta t}_{N}, \tilde{W}_{N}) 
\to 
(\tilde{\eta}, \tilde{\overline{\eta}}, \tilde{\eta}^{\Delta t}, \tilde{\boldsymbol{u}}, \tilde{v}, \tilde{\boldsymbol{u}}^{*}, \tilde{v}^{*}, \tilde{\overline{\boldsymbol{u}}}, \tilde{\overline{v}}, \tilde{\boldsymbol{u}}^{\Delta t}, \tilde{v}^{\Delta t}, \tilde{W}),
\end{equation}
a.s. in $\mathcal{X}$, as $N \to \infty$, where the law of $(\tilde{\eta}, \tilde{\overline{\eta}}, \tilde{\eta}^{\Delta t}, \tilde{\boldsymbol{u}}, \tilde{v}, \tilde{\boldsymbol{u}}^{*}, \tilde{v}^{*}, \tilde{\overline{\boldsymbol{u}}}, \tilde{\overline{v}}, \tilde{\boldsymbol{u}}^{\Delta t}, \tilde{v}^{\Delta t}, \tilde{W})$ is equal to $\mu$.

Furthermore, the following properties hold:
\begin{enumerate}
\item $\tilde{\boldsymbol{u}}_{N} = \tilde{\boldsymbol{u}}_{N}^{*}$, $\tilde{\boldsymbol{u}} = \tilde{\boldsymbol{u}}^{*} = \tilde{\overline{\boldsymbol{u}}} = \tilde{\boldsymbol{u}}^{\Delta t}$ almost surely,
$\tilde{v} = \tilde{v}^{*} = \tilde{\overline{v}} = \tilde{v}^{\Delta t}$ almost surely, and
$\tilde{\eta} = \tilde{\overline{\eta}} = \tilde{\eta}^{\Delta t}$ almost surely.
\item $\tilde{\eta} \in L^{2}(\tilde{\Omega}; W^{1, \infty}(0, T; L^{2}(\Gamma)) \cap L^{\infty}(0, T; H_{0}^{1}(\Gamma)))$,\ 
$\tilde{\boldsymbol{u}} \in L^{2}(\tilde{\Omega}; L^{2}(0, T; H^{1}(\Omega_{f})) \cap L^{\infty}(0, T; L^{2}(\Omega_{f})))$, \ and 
$\tilde{v} \in L^{2}(\tilde{\Omega}; L^{\infty}(0, T; L^{2}(\Gamma)))$.
\item $\tilde{\eta}(0) = \eta_{0}$ almost surely.
\item $\partial_{t}\tilde{\eta} = \tilde{v}$ almost surely. 
\item $(\tilde{\boldsymbol{u}}, \tilde{v}) \in C(0, T; \mathcal{Q}')$ and $(\tilde{\boldsymbol{u}}, \tilde{\eta}) \in \mathcal{W}(0, T)$ almost surely.
\item Define the filtration
\begin{equation}\label{tildefiltration}
\tilde{\mathcal{F}}_{t} = \sigma(\tilde{\eta}(s), \tilde{\boldsymbol{u}}(s), \tilde{v}(s) : 0 \le s \le t).
\end{equation}
Then $\tilde{W}$ is a Brownian motion with respect to $\tilde{\mathcal{F}}_{t}$. 
\item $(\tilde{\boldsymbol{u}}, \tilde{\eta}, \tilde{v})$ is a predictable process with respect to the filtration $\{\tilde{\mathcal{F}}_{t}\}_{0 \le t \le T}$. 
\end{enumerate}
\end{lemma}

\if 1 = 0
Before proving this lemma, we will introduce the following probabilistic lemma.
\begin{lemma}
Let $\{X_{n}\}_{n = 1}^{\infty}$ be a sequence of random variables on a common probability space $(\Omega, \mathcal{F}, \mathbb{P})$ taking values in a Banach space $B$, satisfying a uniform expectation bound
\begin{equation*}
\mathbb{E}(||X_{n}||_{B}^{2}) \le C.
\end{equation*}
Then, there exists a subsequence along which $X_{n}$ is bounded, almost surely, where the subsequence may depend on the outcome $\omega \in \Omega$. 
\end{lemma}

\begin{proof}[Proof of probabilistic lemma]
Let $A_{m}$ denote the event
\begin{equation*}
A_{m} := \{||X_{n}||_{B} \le m \text{ for finitely many } n\}.
\end{equation*}
Note that these events $A_{M}$ are decreasing as $M$ increases. Then define the event
\begin{equation*}
A = \bigcap_{m \in \mathbb{Z}^{+}} A_{m},
\end{equation*}
and notice that for every $\omega \in A^{c}$, there exists some $M \in \mathbb{Z}^{+}$ depending on $\omega$ such that $\omega \in A^{c}_{M}$, and hence, there exists a subsequence (depending on $\omega$) for which $||X_{n}||_{B}$ along this subsequence is bounded by $M$. So it suffices to show that $\mathbb{P}(A) = 0$. 

Define the events $A_{m, n}$ by
\begin{equation*}
A_{m, n} := \{||X_{n}||_{B} > m\}.
\end{equation*}
Then, following the notation in Section 2.3 of Durrett (TODO: Cite), we have that
\begin{equation*}
A_{m} = \liminf_{n \to \infty} A_{m, n}.
\end{equation*}
Then, by Exercise 2.3.1 in Durrett (TODO: Cite),
\begin{equation*}
\mathbb{P}(A_{m}) := \mathbb{P}\left(\liminf_{n \to \infty} A_{m, n}\right) \le \liminf_{n \to \infty} \mathbb{P}(A_{m, n}) \le \frac{C}{m^{2}}.
\end{equation*}
Since $C$ is uniform in $m$, as it is the given uniform bound in expectation, we conclude that
\begin{equation*}
\mathbb{P}(A) = \lim_{m \to \infty} \mathbb{P}(A_{m}) = 0.
\end{equation*}
\end{proof}

\fi 

\begin{proof}
The existence of the probability space $(\tilde{\Omega}, \tilde{\mathcal{F}}, \tilde{\mathbb{P}})$ and the given random variables follows from the previous result on weak convergence in Theorem \ref{weakconv} and the Skorohod representation theorem. So it suffices to prove the given properties. 

\vspace{0.1in}

\noindent \textbf{Property 1:} Because $(\tilde{\boldsymbol{u}}_{N}, \tilde{\boldsymbol{u}}_{N}^{*}) =_{d} (\boldsymbol{u}_{N}, \boldsymbol{u}_{N})$, we have that
$\tilde{\boldsymbol{u}}_{N} - \tilde{\boldsymbol{u}}_{N}^{*} =_{d} 0$ as random variables taking values in $L^{2}(0, T; L^{2}(\Omega_{f}))$,
so $\tilde{\boldsymbol{u}}_{N} = \tilde{\boldsymbol{u}}_{N}^{*}$ a.s. for all $N$. Hence, by taking the limit as $N \to \infty$, we obtain $\tilde{\boldsymbol{u}} = \tilde{\boldsymbol{u}}^{*}$ a.s., since $\tilde{\boldsymbol{u}}_{N} \to \tilde{\boldsymbol{u}}$ and $\tilde{\boldsymbol{u}}_{N}^{*} \to \tilde{\boldsymbol{u}}^{*}$ in $L^{2}(0, T; L^{2}(\Omega_{f}))$ a.s.

Because $\boldsymbol{u}_{N}$ and $\overline{\boldsymbol{u}}_{N}$ actually have different laws from each other, we must use a different argument to conclude that $\tilde{\boldsymbol{u}} = \tilde{\overline{\boldsymbol{u}}}$ a.s. However, we recall the following fact \eqref{ulindiff} from the proof of Lemma \ref{almostsuresub},
\begin{equation*}
\mathbb{E}\left(||\boldsymbol{u}_{N} - \overline{\boldsymbol{u}}_{N}||^{2}_{L^{2}(0, T; L^{2}(\Omega_{f}))}\right) \to 0, \qquad \text{ as } N \to \infty.
\end{equation*}
Hence, by the equivalence of laws, 
\begin{equation*}
\tilde{\mathbb{E}}\left(||\tilde{\boldsymbol{u}}_{N} - \tilde{\overline{\boldsymbol{u}}}_{N}||^{2}_{L^{2}(0, T; L^{2}(\Omega_{f}))}\right) \to 0, \qquad \text{ as } N \to \infty.
\end{equation*}
Therefore, along a further subsequence, $||\tilde{\boldsymbol{u}}_{N} - \tilde{\overline{\boldsymbol{u}}}_{N}||^{2}_{L^{2}(0, T; L^{2}(\Omega_{f}))} \to 0$ almost surely, by a standard Borel Cantelli lemma argument. Since $\tilde{\boldsymbol{u}}_{N} \to \tilde{\boldsymbol{u}}$ and $\tilde{\overline{\boldsymbol{u}}}_{N} \to \tilde{\overline{\boldsymbol{u}}}$ in $L^{2}(0, T; L^{2}(\Omega_{f}))$, we conclude that $\tilde{\boldsymbol{u}} = \tilde{\overline{\boldsymbol{u}}}$ a.s.

The remaining statements follow from the same argument as above. In particular, by using the estimates \eqref{vstardiff}--\eqref{etashiftdiff} 
from the proof of Lemma \ref{almostsuresub}, the equivalence of laws, and the almost sure convergence of the ``tilde" random variables in \eqref{Skorohod}, we obtain the desired result.

\vspace{0.1in}

\noindent \textbf{Property 2:} These properties will all be handled similarly. By the uniform energy estimates in Lemma \ref{uniformbound} and Lemma \ref{uniformetabound}, we have that
\begin{equation*}
\mathbb{E}\left(||\overline{\eta}_{N}||_{W^{1, \infty}(0, T; L^{2}(\Gamma))}^{2}\right) \le C, \ \mathbb{E}\left(||\overline{\eta}_{N}||_{L^{ \infty}(0, T; H_{0}^{1}(\Gamma))}^{2}\right) \le C,
\end{equation*}
\begin{equation*}
\mathbb{E}\left(||\boldsymbol{u}^{\Delta t}_{N}||_{L^{2}(0, T; H^{1}(\Omega_{f}))}^{2}\right) \le C, \ \mathbb{E}\left(||\boldsymbol{u}_{N}||_{L^{\infty}(0, T; L^{2}(\Omega_{f}))}^{2}\right) \le C,
\
\mathbb{E}\left(||v_{N}||_{L^{\infty}(0, T; L^{2}(\Gamma))}^{2}\right) \le C,
\end{equation*}
for a constant $C$ that is independent of $N$. Therefore, by the equivalence of laws, we have that these uniform estimates hold for the random variables on the new probability space, so that 
\begin{equation*}
\tilde{\mathbb{E}}\left(||\tilde{\overline{\eta}}_{N}||_{W^{1, \infty}(0, T; L^{2}(\Gamma))}^{2}\right) \le C, \ \tilde{\mathbb{E}}\left(||\tilde{\overline{\eta}}_{N}||_{L^{ \infty}(0, T; H_{0}^{1}(\Gamma))}^{2}\right) \le C,
\end{equation*}
\begin{equation*}
\tilde{\mathbb{E}}\left(||\tilde{\boldsymbol{u}}^{\Delta t}_{N}||_{L^{2}(0, T; H^{1}(\Omega_{f}))}^{2}\right) \le C, \ \tilde{\mathbb{E}}\left(||\tilde{\boldsymbol{u}}_{N}||_{L^{\infty}(0, T; L^{2}(\Omega_{f}))}^{2}\right) \le C,
\
\tilde{\mathbb{E}}\left(||\tilde{v}_{N}||_{L^{\infty}(0, T; L^{2}(\Gamma))}^{2}\right) \le C,
\end{equation*}
for a constant $C$ that is independent of $N$. Therefore, by this uniform boundedness, we conclude for example that $\tilde{\overline{\eta}}_{N}$ converges weakly star in $L^{2}(\tilde{\Omega}; W^{1, \infty}(0, T; L^{2}(\Gamma)))$ and weakly star in $L^{2}(\tilde{\Omega}; L^{\infty}(0, T; H^{1}_{0}(\Gamma)))$. Since we already have that $\tilde{\overline{\eta}}_{N}$ converges to $\tilde{\overline{\eta}}$ almost surely in $L^{2}(0, T; L^{2}(\Gamma))$ and $\tilde{\overline{\eta}} = \tilde{\eta}$ almost surely by Property 1, by the uniqueness of this limit, we conclude that
$
\tilde{\overline{\eta}}_{N} \rightharpoonup \tilde{\eta}, \ \text{ weakly star in } L^{2}(\tilde{\Omega}; W^{1, \infty}(0, T; L^{2}(\Gamma)))$  and  
$L^{2}(\tilde{\Omega}; L^{\infty}(0, T; H^{1}_{0}(\Gamma)))$.

Similarly,
$\tilde{\boldsymbol{u}}^{\Delta t}_{N} \rightharpoonup \tilde{\boldsymbol{u}}^{\Delta t}$, weakly $L^{2}(\tilde{\Omega}; L^{2}(0, T; H^{1}(\Omega_{f})))$, $\tilde{\boldsymbol{u}}_{N} \rightharpoonup \tilde{\boldsymbol{u}}$ weakly star in $L^{2}(\tilde{\Omega}; L^{\infty}(0, T; L^{2}(\Omega_{f})))$,
and 
$\tilde{v}_{N} \rightharpoonup \tilde{v}$ weakly star in $L^{2}(\tilde{\Omega}; L^{\infty}(0, T; L^{2}(\Gamma)))$.
This establishes Property 2.
\vspace{0.1in}

\noindent \textbf{Property 3:} Since $\tilde{\eta} = \tilde{\overline{\eta}}$ almost surely, it suffices to show that $\tilde{\overline{\eta}}(0) = \eta_{0}$ almost surely. To do this, we use a method similar to the method in the proof of Lemma \ref{continuitylemma}. We define
\begin{equation}\label{DM}
D_{M} = \{\eta \in L^{2}(0, T; L^{2}(\Gamma)) : ||\eta(\cdot) - \eta_{0}||_{L^{2}(0, h, L^{2}(\Gamma))} \le Mh^{3/2}, \text{ for all $0 < h \le T$}\}.
\end{equation}
Because of the uniform bound
$
\mathbb{E}\left(||\overline{\eta}_{N}||^{2}_{W^{1, \infty}(0, T; L^{2}(\Gamma))}\right) \le C \ \text{ for all $N$},
$
from Lemma \ref{uniformetabound}, we have that 
\begin{equation*}
\mathbb{P}(\overline{\eta}_{N} \in D_{M}) \ge 1 - \frac{C}{M^{2}} \qquad \text{ for all $M$ and $N$},
\end{equation*}
by using Chebychev's inequality. This is because if $||\overline{\eta}_{N}||_{W^{1, \infty}(0, T; L^{2}(\Gamma))} \le M$, then from the fact that $\overline{\eta}_{N}(0) = \eta_{0}$ for all $\omega \in \Omega$ and $N$, we have that 
\begin{equation*}
 ||\overline{\eta}(\cdot) - \eta_{0}||_{L^{2}(0, h, L^{2}(\Gamma))} = \left(\int_{0}^{h} ||\overline{\eta}(s) - \eta_{0}||_{L^{2}(\Gamma)}^{2} ds\right)^{1/2} \le \left(\int_{0}^{h} (Ms)^{2} ds\right)^{1/2} \le Mh^{3/2}.
\end{equation*}
Then, by equivalence of laws,
\begin{equation*}
\tilde{\mathbb{P}}(\tilde{\overline{\eta}}_{N} \in D_{M}) \ge 1 - \frac{C}{M^{2}} \qquad \text{ for all $M$ and $N$}.
\end{equation*}
Because $D_{M}$ is a closed set in $L^{2}(0, T; L^{2}(\Gamma))$ and $\tilde{\overline{\eta}}_{N} \to \tilde{\overline{\eta}}$ in $L^{2}(0, T; L^{2}(\Gamma))$ a.s., we conclude that
\begin{equation*}
\tilde{\mathbb{P}}(\tilde{\overline{\eta}} \in D_{M}) \ge \limsup_{N \to \infty} \tilde{\mathbb{P}}(\tilde{\overline{\eta}}_{N} \in D_{M}) \ge 1 - \frac{C}{M^{2}} \ \text{ for all $M$}, \ {\text{which implies}} \ 
\tilde{\mathbb{P}}\left(\tilde{\overline{\eta}} \in \bigcup_{M = 1}^{\infty} D_{M}\right) = 1.
\end{equation*}
Because $\tilde{\overline{\eta}}$ is almost surely continuous on $[0, T]$ taking values in $L^{2}(\Gamma)$ by Property 2, we obtain $\tilde{\overline{\eta}}(0) = \eta_{0}$ almost surely. This is because 
if a continuous function $\eta$ on $[0, T]$ taking values in $L^{2}(\Gamma)$ has $\eta(0) \ne \eta_{0}$, then
\begin{equation*}
||\eta(\cdot) - \eta_{0}||_{L^{2}(0, h; L^{2}(\Gamma))} \ge \frac{d}{2}h^{1/2}, 
\end{equation*}
for all $h$ sufficiently small where $d = ||\eta(0) - \eta_{0}||_{L^{2}(\Gamma)}$, and hence $\eta$ cannot belong to $\bigcup_{M = 1}^{\infty} D_{M}$.

\vspace{0.1in}

\noindent \textbf{Property 4:} To prove this property, we recall from the second equation in the semidiscrete formulation \eqref{semi1} that
\begin{equation*}
\int_{\Gamma} \frac{\eta^{n + 1}_{N} - \eta^{n}_{N}}{\Delta t} \phi dz = \int_{\Gamma} v^{n + \frac{1}{3}}_{N} \phi dz,
\end{equation*}
almost surely for all $\phi \in L^{2}(\Gamma)$. Integrating in time, we obtain for all $N$ that
\begin{equation*}
\int_{0}^{T} \int_{\Gamma} \partial_{t}\overline{\eta}_{N} \phi dz dt = \int_{0}^{T} \int_{\Gamma} v^{*}_{N} \phi dz dt, \qquad \text{ for all } \phi \in C^{1}([0, T); L^{2}(\Gamma)),
\end{equation*}
almost surely. Because each $\overline{\eta}_{N}$ is almost surely a piecewise linear continuous function satisfying $\overline{\eta}(0) = \eta_{0}$, we obtain by integration by parts that almost surely, for all $\phi \in C^{1}([0, T); L^{2}(\Gamma))$, 
\begin{equation*}
-\eta_{0} \cdot \phi(0) - \int_{0}^{T} \int_{\Gamma} \overline{\eta}_{N} \partial_{t}\phi dz dt = \int_{0}^{T} \int_{\Gamma} v^{*}_{N} \phi dz dt,
\end{equation*}
and hence, by equivalence of laws,
\begin{equation*}
-\eta_{0} \cdot \phi(0) - \int_{0}^{T} \int_{\Gamma} \tilde{\overline{\eta}}_{N} \partial_{t}\phi dz dt = \int_{0}^{T} \int_{\Gamma} \tilde{v}^{*}_{N} \phi dz dt.
\end{equation*}
Passing to the limit, we obtain 
\begin{equation*}
-\eta_{0} \cdot \phi(0) - \int_{0}^{T} \int_{\Gamma} \tilde{\eta} \partial_{t}\phi dz dt = \int_{0}^{T} \int_{\Gamma} \tilde{v}\phi dz dt,
\end{equation*}
for all $\phi \in C^{1}([0, T); L^{2}(\Gamma))$, almost surely. This implies that $\partial_{t} \tilde{\eta} = \tilde{v}$ holds almost surely for the limiting solution, since we showed in Property 3 that $\tilde{\eta}(0) = \eta_{0}$ almost surely.

\vspace{0.1in}

\noindent \textbf{Property 5:} The fact that $(\tilde{\boldsymbol{u}}, \tilde{v}) \in C(0, T; \mathcal{Q}')$ almost surely follows from Lemma \ref{continuitylemma}, since the limiting random variables with the tildes have their law given by the probability measure $\mu$. So it remains to show that $(\tilde{\boldsymbol{u}}, \tilde{v}) \in \mathcal{W}(0, T)$, where $\mathcal{W}(0, T)$ is defined in \eqref{W}.


To establish this result, first notice that we already know from Property 2 that $\tilde{\boldsymbol{u}} \in L^{2}(\tilde{\Omega}; L^{\infty}(0, T; L^{2}(\Omega_{f})))$ and $\tilde{\boldsymbol{u}} \in L^{2}(\tilde{\Omega}; L^{2}(0, T; H^{1}(\Omega_{f})))$, and Property 2 already gives the desired result for the structure.
Thus, it remains to show  that $\tilde{\boldsymbol{u}} \in L^{2}(0, T; \mathcal{V}_{F})$ almost surely, where $\mathcal{V}_{F}$ is defined in \eqref{VF}, and that the kinematic coupling condition holds. By Property 4, we must show in particular that $\tilde{\boldsymbol{u}} = \tilde{v} \boldsymbol{e}_{r}$ a.s. on $\Gamma$. 

To do this, define the deterministic function space
\begin{equation*}
\mathcal{H} = \{(\boldsymbol{u}, v) \in L^{2}(0, T; \mathcal{V}_{F}) \times L^{2}(0, T; L^{2}(\Gamma)) : \boldsymbol{u} = v \boldsymbol{e_{r}} \text{ for almost every $t \in [0, T]$}\}.
\end{equation*}
One can check that the linear subspace $\mathcal{H}\subset L^{2}(0, T; H^{1}(\Omega_{f})) \times L^{2}(0, T; L^{2}(\Gamma))$ is closed in the Hilbert space $L^{2}(0, T; H^{1}(\Omega_{f})) \times L^{2}(0, T; L^{2}(\Gamma))$, and hence $\mathcal{H}$ is a Hilbert space with the inner product of $L^{2}(0, T; H^{1}(\Omega_{f})) \times L^{2}(0, T; L^{2}(\Gamma))$. By equivalence of laws and the uniform boundedness in Lemma \ref{uniformbound}, $(\tilde{\boldsymbol{u}}_{N}, \tilde{v}_{N})$ is uniformly bounded in $L^{2}(\tilde{\Omega}; \mathcal{H})$, and hence converges weakly to $(\tilde{\boldsymbol{u}}, \tilde{v}) \in L^{2}(\tilde{\Omega}; \mathcal{H})$ by uniqueness of the limit, since we already have that $(\tilde{\boldsymbol{u}}_{N}, \tilde{v}_{N})$ converges almost surely to $(\tilde{\boldsymbol{u}}, \tilde{v})$ in $L^{2}(0, T; L^{2}(\Omega_{f})) \times L^{2}(0, T; L^{2}(\Gamma))$. This gives the desired result. 

\if 1 = 0

Consider the Banach space
\begin{equation*}
\mathcal{H} = L^{2}(0, T; H^{1}(\Omega_{f})) \times L^{2}(0, T; L^{2}(\Gamma)).
\end{equation*}
Note that this is a deterministic Banach space making no reference to probability. It is in fact a Hilbert space, and hence, it is reflexive. Next, define for each $R > 0$, the set
\begin{multline*}
\mathcal{W}_{R} = \{(\boldsymbol{u}, v) \in L^{2}(0, T; \mathcal{V}_{F}) \times L^{2}(0, T; L^{2}(\Gamma)), \text{ such that } \\
||\boldsymbol{u}||_{L^{2}(0, T; H^{1}(\Omega_{f}))} \le R, \ ||v||_{L^{2}(0, T; L^{2}(\Gamma))} \le R, \ \boldsymbol{u} = v\boldsymbol{e}_{r} \text{ on } \Gamma \text{ for a.e. $t \in [0, T]$}\}.
\end{multline*}

For each arbitrary fixed $R > 0$, $\mathcal{W}_{R}$ is a closed, convex subset of $\mathcal{H}$. This follows from the fact that $\mathcal{V}_{F}$ is closed in $H^{1}(\Omega_{f})$ and from the continuity of the trace operator sending $H^{1}(\Omega_{f})$ to $L^{2}(\Gamma)$. In addition, $\mathcal{W}_{R}$ is clearly a bounded set in $\mathcal{H}$ by definition. Hence, by Proposition 4.8 of Appendix A of Taylor PDE I (TODO: Cite), we conclude that $\mathcal{W}_{R}$ is weakly compact in $\mathcal{H}$ for each $R > 0$.

From equivalence of laws, we have the uniform boundedness results
\begin{equation*}
\mathbb{E}\left(||\tilde{\boldsymbol{u}}_{N}||_{L^{2}(0, T; H^{1}(\Omega_{f}))}^{2}\right) = \mathbb{E}\left(||\boldsymbol{u}_{N}||_{L^{2}(0, T; H^{1}(\Omega_{f}))}^{2}\right) \le C,
\end{equation*}
\begin{equation*}
\mathbb{E}\left(||\tilde{v}_{N}||_{L^{2}(0, T; L^{2}(\Gamma))}^{2}\right) \le T \cdot \mathbb{E}\left(||v_{N}||_{L^{\infty}(0, T; L^{2}(\Gamma))}^{2}\right) \le CT.
\end{equation*}
Therefore, by applying Lemma 5.8 twice sequentially, we obtain almost surely, the existence of a subsequence for which $\tilde{\boldsymbol{u}}_{N}$ and $\tilde{v}_{N}$ are bounded. We will continue to denote the subsequence for which this is true (which depends on $\omega$) by $N$, and in particular, we will not explicitly notate this $\omega$ dependence of the particular subsequence. 

So almost surely (for almost every $\omega \in \Omega$), $(\tilde{\boldsymbol{u}}_{N}(\omega), \tilde{v}_{N}(\omega)) \in \mathcal{W}_{R}$ for some $R$ depending on $\omega$, where the subsequence depends on $\omega$. Because $\mathcal{W}_{R}$ is weakly compact in $\mathcal{H}$ for each $R > 0$, this implies that the almost sure limit under the norm of $L^{2}(0, T; L^{2}(\Omega_{f})) \times L^{2}(0, T; L^{2}(\Gamma))$, which must coincide with the weak limit in $\mathcal{H}$, satisfies
\begin{equation*}
(\tilde{\boldsymbol{u}}(\omega), \tilde{v}(\omega)) \in \mathcal{W}_{R}.
\end{equation*}
Therefore, since this holds almost surely and every function in $\mathcal{W}_{R}$ has $\boldsymbol{u} = v\boldsymbol{e}_{r}$ on $\Gamma$ for almost every $t \in [0, T]$, we have that $\tilde{\boldsymbol{u}} = \tilde{v}$ for almost every $t \in [0, T]$, almost surely.  Furthermore, by the definition of $\mathcal{W}_{R}$, we have that $u \in L^{2}(0, T; \mathcal{V}_{F})$ almost surely. This establishes Property 8. 

\fi

\vspace{0.1in}

\noindent \textbf{Property 6:} First, we sketch the idea. By construction, we have that on the original probability space, $W(t) - W(s)$ is independent of $\sigma(\boldsymbol{u}_{N}(\tau), v_{N}(\tau), \eta_{N}(\tau), \text{ for } 0 \le \tau \le s)$, where we recall that these processes $(\boldsymbol{u}_{N}(\tau), v_{N}(\tau), \eta_{N}(\tau))$ are piecewise constant on intervals of length $\Delta t = T/N$. This is because for a given time $\tau \in [0, T]$, $(\boldsymbol{u}_{N}(\tau), v_{N}(\tau), \eta_{N}(\tau))$ depends only on the values of the Brownian motion at time $\lfloor \frac{\tau}{\Delta t} \rfloor \Delta t$ or earlier, from which the claim follows by the independent increments property of Brownian motion. The idea will be to transfer this independence property over to the new random variables $(\tilde{\boldsymbol{u}}_{N}, \tilde{v}_{N}, \tilde{\eta}_{N})$ on the new probability space $(\tilde{\Omega}, \tilde{F}, \tilde{\mathbb{P}})$ and then take a limit as $N \to \infty$ to get the desired independence in the limit. 

Note that the definition of $\tilde{\mathcal{F}}_{t}$ as
\begin{equation*}
\tilde{\mathcal{F}}_{t} = \sigma(\tilde{\boldsymbol{u}}(s), \tilde{v}(s), \tilde{\eta}(s), \text{ for } 0 \le s \le t)
\end{equation*}
makes sense, since by the above properties, $\tilde{\eta}$ and $(\tilde{\boldsymbol{u}}, \tilde{v})$ are continuous on $[0, T]$ in time, taking values in $L^{2}(\Gamma)$ and $\mathcal{Q}'$ respectively. So it makes sense to refer to values pointwise at specific times, for example as in $\tilde{\boldsymbol{u}}(\tau)$ for a given $\tau \in [0, T]$. However, it is not clear yet, for example, what $\tilde{\boldsymbol{u}}_{N}(\tau)$ would be, since a priori, we only know that $\tilde{\boldsymbol{u}}_{N} \in L^{2}(0, T; L^{2}(\Omega_{f}))$, and hence, each path of $\tilde{\boldsymbol{u}}_{N}$ is only defined up to a version for $t \in [0, T]$. 

To handle this, define the set $K_{N}$ of all functions in $L^{2}(0, T; L^{2}(\Omega_{f}))$ that have a version that is piecewise constant on the intervals of the form $[0, \Delta t]$ and $(n\Delta t, (n + 1)\Delta t]$ for $1 \le n \le N - 1$, where $\Delta t = T/N$. Note that $K_{N}$ is a closed subset of $L^{2}(0, T; L^{2}(\Omega_{f}))$, so by equivalence of laws, 
\begin{equation*}
\tilde{\mathbb{P}}(\tilde{\boldsymbol{u}}_{N} \in K_{N}) = \mathbb{P}(\boldsymbol{u}_{N} \in K_{N}) = 1.
\end{equation*}
Therefore, $\tilde{\boldsymbol{u}}_{N}$ is almost surely piecewise constant on $[0, \Delta t]$ and $(n\Delta t, (n + 1)\Delta t]$ for $1 \le n \le N - 1$. The same argument shows that $\tilde{v}_{N}$ and $\tilde{\eta}_{N}$ also almost surely have versions that are piecewise constant on these same intervals, since $v_{N}$ and $\eta_{N}$ on the original probability space almost surely have this property too. 

Therefore, for each $N$, up to taking a version of $\tilde{\boldsymbol{u}}_{N}$, $\tilde{v}_{N}$, and $\tilde{\eta}_{N}$, we can define random variables $\tilde{\boldsymbol{u}}^{n}_{N}$, $\tilde{v}^{n}_{N}$, and $\tilde{\eta}^{n}_{N}$ for $0 \le n \le N - 1$, satisfying 
\begin{align*}
\tilde{\boldsymbol{u}}_{N}(t, \omega) &= \tilde{\boldsymbol{u}}^{0}_{N}(\omega), \ \text{ if } 0 \le t \le \Delta t \ {\rm and} \ 
\tilde{\boldsymbol{u}}_{N}(t, \omega) = \tilde{\boldsymbol{u}}^{n}_{N}(\omega), \ \text{ if } n\Delta t < t \le (n + 1)\Delta t,\\
\tilde{v}_{N}(t, \omega) &= \tilde{v}^{0}_{N}(\omega), \ \text{ if } 0 \le t \le \Delta t \ {\rm and} \  
\tilde{v}_{N}(t, \omega) = \tilde{v}^{n}_{N}(\omega), \ \text{ if } n\Delta t < t \le (n + 1)\Delta t,\\
\tilde{\eta}_{N}(t, \omega) &= \tilde{\eta}^{0}_{N}(\omega), \ \text{ if } 0 \le t \le \Delta t \ {\rm and} \ 
\tilde{\eta}_{N}(t, \omega) = \tilde{\eta}^{n}_{N}(\omega), \ \text{ if } n\Delta t < t \le (n + 1)\Delta t.
\end{align*}
Furthermore, by the equivalence of laws, the joint distribution of $\tilde{\boldsymbol{u}}^{n}_{N}, \tilde{v}^{n}_{N}, \tilde{\eta}^{n}_{N}$ for $0 \le n \le N - 1$ is the same as that of $\boldsymbol{u}^{n}_{N}, v^{n}_{N}, \eta^{n}_{N}$ for $0 \le n \le N - 1$. Therefore, we can now make sense of $\tilde{\boldsymbol{u}}_{N}(\tau)$ for example for any $\tau \in [0, T]$, by considering the piecewise constant versions of these stochastic processes as given above. When we refer to $\tilde{\boldsymbol{u}}_{N}$, $\tilde{v}_{N}$, and $\tilde{\eta}_{N}$, we will refer to the piecewise constant versions defined above. 

We now show the desired independence. We consider $\tau_{0} \in [0, s]$ and $0 \le s \le t$, and show that $\tilde{\boldsymbol{u}}(\tau_{0})$ and $\tilde{W}(t) - \tilde{W}(s)$ are independent. The same argument will work for $v(\tau_{0})$ and $\eta(\tau_{0})$, so it suffices to show the independence of $\tilde{W}(t) - \tilde{W}(s)$ and $\tilde{\boldsymbol{u}}(\tau_{0})$ for arbitrary $\tau_{0} \in [0, s]$ and $0 \le s \le t$. 

Recall that
$\tilde{\boldsymbol{u}}_{N} \to \tilde{\boldsymbol{u}}$  almost surely in $L^{2}(0, T; L^{2}(\Omega_{f}))$. Define the set
\begin{equation*}
E_{N, n} = \{(t, \omega) \in [0, T] \times \tilde{\Omega} : ||\tilde{\boldsymbol{u}}(t, \omega, \cdot) - \tilde{\boldsymbol{u}}_{N}(t, \omega, \cdot)||_{L^{2}(\Omega_{f})} \ge 2^{-n}\}.
\end{equation*}
For each positive integer $n$, we can choose $N := N(n)$ sufficiently large such that $N(n) > N(n - 1)$ for $n \ge 2$, and 
\begin{equation}\label{prop10borelcantelli}
(dt \times \tilde{\mathbb{P}})(E_{N(n), n}) \le 2^{-n}.
\end{equation}
To see this, one selects $N(n)$ sufficiently large so that 
\begin{equation*}
\tilde{\mathbb{P}}\left(||\tilde{u} - \tilde{u}_{N(n)}||_{L^{2}(0, T; L^{2}(\Omega_{f}))} \le 2^{-2n}\right) \ge 1 - 2^{-2n},
\end{equation*}
and then apply Chebychev's inequality in time. Then, by applying the Borel Cantelli lemma to \eqref{prop10borelcantelli}, we obtain that 
\begin{equation}\label{almostsureS}
\tilde{\boldsymbol{u}}_{N}(t, \omega, \cdot) \to \tilde{\boldsymbol{u}}(t, \omega, \cdot) \qquad \text{ in } L^{2}(\Omega_{f}),
\end{equation}
for all $(t, \omega) \in S \subset [0, T] \times \tilde{\Omega}$ for a set $S$ satisfying $(dt \times \tilde{\mathbb{P}})(S) = T$, where we continue to denote the new subsequence $N(n)$ by $N$. Thus, $([0, T] \times \tilde{\Omega}) - S$ has measure zero with respect to the measure $(dt \times \tilde{\mathbb{P}})$.

Let $S_{0} \subset [0, T]$ be the set of all $t \in [0, T]$ such that $\tilde{\mathbb{P}}((t, \omega) \in S) = 1$. By Fubini's theorem, $S_{0}$ is a measurable subset of $[0, T]$ for which $[0, T] - S_{0}$ has measure zero. Note that for each $t \in S_{0}$, $\tilde{\boldsymbol{u}}_{N}(t, \cdot) \to \tilde{\boldsymbol{u}}(t, \cdot)$ almost surely as random variables taking values in $L^{2}(\Omega_{f})$. 

So if $\tau_{0} \in S_{0}$, we deduce the independence of $\tilde{\boldsymbol{u}}(\tau_{0})$ and $\tilde{W}(t) - \tilde{W}(s)$ as follows. By the fact that $\boldsymbol{u}_{N}(\tau_{0})$ and $W(t) - W(s)$ are independent, we have by equivalence of laws that
\begin{equation*}
\tilde{\boldsymbol{u}}_{N}(\tau_{0}) \text{ and } \tilde{W}_{N}(t) - \tilde{W}_{N}(s) \text{ are independent.}
\end{equation*}
Here, $N$ denotes the subsequence $N(n)$ we used to define $S$ and $S_{0}$. However, since $\tau_{0} \in S_{0}$, we have that $\tilde{\boldsymbol{u}}(\tau_{0})$ is the almost sure limit of $\tilde{\boldsymbol{u}}_{N}(\tau_{0})$, and furthermore, $\tilde{W}(t) - \tilde{W}(s)$ is the almost sure limit of $\tilde{W}_{N}(t) - \tilde{W}_{N}(s)$. So since the almost sure limits of independent random variables are independent, this gives the desired result.

If $\tau_{0} \notin S_{0}$, since $[0, T] - S_{0}$ has measure zero in $[0, T]$, there exists a sequence $\tau_{i} \in S_{0}$ that converges to $\tau_{0}$ as $i \to \infty$, where $\tau_{i} \in [0, s]$. Then, since $\tilde{\boldsymbol{u}}(\tau_{i})$ and $\tilde{W}(t) - \tilde{W}(s)$ are independent for each $i$ and since $\tilde{\boldsymbol{u}}(\tau_{i}) \to \tilde{\boldsymbol{u}}(\tau_{0})$ almost surely by continuity, the result follows. (For the case of $\tau_{0} = 0$, we recall from Lemma \ref{continuitylemma}, that $(\tilde{\boldsymbol{u}}(0), \tilde{v}(0)) = (\boldsymbol{u}_{0}, v_{0})$ almost surely.) 

We use the equivalence of laws to verify the remaining properties of Brownian motion. In particular, we just need to show that $\tilde{W}(t) - \tilde{W}(s)$ is distributed as $N(0, t - s)$. By the equivalence of laws and the fact that $W$ is originally a Brownian motion, $\tilde{W}_{N}(t) - \tilde{W}_{N}(s) =_{d} W(t) - W(s)$, so that $\tilde{W}_{N}(t) - \tilde{W}_{N}(s)$ is distributed as $N(0, t - s)$. Since $\tilde{W}_{N} \to \tilde{W}$ a.s. in $C(0, T; \mathbb{R})$, we obtain that $\tilde{W}_{N}(t) - \tilde{W}_{N}(s) \to \tilde{W}(t) - \tilde{W}(s)$ almost surely, so that $\tilde{W}(t) - \tilde{W}(s)$ is the almost sure limit of random variables distributed as $N(0, t - s)$. Thus, we conclude that $\tilde{W}(t) - \tilde{W}(s)$ must also be distributed as $N(0, t - s)$, which concludes the proof of Property 6. 

\if 1 = 0
Since
\begin{equation*}
\tilde{\boldsymbol{u}}_{N} \to \tilde{\boldsymbol{u}}, \qquad \text{ almost surely in } L^{2}(0, T; L^{2}(\Omega_{f})),
\end{equation*}
we have that $\tilde{\boldsymbol{u}}_{N}(\tau) \to \tilde{\boldsymbol{u}}(\tau)$ in $L^{2}(\Omega; L^{2}(\Omega_{f}))$ for almost every $\tau \in [0, T]$, where we will denote the set of $\tau \in [0, T]$ for which this holds true $S$. 

So if $\tau_{0} \in S$, we deduce the independence of $\tilde{\boldsymbol{u}}(\tau_{0})$ and $\tilde{W}(t) - \tilde{W}(s)$ as follows. By the fact that $\boldsymbol{u}_{N}(\tau_{0})$ and $W(t) - W(s)$ are independent, we have by equivalence of laws that
\begin{equation*}
\tilde{\boldsymbol{u}}_{N}(\tau_{0}) \text{ and } \tilde{W}_{N}(t) - \tilde{W}_{N}(s) \text{ are independent.}
\end{equation*}
However, since $\tau_{0} \in S$, we have that
\begin{equation*}
\tilde{\boldsymbol{u}}_{N}(\tau_{0}) \to \tilde{\boldsymbol{u}}(\tau_{0}) \text{ in } L^{2}(\Omega; L^{2}(\Omega_{f})),
\end{equation*}
so by taking a subsequence, $\tilde{\boldsymbol{u}}(\tau_{0})$ is the almost sure limit of $\tilde{\boldsymbol{u}}_{N}(\tau_{0})$, and furthermore, $\tilde{W}(t) - \tilde{W}(s)$ is the almost sure limit of $\tilde{W}_{N}(t) - \tilde{W}_{N}(s)$. So since the almost sure limits of independent random variables are independent, this gives the desired result.

If $\tau_{0} \notin S$, then since $S$ has measure zero in $[0, T]$, there exists a sequence $\tau_{i}$ that converges almost surely to $\tau_{0}$ as $i \to \infty$, where $\tau_{i} \in [0, s]$. Then, since $\tilde{\boldsymbol{u}}(\tau_{i})$ and $\tilde{W}(t) - \tilde{W}(s)$ are independent for each $i$ and since $\tilde{\boldsymbol{u}}(\tau_{i}) \to \tilde{\boldsymbol{u}}(\tau_{0})$ almost surely by continuity, the result follows. (For the case of $\tau_{0} = 0$, we recall from Lemma 5.6, that $(\tilde{\boldsymbol{u}}(0), \tilde{v}(0)) = (\boldsymbol{u}_{0}, v_{0})$ almost surely.) 

We use the equivalence of laws to verify the remaining properties of Brownian motion. In particular, we just need to show that $\tilde{W}(t) - \tilde{W}(s)$ is distributed as $N(0, t - s)$. By the equivalence of laws and the fact that $W$ is originally a Brownian motion, $\tilde{W}_{N}(t) - \tilde{W}_{N}(s) =_{d} W(t) - W(s)$, so that $\tilde{W}_{N}(t) - \tilde{W}_{N}(s)$ is distributed as $N(0, t - s)$. Since $\tilde{W}_{N} \to \tilde{W}$ a.s. in $C(0, T; \mathbb{R})$, we obtain that $\tilde{W}_{N}(t) - \tilde{W}_{N}(s) \to \tilde{W}(t) - \tilde{W}(s)$ almost surely, so that $\tilde{W}(t) - \tilde{W}(s)$ is the almost sure limit of random variables distributed as $N(0, t - s)$. Thus, we conclude that $\tilde{W}(t) - \tilde{W}(s)$ must also be distributed as $N(0, t - s)$, which concludes the proof of Property 9. 
\fi

\vspace{0.1in}

\noindent \textbf{Property 7:} By the definition of $\tilde{\mathcal{F}}_{t}$, the process $(\tilde{\boldsymbol{u}}, \tilde{v}, \tilde{\eta})$ is adapted to $\tilde{\mathcal{F}}_{t}$. By Property 2, $\tilde{\eta}$ almost surely has continuous paths on $[0, T]$, taking values in $L^{2}(\Omega_{f})$. By Property 5, $(\tilde{\boldsymbol{u}}, \tilde{v})$ almost surely has continuous paths on $[0, T]$, taking values in $\mathcal{Q}'$. Since a continuous adapted process is predictable (see Proposition 5.1 in Chapter IV of Revuz and Yor \cite{RY}), this establishes the desired property. 

This completes the proof of Lemma~\ref{SkorohodLemma}.
\end{proof}

\subsection{Passing to the limit}\label{limit}
We now consider the approximate solutions defined as random variables on the probability space
$(\tilde{\Omega}, \tilde{\mathcal{F}}, \tilde{\mathbb{P}})$, discussed in Lemma~\ref{SkorohodLemma},
and show that the almost sure limit obtained in Lemma~\ref{SkorohodLemma}, satisfies the weak formulation
stated in Definition~\ref{weak}, almost surely on $(\tilde{\Omega}, \tilde{\mathcal{F}}, \tilde{\mathbb{P}})$.
For this purpose, we recall the semidiscrete formulation of the problem from \eqref{semi1}, given by
\begin{multline*}
\int_{\Omega_{f}} \frac{\boldsymbol{u}^{n + 1}_{N} - \boldsymbol{u}^{n}_{N}}{\Delta t} \cdot \boldsymbol{q} d\boldsymbol{x} + 2\mu \int_{\Omega_{f}} \boldsymbol{D}(\boldsymbol{u}^{n + 1}_{N}) : \boldsymbol{D}(\boldsymbol{q}) d\boldsymbol{x} + \int_{\Gamma} \frac{v^{n + 1}_{N} - v^{n}_{N}}{\Delta t} \psi dz + \int_{\Gamma} \nabla \eta^{n + 1}_{N} \cdot \nabla \psi dz \\
= \int_{\Gamma} \frac{W((n + 1)\Delta t) - W(n\Delta t)}{\Delta t} \psi dz + P^{n}_{N, in}\int_{0}^{R} (q_{z})|_{z = 0} dr - P^{n}_{N, out} \int_{0}^{R} (q_{z})|_{z = L} dr, 
\ \forall(\boldsymbol{q}, \psi) \in \mathcal{Q},
\end{multline*}
\begin{equation*}
\int_{\Gamma} \frac{\eta^{n + 1}_{N} - \eta^{n}_{N}}{\Delta t} \phi dz = \int_{\Gamma} v^{n + \frac{1}{3}}_{N} \phi dz, \qquad \forall \phi \in L^{2}(\Gamma),
\end{equation*}
where
$
P^{n}_{N, in/out} = \frac{1}{\Delta t} \int_{n\Delta t}^{(n + 1)\Delta t} P_{in/out}(t) dt.
$
Notice that as stated, this semidiscrete formulation refers to the original variables, defined on the original probability space.
%
Given a general $(\boldsymbol{q}, \psi) \in \mathcal{Q}(0, T)$, we use the semidiscrete formulation at each fixed time and integrate in time from $0$ to $T$ to obtain for all $(\boldsymbol{q}, \psi) \in \mathcal{Q}(0, T)$, 
\begin{eqnarray*}
\int_{0}^{T} \int_{\Omega_{f}} \partial_{t}\overline{\boldsymbol{u}}_{N} \cdot \boldsymbol{q} d\boldsymbol{x} + 2\mu \int_{0}^{T} \int_{\Omega_{f}} \boldsymbol{D}(\boldsymbol{u}^{\Delta t}_{N}) : \boldsymbol{D}(\boldsymbol{q}) d\boldsymbol{x} dt + \int_{0}^{T} \int_{\Gamma} \partial_{t}\overline{v}_{N}\psi dz dt \\
+ \int_{0}^{T} \int_{\Gamma} \nabla \eta^{\Delta t}_{N} \cdot \nabla \psi dz dt = \sum_{n = 0}^{N - 1} \int_{n\Delta t}^{(n + 1)\Delta t} \int_{\Gamma} \frac{W((n + 1)\Delta t) - W(n\Delta t)}{\Delta t} \psi dz dt \\
+ \sum_{n = 0}^{N - 1}\left(\int_{n\Delta t}^{(n + 1)\Delta t} P^{n}_{N, in}\int_{0}^{R} (q_{z})|_{z = 0} dr dt - \int_{n\Delta t}^{(n + 1)\Delta t} P^{n}_{N, out} \int_{0}^{R} (q_{z})|_{z = L} dr dt\right), 
\end{eqnarray*}

\begin{equation*}
\int_{0}^{T} \int_{\Gamma} \partial_{t}\overline{\eta}_{N} \phi dz dt = \int_{0}^{T} \int_{\Gamma} v^{*}_{N} \phi dz dt, 
\quad \forall \phi \in C^{1}(0, T; L^{2}(\Gamma)),
\end{equation*}
where $\overline{\boldsymbol{u}}_{N}, \overline{v}_{N}$ and $\overline{\eta}_{N}$ are the piecewise linear approximations, given by
\eqref{eta_bar} and \eqref{uv_bar}, and $\boldsymbol{u}^{\Delta t}_{N}$ and $\eta^{\Delta t}_{N}$ are the piecewise constant time shifted functions, given by \eqref{etatimeshift} and \eqref{uvtimeshift}.
Now, we convert to the new probability space by noticing that the same identities  hold for the new random
variables defined on the ``tilde'' probability space since the two sets of random variables have the same law on $\mathcal{X}$.
So for all $(\boldsymbol{q}, \psi) \in \mathcal{Q}(0, T)$, on the new probability space $(\tilde{\Omega}, \tilde{\mathcal{F}}, \tilde{\mathbb{P}})$ with the filtration $\{\tilde{\mathcal{F}}_{t}\}_{t \ge 0}$ defined in \eqref{tildefiltration}, we obtain
\begin{eqnarray*}
\int_{0}^{T} \int_{\Omega_{f}} \partial_{t}\tilde{\overline{\boldsymbol{u}}}_{N} \cdot \boldsymbol{q} d\boldsymbol{x} + 2\mu \int_{0}^{T} \int_{\Omega_{f}} \boldsymbol{D}(\tilde{\boldsymbol{u}}^{\Delta t}_{N}) : \boldsymbol{D}(\boldsymbol{q}) d\boldsymbol{x} dt + \int_{0}^{T} \int_{\Gamma} \partial_{t}\tilde{\overline{v}}_{N}\psi dz dt \\
+ \int_{0}^{T} \int_{\Gamma} \nabla \tilde{\eta}^{\Delta t}_{N} \cdot \nabla \psi dz dt = \sum_{n = 0}^{N - 1} \int_{n\Delta t}^{(n + 1)\Delta t} \int_{\Gamma} \frac{\tilde{W}_{N}((n + 1)\Delta t) - \tilde{W}_{N}(n\Delta t)}{\Delta t} \psi dz dt \\
+ \sum_{n = 0}^{N - 1}\left(\int_{n\Delta t}^{(n + 1)\Delta t} P^{n}_{N, in}\int_{0}^{R} (q_{z})|_{z = 0} dr - \int_{n\Delta t}^{(n + 1)\Delta t} P^{n}_{N, out} \int_{0}^{R} (q_{z})|_{z = L} dr dt\right), 
\end{eqnarray*}
\begin{equation*}
\int_{0}^{T} \int_{\Gamma} \partial_{t}\tilde{\overline{\eta}}_{N} \phi dz dt = \int_{0}^{T} \int_{\Gamma} \tilde{v}^{*}_{N} \phi dz dt
\quad
\forall\phi \in C^{1}(0, T; L^{2}(\Gamma)).
\end{equation*}
We can now pass to the limit in all of the integrals, and use the almost sure convergence of the ``tilde'' random variables as follows.

\noindent \textbf{First term:} 
For the functions on the original probability space, note that because $\boldsymbol{q}(T) = 0$, we can integrate by parts to obtain
\begin{equation*}
\int_{0}^{T} \int_{\Omega_{f}} \partial_{t}\overline{\boldsymbol{u}}_{N} \cdot \boldsymbol{q} d\boldsymbol{x} dt = -\int_{0}^{T} \int_{\Omega_{f}} \overline{\boldsymbol{u}}_{N} \cdot \partial_{t}\boldsymbol{q} d\boldsymbol{x} dt - \int_{\Omega_{f}} \boldsymbol{u}_{0} \cdot \boldsymbol{q}(0) d\boldsymbol{x}.
\end{equation*}
By equivalence of laws, this identity also holds with $\tilde{\overline{\boldsymbol{u}}}_{N}$ in place of $\overline{\boldsymbol{u}}_{N}$. Then, because $\tilde{\overline{\boldsymbol{u}}}_{N} \to \tilde{\boldsymbol{u}}$ almost surely in $L^{2}(0, T; L^{2}(\Omega_{f}))$, we can pass to the limit to obtain the desired almost sure convergence,
\begin{equation*}
\int_{0}^{T} \int_{\Omega_{f}} \partial_{t}\tilde{\overline{\boldsymbol{u}}}_{N} \cdot \boldsymbol{q} d\boldsymbol{x} \to -\int_{0}^{T} \int_{\Omega_{f}} \tilde{\boldsymbol{u}} \cdot \partial_{t}\boldsymbol{q} d\boldsymbol{x} dt - \int_{\Omega_{f}} \boldsymbol{u}_{0} \cdot \boldsymbol{q}(0) d\boldsymbol{x}.
\end{equation*}


\noindent \textbf{Third term:} For the third term, we use an argument similar to that for the first term. Since $\psi(T) = 0$, we can integrate by parts,
\begin{equation*}
\int_{0}^{T} \int_{\Gamma} \partial_{t}\overline{v}_{N}\psi dz dt = -\int_{0}^{T} \int_{\Gamma} \overline{v}_{N} \partial_{t}\psi dz dt - \int_{\Gamma} v_{0} \psi(0) dz. 
\end{equation*}
This holds with $\tilde{\overline{v}}_{N}$ in place of $\overline{v}_{N}$ too by equivalence of laws. Because $\tilde{\overline{v}}_{N} \to \tilde{v}$ in $L^{2}(0, T; L^{2}(\Gamma))$ almost surely, we have the desired almost sure convergence,
\begin{equation*}
\int_{0}^{T} \int_{\Gamma} \partial_{t}\tilde{\overline{v}}_{N}\psi dz dt = -\int_{0}^{T} \int_{\Gamma} \tilde{\overline{v}}_{N} \partial_{t}\psi dz dt - \int_{\Gamma} v_{0} \psi(0) dz \to -\int_{0}^{T} \int_{\Gamma} \tilde{v} \partial_{t}\psi dz dt - \int_{\Gamma} v_{0} \psi(0) dz.
\end{equation*}


\noindent \textbf{Second and fourth term with smooth test function:} For the second and fourth term, we have to use an approximation argument, since we only have estimates of convergence of $\tilde{\boldsymbol{u}}_{N}$ and $\tilde{\boldsymbol{u}}^{\Delta t}_{N}$ in $L^{2}(0, T; L^{2}(\Omega_{f}))$ and $\tilde{v}_{N}$ in $L^{2}(0, T; L^{2}(\Gamma))$. 

We will first show the desired convergence under the assumption that $(\boldsymbol{q}, \psi) \in \mathcal{Q}(0, T)$ is spatially smooth at each time in $[0, T]$. Then, on the original probability space,
\begin{equation*}
2\mu \int_{0}^{T} \int_{\Omega_{f}} \boldsymbol{D}(\boldsymbol{u}^{\Delta t}_{N}) : \boldsymbol{D}(\boldsymbol{q}) d\boldsymbol{x} dt = \mu \int_{0}^{T} \int_{\Omega_{f}} \nabla \boldsymbol{u}_{N}^{\Delta t} : \nabla \boldsymbol{q} d\boldsymbol{x} dt = -\mu \int_{0}^{T} \int_{\Omega_{f}} \boldsymbol{u}^{\Delta t}_{N} \cdot \Delta \boldsymbol{q} d\boldsymbol{x} dt,
\end{equation*}
where the last integration by parts has no boundary terms due to the properties of the solution space and test space for the fluid. Then, by the uniform dissipation estimate in Proposition \ref{uniformenergy},
$
\sum_{n = 0}^{N - 1} \mathbb{E}\left(||\boldsymbol{u}^{n + 1}_{N} - \boldsymbol{u}^{n}_{N}||_{L^{2}(\Omega_{f})}^{2}\right) \le C,
$
we have that
\begin{equation*}
\mathbb{E}\left(||\boldsymbol{u}_{N}^{\Delta t} - \boldsymbol{u}_{N}||^{2}_{L^{2}(0, T; L^{2}(\Omega_{f}))}\right) \le C(\Delta t) \to 0, \qquad \text{ as } N \to \infty.
\end{equation*}
By equivalence of laws, the above identities and estimates hold for $\tilde{\boldsymbol{u}}_{N}^{\Delta t}$ in place of $\boldsymbol{u}_{N}^{\Delta t}$. 
By the Borel-Cantelli lemma, we have that
\begin{equation*}
||\tilde{\boldsymbol{u}}^{\Delta t}_{N} - \tilde{\boldsymbol{u}}_{N}||_{L^{2}(0, T; L^{2}(\Omega_{f}))} \to 0 
\ \rm{almost\  surely } \ \text{as $N \to \infty$},
\end{equation*}
taking a subsequence if needed. Because $\tilde{\boldsymbol{u}}_{N}$ converges to $\tilde{\boldsymbol{u}}$ in $L^{2}(0, T; L^{2}(\Omega_{f}))$ as $N \to \infty$, we also have that
\begin{equation*}
||\tilde{\boldsymbol{u}}^{\Delta t}_{N} - \tilde{\boldsymbol{u}}||_{L^{2}(0, T; L^{2}(\Omega_{f}))} \to 0
\ \rm{almost\  surely } \ \text{as $N \to \infty$}
\end{equation*}
 along this subsequence, which allows us to pass to the limit to obtain
\begin{eqnarray}\label{secondsmooth}
2\mu \int_{0}^{T} \int_{\Omega_{f}} \boldsymbol{D}(\tilde{\boldsymbol{u}}^{\Delta t}_{N}) : \boldsymbol{D}(\boldsymbol{q}) d\boldsymbol{x} dt = -\mu \int_{0}^{T} \int_{\Omega_{f}} \tilde{\boldsymbol{u}}^{\Delta t}_{N} \cdot \Delta \boldsymbol{q} d\boldsymbol{x} dt \\
\to -\mu \int_{0}^{T} \int_{\Omega_{f}} \tilde{\boldsymbol{u}} \cdot \Delta \boldsymbol{q} d\boldsymbol{x} dt = 2\mu \int_{0}^{T} \int_{\Omega_{f}} \boldsymbol{D}(\tilde{\boldsymbol{u}}) : \boldsymbol{D}(\boldsymbol{q})d\boldsymbol{x} dt.
\nonumber
\end{eqnarray}

For the fourth term, one can use a similar argument under the assumption that the test function $(\boldsymbol{q}, \psi)$ is spatially smooth. On the original probability space, 
\begin{equation*}
\int_{0}^{T} \int_{\Gamma} \nabla \eta_{N}^{\Delta t} \cdot \nabla \psi dz dt  = -\int_{0}^{T} \int_{\Gamma} \eta_{N}^{\Delta t} \cdot \Delta \psi dz dt.
\end{equation*}
By the numerical dissipation estimate from Lemma \ref{uniformenergy},
$
\sum_{n = 0}^{N - 1} \mathbb{E}\left(||\nabla \eta^{n + \frac{1}{3}}_{N} - \nabla \eta^{n}_{N}||^{2}_{L^{2}(\Gamma)}\right) \le C,
$
so we obtain, by Poincar\'{e}'s inequality, that 
\begin{equation*}
\mathbb{E}\left(||\eta_{N}^{\Delta t} - \eta_{N}||_{L^{2}(0, T; L^{2}(\Gamma))}^{2}\right) \le C(\Delta t) \to 0, \qquad \text{ as } N \to \infty.
\end{equation*}
These estimates hold on the new probability space with $\tilde{\eta}_{N}$ in place of $\eta_{N}$. By the Borel-Cantelli lemma and the convergence of $\tilde{\eta}_{N}$ to $\tilde{\eta}$ in $L^{2}(0, T; L^{2}(\Gamma))$, 
\begin{equation*}
||\tilde{\eta}_{N}^{\Delta t} - \tilde{\eta}||_{L^{2}(0, T; L^{2}(\Gamma))} \to 0, \qquad \text{ almost surely as } N \to \infty,
\end{equation*}
taking a subsequence. This allows us to pass to the limit to obtain the almost sure convergence,
\begin{eqnarray}\label{fourthsmooth}
\int_{0}^{T} \int_{\Gamma} \nabla \tilde{\eta}_{N}^{\Delta t} \cdot \nabla \psi dz dt  = -\int_{0}^{T} \int_{\Gamma} \tilde{\eta}_{N}^{\Delta t} \cdot \Delta \psi dz dt \\
\to -\int_{0}^{T} \int_{\Gamma} \tilde{\eta} \cdot \Delta \psi dz dt = \int_{0}^{T} \int_{\Gamma} \nabla \tilde{\eta} \cdot \nabla \psi dz dt, \qquad \text{ as } N \to \infty.
\nonumber
\end{eqnarray}


\noindent \textbf{Second and fourth term with general test function:} To show the almost sure convergence in the previous step, we assumed that $(\boldsymbol{q}, \psi) \in \mathcal{Q}(0, T)$ was spatially smooth. To get the general convergence, we use an approximation argument. Suppose that $(\boldsymbol{q}, \psi) \in \mathcal{Q}(0, T)$ is not smooth spatially. It suffices to show that
$\int_{0}^{T} \int_{\Omega_{f}} \boldsymbol{D}(\tilde{\boldsymbol{u}}^{\Delta t}_{N}) : \boldsymbol{D}(\boldsymbol{q}) d\boldsymbol{x} dt \to \int_{0}^{T} \int_{\Omega_{f}} \boldsymbol{D}(\tilde{\boldsymbol{u}}) : \boldsymbol{D}(\boldsymbol{q})d\boldsymbol{x} dt$
in probability, and
$\int_{0}^{T} \int_{\Gamma} \nabla \tilde{\eta}_{N}^{\Delta t} \cdot \nabla \psi dz dt  \to \int_{0}^{T} \int_{\Gamma} \nabla \tilde{\eta} \cdot \nabla \psi dz dt$ 
 in probability (see below for the precise definition),
as we would get the desired result from the fact that we then have almost sure convergence along a subsequence. So given any $\epsilon > 0$ and $\delta > 0$, we must show that there exists $N_{0}$ such that for all $N \ge N_{0}$, 
\begin{equation}\label{probconvcond}
\tilde{\mathbb{P}}\left(\left|\int_{0}^{T} \int_{\Omega_{f}} \boldsymbol{D}(\tilde{\boldsymbol{u}}^{\Delta t}_{N}) : \boldsymbol{D}(\boldsymbol{q}) d\boldsymbol{x} dt - \int_{0}^{T} \int_{\Omega_{f}} \boldsymbol{D}(\tilde{\boldsymbol{u}}) : \boldsymbol{D}(\boldsymbol{q})d\boldsymbol{x} dt\right| > \delta\right) \le \epsilon,
\end{equation}
\begin{equation}\label{probconvcond2}
\tilde{\mathbb{P}}\left(\left|\int_{0}^{T} \int_{\Gamma} \nabla \tilde{\eta}^{\Delta t}_{N} \cdot \nabla \psi dz dt - \int_{0}^{T} \int_{\Gamma} \nabla \tilde{\eta} \cdot \nabla \psi dz dt\right| > \delta\right) \le \epsilon.
\end{equation}

To show this, observe that by the uniform dissipation estimate in Proposition \ref{uniformenergy}, we have that
\begin{equation*}
\mathbb{E}\left(\sum_{n = 0}^{N - 1} (\Delta t) \int_{\Omega_{f}} |\boldsymbol{D}(\boldsymbol{u}^{n + 1}_{N})|^{2}d\boldsymbol{x}\right) = \mathbb{E}\left(||\boldsymbol{D}(\boldsymbol{u}^{\Delta t}_{N})||_{L^{2}(0, T; L^{2}(\Omega_{f}))}^{2}\right) \le C.
\end{equation*}
and hence by equivalence of laws,
\begin{equation*}
\tilde{\mathbb{E}}\left(||\boldsymbol{D}(\tilde{\boldsymbol{u}}^{\Delta t}_{N})||_{L^{2}(0, T; L^{2}(\Omega_{f}))}^{2}\right) \le C,
\end{equation*}
for a uniform constant $C$. Since $\tilde{\boldsymbol{u}} \in L^{2}(\Omega; L^{2}(0, T; H^{1}(\Omega_{f})))$ by Property 2 of Lemma \ref{properties}, we conclude that there exists a sufficiently large positive constant $M$ such that for all $N$,
\begin{equation}\label{Mchoice1}
\tilde{\mathbb{P}}\left(||\boldsymbol{D}(\tilde{\boldsymbol{u}}^{\Delta t}_{N})||_{L^{2}(0, T; L^{2}(\Omega_{f}))} \ge M\right) \le \frac{\epsilon}{3}, \qquad \tilde{\mathbb{P}}\left(||\boldsymbol{D}(\tilde{\boldsymbol{u}})||_{L^{2}(0, T; L^{2}(\Omega_{f}))} \ge M\right) \le \frac{\epsilon}{3}.
\end{equation}
For the fourth term involving structure displacements, recall from Lemma \ref{uniformenergy} that 
\begin{equation*}
\mathbb{E}\left(||\nabla \eta^{\Delta t}_{N}||^{2}_{L^{\infty}(0, T; L^{2}(\Gamma))}\right) \le C,
\end{equation*}
and by Property 2 in Lemma \ref{properties}, $\tilde{\eta} \in L^{2}(\tilde{\Omega}; L^{\infty}(0, T; H_{0}^{1}(\Gamma)))$. So using equivalence of laws, $M$ can also be chosen sufficiently large so that for all $N$,
\begin{equation}\label{Mchoice2}
\tilde{\mathbb{P}}\left(||\nabla \tilde{\eta}^{\Delta t}_{N}||_{L^{\infty}(0, T; L^{2}(\Gamma))} \ge M\right) \le \frac{\epsilon}{3}, \qquad \tilde{\mathbb{P}}\left(||\nabla \tilde{\eta}||_{L^{\infty}(0, T; L^{2}(\Gamma))} \ge M\right) \le \frac{\epsilon}{3}.
\end{equation}

Then, choose $(\boldsymbol{\widehat{q}}, \widehat{\psi}) \in \mathcal{Q}(0, T)$ that are smooth spatially at all times in $[0, T]$, such that
\begin{equation}\label{approximation}
||\boldsymbol{D}(\boldsymbol{q}) - \boldsymbol{D}(\boldsymbol{\widehat{q}})||_{L^{2}(0, T; L^{2}(\Omega_{f}))} \le \frac{\delta}{3M}, \qquad ||\nabla \psi - \nabla \widehat{\psi}||_{L^{1}(0, T; L^{2}(\Gamma))} \le \frac{\delta}{3M}.
\end{equation}
Then, the almost sure convergences \eqref{secondsmooth} and \eqref{fourthsmooth}, which hold for this smoother $(\widehat{\boldsymbol{q}}, \widehat{\psi})$, allow us to choose $N_{0}$ sufficiently large such that for all $N \ge N_{0}$,
\begin{equation}\label{secondterm1}
\tilde{\mathbb{P}}\left(\left|\int_{0}^{T} \int_{\Omega_{f}} \boldsymbol{D}(\tilde{\boldsymbol{u}}^{\Delta t}_{N}) : \boldsymbol{D}(\widehat{\boldsymbol{q}}) d\boldsymbol{x} dt - \int_{0}^{T} \int_{\Omega_{f}} \boldsymbol{D}(\tilde{\boldsymbol{u}}) : \boldsymbol{D}(\widehat{\boldsymbol{q}})d\boldsymbol{x} dt\right| > \frac{\delta}{3}\right) \le \frac{\epsilon}{3}, 
\end{equation}
\begin{equation}\label{fourthterm1}
\tilde{\mathbb{P}}\left(\left|\int_{0}^{T} \int_{\Gamma} \nabla \tilde{\eta}^{\Delta t}_{N} \cdot \nabla \widehat{\psi} dz dt - \int_{0}^{T} \int_{\Gamma} \nabla \tilde{\eta} \cdot \nabla \widehat{\psi} dz dt\right| > \frac{\delta}{3} \right) \le \frac{\epsilon}{3}.
\end{equation}
Furthermore, the choice of $(\boldsymbol{\widehat{q}}, \widehat{\psi})$ in \eqref{approximation} and the choice of $M$ in \eqref{Mchoice1} and \eqref{Mchoice2} give that for all $N$,
\begin{equation}\label{secondterm2}
\tilde{\mathbb{P}}\left(\left|\int_{0}^{T}\int_{\Omega_{f}} \boldsymbol{D}(\tilde{\boldsymbol{u}}^{\Delta t}_{N}) : \boldsymbol{D}(\boldsymbol{q}) d\boldsymbol{x} dt - \int_{0}^{T} \int_{\Omega_{f}} \boldsymbol{D}(\tilde{\boldsymbol{u}}^{\Delta t}_{N}) : \boldsymbol{D}(\widehat{\boldsymbol{q}}) d\boldsymbol{x} dt \right| > \frac{\delta}{3}\right) \le \frac{\epsilon}{3},
\end{equation}
\begin{equation}\label{fourthterm2}
\tilde{\mathbb{P}}\left(\left|\int_{0}^{T}\int_{\Gamma} \nabla \tilde{\eta}^{\Delta t}_{N} \cdot \nabla \psi dz dt - \int_{0}^{T}\int_{\Gamma} \nabla \tilde{\eta}_{N}^{\Delta t} \cdot \nabla \widehat{\psi} dz dt \right| > \frac{\delta}{3}\right) \le \frac{\epsilon}{3},
\end{equation}
and
\begin{equation}\label{secondterm3}
\tilde{\mathbb{P}}\left(\left|\int_{0}^{T}\int_{\Omega_{f}} \boldsymbol{D}(\tilde{\boldsymbol{u}}) : \boldsymbol{D}(\boldsymbol{q}) d\boldsymbol{x} dt - \int_{0}^{T} \int_{\Omega_{f}} \boldsymbol{D}(\tilde{\boldsymbol{u}}) : \boldsymbol{D}(\widehat{\boldsymbol{q}}) d\boldsymbol{x} dt \right| > \frac{\delta}{3}\right) \le \frac{\epsilon}{3},
\end{equation}
\begin{equation}\label{fourthterm3}
\tilde{\mathbb{P}}\left(\left|\int_{0}^{T}\int_{\Gamma} \nabla \tilde{\eta} \cdot \nabla \psi dz dt - \int_{0}^{T}\int_{\Gamma} \nabla \tilde{\eta} \cdot \nabla \widehat{\psi} dz dt \right| > \frac{\delta}{3}\right) \le \frac{\epsilon}{3}.
\end{equation}
Combining the estimates \eqref{secondterm1}, \eqref{fourthterm1}, \eqref{secondterm2}, \eqref{fourthterm2}, \eqref{secondterm3}, and \eqref{fourthterm3} establishes the desired estimates \eqref{probconvcond} and \eqref{probconvcond2}, and hence proves the desired convergence in probability.

\if 1 = 0

To show this convergence, we assumed that $\boldsymbol{q}$ was smooth. To get the general convergence, we use an approximation argument. By the uniform dissipation estimate,
\begin{equation*}
\mathbb{E}\left(\sum_{n = 0}^{N - 1} (\Delta t) \mu \int_{\Omega} |\boldsymbol{D}(\boldsymbol{u}^{n + 1}_{N})|^{2}d\boldsymbol{x}\right) = \mathbb{E}\left(||\boldsymbol{D}(\tau_{\Delta t} \boldsymbol{u}_{N})||_{L^{2}(0, T; L^{2}(\Omega))}^{2}\right) \le C.
\end{equation*}
Hence, by Lemma 5.8, we have that $||\boldsymbol{D}(\tau_{\Delta t} \boldsymbol{u}_{N})||_{L^{2}(0, T; L^{2}(\Omega))}$ is bounded along a subsequence almost surely. Therefore, given arbitrary $\boldsymbol{q} \in \mathcal{Q}(0, T)$, we can obtain a sequence $\boldsymbol{q}_{N}$ where $\boldsymbol{q}_{N}$ is smooth at each time such that $||\boldsymbol{q}_{N} - \boldsymbol{q}_{N}||_{C(0, T; H^{1}(\Omega))} \to 0$. This combined with the Korn inequality implies the desired convergence.

\fi 

\if 1 = 0

\noindent \textbf{Fourth term:} We will pass to the limit in the fourth term in a similar manner as the second term. Assuming that $\psi$ is smooth, we have that 

To pass to the limit, we take a sequence $\psi_{n} \to \psi$ that converges in $C(0, T; H_{0}^{1}(\Gamma))$. We then note that because 
\begin{equation*}
\mathbb{E}\left(||\nabla (\tau_{\Delta t} \eta_{N})||_{L^{\infty}(0, T; L^{2}(\Gamma))}\right) \le C,
\end{equation*}
by the uniform energy bound, by Lemma 5.8, along a subsequence depending on the outcome, we have that the quantity $||\nabla (\tau_{\Delta t} \eta_{N})||_{L^{\infty}(0, T; L^{2}(\Gamma))}$ is almost surely bounded. Therefore, we can pass to the limit to obtain the result for general $\psi$.

\fi 
\vspace{0.1in}

\noindent \textbf{Passing to the limit in the stochastic integral.} We want to pass to the limit in the stochastic integral and show that for arbitrary $\psi$ such that $(\boldsymbol{q}, \psi) \in \mathcal{Q}(0, T)$,
\begin{equation*}
\sum_{n = 0}^{N - 1} \int_{n\Delta t}^{(n + 1)\Delta t} \int_{\Gamma} \frac{\tilde{W}_{N}((n + 1)\Delta t) - \tilde{W}_{N}(n\Delta t)}{\Delta t} \psi dz dt \to \int_{0}^{T} \left(\int_{\Gamma} \psi dz\right) d\tilde{W}, \qquad \text{ a.s. as $N \to \infty$}.
\end{equation*}
Note that because $\psi$ is deterministic, we can express the right hand side as a stochastic integral,
\begin{eqnarray*}
\sum_{n = 0}^{N - 1} \int_{n\Delta t}^{(n + 1)\Delta t} \int_{\Gamma} \frac{\tilde{W}_{N}((n + 1)\Delta t) - \tilde{W}_{N}(n\Delta t)}{\Delta t} \psi dz dt 
\hskip 1in\\
= \int_{0}^{T} \sum_{n = 0}^{N - 1} \left(\frac{1}{\Delta t} \int_{n\Delta t}^{(n + 1)(\Delta t)} \int_{\Gamma} \psi(s, z) dz ds\right) 1_{t \in (n\Delta t, (n + 1)\Delta t]}(t) d\tilde{W}_{N}(t).
\end{eqnarray*}
Because convergence in probability implies convergence almost surely along a subsequence, it thus suffices to prove that
\begin{equation*}
\int_{0}^{T} \sum_{n = 0}^{N - 1} \left(\frac{1}{\Delta t} \int_{n\Delta t}^{(n + 1)(\Delta t)} \int_{\Gamma} \psi(s, z) dz ds\right) 1_{t \in (n\Delta t, (n + 1)\Delta t]}(t) d\tilde{W}_{N} \to \int_{0}^{T}\left(\int_{\Gamma} \psi dz\right) d\tilde{W},
\end{equation*}
as $N \to \infty$ in probability.
So we must show that given any $\delta > 0$ and any $\epsilon > 0$, there exists $N_{0}$ sufficiently large such that for all $N \ge N_{0}$
\begin{equation*}
\tilde{\mathbb{P}}\left(\left|\int_{0}^{T} \sum_{n = 0}^{N - 1} \left(\frac{1}{\Delta t} \int_{n\Delta t}^{(n + 1)(\Delta t)} \int_{\Gamma} \psi(s, z) dz ds\right) 1_{t \in (n\Delta t, (n + 1)\Delta t]}(t) d\tilde{W}_{N} - \int_{0}^{T}\left(\int_{\Gamma} \psi dz\right) d\tilde{W}\right| > \delta \right) < \epsilon.
\end{equation*}

We accomplish this through two estimates. We claim that we can choose $N_{0}$ sufficiently large such that 
\begin{equation}\label{Brownianest1}
\tilde{\mathbb{P}}\left(\left|\int_{0}^{T} \sum_{n = 0}^{N - 1} \left(\frac{1}{\Delta t} \int_{n\Delta t}^{(n + 1)(\Delta t)} \int_{\Gamma} \psi(s, z) dz ds\right) 1_{t \in (n\Delta t, (n + 1)\Delta t]}(t) d\tilde{W}_{N} - \int_{0}^{T}\left(\int_{\Gamma} \psi dz\right) d\tilde{W}_{N}\right| > \frac{\delta}{2} \right) < \frac{\epsilon}{2},
\end{equation}
and
\begin{equation}\label{Brownianest2}
\tilde{\mathbb{P}}\left(\left|\int_{0}^{T} \left(\int_{\Gamma} \psi dz\right) d\tilde{W}_{N} - \int_{0}^{T} \left(\int_{\Gamma} \psi dz\right) d\tilde{W}\right| > \frac{\delta}{2} \right) < \frac{\epsilon}{2},
\end{equation}
for all $N \ge N_{0}$. 

For the first estimate \eqref{Brownianest1}, it suffices to use the Itô isometry along with the fact that 
\begin{equation*}
\left|\left|\sum_{n = 0}^{N - 1} \left(\frac{1}{\Delta t} \int_{n\Delta t}^{(n + 1)(\Delta t)} \int_{\Gamma} \psi(s, z) dz ds\right) 1_{t \in (n\Delta t, (n + 1)\Delta t]}(t) - \int_{\Gamma} \psi dz\right|\right|_{L^{2}(0, T)} \to 0, \qquad \text{ as } N \to \infty,
\end{equation*}
to conclude that
\begin{equation*}
\tilde{\mathbb{E}}\left(\left|\int_{0}^{T} \sum_{n = 0}^{N - 1} \left(\frac{1}{\Delta t} \int_{n\Delta t}^{(n + 1)(\Delta t)} \int_{\Gamma} \psi(s, z) dz ds\right) 1_{t \in (n\Delta t, (n + 1)\Delta t]}(t) d\tilde{W}_{N} - \int_{0}^{T}\left(\int_{\Gamma} \psi dz\right) d\tilde{W}_{N}\right|^{2}\right) \to 0,
\end{equation*}
as $N \to \infty$. The first estimate \eqref{Brownianest1} thus follows from taking $N_{0}$ sufficiently large to make this expectation sufficiently small, and then using Chebychev's inequality. 

For the second estimate, note that we can approximate $\int_{\Gamma}\psi(t, z) dz := g(t)$ by a \textit{deterministic} step function 
\begin{equation*}
g_{m}(t) = g\left(\frac{kT}{m}\right) \qquad \text{ if } \frac{kT}{m} < t \le \frac{(k + 1)T}{m}.
\end{equation*}
By the continuity of $g(t)$, we can select $m$ sufficiently large such that
\begin{equation*}
\left|\left|g(t) - g_{m}(t)\right|\right|_{L^{2}(0, T)}^{2} < \frac{\epsilon}{6} \cdot \left(\frac{\delta}{6}\right)^{2}.
\end{equation*}
Then, by the Itô isometry and Chebychev's inequality,
\begin{equation}\label{Brownianest2part1}
\tilde{\mathbb{P}}\left(\left|\int_{0}^{T} \left(\int_{\Gamma} \psi dz\right) d\tilde{W}_{N} - \int_{0}^{T} g_{m}(t) d\tilde{W}_{N}\right| > \frac{\delta}{6} \right) < \frac{\epsilon}{6},
\end{equation}
for all $N$, and
\begin{equation}\label{Brownianest2part2}
\tilde{\mathbb{P}}\left(\left|\int_{0}^{T} \left(\int_{\Gamma} \psi dz\right) d\tilde{W} - \int_{0}^{T} g_{m}(t) d\tilde{W}\right| > \frac{\delta}{6} \right) < \frac{\epsilon}{6}.
\end{equation}

So it remains to choose $N_{0}$ sufficiently large such that for all $N \ge N_{0}$,
\begin{equation}\label{Brownianest2part3}
\tilde{\mathbb{P}}\left(\left|\int_{0}^{T} g_{m}(t) d\tilde{W}_{N} - \int_{0}^{T} g_{m}(t) d\tilde{W}\right| > \frac{\delta}{6} \right) < \frac{\epsilon}{6}.
\end{equation}
We note that $|g_{m}(t)| \le K$ for some constant $K$ that is deterministic, as $g_{m}(t)$ is a deterministic function of time. 
Also, note that
\begin{equation*}
\int_{0}^{T}g_{m}(t) d\tilde{W}_{N} = \sum_{k = 0}^{m - 1} g\left(\frac{kT}{m}\right) \cdot \left(\tilde{W}_{N}\left(\frac{(k + 1)T}{m}\right) - \tilde{W}_{N}\left(\frac{kT}{m}\right)\right),
\end{equation*}
with an analogous formula for the integration against $\tilde{W}$. Hence,
\begin{gather*}
\left|\int_{0}^{T} g_{m}(t) d\tilde{W}_{N} - \int_{0}^{T} g_{m}(t) d\tilde{W}\right| \\
\le \sum_{k = 0}^{m - 1} \left|g\left(\frac{kT}{m}\right) \cdot \left[\left(\tilde{W}_{N}\left(\frac{(k + 1)T}{m}\right) - \tilde{W}_{N}\left(\frac{kT}{m}\right)\right) - \left(\tilde{W}\left(\frac{(k + 1)T}{m}\right) - \tilde{W}\left(\frac{kT}{m}\right)\right)\right]\right| \\
\le \sum_{k = 0}^{m - 1} 2K||\tilde{W} - \tilde{W}_{N}||_{C(0, T; \mathbb{R})} \le 2Km \cdot ||\tilde{W} - \tilde{W}_{N}||_{C(0, T; \mathbb{R})}.
\end{gather*}
Because $\tilde{W}_{N} \to \tilde{W}$ in $C(0, T; \mathbb{R})$ almost surely, there exists $N_{0}$ sufficiently large such that
\begin{equation*}
\tilde{\mathbb{P}}\left(||\tilde{W} - \tilde{W}_{N}||_{C(0, T; \mathbb{R})} > \frac{\delta}{12Km}\right) < \frac{\epsilon}{6}, \qquad \text{ for all } N \ge N_{0}.
\end{equation*}
Therefore, 
\begin{equation*}
\tilde{\mathbb{P}}\left(\left|\int_{0}^{T} g_{m}(t) d\tilde{W}_{N} - \int_{0}^{T} g_{m}(t) d\tilde{W}\right| > \frac{\delta}{6} \right) < \frac{\epsilon}{6}, \qquad \text{ for all } N \ge N_{0}.
\end{equation*}
The estimates \eqref{Brownianest2part1}, \eqref{Brownianest2part2}, and \eqref{Brownianest2part3} thus imply the desired estimate in \eqref{Brownianest2}.

\vspace{0.1in}

\noindent \textbf{Convergence of the pressure term.} Finally, we show that
\begin{equation}\label{pressureconv}
\sum_{n = 0}^{N - 1} \int_{n\Delta t}^{(n + 1)\Delta t} P^{n}_{N, in}\left(\int_{0}^{R} (q_{z})|_{z = 0} dr\right) dt \to \int_{0}^{T} P_{in}(t) \left(\int_{0}^{R} (q_{z})|_{z = 0} dr\right) dt, \qquad \text{ as } N \to \infty.
\end{equation}
The same argument will work for the outlet pressure term.

Define the following piecewise approximation of the test function $\boldsymbol{q}$,
\begin{equation*}
\boldsymbol{q}^{m}(t, \cdot) = \boldsymbol{q}\left(\frac{kT}{m}, \cdot \right), \qquad \text{ if } \frac{kT}{m} < t \le \frac{(k + 1)T}{m}.
\end{equation*} 
For any positive integer $N$, 
\begin{gather*}
\int_{0}^{T} P_{in}(t) \left(\int_{0}^{R} (q^{N}_{z})|_{z = 0} dr\right) dt - \sum_{n = 0}^{N - 1} \int_{n\Delta t}^{(n + 1)\Delta t} P^{n}_{N, in}\left(\int_{0}^{R} (q^{N}_{z})|_{z = 0} dr\right) dt \\
= \sum_{n = 0}^{N - 1} \int_{n\Delta t}^{(n + 1)\Delta t} (P_{in}(t) - P^{n}_{N, in})\left(\int_{0}^{R} (q^{N}_{z})|_{z = 0} dr\right) dt \\
= \sum_{n = 0}^{N - 1} \left(\int_{0}^{R} (q^{N}_{z})|_{z = 0} dr\right) \int_{n\Delta t}^{(n + 1)\Delta t} (P_{in}(t) - P^{n}_{N, in})dt = 0.
\end{gather*}
To establish \eqref{pressureconv}, it suffices to show that
\begin{equation}\label{pressureconv1}
\int_{0}^{T} P_{in}(t) \left(\int_{0}^{R} (q_{z})|_{z = 0} dr\right) dt - \int_{0}^{T} P_{in}(t) \left(\int_{0}^{R} (q^{N}_{z})|_{z = 0} dr\right) dt \to 0, \qquad \text{ as } N \to \infty,
\end{equation}
\begin{equation}\label{pressureconv2}
\sum_{n = 0}^{N - 1} \int_{n\Delta t}^{(n + 1)\Delta t} P^{n}_{N, in}\left(\int_{0}^{R} (q_{z})|_{z = 0} dr\right) dt - \sum_{n = 0}^{N - 1} \int_{n\Delta t}^{(n + 1)\Delta t} P^{n}_{N, in}\left(\int_{0}^{R} (q^{N}_{z})|_{z = 0} dr\right) dt \to 0, \qquad \text{ as } N \to \infty.
\end{equation}

For \eqref{pressureconv1}, we compute
\begin{gather}\nonumber
\left|\int_{0}^{T} P_{in}(t) \left(\int_{0}^{R} (q_{z})|_{z = 0} dr\right) dt - \int_{0}^{T} P_{in}(t) \left(\int_{0}^{R} (q^{N}_{z})|_{z = 0} dr\right) dt \right| \\
\nonumber
= \left|\int_{0}^{T} P_{in}(t) \left(\int_{0}^{R} (q_{z} - q_{z}^{N})|_{z = 0} dr\right) dt\right| \le ||P_{in}||_{L^{2}(0, T)} \left(\int_{0}^{T} \left(\int_{0}^{R} (q_{z} - q_{z}^{N})|_{z = 0} dr \right)^{2} dt\right)^{1/2} \\
\label{pressure1calc}
\le C ||P_{in}||_{L^{2}(0, T)} \left(\int_{0}^{T} ||\boldsymbol{q} - \boldsymbol{q}^{N}||^{2}_{H^{1}(\Omega_{f})} dt\right)^{1/2}.
\end{gather}
Because $\boldsymbol{q}$ is continuous taking values in $\mathcal{V}_{F}$ equipped with the norm of $H^{1}(\Omega_{f})$, we have that 
$||\boldsymbol{q} - \boldsymbol{q}^{N}||_{H^{1}(\Omega_{f})} \to 0$ uniformly on $[0, T]$ as $N \to \infty$, which establishes the desired limit.
Similary, to estabish \eqref{pressureconv2} we calculate
\begin{gather*}\label{pressure2calc}
 \left|\sum_{n = 0}^{N - 1} P^{n}_{N, in} \int_{n\Delta t}^{(n + 1)\Delta t} \left(\int_{0}^{R} (q_{z} - q^{N}_{z})|_{z = 0} dr\right) dt\right| 
\le \left|\sum_{n = 0}^{N - 1} (\Delta t)^{1/2} P^{n}_{N, in} \left(\int_{n\Delta t}^{(n + 1)\Delta t} \left(\int_{0}^{R} (q_{z} - q^{N}_{z})|_{z = 0} dr\right)^{2} dt\right)^{1/2}\right| \\
\le C\left|\sum_{n = 0}^{N - 1} (\Delta t)^{1/2} P^{n}_{N, in} \left(\int_{n \Delta t}^{(n + 1)\Delta t} ||\boldsymbol{q} - \boldsymbol{q}^{N}||^{2}_{H^{1}(\Omega_{f})} dt \right)^{1/2}\right| \\
\le C\left(\sum_{n = 0}^{N - 1} (\Delta t)^{1/2} |P^{n}_{N, in}|\right) \cdot \max_{0 \le n \le N - 1} \left(\int_{n \Delta t}^{(n + 1)\Delta t} ||\boldsymbol{q} - \boldsymbol{q}^{N}||^{2}_{H^{1}(\Omega_{f})} dt \right)^{1/2} \\
\le C \left(\sum_{n = 0}^{N - 1} \frac{1}{(\Delta t)^{1/2}} \int_{n\Delta t}^{(n + 1)\Delta t} |P_{in}(t)| dt\right) \cdot \max_{0 \le n \le N - 1} \left(\int_{n \Delta t}^{(n + 1)\Delta t} ||\boldsymbol{q} - \boldsymbol{q}^{N}||^{2}_{H^{1}(\Omega_{f})} dt \right)^{1/2} \\
\le C ||P_{in}||_{L^{2}(0, T)} \cdot \max_{0 \le n \le N - 1} \left(\int_{n \Delta t}^{(n + 1)\Delta t} ||\boldsymbol{q} - \boldsymbol{q}^{N}||^{2}_{H^{1}(\Omega_{f})} dt \right)^{1/2}.
\end{gather*}
Again, because $\boldsymbol{q}$ is continuous taking values in $\mathcal{V}_{F}$ equipped with the norm of $H^{1}(\Omega_{f})$, we have that $||\boldsymbol{q} - \boldsymbol{q}^{N}||_{H^{1}(\Omega_{f})} \to 0$ uniformly on $[0, T]$ as $N \to \infty$, which establishes the desired limit. 

We have, therefore, established
the existence of a weak solution to the stochastic fluid-structure interaction problem in a probabilistically weak sense, as in Definition \ref{weak}.

\section{Return to the original probability space}\label{original_space}

We have thus constructed a stochastic process $(\tilde{\boldsymbol{u}}, \tilde{\eta})$,
which satisfies the weak formulation of the continuous problem almost surely on the ``tilde'' probability space
determined by the Skorohod representation theorem. 
However, we want to bring the solution back to the original probability space. In particular, we must get convergence of the original approximate solutions $(\boldsymbol{u}_{N}, v_{N}, \eta_{N})$ on the original probability space $(\Omega, \mathcal{F}, \mathbb{P})$ with the original given complete filtration $\{\mathcal{F}_{t}\}_{t \ge 0}$ and the original Brownian motion $W(t)$.

To do this, we will use a standard \textit{Gy\"{o}ngy-Krylov argument} based on the following lemma, see Lemma 1.1 in \cite{GK} and Proposition 6.3 in \cite{LNT}.
\begin{lemma}[Gy\"{o}ngy-Krylov lemma]\label{GK}
Let $\{X_{n}\}_{n = 1}^{\infty}$ be a sequence of random variables defined on a probability space $(\Omega, \mathcal{F}, \mathbb{P})$ taking values in a separable Banach space $B$. Then {\bf{$X_n$ converges in probability}} to some $B$-valued random variable $X^*$
if and only if for every two subsequences $X_{l_{k}}$ and $X_{m_{k}}$ of $X_{n}$, there exists a further subsequence of $x_{k} = (X_{l_k},X_{m_k})$ whose laws converge weakly to a probability measure 
$\nu$ on $B \times B$ that is supported on the diagonal  
$\{(x, y) \in B \times B : x = y\}$.
\end{lemma}

In other words, the statement of the Gy\"{o}ngy-Krylov lemma holds if and only if for every two subsequences $X_{l_k}$ and $X_{m_k}$,
there exists a further subsequence such that 
the joint probability measures associated with $x_k = (X_{l_k}, X_{m_k})$ on $B \times B$, defined by 
\begin{equation*}
\nu_{x_k} = \nu_{X_{l_k},X_{m_k}}(A_{1} \times A_{2}) = \mathbb{P}(X_{l_k} \in A_{1}, X_{m_k} \in A_{2}), \quad A_1, A_2 \in {\cal{B}}(B),
\end{equation*}
where $\mathcal{B}(B)$ is the Borel sigma algebra on $B$, converge weakly along this \textit{further subsequence} to some probability measure $\nu$, where $\nu$ is such that 
\begin{equation}\label{diagonal}
\nu(\{(x, y) \in B \times B: x = y\}) = 1.
\end{equation}
Thus, the limits of any two convergent subsequences have to be ``the same'' with probability 1. 

Once we show convergence in probability of our original sequence using the Gy\"{o}ngy-Krylov lemma, we will have almost sure convergence along a subsequence of our approximate solutions on the \textit{original probability space}. Then, using the fact that our approximate solutions converge almost surely along a subsequence on the original probability space, we can adapt the arguments in Section~\ref{limit} in order to show that the limiting weak solution on the original probability space satisfies the weak form of the continuous problem almost surely, so that the limiting solution is a weak solution in a probabilistically strong sense.

Thus, what remains to be shown is that the diagonal condition in the Gy\"{o}ngy-Krylov lemma holds. 
Since our problem is linear and the stochastic noise is additive, using the Skorohod representation theorem, one can show that
the diagonal condition is equivalent to showing {\emph{deterministic uniqueness}} holding pathwise. 
To demonstrate this, we first prove deterministic uniqueness, and then use it to show how this implies the diagonal condition.

\subsection{Uniqueness of the deterministic linear problem}\label{deterministic}

\if 1 = 0

To prove this result, we first recall the following lemma, which is the Korn equality ({\color{red}{TODO:}} Cite Canic and Muha). Recall the definition of the solution space for the fluid in space,
\begin{equation*}
\mathcal{V}_{F} = \{\boldsymbol{u} = (u_{z}, u_{r}) \in H^{1}(\Omega_{f})^{2}: \nabla \cdot \boldsymbol{u} = 0, \ u_{z} = 0 \text{ on } \Gamma, \ u_{r} = 0 \text{ on } \Omega_{f} \text{ \textbackslash} \ \Gamma, \ \partial_{r}u_{z} = 0 \text{ on } \Gamma_{b}.\}
\end{equation*}

\begin{lemma}
For $\boldsymbol{u} \in \mathcal{V}_{F}$, we have that
\begin{equation*}
2\int_{\Omega_{f}} |\boldsymbol{D}(\boldsymbol{u})|^{2} dz dr = \int_{\Omega_{f}} |\nabla \boldsymbol{u}|^{2} dz dr.
\end{equation*}
\end{lemma}

\begin{proof}
Noting that density would yield the desired result, we can assume that $\boldsymbol{u}$ is smooth. Since
\begin{equation*}
\boldsymbol{D}(\boldsymbol{u}) = \frac{1}{2}(\nabla \boldsymbol{u} + \nabla^{t} \boldsymbol{u}),
\end{equation*}
it suffices to show that
\begin{equation*}
\int_{\Omega_{f}} \nabla \boldsymbol{u} : \nabla^{t} \boldsymbol{u} dz dr = 0.
\end{equation*}
So we want to show that
\begin{equation*}
\int_{\Omega_{f}} \frac{\partial u_{z}}{\partial z} \frac{\partial u_{z}}{\partial z} + \frac{\partial u_{z}}{\partial r}\frac{\partial u_{r}}{\partial z} + \frac{\partial u_{r}}{\partial z} \frac{\partial u_{z}}{\partial r} + \frac{\partial u_{r}}{\partial r} \frac{\partial u_{r}}{\partial r} dz dr = 0.
\end{equation*}
We move the derivative on the first term in each product onto the separate term, and by the divergence free condition, are left only with the boundary integral. So we must show that
\begin{equation*}
I_{\Gamma} + I_{\Gamma_{in}} + I_{\Gamma_{b}} + I_{\Gamma_{out}} = 0,
\end{equation*}
where
\begin{align*}
I_{\Gamma} &= \int_{\Gamma} u_{z} \frac{\partial u_{r}}{\partial z} + u_{r} \frac{\partial u_{r}}{\partial r} dz, \\
I_{\Gamma_{in}} &= -\int_{\Gamma_{in}} u_{z} \frac{\partial u_{z}}{\partial z} + u_{r} \frac{\partial u_{z}}{\partial r} dr, \\
I_{\Gamma_{b}} &= -\int_{\Gamma_{b}} u_{z} \frac{\partial u_{r}}{\partial z} + u_{r} \frac{\partial u_{r}}{\partial r} dz, \\
I_{\Gamma_{out}} &= \int_{\Gamma_{out}} u_{z} \frac{\partial u_{z}}{\partial z} + u_{r} \frac{\partial u_{z}}{\partial r} dr. \\
\end{align*}
We handle these integrals as follows.
\begin{itemize}
\item For $I_{\Gamma}$, we observe that $u_{z} = 0$ on $\Gamma$. Thus, using the divergence free condition, and the fact that $\frac{\partial u_{z}}{\partial z} = 0$ on $\Gamma$, we have that 
\begin{equation*}
I_{\Gamma} = -\int_{\Gamma} u_{r} \frac{\partial u_{z}}{\partial z} dz = 0.
\end{equation*}
\item For $I_{\Gamma_{in}}$, we have that $u_{r} = 0$ on $\Gamma_{in}$. So by the divergence free condition and the fact that $\frac{\partial u_{r}}{\partial r} = 0$ on $\Gamma_{in}$, we obtain
\begin{equation*}
I_{\Gamma_{in}} = \int_{\Gamma_{in}} u_{z} \frac{\partial u_{r}}{\partial r} dr = 0.
\end{equation*}
An analogous argument shows that $I_{\Gamma_{out}} = 0$. 
\item For $I_{\Gamma_{b}}$, we have that $u_{z} = 0$ and $\frac{\partial u_{z}}{\partial r} = 0$ on $\Gamma_{b}$, so $I_{\Gamma_{b}} = 0$. 
\end{itemize}
This concludes the proof of the Korn equality. 
\end{proof}

The next lemma we will need is a spectral result about functions satisfying the kinematic coupling condition. 

\fi

\begin{lemma}[Uniqueness for the deterministic problem]\label{detunique}
Suppose that $\boldsymbol{u} \in L^{\infty}(0, T; L^{2}(\Omega_{f})) \cap L^{2}(0, T; \mathcal{V}_{F})$, $\eta \in W^{1, \infty}(0, T; L^{2}(\Gamma)) \cap L^{\infty}(0, T; \mathcal{V}_{S})$, and $\boldsymbol{u}|_{\Gamma} = \partial_{t}\eta \boldsymbol{e}_{r}$. Suppose also that $(\boldsymbol{u}, \partial_{t}\eta) \in C(0, T; \mathcal{Q}')$, with $\eta(0) = 0$. If for all $(\boldsymbol{q}, \psi) \in \mathcal{Q}(0, T)$, 
\begin{equation*}
- \int_{0}^{T} \int_{\Omega_{f}} \boldsymbol{u} \cdot \partial_{t}\boldsymbol{q} d\boldsymbol{x} dt + 2\mu \int_{0}^{T} \int_{\Omega_{f}} \boldsymbol{D}(\boldsymbol{u}) : \boldsymbol{D}(\boldsymbol{q}) d\boldsymbol{x} dt - \int_{0}^{T} \int_{\Gamma} \partial_{t}\eta\partial_{t}\psi dz dt + \int_{0}^{T} \int_{\Gamma} \nabla \eta \cdot \nabla \psi dz dt = 0,
\end{equation*}
then $(\boldsymbol{u}, \eta) = 0$. 
\end{lemma}

\if 1 = 0 

\begin{proof}
The idea is that we want to substitute in $(\boldsymbol{q}, \psi) = (\boldsymbol{u}, \partial_{t}\eta)$ into the above equality and use the resulting energy inequality. However, there is not enough regularity to justify this. Therefore, we smooth out using convolution in time against a compactly supported function. 

Define an even compactly supported function $\phi \in C_{0}^{\infty}(\mathbb{R})$ such that $\phi$ is even, $\phi = 0$ for $|x| \ge 1$, $\phi > 0$ for $x \in (-1, 1)$, and 
\begin{equation*}
\int_{\mathbb{R}} \phi(t) dt = 1.
\end{equation*}
We then define for $\epsilon > 0$, 
\begin{equation*}
\phi_{\epsilon}(t) := \frac{1}{\epsilon} \phi\left(\frac{t}{\epsilon}\right),
\end{equation*}
so that $\phi_{\epsilon}(t)$ is supported in $[-\epsilon, \epsilon]$. We then define the regularized solution for $t \in [\epsilon, T - \epsilon]$ by 
\begin{equation*}
\boldsymbol{u}_{\epsilon}(t, \cdot) = \int_{\mathbb{R}} \phi_{\epsilon}(t - s) \boldsymbol{u}(s, \cdot) ds = \int_{0}^{T} \phi_{\epsilon}(t - s) \boldsymbol{u}(s, \cdot) ds,
\end{equation*}
\begin{equation*}
\eta_{\epsilon}(t, \cdot) = \int_{\mathbb{R}} \phi_{\epsilon}(t - s) \eta(s, \cdot) ds = \int_{0}^{T} \phi_{\epsilon}(t - s) \eta(s, \cdot) ds.
\end{equation*}
Note that since $\boldsymbol{u} = \partial_{t}\eta$ for all every $t \in [0, T]$, we have that $\boldsymbol{u}_{\epsilon} = \partial_{t}\eta_{\epsilon}$ for all $t \in [\epsilon, T - \epsilon]$. We now proceed with the uniqueness proof in the following steps.

\vspace{0.2in}

\noindent \textbf{Step 1:} We first verify the following identity for the regularized solutions. We claim that for all $(\boldsymbol{q}, \psi) \in \mathcal{Q}(0, T)$ such that 
\begin{equation*}
(\boldsymbol{q}(t), \psi(t)) = 0, \qquad \text{ for all } t \in [0, \epsilon] \cup [T - \epsilon, T],
\end{equation*}
we have that
\begin{multline}\label{step1unique}
- \int_{\epsilon}^{T - \epsilon} \int_{\Omega_{f}} \boldsymbol{u}_{\epsilon} \cdot \partial_{t}\boldsymbol{q} d\boldsymbol{x} dt + 2\mu \int_{\epsilon}^{T - \epsilon} \int_{\Omega_{f}} \boldsymbol{D}(\boldsymbol{u}_{\epsilon}) : \boldsymbol{D}(\boldsymbol{q}) d\boldsymbol{x} dt \\
- \int_{\epsilon}^{T - \epsilon} \int_{\Gamma} \partial_{t}\eta_{\epsilon}\partial_{t}\psi dz dt + \int_{\epsilon}^{T - \epsilon} \int_{\Gamma} \nabla \eta_{\epsilon} \cdot \nabla \psi dz dt = 0.
\end{multline}
Define for arbitrary $(\boldsymbol{q}, \psi) \in \mathcal{Q}(0, T)$ and $s \in [0, T]$, the regularized test functions, 
\begin{equation*}
\boldsymbol{q}_{\epsilon}(s) := \int_{\epsilon}^{T - \epsilon} \phi_{\epsilon}(t - s) \boldsymbol{q}(t) dt, \qquad \psi_{\epsilon}(s) := \int_{\epsilon}^{T - \epsilon} \phi_{\epsilon}(t - s) \psi(t) dt,
\end{equation*}
and note that $(\boldsymbol{q}_{\epsilon}, \psi_{\epsilon}) \in \mathcal{Q}(0, T)$. The left hand side of \eqref{step1unique}, by using Fubini's theorem to move the convolution onto the test function, is equal to
\begin{equation*}
-\int_{0}^{T} \int_{\Omega_{f}} \boldsymbol{u} \cdot \partial_{t}\boldsymbol{q}_{\epsilon} d\boldsymbol{x} dt + 2\mu \int_{0}^{T} \int_{\Omega_{f}} \boldsymbol{D}(\boldsymbol{u}) : \boldsymbol{D}(\boldsymbol{q}_{\epsilon}) d\boldsymbol{x} dt - \int_{0}^{T} \int_{\Gamma} \partial_{t}\eta \partial_{t}\psi_{\epsilon} dz dt + \int_{0}^{T} \int_{\Gamma} \nabla \eta \cdot \nabla \psi_{\epsilon} dz dt.
\end{equation*}
However, since $(\boldsymbol{q}_{\epsilon}, \psi_{\epsilon}) \in \mathcal{Q}(0, T)$, this is equal to zero by assumption, which establishes the desired identity \eqref{step1unique}.

\vspace{0.2in}

\noindent \textbf{Step 2:} We claim that $(\boldsymbol{u}_{\epsilon}, \eta_{\epsilon})$ satisfies the following identity for all $(\boldsymbol{q}, \psi) \in \mathcal{Q}(0, T)$:
\begin{multline}\label{step2unique}
- \int_{\epsilon}^{T - \epsilon} \int_{\Omega_{f}} \boldsymbol{u}_{\epsilon} \cdot \partial_{t}\boldsymbol{q} d\boldsymbol{x} dt + 2\mu \int_{\epsilon}^{T - \epsilon} \int_{\Omega_{f}} \boldsymbol{D}(\boldsymbol{u}_{\epsilon}) : \boldsymbol{D}(\boldsymbol{q}) d\boldsymbol{x} dt \\
- \int_{\epsilon}^{T - \epsilon} \int_{\Gamma} \partial_{t}\eta_{\epsilon}\partial_{t}\psi dz dt + \int_{\epsilon}^{T - \epsilon} \int_{\Gamma} \nabla \eta_{\epsilon} \cdot \nabla \psi dz dt = \int_{\Omega_{f}} \boldsymbol{u}_{\epsilon}(\epsilon) \cdot \boldsymbol{q}(\epsilon) d\boldsymbol{x} + \int_{\Gamma} \partial_{t}\eta_{\epsilon}(\epsilon)\phi(\epsilon) dz \\
- \int_{\Omega_{f}} \boldsymbol{u}_{\epsilon}(T - \epsilon) \cdot \boldsymbol{q}(T - \epsilon) d\boldsymbol{x} - \int_{\Gamma} \partial_{t}\eta_{\epsilon}(T - \epsilon)\phi(T - \epsilon) dz.
\end{multline}

To see this, we can use a limiting argument. Let us define the cutoff functions $\beta_{\epsilon, \delta}: [0, T] \to \mathbb{R}$ for $\delta > 0$ as follows. 
\begin{equation*}
\beta_{\delta}(t) = 0, \qquad \text{ if } t \in [0, \epsilon] \cup [T - \epsilon, T],
\end{equation*}
\begin{equation*}
\beta_{\delta}(t) = 1, \qquad \text{ if } t \in [\epsilon + \delta, T - \epsilon - \delta],
\end{equation*}
\begin{equation*}
\beta_{\delta}(t) = \int_{-\delta}^{2(t - \epsilon) - \delta} \phi_{\delta/2}(s) ds, \qquad \text{ if } t \in [\epsilon, \epsilon + \delta],
\end{equation*}
\begin{equation*}
\beta_{\delta}(t) = 1 - \int_{-\delta}^{2(t - T + \epsilon) + \delta} \phi_{\delta/2}(s) ds, \qquad \text{ if } t \in [T - \epsilon - \delta, T - \epsilon].
\end{equation*}

To prove the identity \eqref{step2unique}, we note that for our fixed but arbitrary $\epsilon > 0$, and for any $\delta > 0$ and $(\boldsymbol{q}, \psi) \in \mathcal{Q}(0, T)$, we have that
\begin{equation*}
(\boldsymbol{q}_{\epsilon, \delta}, \psi_{\epsilon, \delta}) := (\beta_{\epsilon, \delta}(t)\boldsymbol{q}(t), \beta_{\epsilon, \delta}(t)\psi(t)) \in \mathcal{Q}(0, T)
\end{equation*}
and is equal to zero for $t \in [0, \epsilon] \cup [T - \epsilon, T]$ by construction. So by the result of Step 1, we have that
\begin{multline}\label{step2uniqueapprox}
- \int_{\epsilon}^{T - \epsilon} \int_{\Omega_{f}} \boldsymbol{u}_{\epsilon} \cdot \partial_{t}\boldsymbol{q}_{\epsilon, \delta} d\boldsymbol{x} dt + 2\mu \int_{\epsilon}^{T - \epsilon} \int_{\Omega_{f}} \boldsymbol{D}(\boldsymbol{u}_{\epsilon}) : \boldsymbol{D}(\boldsymbol{q}_{\epsilon, \delta}) d\boldsymbol{x} dt \\
- \int_{\epsilon}^{T - \epsilon} \int_{\Gamma} \partial_{t}\eta_{\epsilon}\partial_{t}\psi_{\epsilon, \delta} dz dt + \int_{\epsilon}^{T - \epsilon} \int_{\Gamma} \nabla \eta_{\epsilon} \cdot \nabla \psi_{\epsilon, \delta} dz dt = 0.
\end{multline}
However, one obtains the desired equality \eqref{step2unique} by taking the limit as $\delta \to 0$. This is because $\partial_{t}\boldsymbol{q}_{\epsilon, \delta}$ for example, defined on $t \in [\epsilon, T - \epsilon]$, converges as $\delta \to 0$ to 
\begin{equation*}
\boldsymbol{q}(\epsilon)\delta_{t = \epsilon} - \boldsymbol{q}(T - \epsilon) \delta_{t = T - \epsilon} + 1_{t \in [\epsilon, T - \epsilon]} \partial_{t}\boldsymbol{q},
\end{equation*}
where $\delta_{t = \epsilon}$ for example denotes a delta function in time at $t = \epsilon$. 

\vspace{0.2in}

\noindent \textbf{Step 3:} We now obtain an energy estimate for the regularized solutions $(\boldsymbol{u}_{\epsilon}, \eta_{\epsilon})$. Because $(\boldsymbol{u}_{\epsilon}, \eta_{\epsilon})$ is now infinitely differentiable in time due to the convolution, we can integrate by parts backwards in \eqref{step2unique} to obtain that for all $(q, \psi) \in \mathcal{Q}(0, T)$, 
\begin{multline*}
\int_{\epsilon}^{T - \epsilon} \int_{\Omega_{f}} \partial_{t} \boldsymbol{u}_{\epsilon} \cdot \boldsymbol{q} d\boldsymbol{x} dt + 2\mu \int_{\epsilon}^{T - \epsilon} \int_{\Omega_{f}} \boldsymbol{D}(\boldsymbol{u}_{\epsilon}) : \boldsymbol{D}(\boldsymbol{q}) d\boldsymbol{x} dt \\
+ \int_{\epsilon}^{T - \epsilon} \int_{\Gamma} (\partial_{tt} \eta_{\epsilon}) \psi dz dt + \int_{\epsilon}^{T - \epsilon} \int_{\Gamma} \nabla \eta_{\epsilon} \cdot \nabla \psi dz dt = 0.
\end{multline*}

\end{proof}

\fi

\begin{proof}
Observe first that to get the usual energy equality, we would want to formally substitute in $(\boldsymbol{u}, \partial_{t}\eta)$ for $(\boldsymbol{q}, \psi)$. However, since $(\boldsymbol{q}, \psi)$ must have $\psi(t) \in H_{0}^{1}(\Gamma)$ by the definition of the test space $\mathcal{Q}(0, T)$, we do not have enough regularity to do this. Therefore, we use a different approach of taking an antiderivative, which is an approach used for example in establishing uniqueness of weak solutions for general hyperbolic equations (see Section 7.2 in \cite{Evans}).

Consider an arbitrary $s$ such that $0 \le s \le T$. We use the following test function,
\begin{equation*}\label{wavefundamental}
(\boldsymbol{q}_{0}(t), \psi_{0}(t)) = 
\left\{
\begin{array}{cl}
\displaystyle{\left(\int_{t}^{s} \left(\int_{0}^{\tau} \boldsymbol{u}(\sigma) d\sigma\right) d\tau, \int_{t}^{s} \eta(\tau) d\tau \right)} \qquad &\text{ if } 0 \le t \le s,
\\
\displaystyle{(0, 0)} \qquad &\text{ if } s \le t \le T.
\end{array}
\right.
\end{equation*}
Recall that $\eta(0) = 0$ by assumption. Note that since
\begin{equation*}
\int_{0}^{\tau} \boldsymbol{u}(\sigma) d\sigma \Big\vert_{\Gamma} = \int_{0}^{\tau} \partial_{t}\eta(\sigma) d\sigma = \eta(\tau)
\end{equation*}
for all $\tau \in [0, T]$, the function $(\boldsymbol{q}_{0}, \psi_{0})$ satisfies the necessary kinematic coupling condition for $\mathcal{Q}(0, T)$. While this test function is only piecewise differentiable, it is easy to show by an approximation argument that the weak formulation should still hold with this test function by approximating it with differentiable functions. For notational simplicity, we define
\begin{equation*}
\boldsymbol{U}(t) = \int_{0}^{t} \boldsymbol{u}(\sigma) d\sigma.
\end{equation*} 

Substituting the test function into the weak formulation, we obtain for all $s \in [0, T]$, 
\begin{equation*}
\int_{0}^{s} \int_{\Omega_{f}} \boldsymbol{u} \cdot \boldsymbol{U} d\boldsymbol{x} dt + 2\mu \int_{0}^{s} \int_{\Omega_{f}} \boldsymbol{D}(\boldsymbol{u}) : \boldsymbol{D}(\boldsymbol{q}_{0}) d\boldsymbol{x} dt + \int_{0}^{s} \int_{\Gamma} \partial_{t}\eta \cdot \eta dz dt + \int_{0}^{s} \int_{\Gamma} \nabla \eta \cdot \nabla \psi_{0} dz dt = 0,
\end{equation*}
where we note that $\partial_{t}\boldsymbol{q}_{0}(t) = -U(t)$ and $\partial_{t}\psi_{0}(t) = -\eta(t)$, for $t \in [0, s)$. We handle the four terms on the left hand side as follows.

\begin{itemize}
\item \textbf{First term:} We note that $\boldsymbol{u} = \partial_{t}\boldsymbol{U}$. Hence, using the fact that $\boldsymbol{U}(0) = 0$, we get
\begin{gather*}
\int_{0}^{s} \int_{\Omega_{f}} \boldsymbol{u} \cdot \boldsymbol{U} d\boldsymbol{x} dt = \int_{0}^{s} \frac{d}{dt} \left(\frac{1}{2} ||\boldsymbol{U}||_{L^{2}(\Omega_{f})}^{2}\right) dt 
= \frac{1}{2} ||\boldsymbol{U}(s)||_{L^{2}(\Omega_{f})}^{2} - \frac{1}{2} ||\boldsymbol{U}(0)||_{L^{2}(\Omega_{f})}^{2} = \frac{1}{2} ||\boldsymbol{U}(s)||_{L^{2}(\Omega_{f})}^{2}.
\end{gather*}
\item \textbf{Second term:} For the second term, we again use that $\boldsymbol{u} = \partial_{t}\boldsymbol{U}$. Therefore,
\begin{equation*}
2\mu \int_{0}^{s} \int_{\Omega_{f}} \boldsymbol{D}(\boldsymbol{u}) : \boldsymbol{D}(\boldsymbol{q}_{0}) d\boldsymbol{x} dt = 2\mu \int_{0}^{s} \int_{\Omega_{f}} \boldsymbol{D}(\partial_{t}\boldsymbol{U}) : \boldsymbol{D}(\boldsymbol{q}_{0}) d\boldsymbol{x} dt.
\end{equation*}
We integrate by parts in time. Note that $\boldsymbol{U}(0) = 0$ and $\boldsymbol{q}_{0}(s) = 0$, so there are no boundary terms from the integration by parts. Hence, using the fact that $\partial_{t}\boldsymbol{q}_{0} = -\boldsymbol{U}$, we obtain
\begin{equation*}
2\mu \int_{0}^{s} \int_{\Omega_{f}} \boldsymbol{D}(\boldsymbol{u}) : \boldsymbol{D}(\boldsymbol{q}_{0}) d\boldsymbol{x} dt = -2\mu \int_{0}^{s} \int_{\Omega_{f}} \boldsymbol{D}(\boldsymbol{U}) : \boldsymbol{D}(\partial_{t}\boldsymbol{q}_{0}) d\boldsymbol{x} dt = 2\mu \int_{0}^{s} \int_{\Omega_{f}} |\boldsymbol{D}(\boldsymbol{U})|^{2} d\boldsymbol{x} dt.
\end{equation*}
\item \textbf{Third term:} We immediately have that
\begin{equation*}
\int_{0}^{s} \int_{\Gamma} \partial_{t}\eta \cdot \eta dz dt = \frac{1}{2}||\eta(s)||_{L^{2}(\Gamma)}^{2} - \frac{1}{2}||\eta(0)||_{L^{2}(\Gamma)}^{2} = \frac{1}{2}||\eta(s)||_{L^{2}(\Gamma)}^{2}.
\end{equation*}
\item \textbf{Fourth term:} Since $\eta = -\partial_{t}\psi_{0}$, we have that
\begin{equation*}
\int_{\Gamma} \nabla \eta \cdot \nabla \psi_{0} dz = -\frac{1}{2} \frac{d}{dt}\left(||\nabla \psi_{0}||_{L^{2}(\Gamma)}^{2}\right),
\end{equation*}
and hence, using the fact that $\psi_{0}(s) = 0$, we get that 
\begin{equation*}
\int_{0}^{s} \int_{\Gamma} \nabla \eta \cdot \nabla \psi_{0} dz dt = \frac{1}{2}||\nabla \psi_{0}(0)||^{2}_{L^{2}(\Gamma)}.
\end{equation*}
\end{itemize}

Therefore, for all $0 \le s \le T$,  the entire expression (energy) can now be written as
\begin{equation*}
\frac{1}{2}||\boldsymbol{U}(s)||^{2}_{L^{2}(\Omega_{f})} + 2\mu \int_{0}^{s} \int_{\Omega_{f}} |\boldsymbol{D}(\boldsymbol{U})|^{2} d\boldsymbol{x} dt + \frac{1}{2}||\eta(s)||^{2}_{L^{2}(\Gamma)} + \frac{1}{2}||\nabla \psi_{0}(0)||^{2}_{L^{2}(\Gamma)} = 0.
\end{equation*}
Thus, we conclude that $\boldsymbol{U}(s) = 0$ and $\eta(s) = 0$ for all $s \in [0, T]$. From the definition of $\boldsymbol{U}$, we conclude that $\boldsymbol{u}(t) = \partial_{t}\boldsymbol{U}(t) = 0$ for all $t \in [0, T]$ also, which completes the proof. 
\end{proof}

\subsection{Verifying the diagonal condition of the Gy\"{o}ngy-Krylov lemma}\label{GKlemma}

Now that we have established a uniqueness result, we can construct a solution on the original probability space $(\Omega, \mathcal{F}, \mathbb{P})$ by invoking a standard argument involving the Gy\"{o}ngy-Krylov argument (Lemma \ref{GK}), to show that the random variables 
$(\eta_{N}, \boldsymbol{u}_{N}, v_{N})$ defined on the original probability space converge in probability, and hence converge almost surely along a subsequence \textit{in the original topology}. 

\if 1 = 0

To do this, we first recall the Gy\"{o}ngy-Krylov lemma, which is Lemma 1.1 in \cite{GK}.

\if 1 = 0
\begin{lemma}[Gy\"{o}ngy-Krylov lemma]
Let $\{X_{n}\}_{n = 1}^{\infty}$ be a sequence of random variables defined on a probability space $(\Omega, \mathcal{F}, \mathbb{P})$ taking values in a Banach space $B$. For positive integers $m$ and $n$, define the joint probability measures $\nu_{m, n}$ on $B \times B$ by 
\begin{equation*}
\nu_{m, n}(A_{1} \times A_{2}) = \mathbb{P}(X_{m} \in A_{1}, X_{n} \in A_{2}).
\end{equation*}
Suppose that the following \textit{diagonal condition} holds: for any arbitrary subsequences $\{m_{k}\}_{k = 1}^{\infty}$ and $\{n_{k}\}_{k = 1}^{\infty}$, there exists a further subsequence such that the joint probability laws $\nu_{m_{k_{l}}, n_{k_{l}}}$ along this subsequence
converge weakly, as $l \to \infty$, to a limiting probability measure $\nu$, such that
\begin{equation*}
\nu(\Delta) = 1,\ {\text{where}} \  \Delta = \{{\color{blue}{(X, X) : X \in B}}\} \ \text{denotes the diagonal of}\  B \times B.
\end{equation*}
 Then, $X_{n}$ converges in probability to some $B$-valued random variable {\color{blue}{$X^*$}} as $n \to \infty$. 
\end{lemma}
\fi

{\color{blue}{
\begin{lemma}[Gy\"{o}ngy-Krylov lemma]
Let $X_{n}, {n = 1}\dots{\infty}$, be a sequence of random variables defined on a probability space $(\Omega, \mathcal{F}, \mathbb{P})$ taking values in a Banach space $B$. Then {\bf{$X_n$ converges in probability}} to some $B$-valued random variable $X^*$
if and only if for every two subsequences $X_l$ and $X_m$ of $X_{n}$, there exists a subsequence $x_k = (X_{l_k},X_{m_k})$ in $B\times B$,
converging weakly to a random variable $x^*$ supported on the diagonal  
$\{(x, y) : x = y\}$.
\end{lemma}

In other words, the statement of Gy\"{o}ngy-Krylov lemma holds if and only if for every two subsequences $X_l$ and $X_m$,
there exists a further subsequence $x_k = (X_{l_k},X_{m_k})$, such that 
the joint probability measures $\nu_{l,m}$, defined by 
\begin{equation*}
\nu_{X_{l_k},X_{m_k}}(A_{1} \times A_{2}) = \mathbb{P}(X_{l_k} \in A_{1}, X_{m_k} \in A_{2}).
\end{equation*}
converge to some probability measure $\nu$ as $k\to\infty$:
$$
\nu_{X_{l_k},X_{m_k}} \to \mu, \ {\rm{as}} \ k\to\infty,
$$
where $\nu$ is such that 
$$
\nu(\{(X, Y) : X = Y\}) = 1.
$$
Thus, the limits of any two convergent subsequences have to be ``the same'' with probability 1. 

To verify this diagonal condition, we will show {\emph{deterministic uniqueness}}, which holds pathwise, and combine it with 
the Skorohod's embedding theorem. 
}}

We will use this lemma to {\color{blue}{show that the random variables $(\eta_{N}, \boldsymbol{u}_{N}, v_{N})$ converge 
in probability, and then almost surely, in the original topology}} . The following argument using the previous Gy\"{o}ngy-Krylov lemma,
will be used, see for example \cite{LNT}. {\color{blue}{For each $M,N = 1, 2, \dots$ we define 
\begin{equation*}
\nu_{M, N} = \mu_{M} \times \mu_{N}
\end{equation*}
to be the probability measure defined on phase space $\mathcal{X} \times \mathcal{X}$, where we recall the definitions of $\mathcal{X}$ and $\mu_{N}$ from \eqref{phase} and \eqref{muN}. By the tightness of $\{\mu_{N}\}_{N = 1}^{\infty}$ on $\mathcal{X}$ proved in Lemma \ref{KRcompact}, we have that for any two given subsequences $M_{k}$ and $N_{k}$, the collection $\{\nu_{M_{k}, N_{k}}\}_{k = 1}^{\infty}$ is also tight and hence has a weakly convergence subsequence, which we will continue to denote by $\nu_{M_{k}, N_{k}}$, which converges to some limiting probability measure $\nu$ on $\mathcal{X} \times \mathcal{X}$. We want to show that 
\begin{equation}\label{diagonal}
\nu(\{(X, X) : X \in \mathcal{X}\}) = 1.
\end{equation}
{\color{pink}{This will imply that $(\eta_{N}, \boldsymbol{u}_{N}, v_{N})$ converge 
in probability, and then almost surely to some $(\eta^*, \boldsymbol{u}^*, v^*)$, in the original topology. However, we do not know that
the realizations $(\eta^*, \boldsymbol{u}^*, v^*)$ satisfy the weak form of our original problem. We want to show that they satisfy the 
weak form almost surely, for almost all realizations $(\eta^*, \boldsymbol{u}^*, v^*)$.
To do that, we would 
}}
\fi

\if 1 = 0
{\color{red} {To verify \eqref{diagonal}, we will use the approach of translating the diagonal condition \eqref{diagonal} from the Gy\"{o}ngy-Krylov lemma to a condition on deterministic uniqueness, which can be much more readily verified. In particular, by invoking the Skorokhod embedding theorem on the ordered pairs of appropriate solutions, $(X_{M}, X_{N})$, where $X_{M} = (\eta_{M}, \overline{\eta}_{M}, \boldsymbol{u}_{M}, v_{M}, \boldsymbol{u}_{M}, v_{M}^{*}, \overline{\boldsymbol{u}}_{M}, \overline{v}_{M}, W)$ for example and similarly with $X_{N}$, we can show that the diagonal condition \eqref{diagonal} is equivalent to showing uniqueness in law of solutions to the original linear stochastic fluid-structure interaction problem, which is equivalent to showing uniqueness of solutions to the \textit{deterministic} analogue of this linear fluid-structure interaction problem due to the fact that the stochastic noise is additive. Because we have already shown deterministic uniqueness in Sec.~\ref{deterministic}, it only remains to demonstrate how the Skorokhod embedding theorem can be used to show that the diagonal condition \eqref{diagonal} is equivalent to showing deterministic uniqueness, and we will explicitly demonstrate this now.}
}
\fi
Because we have already shown deterministic uniqueness in Sec.~\ref{deterministic}, it only remains to demonstrate how the Skorohod representation theorem can be used to show that the diagonal condition \eqref{diagonal} from the Gy\"{o}ngy-Krylov lemma
is equivalent to showing deterministic uniqueness.

For this purpose, denote by $\{X^1_{M_{k}}\}_{k = 1}^{\infty}$ and $\{X^2_{N_{k}}\}_{k = 1}^{\infty}$ any two subsequences of our random variables (approximate solutions) 
defined on the original probability space $({\Omega}, {\mathcal{F}}, {\mathbb{P}})$:
\begin{align*}
X^1_{M_{k}} &= (\eta_{M_{k}}, \overline{\eta}_{M_{k}}, \eta^{\Delta t}_{M_{k}}, \boldsymbol{u}_{M_{k}}, v_{M_{k}}, \boldsymbol{u}_{M_{k}}, v_{M_{k}}^{*}, \overline{\boldsymbol{u}}_{M_{k}}, \overline{v}_{M_{k}}, \boldsymbol{u}^{\Delta t}_{M_{k}}, v^{\Delta t}_{M_{k}}, W), \\
X^2_{N_{k}} &= (\eta_{N_{k}}, \overline{\eta}_{N_{k}}, \eta^{\Delta t}_{N_{k}}, \boldsymbol{u}_{N_{k}}, v_{N_{k}}, \boldsymbol{u}_{N_{k}}, v_{N_{k}}^{*}, \overline{\boldsymbol{u}}_{N_{k}}, \overline{v}_{N_{k}}, \boldsymbol{u}^{\Delta t}_{N_{k}}, v^{\Delta t}_{N_{k}}, W).
\end{align*}
Recall that the laws corresponding to each of these these two sequences of random variables \textit{individually} converge to the law $\mu$. However, to verify the diagonal condition in the Gy\"{o}ngy-Krylov lemma, we must examine the \textit{joint laws} of these random variables $(X^{1}_{M_{k}}, X^{2}_{N_{k}})$. 

Hence, we consider the joint probability measures (or joint laws) $\{\nu_{X_{M_{k}}^{1}, X_{N_{k}}^{2}}\}_{k = 1}^{\infty}$ on $\mathcal{X} \times \mathcal{X}$, associated with the subsequence $(X_{M_{k}}^{1}, X_{N_{k}}^{2})$. By the tightness of the original probability measures $\mu_{N}$, established in the proof of Theorem \ref{weakconv}, we have that the collection of joint laws $\{\nu_{X_{M_{k}}^{1},  X_{N_{k}}^{2}}\}_{k = 1}^{\infty}$ is also tight, and hence converges weakly to a probability measure $\nu$ on $\mathcal{X} \times \mathcal{X}$ along a further subsequence, which we will continue to denote by the same indexing for notational simplicity.
Then, by the Skorohod representation theorem, there exists a probability space $(\tilde{\Omega}, \tilde{\mathcal{F}}, \tilde{\mathbb{P}})$ and random variables
\begin{align*}
\tilde{X}_{M_{k}}^{1} &= (\tilde{\eta}_{M_{k}}^{1}, \tilde{\overline{\eta}}_{M_{k}}^{1}, \tilde{\eta}^{\Delta t, 1}_{M_{k}}, \tilde{\boldsymbol{u}}_{M_{k}}^{1}, \tilde{v}_{M_{k}}^{1}, \tilde{\boldsymbol{u}}_{M_{k}}^{*, 1}, \tilde{v}_{M_{k}}^{*, 1}, \tilde{\overline{\boldsymbol{u}}}_{M_{k}}^{1}, \tilde{\overline{v}}_{M_{k}}^{1}, \tilde{\boldsymbol{u}}^{\Delta t, 1}_{M_{k}}, \tilde{v}^{\Delta t, 1}_{M_{k}}, \tilde{W}_{M_{k}}^{1}),
\\
\tilde{X}_{N_{k}}^{2} &= (\tilde{\eta}_{N_{k}}^{2}, \tilde{\overline{\eta}}_{N_{k}}^{2}, \tilde{\eta}^{\Delta t, 2}_{N_{k}}, \tilde{\boldsymbol{u}}_{N_{k}}^{2}, v_{N_{k}}^{2}, \tilde{\boldsymbol{u}}_{N_{k}}^{*, 2}, \tilde{v}_{N_{k}}^{*, 2}, \tilde{\overline{\boldsymbol{u}}}_{N_{k}}^{2}, \tilde{\overline{v}}_{N_{k}}^{2}, \tilde{\boldsymbol{u}}^{\Delta t, 2}_{N_{k}}, \tilde{v}^{\Delta t, 2}_{N_{k}}, \tilde{W}_{N_{k}}^{2}),
\end{align*}
such that
\begin{equation}\label{jointequiv}
(\tilde{X}^{1}_{M_{k}}, \tilde{X}^{2}_{N_{k}}) =_{d} (X^{1}_{M_{k}}, X^{2}_{N_{k}}),
\end{equation}
and
$(\tilde{X}^{1}_{M_{k}}, \tilde{X}^{2}_{N_{k}}) \to (\tilde{X}^{1}, \tilde{X}^{2})$ in 
$\mathcal{X} \times \mathcal{X}$ almost surely as $k \to \infty$, where
\begin{align*}
\tilde{X}^{1} = (\tilde{\eta}^{1}, \tilde{\overline{\eta}}^{1}, \tilde{\eta}^{\Delta t, 1}, \tilde{\boldsymbol{u}}^{1}, \tilde{v}^{1}, \tilde{\boldsymbol{u}}^{*, 1}, \tilde{v}^{*, 1}, \tilde{\overline{\boldsymbol{u}}}^{1}, \tilde{\overline{v}}^{1}, \tilde{\boldsymbol{u}}^{\Delta t, 1}, \tilde{v}^{\Delta t, 1}, \tilde{W}^{1}),
\\
\tilde{X}^{2} = (\tilde{\eta}^{2}, \tilde{\overline{\eta}}^{2}, \tilde{\eta}^{\Delta t, 2}, \tilde{\boldsymbol{u}}^{2}, \tilde{v}^{2}, \tilde{\boldsymbol{u}}^{*, 2}, \tilde{v}^{*, 2}, \tilde{\overline{\boldsymbol{u}}}^{2}, \tilde{\overline{v}}^{2}, \tilde{\boldsymbol{u}}^{\Delta t, 2}, \tilde{v}^{\Delta t, 2}, \tilde{W}^{2}),
\end{align*} 
are random variables
on $(\tilde{\Omega}, \tilde{\mathcal{F}}, \tilde{\mathbb{P}})$, and $\nu$ is the law of $(\tilde{X}_{1}, \tilde{X}_{2})$. 

We want to show that $\nu$ is supported on the diagonal. It suffices to show that
$\tilde{\mathbb{P}}(\tilde{X}^{1} = \tilde{X}^{2}) =1$.
We do this in three steps.

{\bf{Step 1.}} First we notice that 
\textit{$\tilde{X}^{1}$ is a weak solution in a probabilistically weak sense with respect to the stochastic basis $(\tilde{\Omega}, \tilde{\mathcal{F}}, \{\tilde{\mathcal{F}}_{t}^{1}\}_{t \ge 0}, \tilde{\mathbb{P}}, \tilde{W}_{1})$} in the sense of Definition \ref{weak}. This follows from the results of Lemma \ref{properties}. 
Namely,  the results of Lemma \ref{properties}
imply that
$\tilde{\eta}^{1} = \tilde{\overline{\eta}}^{1} = \tilde{\eta}^{\Delta t, 1}$, $\tilde{\boldsymbol{u}}^{1} = \tilde{\boldsymbol{u}}^{*, 1} = \tilde{\overline{\boldsymbol{u}}}^{1} = \tilde{\boldsymbol{u}}^{\Delta t, 1}$, $\tilde{v}^{1} = \tilde{v}^{*, 1} = \tilde{\overline{v}}^{1} = \tilde{v}^{\Delta t, 1}$, and $\partial_{t}\tilde{\eta}^{1} = \tilde{v}^{1}$ almost surely.
Furthermore, $(\tilde{\boldsymbol{u}}^{1}, \tilde{\eta}^{1}) \in  \mathcal{W}(0, T)$ and $(\tilde{\boldsymbol{u}}^{1}, \tilde{v}^{1}) \in C(0, T; \mathcal{Q}')$, satisfying the initial condition $\tilde{\eta}^{1}(0) = \eta_{0}$ almost surely.
Furthermore, {$\tilde{X}^{1}$ is a weak solution in a probabilistically weak sense with respect to the stochastic basis $(\tilde{\Omega}, \tilde{\mathcal{F}}, \{\tilde{\mathcal{F}}_{t}^{1}\}_{t \ge 0}, \tilde{\mathbb{P}}, \tilde{W}_{1})$} in the sense of Definition \ref{weak}. The same is true for the components of $\tilde{X}^{2}$, with respect to $(\tilde{\Omega}, \tilde{\mathcal{F}}, \{\tilde{\mathcal{F}}_{t}^{2}\}_{t \ge 0}, \tilde{\mathbb{P}}, \tilde{W}_{2})$. Here, the filtrations $\{\tilde{\mathcal{F}}_{t}^{1}\}_{t \ge 0}$ and $\{\tilde{\mathcal{F}}_{t}^{2}\}_{t \ge 0}$ are defined by \eqref{tildefiltration} with the appropriate limiting random variables with superscripts ``1" and ``2" respectively. 

{\bf{Step 2.}} Here we notice  that the limiting white noise satisfies $\tilde{W}_{1} = \tilde{W}_{2}$. This follows directly from 
\eqref{jointequiv}, which implies 
$\tilde{W}^{1}_{M_{k}} = \tilde{W}^{2}_{N_{k}}$ almost surely, since the law of $(\tilde{W}^{1}_{M_{k}}, \tilde{W}^{2}_{N_{k}})$ is the same as that of $(W, W)$. 
Thus, by the convergence of $\tilde{W}^{1}_{M_{k}}$ and $\tilde{W}^{2}_{N_{k}}$ in $C(0, T; \mathbb{R})$ almost surely to $\tilde{W}^{1}$ and $\tilde{W}^{2}$, we have that $\tilde{W}^{1} = \tilde{W}^{2}$ almost surely in $C(0, T; \mathbb{R})$. This will allow us to make sense of the difference of the stochastic integrals with respect to $\tilde{W}_{1}$ and $\tilde{W}_{2}$ in the weak formulations on the ``tilde" probability space.

{\bf{Step 3.}} Finally, we use deterministic uniqueness to obtain the diagonal condition. 
We consider the difference $(\tilde{\eta}^{1} - \tilde{\eta}^{2}, \tilde{\boldsymbol{u}}^{1} - \tilde{\boldsymbol{u}}^{2})$. 
By subtracting the weak formulations defining
$(\tilde{\boldsymbol{u}}^{1}, \tilde{\eta}^{1})$ and $(\tilde{\boldsymbol{u}}^{2}, \tilde{\eta}^{2})$
as probabilistically weak solutions, given in Definition \ref{weak}, 
and by using the result of Step 2 above, we obtain that $(\tilde{\boldsymbol{u}}^{1} - \tilde{\boldsymbol{u}}^{2}, \tilde{\eta}^{1} - \tilde{\eta}^{2})$ almost surely satisfies for all $(\boldsymbol{q}, \psi) \in \mathcal{Q}(0, T)$, 
\begin{gather*}
- \int_{0}^{T} \int_{\Omega_{f}} (\boldsymbol{u}_{1} - \boldsymbol{u}_{2}) \cdot \partial_{t}\boldsymbol{q} d\boldsymbol{x} dt + 2\mu \int_{0}^{T} \int_{\Omega_{f}} \boldsymbol{D}(\boldsymbol{u}_{1} - \boldsymbol{u}_{2}) : \boldsymbol{D}(\boldsymbol{q}) d\boldsymbol{x} dt \\
- \int_{0}^{T} \int_{\Gamma} \partial_{t}(\eta_{1} - \eta_{2})\partial_{t}\psi dz dt + \int_{0}^{T} \int_{\Gamma} \nabla (\eta_{1} - \eta_{2}) \cdot \nabla \psi dz dt = 0,
\end{gather*}
with $\tilde{\eta}^{1} - \tilde{\eta}^{2} = 0$ almost surely. Therefore, by using the uniqueness result in Lemma \ref{detunique}, we conclude that $\tilde{\eta}^{1} = \tilde{\eta}^{2}$ and $\tilde{\boldsymbol{u}}^{1} = \tilde{\boldsymbol{u}}^{2}$ almost surely. Since $\tilde{v}^{1} = \partial_{t}\tilde{\eta}^{1}$ and $\tilde{v}^{2} = \partial_{t}\tilde{\eta}^{2}$, we also obtain that $\tilde{v}^{1} = \tilde{v}^{2}$ almost surely. This allows us to conclude that $\tilde{\mathbb{P}}(\tilde{X}^{1} = \tilde{X}^{2})  = 1$, which implies that the
limiting joint probability measure (or law) $\nu$ is supported on the diagonal. 

This completes the verification of the diagonal condition of the Gy\"{o}ngy-Krylov lemma. 

\subsection{Existence of a weak solution in a probabilistically strong sense}\label{final}

The existence of a weak solution in a probabilistically strong sense,  given by Definition~\ref{strong},
now follows from the Gy\"{o}ngy-Krylov lemma in Lemma \ref{GK}. More precisely, by the Gy\"{o}ngy-Krylov lemma, the original sequence 
$
(\eta_{N}, \overline{\eta}_{N}, \eta^{\Delta t}_{N}, \boldsymbol{u}_{N}, v_{N}, \boldsymbol{u}_{N}, v_{N}^{*}, \overline{\boldsymbol{u}}_{N}, \overline{v}_{N}, \boldsymbol{u}^{\Delta t}_{N}, v^{\Delta t}_{N}, W)
$
converges in probability to some random variable
$
(\eta, \overline{\eta}, \eta^{\Delta t}, \boldsymbol{u}, v, \boldsymbol{u}^{*}, v^{*}, \overline{\boldsymbol{u}}, \overline{v}, \boldsymbol{u}^{\Delta t}, v^{\Delta t}, W),
$
where the last component must be $W$ up to a null set, since the limit in probability of any constant sequence is almost surely exactly that constant. 

Since convergence in probability implies almost sure convergence along a subsequence,
we conclude that along a subsequence which we continue to denote by $N$, we have that
\begin{equation}\label{strongconv}
(\eta_{N}, \overline{\eta}_{N}, \eta^{\Delta t}_{N}, \boldsymbol{u}_{N}, v_{N}, \boldsymbol{u}_{N}, v_{N}^{*}, \overline{\boldsymbol{u}}_{N}, \overline{v}_{N}, \boldsymbol{u}^{\Delta t}_{N}, v^{\Delta t}_{N}, W) \to (\eta, \overline{\eta}, \eta^{\Delta t}, \boldsymbol{u}, v, \boldsymbol{u}^{*}, v^{*}, \overline{\boldsymbol{u}}, \overline{v}, \boldsymbol{u}^{\Delta t}, v^{\Delta t}, W), \textit{ a.s. in } \mathcal{X}.
\end{equation}
To show that this limit is a weak solution in the sense of Definition~\ref{strong}, we use the same 
arguments as in Lemma \ref{properties}. All of the properties from Definition~\ref{strong} follow from Lemma \ref{properties},
except for uniqueness and showing that $(\boldsymbol{u}, v, \eta)$ is $\mathcal{F}_{t}$-adapted. 

{\bf{Uniqueness}} follows from the deterministic uniqueness result of Lemma \ref{detunique}.

{\bf{$\mathcal{F}_{t}$-adaptedness of $(\boldsymbol{u}, v, \eta)$:}}
Note that this is not provided by Lemma \ref{properties}, as we want to show that this solution is adapted to the \textit{original} filtration $\{\mathcal{F}_{t}\}_{t \ge 0}$, while the filtration defined in \eqref{tildefiltration} is not necessarily the same filtration. 

To verify this, we note that by construction, $(\boldsymbol{u}_{N}, v_{N}, \eta_{N})$ is adapted to the given complete filtration $\{\mathcal{F}_{t}\}_{t \ge 0}$. We want to pass to the limit as $N \to \infty$. By the convergence in \eqref{strongconv}, 
\begin{gather*}
\boldsymbol{u}_{N} \to \boldsymbol{u}, \qquad \text{ almost surely in } L^{2}(0, T; L^{2}(\Omega_{f})),
\\
v_{N} \to v, \qquad \text{ almost surely in } L^{2}(0, T; L^{2}(\Gamma)),
\\
\eta_{N} \to \eta, \qquad \text{ almost surely in } L^{2}(0, T; L^{2}(\Gamma)).
\end{gather*}
By the same argument used to establish \eqref{almostsureS} for example, we obtain that for a measurable set $S \subset [0, T] \times \Omega$ with $(dt \times \mathbb{P})(S) = T$, 
\begin{gather}
\boldsymbol{u}_{N_{k}}(t, \omega, \cdot) \to \boldsymbol{u}(t, \omega, \cdot) \ \text{ in } L^{2}(\Omega_{f}), \qquad v_{N_{k}}(t, \omega, \cdot) \to v(t, \omega, \cdot) , \ \eta_{N_{k}}(t, \omega, \cdot) \to \eta(t, \omega, \cdot) \ \text{ in } L^{2}(\Gamma)
\label{adaptedconv}
\end{gather}
along a common subsequence $N_{k}$. In particular, $([0, T] \times \Omega) - S$ has measure zero with respect to the product measure $dt \times \mathbb{P}$. 

Define $S_{0} \subset [0, T]$ to be all times $t \in [0, T]$ for which $\mathbb{P}((t, \omega) \in S) = 1$, so that the time slice at time $t$ has full measure in probability. $S_{0}$ is measurable in $[0, T]$ and contains almost every time in $[0, T]$ by Fubini's theorem. So for all $t \in S_{0}$, the convergences \eqref{adaptedconv} are almost sure convergences. 

Because $\{\mathcal{F}_{t}\}_{t \ge 0}$ is a complete filtration by assumption, the almost sure limit of $\mathcal{F}_{t}$-measurable random variables must also be $\mathcal{F}_{t}$-measurable, since $\mathcal{F}_{t}$ contains all null sets of $(\Omega, \mathcal{F}, \mathbb{P})$. So for all $t \in S_{0}$, $\boldsymbol{u}(t)$, $v(t)$, and $\eta(t)$ are $\mathcal{F}_{t}$-measurable since $\boldsymbol{u}_{N_{k}}(t)$, $v_{N_{k}}(t)$, and $\eta_{N_{k}}(t)$ are $\mathcal{F}_{t}$-measurable by construction. 

To show $\boldsymbol{u}(t)$, $v(t)$, and $\eta(t)$ are $\mathcal{F}_{t}$-measurable for $t \notin S_{0}$, we use the fact that $S_{0}$ has full measure in $[0, T]$ and is hence dense. We can assume $t \ne 0$, since at $t = 0$, $(\boldsymbol{u}(0), v(0), \eta(0)) = (\boldsymbol{u}_{0}, v_{0}, \eta_{0})$ almost surely so the result holds. So for $t \notin S_{0}$ and $t \ne 0$, we can construct $t_{n} \in S_{0}$ such that $t_{n} \nearrow t$. By the fact that $(\boldsymbol{u}, v) \in C(0, T; \mathcal{Q}')$ and $\eta$ is Lipschitz continuous almost surely, we have that $(\boldsymbol{u}(t), v(t), \eta(t))$ is the almost sure limit of $(\boldsymbol{u}(t_{n}), v(t_{n}), \eta(t_{n}))$, which are $\mathcal{F}_{t}$-measurable since $\mathcal{F}_{t_{n}} \subset \mathcal{F}_{t}$, as $t_{n} \le t$. This establishes the adaptedness of $(\boldsymbol{u}, v, \eta)$ to the given complete filtration $\{\mathcal{F}_{t}\}_{t \ge 0}$. 

In conclusion, we have now shown that  $(\boldsymbol{u}, v, \eta)$ has all of the required properties needed to be a weak solution in a probabilistically strong sense to the given fluid-structure interaction problem with respect to the Brownian motion $W$ with complete filtration $\{\mathcal{F}_{t}\}_{t \ge 0}$, as in Definition \ref{strong}. This completes the proof of the main result, stated in Theorem~\ref{Main theorem},
and restated here:
\vskip 0.1in

\begin{theorem}[{\bf{Main Result}}]
Let $\boldsymbol{u}_{0} \in L^{2}(\Omega_{f})$, $v_{0} \in L^{2}(\Gamma)$, and $\eta_{0} \in H_{0}^{1}(\Gamma)$. Let $P_{in/out} \in L^{2}_{loc}(0, \infty)$ and let $(\Omega, \mathcal{F}, \mathbb{P})$ be a probability space with a Brownian motion $W$ with respect to a given complete filtration $\{\mathcal{F}_{t}\}_{t \ge 0}$. Then, for any $T > 0$, there exists a unique weak solution in a probabilistically strong sense to the given stochastic fluid-structure interaction problem \eqref{eq}--\eqref{idata}. 
\end{theorem}

\section{Conclusions}
In this manuscript, we presented a constructive proof of the existence of a weak solution in a probabilistically strong sense, 
to a benchmark stochastic fluid-structure interaction (SFSI) problem \eqref{eq}--\eqref{idata}. 
An example of such a problem is the flow of blood in coronary arteries that sit on the surface of the heart
and contract and expand under the outside forcing due to the heart muscle contraction and expansion.
Dynamic patient images show significant stochastic effects in the heart contractions,
which can be captured by an SFSI model such as the one studied in this work. 
Our well-posedness result indicates that stochastic FSI models are robust in the sense that a unique weak solution 
in the sense of Definition~\ref{strong} will exist even when the problem is stochastically forced by
a rough time-dependent white noise, as considered in this work.

In addition to the importance of this work in terms of modeling real-life fluid structure interaction phenomena with stochastic noise, 
to the best of our knowledge the results of this work present a first constructive existence proof of a unique weak solution 
in a probabilistically strong sense
to a stochastically forced and fully coupled FSI problem, as defined in Definition~\ref{strong}.

In contrast to the deterministic case, the proof based on the operator splitting strategy presented in this work
has several new interesting components, which we summarize below.
\begin{enumerate}
\item The energy estimates are given in expectation, and do not necessarily hold pathwise. 
As a consequence, weak precompactness can be deduced only for the {\emph{probability measures}}, or {\emph{laws}} 
associated with the approximate solution sequences, and not the sequences themselves. 
\item The energy estimate has an extra term that accounts for the energy pumped into the problem by the stochastic forcing in expectation. 
\item Solving the subproblems resulting from the operator splitting scheme must be done
in the ``correct'' order to obtain a {\emph{stable scheme}}.
In particular, the order is: (1) the structure subproblem, (2) the stochastic subproblem, and (3) the fluid subproblem.
With this order, we can properly interpret the terms involving time increments of the stochastic forcing as a stochastic integral, due to the measurability properties of the approximate solutions, which allows us to show stability.
\item To establish
weak convergence of probability measures, one must show that the probability measures are {\bf{tight}}, which requires the use of a {\emph{compactness
result}} alla Aubin-Lions, even though the coupled problem is linear.
\item Once weak convergence of the probability measures (laws) associated with the approximate solutions is established, probabilistic techniques
based on the Skorohod representation theorem and the Gy\"ongy-Krylov lemma have to be employed to obtain almost sure convergence along a subsequence to a weak solution. 
\end{enumerate}

\section*{Acknowledgements}
This work was partially supported by the National Science Foundation under grants DMS-1853340 and DMS-2011319. 

\bibliography{linearsplit}
\bibliographystyle{plain}

\end{document}